\documentclass[11pt, letterpaper]{amsart}
\usepackage{amsthm,amssymb,mathrsfs,mathtools,empheq,mathabx}
\usepackage[shortlabels]{enumitem}
\usepackage[T1]{fontenc}
\usepackage{pinlabel}
\usepackage[notref,notcite, final]{showkeys}
\mathtoolsset{showonlyrefs}
\usepackage{color}
\usepackage{geometry}

\newtheorem{mainthm}{Theorem}
\newtheorem{maincor}{Corollary} 
 
\newtheorem{theorem}{Theorem}[section]
\newtheorem*{theorem*}{Theorem}
\newtheorem{corollary}[theorem]{Corollary}

\newtheorem{lemma}[theorem]{Lemma}
\newtheorem{proposition}[theorem]{Proposition}
\newtheorem*{proposition*}{Proposition}

\newtheorem*{conjecture*}{Conjecture}

\theoremstyle{definition}
\newtheorem{definition}[theorem]{Definition}
\newtheorem{remark}[theorem]{Remark}

\numberwithin{equation}{section}

\def\fa {\mathfrak{a}}
\def\fv {\mathfrak{v}}
\def\fb {\mathfrak{b}}
\def\fw {\mathfrak{w}}

\def\bN {\mathbb{N}}

\def\bR {\mathbb{R}}

\def\cE {\mathcal{E}}
\def\cF {\mathcal{F}}

\def\cM {\mathcal{M}}

\def\scrL{\mathscr{L}}

\def\scrU{\mathscr{U}}

\def\grad {{\nabla}}

\def\la {\langle}
\def\ra {\rangle}

\newcommand{\tx}[1]{\mathrm{#1}}
\newcommand{\wto}{\rightharpoonup}

\newcommand{\wt}[1]{\widetilde{#1}}
\newcommand{\bs}[1]{\boldsymbol{#1}}
\newcommand{\conj}[1]{\overline{#1}}
\newcommand{\sh}[1]{#1^\sharp}
\newcommand{\fl}[1]{#1^\flat}

\newcommand{\Id}{\operatorname{Id}}

\newcommand{\eee}{e}

\newcommand{\ud}{\mathrm{\,d}}
\newcommand{\vd}{\mathrm{d}}
\newcommand{\up}{\mathrm{p}}
\newcommand{\uk}{\mathrm{k}}

\newcommand{\vD}{\mathrm{D}}
\newcommand{\dd}[1]{{\frac{\vd}{\vd{#1}}}}

\newcommand{\lin}{\mathrm{L}}

\definecolor{deepgreen}{cmyk}{1,0,1,0.5}

\newcommand{\E}{\mathcal{E}}

\newcommand{\R}{\mathbb{R}}

\newcommand{\al}{\alpha}
\newcommand{\be}{\beta}

\newcommand{\de}{\delta}

\newcommand{\om}{\omega}
\newcommand{\lam}{\lambda}
\newcommand{\Om}{\Omega}

\newcommand{\p}{\partial}

\newcommand{\loc}{\operatorname{loc}}

\makeatletter

\newcommand{\Rmnum}[1]{\expandafter\@slowromancap\romannumeral #1@}
\makeatother

\newcommand{\ti}{\widetilde}

\newcommand{\abs}[1]{\left\lvert{#1}\right\rvert}
\newcommand{\ds}{\displaystyle}

\newcommand{\EQ}[1]{\begin{equation}\begin{split} #1 \end{split}\end{equation}}
\newcommand{\pmat}[1]{\begin{pmatrix} #1 \end{pmatrix}}

\newcommand{\Del}[1]{}

\newcommand{\mand}{{\ \ \text{and} \ \  }}

\newcommand{\mif}{{\ \ \text{if} \ \ }}

\newcommand{\mas}{{\ \ \text{as} \ \ }}

\newcommand{\eps}{\epsilon}

\newcommand{\bfd}{{\bf d}}

\newcommand{\calB}{\mathcal B}

\newcommand{\calE}{\mathcal E}

\newcommand{\calM}{\mathcal M}
\newcommand{\calN}{\mathcal N}

\newcommand{\calT}{\mathcal T}

\newcommand{\calX}{\mathcal X}

\newcommand{\qedno}{
\begin{flushright}
$\boxslash$
\end{flushright}
}

\title[Classification of kink clusters]
{Classification of kink clusters \\ for scalar fields in dimension 1+1}\author{Jacek Jendrej}
\author{Andrew Lawrie}


\thanks{J. Jendrej is supported by ANR-18-CE40-0028 project ``ESSED'' and ERC Starting Grant ``INSOLIT'' 101117126.  A. Lawrie is supported by NSF grant DMS-2247290.}

\begin{document}

\begin{abstract}
We consider a real scalar field equation in dimension $1+1$
with an even, positive self-interaction potential having two
non-degenerate zeros (vacua) $1$ and $-1$.
Such a model
admits non-trivial static solutions called kinks and antikinks.
We define a kink $n$-cluster to be a solution approaching, for large positive times, a superposition of $n$ alternating
kinks and antikinks whose velocities converge to $0$.
They can be equivalently characterized as the solutions of minimal possible energy
containing $n-1$ transitions between the vacua,
or as the solutions whose kinetic energy decays to $0$ in large time.

Our first main result is a determination of the main-order asymptotic behavior of any kink $n$-cluster. The proof relies on a reduction,
using appropriately chosen modulation parameters, to an $n$-body problem with attractive exponential interactions.
We then construct a kink $n$-cluster for any prescribed initial positions of the
kinks and antikinks, provided that their mutual distances are sufficiently large.
Next, we prove that the set of all the kink $n$-clusters is an $n$-dimensional topological manifold,
and we show how it can be parametrized by the positions of the kinks in the configuration.
The proof relies on energy estimates and the contraction mapping principle, using the Lyapunov-Schmidt reduction technique.
Finally, we show that kink clusters are universal profiles for the formation/collapse of multikink configurations.
In this sense, they can be interpreted as forming the stable/unstable manifold of the multikink state given by a superposition of $n$ infinitely separated alternating kinks and antikinks.  

%The proof of the dynamical characterization of kink $n$-clusters  relies on a reduction,
%using appropriately chosen modulation parameters,
%to an $n$-body problem
%with attractive exponential interactions.
\end{abstract}

\maketitle

\setcounter{tocdepth}{1}
\tableofcontents

%-------------------------INTRODUCTION------------------------------------------%

\section{Introduction}
\label{sec:intro}
\subsection{Setting of the problem}
\label{ssec:setting}
We study scalar field equations in dimension $1+1$, which are associated to the Lagrangian action
\begin{equation}
\label{eq:lagrange}
\scrL(\phi) = \int_{-\infty}^{\infty}\int_{-\infty}^{\infty} \Big(\frac 12(\partial_t \phi)^2 - \frac 12 (\partial_x\phi)^2 - U(\phi)\Big)\, \ud  x \ud t,
\end{equation}
where the self-interaction potential $U: \bR \to [0, +\infty)$ is a given smooth function.
The unknown field $\phi = \phi(t, x)$ is assumed to be real-valued.
The resulting Euler-Lagrange equation is
\begin{equation}
\label{eq:csf-2nd}
\partial_t^2 \phi(t, x) - \partial_x^2 \phi(t, x) + U'(\phi(t, x)) = 0, \qquad (t, x) \in \bR\times \bR,\ \phi(t, x) \in \bR.
\end{equation}
We assume that
\begin{itemize}
\item
$U$ is an even function,
\item $U(\phi) > 0$ for all $\phi \in (-1, 1)$,
\item  $U(-1) = U(1) = 0$ and $U''(1) = U''(-1) = 1$.
\end{itemize}
The~zeros of $U$ are called the \emph{vacua}.
Linearization of \eqref{eq:csf-2nd} around each of the vacua
$1$ and $-1$, $\phi = \pm 1 + g$, yields the free linear Klein-Gordon equation of mass $1$:
\begin{equation}
\label{eq:free-kg}
\partial_t^2 g_{\lin}(t, x) - \partial_x^2 g_{\lin}(t, x) + g_{\lin}(t, x) = 0.
\end{equation}

Two well-known examples of~\eqref{eq:csf-2nd} satisfying the hypotheses above are the \emph{sine-Gordon equation}
\begin{equation}
\label{eq:sg} 
\partial_t^2 \phi(t, x) - \partial_x^2 \phi(t, x) -\frac{1}{\pi}\sin(\pi\phi(t, x)) = 0,
\end{equation}
where we have taken $U(\phi) = \frac{1}{\pi^2}\big(1+ \cos(\pi\phi)\big)$, and the \emph{$\phi^4$ model}
\begin{equation}
\label{eq:phi4-m} 
\partial_t^2 \phi(t, x) - \partial_x^2 \phi(t, x) - \frac 12\phi(t, x) + \frac 12\phi(t, x)^3 = 0,
\end{equation}
for which $U(\phi) = \frac{1}{8}(1- \phi^2)^2$.

The equation \eqref{eq:csf-2nd} can be rewritten as a system of first order in $t$:
\begin{equation}
\label{eq:csf-1st}
\partial_t \begin{pmatrix} \phi(t, x) \\ \dot\phi(t, x)\end{pmatrix} = \begin{pmatrix} \dot\phi(t, x) \\ \partial_x^2 \phi(t, x) - U'(\phi(t, x))\end{pmatrix}.
\end{equation}
We denote $\bs \phi_0 = (\phi_0, \dot \phi_0)^\tx T$ an element of the phase space
(in the sequel, we omit the transpose in the notation).
The potential energy $E_p$, the kinetic energy $E_k$ and the total energy $E$
of a state are given by
\begin{align}
E_{\up}(\phi_0) &= \int_{-\infty}^{+\infty}\Big(\frac 12 (\partial_x\phi_0(x))^2 + U(\phi_0(x))\Big)\ud x, \\
E_{ \uk}(\dot \phi_0) &= \int_{-\infty}^{+\infty}\frac 12( \dot\phi_0(x))^2 \ud x, \\
E(\bs\phi_0) &= \int_{-\infty}^{+\infty}\Big(\frac 12(\dot \phi_0(x))^2+\frac 12 (\partial_x\phi_0(x))^2 + U(\phi_0(x))\Big)\ud x.
\end{align}
Denoting $\bs \phi(t, x) \coloneqq (\phi(t, x), \dot \phi(t, x))$, the system \eqref{eq:csf-1st} can be reformulated in the Hamiltonian form as
\begin{equation}
\label{eq:csf}
\partial_t \bs \phi(t) = \bs J\vD E(\bs \phi(t)),
\end{equation}
where $\bs J \coloneqq \begin{pmatrix}0 & 1 \\ {-1} & 0 \end{pmatrix}$ is the standard symplectic form
and $\vD$ is the Fr\'echet derivative for the $L^2\times L^2$ inner product.
In particular, $E$ is a conserved quantity, and we denote $E(\bs \phi)$ the energy of a solution $\bs \phi$ of \eqref{eq:csf}.

The set of finite energy states $\bs \phi_0 = (\phi_0, \dot \phi_0)$ contains the following affine spaces:
\begin{equation}
\begin{aligned}
\cE_{1, 1} &\coloneqq \{(\phi_0, \dot \phi_0): E(\phi_0, \dot \phi_0) < \infty\ \text{and}\ \lim_{x \to -\infty}\phi_0(x) = 1, \lim_{x \to \infty}\phi_0(x) = 1\}, \\
\cE_{-1, -1} &\coloneqq \{(\phi_0, \dot \phi_0): E(\phi_0, \dot \phi_0) < \infty\ \text{and}\ \lim_{x \to -\infty}\phi_0(x) = -1, \lim_{x \to \infty}\phi_0(x) = -1\}, \\
\cE_{1, -1} &\coloneqq \{(\phi_0, \dot \phi_0): E(\phi_0, \dot \phi_0) < \infty\ \text{and}\ \lim_{x \to -\infty}\phi_0(x) = 1, \lim_{x \to \infty}\phi_0(x) = -1\}, \\
\cE_{-1, 1} &\coloneqq \{(\phi_0, \dot \phi_0): E(\phi_0, \dot \phi_0) < \infty\ \text{and}\ \lim_{x \to -\infty}\phi_0(x) = -1, \lim_{x \to \infty}\phi_0(x) = 1\}.
\end{aligned}
\end{equation}
In the case of the $\phi^4$ model, these are all the finite-energy states,
but in general there can be other states of finite energy,
for example if $U$ has other vacua than $1$ and $-1$. 
Here, the states which we consider will always belong to one of the four affine spaces listed above, which are all parallel to the space 
\EQ{
\calE \coloneqq  H^1 \times L^2. 
}

Equation \eqref{eq:csf} admits static solutions. They are the critical points
of the potential energy. The trivial ones are the vacuum fields $\phi(t, x) = \pm 1$.
The solution $\phi(t, x) = 1$ (resp. $\phi(t, x) = -1$) has zero energy and is the ground state in $\cE_{1, 1}$
(resp. $\cE_{-1, -1}$).

There are also non-constant static solutions $\phi(t, x)$
connecting the two vacua, that is
\begin{equation}
\label{eq:static-nn'}
\lim_{x \to -\infty}\phi(t, x) = \mp 1, \quad \lim_{x \to \infty}\phi(t, x) = \pm 1.
\end{equation}
One can describe all these solutions.
There exists a unique increasing odd function $H: \bR \to (-1, 1)$ such that
all the solutions of \eqref{eq:static-nn'} are given by $\phi(t, x) = \pm H(x - a)$ for some $a \in \bR$.
The basic properties of $H$ are given in Section~\ref{sec:static}.
Its translates are called the \emph{kinks} and are the ground states in $\cE_{-1, 1}$.
The translates of the function $-H$ are called the \emph{antikinks} and are the ground states in $\cE_{1, -1}$.

%The condition \eqref{eq:static-nn'} defines a \emph{topological class}, since for any continuous path
%of finite energy states either none or all of them satisfy \eqref{eq:static-nn'}.

%The condition defines a \emph{topological class}, since for any continuous path of finite energy states either none or all of them satisfy. In general, minimizers of the energy in a topological class that does not contain vacua are called topological solitons. Topological solitons were introduced in the physics literature by Skyrme as candidates for particles in a classical field theory; see ~\cite{Skryme 62, MS}. Kinks and antikinks are one dimensional examples, and in higher dimensions examples include vortices, harmonic maps,  monopoles, Skyrmions, and instantons. In this context, it is natural to investigate to what extent multikinks (defined below) can be effectively described as a system of interacting point particles. In other words, can their dynamical behavior be captured by means of a small number of parameters, for instance their positions and momenta.

The condition~\eqref{eq:static-nn'} defines a \emph{topological class}, since for any continuous path of finite energy states, either all or none of them satisfy~\eqref{eq:static-nn'}. In general, minimizers of the energy in a topological class that does not contain vacua are called topological solitons. Topological solitons were introduced in the physics literature by Skyrme as candidates for particles in classical field theories; see~\cite{Skyrme, MS}. Kinks and antikinks are one dimensional examples, and in higher dimensions examples include vortices, harmonic maps,  monopoles, Skyrmions, and instantons. In this context, it is natural to investigate to what extent multikinks (defined below) can be effectively described as a system of interacting point particles. In other words, can their dynamical behavior be captured by means of a small number of parameters, for instance their positions and momenta. 

The Lagrangian~\eqref{eq:lagrange} is invariant under the Poincar\'e group, from which we see that applying a Lorentz boost to the stationary kink
$\bs\phi(t) = (H(x - a), 0)$, one obtains a traveling kink.
Given $(a, v) \in \bR \times (-1, 1)$, we denote
\begin{equation}
\label{eq:Hav}
H(a, v; x) \coloneqq  H(\gamma_v(x - a)), \qquad \gamma_v \coloneqq  (1 - v^2)^{-\frac 12}
\end{equation}
and
\begin{equation}
\label{eq:bsHav}
\begin{aligned}
\bs H(a, v; x) &\coloneqq  \big( H(a, v; x), -v\partial_x H(a, v; x) \big) \\
&= \big(H(\gamma_v(x - a)), -v\gamma_v\partial_x H(\gamma_v(x - a))\big).
\end{aligned}
\end{equation}
Then $\bs \phi(t) = \bs H(a + vt, v)$ is a solution of~\eqref{eq:csf},
in other words
\begin{equation}
\label{eq:Hv-id}
-v\partial_x \bs H(a, v) = v\partial_a \bs H(a, v) = \bs J\vD E(\bs H(a, v)).
\end{equation}
%Note the following identity, which will be used several times:
%\begin{equation}
%\label{eq:dadx-H}
%\partial_a \bs H(a, v; x) = -\partial_x \bs H(a,v; x).
%\end{equation}
The \emph{momentum} of a state $\bs \phi_0 = (\phi_0, \dot\phi_0)$
is given by
\begin{equation}
P(\bs\phi_0) = -\int_{-\infty}^\infty \dot \phi_0(x)\partial_x \phi_0(x) \ud x = -\frac 12 \la \partial_x \bs \phi_0, \bs J\bs \phi_0\ra. 
\end{equation}
It is well defined for every state of finite energy and is a conserved quantity. Using \eqref{eq:equipartition},
it is easy to compute  that
\begin{align}
\label{eq:energy-Hv}
E(\bs H(a, v)) &= \gamma_v M, \\
\label{eq:momentum-Hv}
P(\bs H(a, v)) &= \gamma_v Mv.
\end{align}
A traveling kink is thus similar to a relativistic point particle
of mass $M$ and velocity $v$, and the formulas above
are the usual formulas of Special Relativity.
%and  note the identities, 
%\EQ{ \label{eq:ab-dxda} 
%\p_x \bs \al ( a, v) = - \p_a \bs \al(a, v), \quad \p_x \bs \beta ( a, v) = - \p_a \bs \beta(a, v). 
%}
%Since $\bs J$ is skew-symmetric, we have
%\begin{equation}
%\label{eq:identities-intro}
%\la \bs \alpha(a, v), \partial_a \bs H(a, v)\ra =
%\la \bs \beta(a, v), \partial_v \bs H(a, v)\ra = 0.
%\end{equation}

%In general, minimizers of the energy in a given topological class are referred to as \emph{topological solitons},
%see \cite[Section 4.1]{MS}.
%Kinks and antikinks are among the simplest examples of topological solitons (they are one-dimensional) and this perhaps explains why the wave equation \eqref{eq:csf-2nd} is widely studied both as a model problem in physics and due to its own merit as an interesting and challenging mathematical problem. For example, the question of nonlinear stability of the kink for the $\phi^4$-model~\eqref{eq:phi4-m} is classical, but still open for general smooth perturbations; see the work of the second author with Martel and Mu\~noz~\cite{KMM} where stability of the $\phi^4$ kink was proved under odd perturbations,
%as well as \cite{Buslaev-Perelman, SW99, Delort-Masmoudi, Germain-Pusateri, BKMM, MR2770013, MR2835867, Mizumachi} and the survey \cite{Cuccagna-survey} for related results.
%On the mathematical physics side, we refer the reader to~\cite{PhysRevD.20.3120, MS}
%%, \cite{shnir_2018}
%and the references therein for specific examples and their motivations. 
%\red{WORK ON THIS PARAGRAPH}

\subsection{The notion of a kink cluster} 

By the variational characterization of $H$ and its translates as the ground states in $\cE_{-1, 1}$,
one can view them as the transitions between the two vacua $-1$ and $1$
having the minimal possible energy $E = E_p(H)$.
Given a natural number $n$, we are interested in solutions of \eqref{eq:csf-2nd} containing,
asymptotically as $t \to \infty$, $n$ such transitions.
Since energy $E_p(H)$ is needed for each transition, we necessarily have $E(\bs \phi) \geq nE_p(H)$. We call \emph{kink clusters} the solutions for which equality holds.

\begin{definition}[Kink $n$-cluster]
\label{def:cluster}
Let $n \in \{0, 1, \ldots\}$.
We say that a solution $\bs\phi$ of \eqref{eq:csf} is a \emph{kink $n$-cluster} if $E(\bs\phi) \leq nE_p(H)$ and there exist real-valued functions
$x_0(t) \leq x_1(t) \leq \ldots \leq x_n(t)$ such that
 $$\lim_{t\to\infty}\phi(t, x_k(t)) = (-1)^k\qquad\text{for all }k \in \{0, 1, \ldots, n\}.$$
\end{definition}
Note that the kink $0$-clusters are the constant solutions $\phi \equiv 1$ and the kink $1$-clusters are the antikinks.
We say that $\phi$ is a kink cluster if it is a kink $n$-cluster for some $n \in \{0, 1, \ldots\}$.

From the heuristic discussion above, the shape of each transition
in a kink cluster has to be close to optimal,
that is close to a kink or an antikink.
Before we give a precise statement of this fact,
we introduce the so-called \emph{multikink configurations}.

 Let $\vec a = (a_1, \ldots, a_n) \in \bR^n$ be the positions of the transitions
and $\vec v = (v_1, \ldots, v_n) \in (-1, 1)^n$ the Lorentz parameters
(following \cite[Chapter 5]{MS}, we use the letter
$a$ for the translation parameter; it should not be confused
with ``acceleration'' which will be given no symbol in this paper).
We always assume $a_1 \leq a_2 \leq \ldots \leq a_n$.
It will be convenient to abbreviate $\bs H_k \coloneqq  \bs H(a_k, v_k)$,
$\partial_a \bs H_k \coloneqq  \partial_a \bs H(a_k, v_k)$, $\partial_v \bs H_k \coloneqq  \partial_v \bs H(a_k, v_k)$
for $k \in \{1, \ldots, n\}$. We also denote $\bs 1 \coloneqq  (1, 0)$ and
\begin{align}
\label{eq:Ha-def}
\bs H(\vec a, \vec v; x) \coloneqq  \bs 1 + \sum_{k=1}^{n} (-1)^k (\bs H_k(x) + \bs 1).
\end{align}
For example, if $n = 0$, then $\vec a, \vec v$ are empty vectors and $\bs H(\vec a, \vec v)$ is the vacuum $\bs 1$.
If $n = 1$, $\vec a = (a_1)$ and $\vec v = (v_1)$, then $\bs H(\vec a, \vec v) = -\bs H(a_1, v_1)$ is an antikink.
If $n = 2$, $\vec a = (a_1, a_2)$ with $a_2 - a_1 \gg 1$ and $\vec v = (v_1, v_2)$, then
$\bs H(\vec a, \vec v) = \bs 1 - \bs H(a_1, v_1) + \bs H(a_2, v_2)$ has the shape of an antikink near $x = a_1$,
and of a kink near $x = a_2$. These are the \emph{kink-antikink pairs}, which we studied with Kowalczyk in \cite{JKL1}.

Employing the term introduced by Martel and Rapha\"el in \cite{MaRa18},
we  study multikinks in the regime of \emph{strong interactions}, so we will always assume that Lorentz parameters $\vec v$ are small, and the positions of the kinks $\vec a$ are well-separated. We weigh these requirements as follows. 
\begin{definition}[Weight of modulation parameters]
For all $(\vec a, \vec v) \in \bR^n\times \bR^n$, we set
\begin{equation}
\label{eq:rhoa-def}
\rho(\vec a, \vec v) \coloneqq  \sum_{k=1}^{n-1} \eee^{-(a_{k+1} - a_k)} + \sum_{k=1}^n v_k^2.
\end{equation}
%For $(\vec a, \vec v): I \to \bR^n\times \bR^n$, we define $\rho : I \to (0, \infty)$ by
%\begin{equation}
%\label{eq:rhot-def}
%\rho(t) \coloneqq  \rho(\vec a(t), \vec v(t)) = \sum_{k=1}^{n-1} \eee^{-(a_{k+1}(t) - a_k(t))} + \sum_{k=1}^n v_k(t)^2.
%\end{equation}
\end{definition}

With these notations we define a distance to the family of $n$-kink configurations. Given $\bs \phi \in \E_{1, (-1)^n}$, let 
\EQ{ \label{eq:bfd-def} 
\bfd( \bs \phi) \coloneqq  \inf_{(\vec a, \vec v)} \Big( \| \bs \phi - \bs H(\vec a, \vec v) \|_{\E}^2 + \rho(\vec a, \vec v) \Big).
}

%For $\vec a = (a_1, \ldots, a_n)$ such that $a_1 \leq \ldots \leq a_n$, we denote
%\begin{equation}
%H(\vec a) \coloneqq 1 + \sum_{k=1}^{n} (-1)^k\big(H(\cdot - a_k) + 1\big)
%\end{equation}
%(we chose the ``additive ansatz'', see \cite[Section 1.7]{vachaspati}
%for a comparison with a different ``product ansatz'', which we could also use
%without introducing any changes in the statement of our results below).
\begin{proposition}
\label{prop:close-to-H}
A solution $\bs\phi$ of \eqref{eq:csf} is a kink $n$-cluster if and only if %there exist continuous functions $a_1, \ldots, a_n: \bR \to \bR$ and $v_1, \dots v_n: \R \to (-1, 1)$ 
%such that
%\begin{equation}
%\label{eq:close-to-H}
%\begin{aligned}
%\lim_{t \to \infty}\Big(\big\|\bs \phi(t) - H(\vec a(t), \vec v(t))\big\|_{\E}^2 + \sum_{k=1}^{n-1}\eee^{-(a_{k+1}(t) - a_k(t))} + \sum_{ k =1}^n | v_k(t)|^2\Big) = 0, 
%\end{aligned}
%\end{equation}
%that is, if and only if 
\begin{equation}
\label{eq:d-conv-0}
\lim_{t \to \infty} \bfd( \bs \phi(t)) = 0.
\end{equation} 
\end{proposition}
In other words, kink clusters can be equivalently defined as solutions approaching,
as $t \to \infty$, a superposition of a finite number of alternating kinks and antikinks,
whose mutual distances tend to $\infty$ and which travel with speeds converging to $0$.
In contrast to multikink solutions consisting of Lorentz-boosted kinks (traveling
with asymptotically non-zero speed) constructed in \cite{CJ2},
the dynamics of kink clusters are driven solely by interactions between the kinks and antikinks -- this is the regime of \emph{strong interactions}.
Proposition~\ref{prop:close-to-H} is proved in Section~\ref{ssec:close-to-H}.
It implies in particular that the energy of a kink $n$-cluster equals $nE_p(H)$.

In Section~\ref{ssec:asym-stat}, we provide another characterization of kink clusters, namely as \emph{asymptotically
static solutions}, by which we mean solutions whose kinetic energy
converges to $0$ as $t \to \infty$. In celestial mechanics such solutions are referred to as the~\emph{parabolic motions} of the system, see for instance \cite[Section 2.4]{AKN-Celestial} and~\cite{MaVe09, Saari2}. 
For simplicity, we restrict our attention
to the $\phi^4$ self-interaction potential $U(\phi) \coloneqq \frac 18(1-\phi^2)^2$.
\begin{proposition}
\label{prop:asym-stat}
Let $U(\phi) \coloneqq \frac 18(1-\phi^2)^2$.
A solution $\bs\phi$ of \eqref{eq:csf} satisfies $\lim_{t\to \infty}\|\partial_t \phi(t)\|_{L^2}^2 = 0$ if and only if $\bs \phi$ or
$-\bs\phi$ is a kink cluster.
\end{proposition}

Given a solution $\bs \phi(t)$ of \eqref{eq:csf} that is close to an $n$-kink configuration at time $t$, i.e., with $\bfd( \bs \phi(t))$ sufficiently small, we can decompose it as
\EQ{
\bs \phi(t) = \bs H(\vec a(t), \vec v(t)) + \bs h(t) 
}
with $\|\bs h(t) \|_{\E}^2 + \rho(\vec a(t), \vec v(t)) \lesssim \bfd(\bs \phi(t))$. 
This decomposition is clearly not unique, but we will make a specific choice of \emph{uniquely determined parameters} enforcing a convenient choice of orthogonality conditions, given by the following standard lemma, which is proved in Lemma~\ref{lem:static-mod} in Section~\ref{sec:mod}. % proof is deferred to Section~\ref{sec:proofs} (see Lemma~\ref{lem:mod-av} and Remark~\ref{rem:mod}).
We denote
\begin{equation} \label{eq:alpha-beta-def}
\bs\alpha(a, v) \coloneqq  \bs J\partial_a \bs H(a, v), \qquad \bs\beta(a, v) \coloneqq  \bs J\partial_v \bs H(a, v),
\end{equation}
and use the shorthand $\bs \alpha_k \coloneqq  \bs\alpha(a_k, v_k)$ and $\bs\beta_k \coloneqq  \bs\beta(a_k, v_k)$.

\begin{lemma} [Modulation parameters] \label{lem:mod-av-intro} %\emph{\cite[Lemma 4.2]{JL9}}
There exist $\eta_0, \eta_1, C_0>0$ with the following properties. For all $\bs \phi_0 \in \E_{1, (-1)^n}$  such that 
\EQ{
\bfd( \bs \phi_0) < \eta_0,
}
 there exist unique $(\vec a, \vec v) = (\vec a( \bs \phi_0), \vec v(\bs \phi_0)) \in \R^n \times \R^n$  with $a_1 \le a_2 \le \dots \le a_n$ such that 
\EQ{
 \big\| \bs \phi_0 - \bs H( \vec a, \vec v) \big\|_{\E}^2 + \rho(\vec a, \vec v) < \eta_1, 
}
and 
\EQ{ \label{eq:ortho-av-d} 
\la \bs \phi_0 - \bs H(\vec a , \vec v), \, \bs \alpha(a_k, v_k) \ra = 
\la \bs \phi_0 - \bs H(\vec a, \vec v) ,\, \bs \beta(a_k, v_k) \ra &= 0, 
}
for all $k  = 1, \dots, n$. The pair $(\vec a, \vec v)$ satisfies the estimates
\EQ{ \label{eq:av-d} 
 \big\| \bs \phi_0 - \bs H( \vec a, \vec v) \big\|_{\E}^2 &+ \rho(\vec a, \vec v) < C_0  \bfd( \bs \phi_0) .
} 
Finally, the map $ \E_{1, (-1)^n} \owns \bs \phi_0 \mapsto  (\vec a( \bs \phi_0), \vec v(\bs \phi_0)) \in \R^n \times \R^n$ is of class $C^1$. 

\end{lemma}

% which we denote by  $$\vec a(t)= \vec a( \bs \phi(t)),\quad  \vec v(t)) = \vec v( \bs \phi(t))$$that are \emph{determined uniquely} by requiring the orthogonality conditions
%\EQ{
%\la \bs h(t), \, \bs \al( a_k(t), v_k(t)) \ra = \la \bs h(t) ,\, \bs \beta( a_k(t), v_k(t)) \ra = 0
%}
%for each $k =1, \dots, n$ (see Lemma~\ref{lem:mod-av}, and in fact Remark~\ref{rem:mod}, for a precise statement). 

We will refer  to   $a_k(\bs \phi(t))$ as the \emph{position} of the $k$th kink and $v_k( \bs \phi(t))$ as its \emph{velocity} at time $t$.

\subsection{Main results}
\label{ssec:results}
%It is based on coercivity properties of the kinks,
%related to the so-called \emph{Bogomolny trick}.
 By Proposition~\ref{prop:close-to-H} and Lemma~\ref{lem:mod-av-intro}, for any kink $n$-cluster  $\bs \phi$, the modulation parameters $(\vec a(\bs \phi(t)), \vec v( \bs \phi(t)))$ are defined for all sufficiently large times. The first main result, which is is crucial to the other conclusions in this section, is the determination
of the leading order dynamics of the positions and velocities of the kinks. 

%\begin{mainthm}[Asymptotic dynamics of kink $n$-clusters]
%\label{thm:asymptotics}
%Let $\kappa > 0$ be given by Proposition~\ref{prop:prop-H}
%and $M \coloneqq E_p(H)$.
%If $\bs\phi$ is a kink $n$-cluster, then there exist
%continuously differentiable functions $a_1, \ldots, a_n: \bR \to \bR$
%such that $g(t) \coloneqq \phi(t) - H(\vec a(t))$ satisfies
%\begin{equation}
%\begin{aligned}
%\lim_{t\to\infty}\bigg(&\max_{1\leq k < n}\bigg|\big(a_{k+1}(t) - a_k(t)\big) - \Big(2\log(\kappa t) - \log\frac{Mk(n-k)}{2}\Big)\bigg| \\
%+ &\max_{1\leq k\leq n}|ta_k'(t) + (n+1 - 2k)|
%+ t\|\partial_t g(t)\|_{L^2} + t\|g(t)\|_{H^1}\bigg) = 0.
%\end{aligned}
%\end{equation}
%\end{mainthm}

Define the number $\kappa>0$ by
\EQ{ \label{eq:kappa-def} 
\kappa\coloneqq  \exp\Big( \int_0^1 \big( \frac{1}{ \sqrt{2U(y)}} - \frac{1}{1 - y} \Big) \, \ud y  \Big).
}

 \begin{mainthm}[Asymptotic dynamics of kink $n$-clusters]  \label{thm:asymptotics}
Let $\kappa > 0$ be given by~\eqref{eq:kappa-def} 
and $M \coloneqq E_p(H)$.
Let $\bs\phi$ be a kink $n$-cluster and let $(\vec a(t), \vec v(t)) = (\vec a( \bs \phi(t)), \vec v( \bs \phi(t))$ be the positions and velocities of the kinks given by Lemma~\ref{lem:mod-av-intro}. Then, 
\begin{equation} \label{eq:JL9} 
\begin{aligned}
\lim_{t\to\infty}\bigg(&\max_{1\leq k < n}\bigg|\big(a_{k+1}(t) - a_k(t)\big) - \Big(2\log(\kappa t) - \log\frac{Mk(n-k)}{2}\Big)\bigg| \\
+ &\max_{1\leq k\leq n}|tv_k(t) + (n+1 - 2k)|
 + t\| \bs \phi(t) - \bs H( \vec a(t), \vec v(t))\|_{H^1 \times L^2}\bigg) = 0.
\end{aligned}
\end{equation}
\end{mainthm}

%The decomposition $\phi(t) = H(\vec a(t)) + g(t)$ used in the statement above
%is clearly not unique. We will use a specific choice of $\vec a(t)$ determined
%by the \emph{orthogonality conditions}
%\begin{equation}
%\label{eq:g-orth-intro}
%\int_{-\infty}^\infty\partial_x H(x - a_k(t)) g(t, x)\ud x = 0, \qquad\text{for all }k \in \{1, \ldots, n\}.
%\end{equation}
%This way, whenever $\phi(t)$ is close to a multikink configuration,
%the uniquely determined number $a_k(t)$ indicates the ``position'' of the $k$-th kink.

Our next result concerns the problem of existence of kink $n$-clusters.
We prove that, for any choice of $n$ points on the line sufficiently distant from each other,
there exists a kink $n$-cluster such that the initial positions of the (anti)kinks
are given by the $n$ chosen points.
The result is inspired by the work
of Maderna and Venturelli~\cite{MaVe09} on the Newtonian $n$-body problem.
%Before we give the precise statement, we introduce the following notion.
%\begin{definition}[Distance to a multikink configuration]
%\label{def:d-def}
%For all $\bs \phi_0 \in \cE_{1, (-1)^n}$, the distance from $\bs \phi_0 = (\phi_0, \dot\phi_0)$
%to the set of multikink configurations is defined by
%\begin{equation}
%\label{eq:d-def}
%\bfd(\bs \phi_0) \coloneqq \inf_{\vec b \in \bR^n}\Big(\|\dot \phi_0\|_{L^2}^2 + \|\phi_0 - H(\vec b)\|_{H^1}^2 + \sum_{k=1}^{n-1} \eee^{-(b_{k+1} - b_k)}\Big).
%\end{equation}
%\end{definition}
%Note that, by Proposition~\ref{prop:close-to-H}, if $\bs\phi$ is a kink $n$-cluster,
%then $\lim_{t\to\infty}\bfd(\bs\phi(t)) = 0$.
%We stress that the position parameters determined by the orthogonality
%conditions do not necessarily achieve the infimum above,
%but they do achieve it up to a constant, see Lemma~\ref{lem:static-mod}.
\begin{mainthm}[Existence of kink $n$-clusters]
\label{thm:any-position}
There exist $C_0, L_0 > 0$ such that the following is true.
If $L \geq L_0$ and $\vec a_0 = (a_{0, 1}, \dots, a_{0, n})   \in \bR^n$ satisfies $a_{0, k+1} - a_{0, k} \geq L$ for all
$k \in \{1, \ldots, n-1\}$, then there exists
a solution $\bs \phi(t)$ of \eqref{eq:csf} satisfying:
\begin{itemize}
\item $\bfd(\bs \phi(t)) \leq C_0/(e^L + t^2)$ for all $t \geq 0$; in particular, $\bs \phi$ is a kink $n$-cluster;
\item  $\vec a( \bs \phi(0)) = \vec a_0$, where $\vec a(\bs \phi(0))$ is given by Lemma~\ref{lem:mod-av-intro}.
\end{itemize}
%$\bs g_0 = (g_0, \dot g_0) \in \cE$ and $\vec v_0 \in  (-1, 1)^n$ satisfying $\|\bs g_0\|_{\cE}^2 \leq C_0 \eee^{-L}$ and the orthogonality conditions
%\begin{equation}
%\label{eq:g0-orth}
%\int_{-\infty}^\infty \partial_x H(x - a_{0, k})g_0(x)\ud x = 0 \qquad\text{for all }k \in \{1, \ldots, n\}
%\end{equation}
%\EQ{
%0 = \la \bs g_0,\, \bs \alpha( a_{0, k}, v_{0, k}) \ra = \la \bs g_0,\, \bs \beta( a_{0, k}, v_{0, k}) \ra,
%}
%such that the solution of \eqref{eq:csf} corresponding to the initial data
%\begin{equation}
%\bs\phi(0) = \bs \phi_0 \coloneqq \big(H(\vec a_0, \vec v_0) + g_0, \dot g_0\big)
%\end{equation}
%is a kink cluster and satisfies $\bfd(\bs \phi(t)) \leq C_0/(e^L + t^2)$ for all $t \geq 0$.
\end{mainthm}
%\begin{remark}
%In the statement above, we have in particular $\lim_{t\to\infty} \bfd(\bs\phi(t)) = 0$,
%in other words $\bs \phi$ is a kink $n$-cluster.
%\end{remark}

\begin{remark}
In the case of the sine-Gordon equation \eqref{eq:sg}, which is completely integrable,
it is in principle possible to obtain explicit kink clusters.
For $n \in \{2, 3\}$, such examples of kink clusters were given in \cite{Tomasz}. Nevertheless, Theorems~\ref{thm:asymptotics}
and \ref{thm:any-position} are new even for the sine-Gordon equation.
\end{remark}

We expect that for given widely separated initial positions $\vec a_0$, there is actually a \emph{unique}
choice of $\vec v_0$ and $\bs g_0$ (small respectively in $\bR^n$ and $\cE$) which leads to a kink $n$-cluster.
This is clearly true for $n = 1$, since in this case we necessarily have $\vec v_0 = \bs g_0 = 0$. For $n = 2$,
uniqueness of $\bs g_0$ can be obtained as
a consequence of our work with Kowalczyk~\cite{JKL1}.
 Our next result shows that uniqueness does hold for certain initial position vectors $\vec a_0$. Indeed, we prove that kink $n$-clusters can be parameterized uniquely (and continuously) by  \emph{only the positions} of the kinks, at least for initial configurations that are sufficiently close to the leading order asymptotics from Theorem~\ref{thm:asymptotics}. %That is, given an initial position vector $\vec a_0 = a(t_0)$ 

\begin{mainthm}[Classification of kink $n$-clusters] \label{thm:main} There exist $\tau_0, \eta_0 >0$ such that for each $t_0 \ge \tau_0$ and each $\vec a_0 \in \bR^n$ with 
\EQ{ \label{eq:a0-initial} 
\max_{1\leq k < n}\bigg|\big(a_{k+1, 0} - a_{k, 0}\big) - \Big(2\log(\kappa t_0) - \log\frac{Mk(n-k)}{2}\Big)\bigg| < \eta_0,
}
there exists a \emph{unique} %initial velocity vector $\vec v_0 \in \R^n$ and a 
 solution  $\bs \phi(\vec a_0, t_0; t)$ of~\eqref{eq:csf} satisfying: 
\begin{itemize} 
\item $\bfd(\bs \phi(\vec a_0, t_0;  t)) < \eta_0$ for all $t \ge t_0$ and $\lim_{t \to \infty} \bfd( \bs \phi(\vec a_0, t_0; t)) = 0$;
\item  $\vec a( \bs \phi(a_0, t_0; t_0) = \vec a_0$.% and  
%\EQ{
% \la \bs \phi(t) - \bs H( \vec a(t), \vec v(t)) ,\, \bs \alpha(a_k(t), v_k(t)) \ra &= 0 \\
%   \la \bs \phi(t) - \bs H( \vec a(t), \vec v(t)) ,\, \bs \beta(a_k(t), v_k(t)) \ra  & =  0 , \quad \forall k, \quad \forall t \ge t_0.
%} 
%and such that 
%\EQ{
% \| \bs \phi(t) - \bs H(\vec a(t), \vec v(t)) \|_{\E}^2 + \rho(\vec a(t), \vec v(t)) < \eta_0, 
%}
%for all $t \ge t_0$.  
\end{itemize} 
\end{mainthm}

\begin{remark} By Theorem~\ref{thm:asymptotics}, the trajectories $(\vec a(t), \vec v(t))$ associated to every kink $n$-cluster $\bs \phi$ eventually satisfy~\eqref{eq:a0-initial}, meaning that for all $t_0$ sufficiently large (depending on $\bs \phi$), $\vec a_0 = \vec a( \bs \phi(t_0))$ satisfies~\eqref{eq:a0-initial}. 
\end{remark}

As a by product of the proof of Theorem~\ref{thm:main}, we obtain the following refined information about the dynamics of $\bs \phi( \vec a_0, t_0; t)$.

\begin{maincor}[Refined dynamics]\label{cor:estimates} Let $\tau_0, \eta_0$ be as in Theorem~\ref{thm:main}, $t_0 \ge \tau_0$, let $\vec a_0$ be as in~\eqref{eq:a0-initial}, let $\bs \phi( \vec a_0, t_0)$ be the kink $n$-cluster given by Theorem~\ref{thm:main}, and let  $(\vec a(t), \vec v(t)) = ( \vec a( \bs \phi(\vec a_0, t_0; t)), \vec v( \bs \phi(\vec a_0, t_0; t))$. Then
%$\bfd( \vec \phi(t)) \le \de_0$ for all $t \ge t_0$ with  
\EQ{ \label{eq:h-decay-thm} 
\lim_{t \to \infty} \Big( t^\gamma \| \bs \phi(\vec a_0, t_0; t) - \bs H(\vec a(t), \vec v(t)) \|_{\E} \Big)  =0
}
for each $\gamma<2$. 
In addition to \eqref{eq:JL9}, the parameters $(\vec a(t), \vec v(t))$  satisfy 
%\EQ{ \label{eq:traj-thm} 
%\lim_{t\to\infty}\bigg(&\max_{1\leq k < n}\bigg|\big(a_{k+1}(t) - a_k(t)\big) - \Big(2\log(\kappa t) - \log\frac{Mk(n-k)}{2}\Big)\bigg|  \\
%&+ \max_{1\leq k\leq n}|tv_k(t) + (n+1 - 2k)| \bigg) = 0
%}
%and 
\EQ{ \label{eq:traj'-thm} 
 \lim_{t \to \infty} t^2 \Big(| \vec a\, '(t)  - \vec v(t)| + | \vec v\,'(t)|\Big) < \infty.
}
%for all $t \ge t_0$. 
\end{maincor} 

%\subsubsection{Kink clusters as an invariant manifold} 

The kink $n$-clusters can be interpreted as forming the stable manifold of the asymptotic state given by infinitely separated multikink configurations. (Note that the unstable manifold can be obtained by applying the time reversal symmetry.) This topological manifold is $n$-dimensional and can be parameterized by the positions of the kinks. We formulate this as follows.

%In this language, Theorem 3 characterizes the trajectories in the phase space that enter (or in reverse time, exit) a small neighbourhood of the critical point H8, by arming that a such a trajectory, while still far away from the critical point, must be close to its (un)stable manifold. For hyperbolic critical points, this property is a consequence of the Hartman-Grobman theorem. In our setting, the soliton interactions play an analogous role as exponential (in)stability in the hyperbolic case.

Let $\calM_n \subset \E_{1, (-1)^n}$ be the set of all the possible initial conditions leading to a kink $n$-cluster. In other words, $\bs \phi_0 \in \E_{1, (-1)^n}$ belongs to $\calM_n$ if and only if the corresponding solution $\bs \phi(t)$ of~\eqref{eq:csf} with initial data $\bs \phi(0) = \bs \phi_0$ is a kink $n$-cluster. Clearly, $\calM_n$ is an invariant set.

%For each $\eps>0$ we denote by $\calM_{n, \eps}$ the intersection of $\calM_n$ with the open subset of $\E$ consisting of states $\bs \phi_0$ for which $\bfd(\bs \phi_0) < \eps$.

We fix $\eta_0, \tau_0$ as in Theorem~\ref{thm:main} and define an open subset  $\calT(\eta_0) \subset \R^n$ by 
\begin{multline}  \label{eq:T0-def} 
\calT(\eta_0) \coloneqq  \Big\{ \vec a \in \R^n \mid \exists \, t_0 \ge \tau_0 \, \, \textrm{such that} \\ \max_{1 \le k < n}\bigg|\big(a_{k+1, 0} - a_{k, 0}\big) - \Big(2\log(\kappa t_0) - \log\frac{Mk(n-k)}{2}\Big)\bigg| < \eta_0 \Big\} .
\end{multline} 

By Theorem~\ref{thm:main}, to each $\vec a_0 \in \calT(\eta_0)$ and each $t_0 \ge \tau_0$ such that~\eqref{eq:a0-initial} is satisfied we can associate an element of $\calM_n$ via the mapping 
\EQ{ \label{eq:homeo} \bs \Phi: \calT(\eta_0)  \to \calM_n; \quad 
\calT(\eta_0) \ni \vec a_0 \mapsto  \bs \Phi( \vec a_0) = \bs \phi(\vec a_0, t_0; t_0) \in \calM_n.
}

%\begin{remark}  \label{rem:tt}
%
%\end{remark} 

\begin{remark} \label{rem:tt-homeo}
Observe that for a fixed $\vec a_0 \in \R^n$, once~\eqref{eq:a0-initial} is satisfied for some $t_0 \ge \tau_0$ and $\eta_0$ as in Theorem~\ref{thm:main}, it will also be satisfied for nearby $\ti t_0$. However, it will be evident from the proof that the trajectories $\bs \phi( \vec a_0, t_0; t)$ and $\bs \phi(\vec a_0, \ti t_0; t)$ are time translations of each other, meaning, for example,  if $t_0 < \ti t_0$ we have $\bs \phi( \vec a_0, t_0; t) = \bs \phi(\vec a_0, \ti t_0; t+ (\ti t_0 - t_0))$. 
%In particular, $\bs \phi( \vec a_0, t_0; t_0) = \bs \phi(\vec a_0, \ti t_0; \ti t_0 )$.
 It follows that the map  $\Phi$ is independent of the choice of $t_0 \ge \tau_0$ such that~\eqref{eq:a0-initial} holds. Indeed, for a fixed $\vec a_0 \in \calT(\eta_0)$, different choices $t_0, \ti t_0 \ge \tau_0$ such that~\eqref{eq:a0-initial} holds are associated to the same element of the phase space, i.e., $\bs \phi(\vec a_0, t_0; t_0) = \bs \phi( \vec a_0, \ti t_0; \ti t_0)$.  In other words, the mapping~\eqref{eq:homeo} associates to each element $\vec a_0 \in \calT(\eta_0)$  a unique element of $\calM_n$, which we can denote more succinctly by $\bs \phi( \vec a_0)$. 
\end{remark} 

% For a fixed $\vec a_0 \in \calT(\eta_0)$ and distinct $t_0, t_1 \ge \tau_0$ such that~\eqref{eq:a0-initial} is satisfied for both $t_0, t_1$ we will see in the proof that 
%\EQ{
%\bs \phi( \vec a_0; t_0) = \bs \phi(\vec a_0; t_1)
%}
%and the corresponding trajectories are time translations  
%
%By Theorem~\ref{thm:main} we can define a mapping $ \calT( \eta_0) \to \calM_n$ via the association 
%\EQ{ \label{eq:homeo} 
%\calT(\eta_0) \ni \vec a_0 \mapsto \bs \phi(\vec a_0; t_0(\vec a_0)) \in \calM_n 
%}
%where $t_0 = t_0(\vec a_0)$ above is defined by \red{fix} 

\begin{maincor}[Kink $n$-clusters as an invariant topological manifold] \label{cor:stable}The set $\calM_n$ is an $n$-dimensional topological manifold. The mapping~\eqref{eq:homeo} is a homeomorphism onto its image $\bs \Phi( \calT(\eta_0))$, which we denote by $\bs \Phi( \calT(\eta_0))=:\calM_{n, \loc} \subset \calM_n$. Moreover, for each trajectory $\bs \phi(t) \subset \calM_n$ there exists a time $t_0>0$ such that $\bs \phi(t) \in \calM_{n, \loc}$ for all $t \ge t_0$. 

% The open set $\calT(\eta_0) \subset \R^n$ defined in~\eqref{eq:T0-def} is homeomorphic to its image  $\calM_{n, \textrm{loc}} \subset \calM_{n}$  under the mapping~\eqref{eq:homeo}. 

\end{maincor}

\begin{remark} 
%We will see in the proof that the initial position vectors $\vec a_0$ of kink $n$-clusters as in~\eqref{eq:a0-initial} that are close to their asymptotic state parameterize $\calM_{n, \eps}$. We believe that these parameterize the entire manifold $\calM_n$.  To prove this, we would need to rule out self-intersections of $\calM_n$, but we do not pursue this here. 
We believe that $\calM_n$ is a \emph{smooth} $n$-dimensional invariant manifold, but we do not pursue its regularity here. We  believe the submanifold $\calM_{n} \cap \{ \bs \phi \in \E \mid \bfd( \bs \phi) < \eps\}$ (note that this set contains $\calM_{n, \loc}$) can also be parameterized uniquely by the positions of the kinks via the map $\bs \phi \mapsto \vec a(\bs \phi)$, (provided $\eps>0$ is sufficiently small so that this map can be defined via Lemma~\ref{lem:mod-av-intro}). To prove this we would need to rule out self intersections, i.e., show there are no elements  $\bs \phi, \sh{\bs \phi}$ of $\calM_{n} \cap \{ \bs \phi \in \E \mid \bfd( \bs \phi) < \eps\}$ for which $\vec a(\bs \phi) = \vec a( \sh{\bs \phi})$ but $\vec v( \bs \phi) \neq \vec v( \sh{\bs \phi})$, which we do not accomplish here.  Rather we show this non-self intersection property holds only within $\calM_{n, \loc}$, i.e., for kink $n$-clusters that are close enough to their asymptotic configurations as in~\eqref{eq:a0-initial}. %\red{this remark is perhaps too imprecise} 
\end{remark}

The manifold $\calM_{n}$ and its submanifold $\calM_{n, \loc}$ play natural roles in the sense of dynamical systems. If we extend the phase space by a state $\bs H^\infty_n$ corresponding to the limit of $\bs H(\vec a, \vec v)$ as $\rho(\vec a, \vec v) \to 0$, the function $\bfd(\bs \phi)$ gives a distance to $\bs H^{\infty}_n$ and the trajectories $ \bs \phi(t)  \in \calM_{n}$ satisfy  $\lim_{t \to \infty}\bfd ( \bs \phi(t)) = 0$, in other words $\calM_n$ is the \emph{stable manifold} of $\bs H^\infty_n$. Since every kink $n$-cluster enters and then never leaves the submanifold $\calM_{n, \loc}$ after some finite time, $\calM_{n, \loc}$ forms the \emph{local stable manifold} of $\bs H^\infty_n$.
The (local) unstable manifold can be obtained by applying the time reversal symmetry.
% In the companion paper~\cite{JL9} we showed in~\cite[Theorem 3]{JL9} that the manifolds $\calM_{n, \loc}$ encode the dynamics of solutions either entering or exiting a small neighborhood of the ``critical point'' $\bs H^\infty_n$, which is a standard  property of the stable/unstable manifolds. We refer the reader to~\cite{JL9} for further discussion. 

A typical feature of the stable/unstable manifold of a stationary state is that it governs the behavior of trajectories that enter/exit its small neighborhood,
in the sense that such a trajectory, while still/already far away from the stationary state, must be close to its stable/unstable manifold.
(For hyperbolic stationary states, this property is a consequence
of the Hartman-Grobman theorem.)
In the next theorem, we show that the manifold of kink clusters discussed above
plays an analogous role for the formation/collapse of multikink configuration.
\begin{mainthm}[Universal profiles of multikink formation] 
\label{thm:unstable}
Let $\eta > 0$ be sufficiently small and let $\bs\phi_m$ be a sequence of solutions of \eqref{eq:csf} defined on time intervals $[0, T_m]$
satisfying the following assumptions:
\begin{enumerate}[(i)]
\item $\lim_{m\to \infty}\bfd(\bs\phi_m(T_m)) = 0$,
\item $\bfd(\bs \phi_m(t)) \leq \eta$ for all $t \in [0, T_m]$,
\item $\bfd(\bs \phi_m(0)) = \eta$.
\end{enumerate}
Then, after extraction of a subsequence, there exist $0 = n^{(0)} < n^{(1)} < \ldots < n^{(\ell)} = n$,
finite energy states $\bs P_0^{(1)}, \ldots, \bs P_0^{(\ell)}$ and sequences of real numbers $(X_m^{(1)})_m, \ldots, (X_m^{(\ell)})_m$ such that
\begin{enumerate}[(i)]
\item for all $j \in \{1, \ldots, \ell\}$, the solution $\bs P^{(j)}$ of \eqref{eq:csf}
for the initial data $\bs P^{(j)}(0) = \bs P_0^{(j)}$ is a cluster of $n^{(j)} - n^{(j-1)}$ kinks,
\item for all $j\in \{1, \ldots, \ell-1\}$, $\lim_{m\to \infty}\big( X_m^{(j+1)} - X_m^{(j)}\big) = \infty$,
\item $
\lim_{m\to \infty} \Big\|\bs \phi_m(0) - \Big(\bs 1 + \sum_{j = 1}^\ell(-1)^{n^{(j-1)}} \big(\bs P_0^{(j)}(\cdot - X_m^{(j)}) - \bs 1\big)\Big)\Big\|_\cE = 0.
$
\end{enumerate}
\end{mainthm}

%Theorem~\ref{thm:unstable} can be understood to mean that kink clusters have properties similar to the stable/unstable manifolds of a hyperbolic stationary state. This analogy is most easily understood in the case $n =2$, which we explain here.

%Theorem~\ref{thm:unstable} can be understood to mean that kink clusters have properties similar to the stable/unstable manifolds of a hyperbolic stationary state.
The analogy with stable manifolds is most easily understood in the case $n =2$, as Theorem~\ref{thm:unstable} characterizes trajectories in the phase space that enter a small neighborhood of the asymptotic state $\bs H^\infty_2$, by affirming that a such a trajectory,
while still far away from $\bs H^\infty_2$,
must be close to the manifold $\cM_2$.
In our setting, nonlinear soliton interactions play an analogous role
as exponential (in)stability in the hyperbolic case. 

The analogy described above carries over to $n>2$, but is slightly more complicated, since at the ``exit'' time $t =0$ the solution $\bs\phi_m(0)$ is close to a superposition of well-separated kink clusters, rather than to a single one. Intuitively, for $n > 2$ it can happen
that only some of the neighboring kinks ``collapse'', while
the distances between other neighboring kinks remain large.

\subsection{Structure of the paper and main ideas}
In Section~\ref{sec:static}, we recall the basic properties of the stationary solutions
and compute the first non-trivial term in the asymptotic expansion
of the potential energy of a multikink configuration, $E(\bs H(\vec a, \vec v))$,
as the distances between the kinks tend to infinity.

Section~\ref{sec:cauchy} is devoted to a brief presentation of the well-posedness theory
of the equation \eqref{eq:csf}.

In Section~\ref{sec:mod}, we implement the \emph{modulation method},
also called the ``method of collective coordinates'' in the physics literature. Every kink $n$-cluster asymptotically decomposes into the form
\EQ{
\bs \phi(t) = \bs H( \vec a(t), \vec v(t)) + \bs g(t)
}
where $(\vec a(t), \vec v(t))$ are trajectories related to  $\bs h(t)$ by the orthogonality conditions~\eqref{eq:ortho-av-d} and satisfies 
\EQ{
\bfd( \bs \phi(t)) \to 0 \mas t \to \infty,
}
by Propositon~\ref{prop:close-to-H}. 
The idea is to rewrite \eqref{eq:csf} as a coupled system of equations for
the \emph{modulation parameters} $\vec a(t), \vec v(t)$ and the \emph{remainder} $\bs g(t)$. Indeed, using Lemma~\ref{lem:mod-av-intro} we see that~\eqref{eq:csf} can be converted to a system for $(\vec a(t), \vec v(t), \bs g(t)) \in C^1([t_0; \infty); \R^{2n}) \times  C([t_0, \infty); \E)$ 
\EQ{ \label{eq:hav-system} 
\p_t \bs g(t) &=  \bs  J \vD E(\bs H(\vec a(t), \vec v(t)) + \bs g(t)) -\sum_{ k =1}^n (-1)^k \big( a_k'(t) \p_{a_k} \bs H_k(t) + v_k'(t) \p_{v_k} \bs H_k(t)\big)\\
0& = \la \bs g(t), \, \bs \alpha_k(t) \ra 
 = \la \bs g(t) ,\, \bs \beta_k(t)  \ra , \quad  \, \forall k \in \{1, \dots, n\}.
}
%Expanding, extracting the linear terms, and noting the identities
%\EQ{
%\p_t \bs H(\vec a(t), \vec v(t)) = \sum_{ k =1}^n (-1)^k \big( a_k'(t) \p_{a_k} \bs H_k(t) + v_k'(t) \p_{v_k} \bs H_k(t)\big)
%}
%and 
%\EQ{
%\bs J \vD E(\bs H(\vec a(t), \vec v(t))) = \sum_{ k =1}^n (-1)^k v_k \p_{a_k} \bs H_k - \pmat{ 0 \\ U'( H( \vec a, \vec v)) - \sum_{k=1}^n (-1)^kU'( H_k)} 
%}
%the equation for $\bs h(t)$ above can be written as 
%\EQ{ \label{eq:h-eq-intro}
%\p_t \bs h(t) &= \bs  J \vD^2 E( \bs H(\vec a(t), \vec v(t))) \bs h(t) -  \pmat{ 0 \\ U'( H( \vec a, \vec v)) - \sum_{k=1}^n (-1)^kU'( H_k)} \\  
%&\quad  - \sum_{ k =1}^n (-1)^k \Big( (a_k'(t) - v_k(t))\p_{a_k} \bs H_k(t) + v_k'(t) \p_{v_k} \bs H_k(t)\Big) \\
%&\quad +  \bs J \big( \vD E(\bs H(\vec a(t), \vec v(t)) + \bs h(t)) - \vD E( \bs H(\vec a(t), \vec v(t))) - \vD^2 E(\bs H(\vec a(t), \vec v(t)))  \bs h(t) \big) .
%}
 To obtain the equations for the trajectories $(\vec a(t), \vec v(t))$ we can pair the equation above with $\bs \al_k(t)$ and $ \bs \beta_k(t)$ while using the identities 
\EQ{ \label{eq:av-eq-intro}
a_k' \la \bs g, \,  \p_{a} \bs \alpha(a_k, v_k) \ra + v_k ' \la \bs g, \, \p_v \bs \alpha(a_k, v_k) \ra  &=\la  \p_t \bs g, \, \bs \alpha(a_k, v_k) \ra \\
  a_k' \la \bs g, \,  \p_{a} \bs \beta(a_k, v_k) \ra + v_k ' \la \bs g, \, \p_v \bs \beta(a_k, v_k) \ra  &=\la  \p_t \bs g, \, \bs \beta(a_k, v_k) \ra,
%  a_k'(t) \la \bs g(t), \,  \p_{a} \bs \alpha(a_k(t), v_k(t)) \ra + v_k '(t) \la \bs g(t), \, \p_v \bs \alpha(a_k(t), v_k(t)) \ra  &=\la  \p_t \bs g(t), \, \bs \alpha(a_k(t), v_k(t)) \ra \\
%  a_k'(t) \la \bs g(t), \,  \p_{a} \bs \beta(a_k(t), v_k(t)) \ra + v_k '(t) \la \bs g(t), \, \p_v \bs \beta(a_k(t), v_k(t)) \ra  &=\la  \p_t \bs g(t), \, \bs \beta(a_k(t), v_k(t)) \ra,
}
which are obtained by differentiation in time of the second line in the coupled system~\eqref{eq:hav-system}.

% By substituting~\eqref{eq:hav-system} in place of $\p_t \bs g(t)$ on the right hand side of~\eqref{eq:av-eq-intro}, we obtain a system of ODEs for the modulation parameters. 
 Rather than write down the nonlinear  system for $(\vec a(t), \vec v(t))$ explicitly here, we instead build intuition for it as follows. Let us drop for a moment the Lorentz parameters from the ansatz, and introduce the notation $ \bs H( \vec a) \coloneqq  \bs H( \vec a, \vec 0)$. If we consider an isolated system of $n$ points on the real line, located at $a_1, a_2, \ldots, a_n$, whose masses are equal to $M > 0$ and the total potential energy of the system is given by $E_{\up}(H(a_1, \ldots, a_n))$,
then the principles of Newtonian mechanics assert that the positions of the masses evolve according to the system of differential equations
\begin{equation}
\label{eq:newton-ode}
a_k'(t) = M^{-1}p_k(t), \qquad p_k'(t) = F_k(a_1(t), \ldots, a_n(t)) = -\partial_{a_k}E_{\up}(H(a_1, \ldots, a_n)).
\end{equation}
The main conclusion of Section~\ref{sec:mod} is Lemma~\ref{lem:ref-mod},
which states that the modulation parameters of a kink cluster
\emph{approximately} satisfy this system of ODEs with $M \vec v(t) = \vec p(t)$.
This step relies on appropriately defined \emph{localized momenta},
a method first used in \cite{J-18p-gkdv} in a similar context,
and inspired by \cite[Proposition 4.3]{RaSz11}.

After replacing $E_{\up}(H(\vec a))$ by its leading term computed in Section~\ref{sec:static},
the system \eqref{eq:newton-ode} becomes an $n$-body problem with \emph{attractive exponential
nearest-neighbour interactions}
\begin{equation}
\label{eq:attractive-toda}
a_k'(t) = M^{-1}p_k(t), \qquad p_k'(t) = 2\kappa^2\big(e^{-(a_{k+1}(t) - a_k(t))} - e^{-(a_k(t) - a_{k-1}(t))}\big),
\end{equation}
where by convention $a_0(t) \coloneqq -\infty$ and $a_{n+1}(t) \coloneqq \infty$.
In the case of repulsive interactions, we would have the well-known \emph{Toda system} introduced in \cite{Toda70}.
Section~\ref{sec:n-body} is devoted to the study of the long-time behaviour of
solutions of this $n$-body problem. The solutions corresponding to kink clusters
are the ones for which the distances between the masses tend to $\infty$
and their momenta tend to $0$ as $t \to \infty$.
We recall that such solutions are called \emph{parabolic motions}.
One can check that \eqref{eq:attractive-toda} has an exact solution
(defined up to translation, the mass center being chosen arbitrarily)
\begin{equation}
\label{eq:explicit-parabolic}
a_{k+1}(t) - a_k(t) = 2\log(\kappa t) - \log\frac{Mk(n-k)}{2}, \qquad p_k(t) = {-}M\frac{n+1-2k}{t}.
\end{equation}
In Section~\ref{sec:n-body}, we prove that the leading order
of the asymptotic behavior of $a_{k+1}(t) - a_k(t)$ and $v_k(t)$ for any parabolic motion
coincides with \eqref{eq:explicit-parabolic} with $ M \vec v(t) = \vec p(t)$.
Our analysis is sufficiently robust to be valid also in the presence of
the error terms obtained from Lemma~\ref{lem:ref-mod},
which leads to a proof of Theorem~\ref{thm:asymptotics}.

\begin{remark}
Note that \eqref{eq:d-conv-0} is equivalent to
\begin{equation}
\lim_{t\to\infty} \inf_{\vec a} \Big( \| \bs \phi - \bs H(\vec a, \vec 0) \|_{\E}^2 + \rho(\vec a, \vec 0) \Big) = 0.
\end{equation}
In other words, as long as we are only interested in the (qualitative) smallness of $\bfd(\bs\phi)$,
we can assume without loss of generality that $\vec v = 0$. For this reason, the expressions $\bs H(\vec a, \vec 0)$ and $\rho(\vec a, \vec 0)$ will appear many times.
For the sake of brevity, throughout the paper we will denote
\begin{equation}
\label{eq:notation-v-0}
H(\vec a) \coloneqq H(\vec a, \vec 0), \qquad \bs H(\vec a) \coloneqq (H(\vec a), 0) = \bs H(\vec a, \vec 0), \qquad \rho(\vec a) \coloneqq \rho(\vec a, \vec 0),
\end{equation}
and similar doubling of notation will be introduced for a few other objects.

We are aware that this simultaneous use of the ``relativistic'' multikink configurations $H(\vec a(t), \vec v(t))$ and their ``non-relativistic'' version $H(\vec a(t))$ is an unpleasant feature of the paper,
but it also reflects a fundamental conceptual problem. Namely, on one hand we deal with a relativistic equation \eqref{eq:csf-2nd}, which favors the use of the relativistic version. On the other hand, the $n$-body problem with instantaneously acting forces, like \eqref{eq:attractive-toda}, is a purely non-relativistic concept, hence in many parts of the paper the non-relativistic version is better suited.
\end{remark}
\begin{remark}
Let us mention that H\'enon \cite{Henon74} found $n$ independent conserved quantities
for the Toda system (both in the repulsive and in the attractive case).
For parabolic motions, all these quantities are equal to $0$,
which allows to reduce the problem to a system of $n$ equations of 1st order.
Probably, this approach could lead to some simplifications in determining
the asymptotic behavior of the parabolic motions of the system \eqref{eq:attractive-toda},
and perhaps also of the approximate system satisfied by the modulation parameters.
Our arguments do not explicitly rely on the conservation laws related to the complete
integrability of the Toda system, and we expect that part of the analysis
will be applicable also in the cases where the modulation equations are not related
to any completely integrable system of ODEs.
\end{remark}
Theorem~\ref{thm:any-position} is proved in Section~\ref{sec:any-position}.
The overall proof scheme is taken from Martel~\cite{Martel05},
see also the earlier work of Merle~\cite{Merle90}, and contains two steps:
\begin{itemize}
\item for any $T > 0$, prove existence of a solution $\bs \phi$ satisfying
the conclusions of Theorem~\ref{thm:any-position}, but only on the finite time interval $t \in [0, T]$,
\item take a sequence $T_m \to \infty$ and consider a weak limit
of the solutions $\bs \phi_m$ obtained in the first step with $T = T_m$.
\end{itemize}

The first step relies on a novel application of the Poincar\'e-Miranda theorem,
which is essentially a version of Brouwer's fixed point theorem.
We choose data close to a multikink configuration at time $t = T$ and control
how it evolves backwards in time.
It could happen that the multikink collapses before reaching the time $t = 0$.
For this reason, we introduce an appropriately defined ``exit time'' $T_1$.
The mapping which assigns the positions of the (anti)kinks at time $T_1$
to their positions at time $T$ turns out to be continuous and, for topological reasons,
surjective in the sense required by Theorem~\ref{thm:any-position}.

In the second step, it is crucial to dispose of some \emph{uniform} estimate on the sequence $\bs\phi_m$. In our case, the relevant inequality is $\bfd(\bs\phi_m(t)) \lesssim (\eee^L + t^2)^{-1}$ with a universal constant.
The existence of such a uniform bound is related to what we would call the ``ejection property''
of the system. Intuitively, once $\bfd(\bs\phi(t))$ starts to grow,
it has to continue growing at a definite rate until the multikink configuration collapses.

In Section~\ref{sec:profiles}, we give a proof of Theorem~\ref{thm:unstable}.
The identification of the clusters presents no difficulty: the positions of any two consecutive (anti)kinks at time $t = 0$, after taking a subsequence in $m$, either remain at a bounded distance
or separate with their distance growing to infinity as $m\to\infty$.
This dichotomy determines whether
they fall into the same cluster or to distinct ones.
The next step is to again make use of the ejection property in order to obtain
bounds on $\bfd(\bs\phi(t))$ independent of $m$, for any $t \geq 0$.
By standard localization techniques involving the finite speed of propagation,
these bounds are inherited by each of the clusters.
We mention that the proof of strong convergence in Theorem~\ref{thm:unstable} (iii)
is based on a novel application of the well-known principle from the Calculus of Variations
affirming that, for a strictly convex functional $\cF$, if $\bs g_m \wto \bs g$
and $\cF(\bs g_m) \to \cF(\bs g)$, then $\bs g_m \to \bs g$.

The proof of Theorem~\ref{thm:main}, which spans Sections~\ref{sec:cauchy-fin}--\ref{sec:proofs}, is a generalization, refinement, and simplification of the method introduced by the authors with Kowalczyk in~\cite{JKL1}, where we studied the case of kink-antikink pairs, i.e., $n=2$, and is based on a Lyapunov-Schmidt reduction implemented in our context in the following way.

Fix $\vec a_0$ as in~\eqref{eq:a0-initial} with $t_0$ large, and consider the set of all $C^1$ trajectories $(\vec a(t), \vec v(t))$ with $\vec a(t_0) = \vec a_0$ that satisfy 
\EQ{
\label{eq:traj-intro}
|\vec a\,'(t) - \vec v(t)| + |\vec v\,'(t)| + |\vec v(t)|^2 + \sum_{ k =1}^{n-1}\exp({-(a_{k+1}(t) - a_k(t))}) \lesssim t^{-2}.
}
For each given choice of trajectories $(\vec a(t), \vec v(t))$ satisfying~\eqref{eq:traj-intro} we solve the \emph{projected equation} for the triple $(\bs h(t), \vec \lambda(t), \vec \mu(t))$ 
\EQ{ \label{eq:proj-intro} 
 \partial_t \bs h(t) &= \bs J\vD E(\bs H(\vec a(t), \vec v(t))+\bs h(t)) %- \p_t \bs H( \vec a(t), \vec v(t)) \\
 + \sum_{k=1}^n\big(\lambda_k(t) \partial_{a}\bs H_k(t)+ \mu_k(t)\partial_{v}\bs H_k(t)\big) \\
 0& = \la \bs h(t), \, \bs \alpha_k(t) \ra 
 = \la \bs h(t) ,\, \bs \beta_k(t)  \ra , \quad  \, \forall k \in \{1, \dots, n\}.
}
The pair of multipliers $(\vec \lam(t), \vec \mu(t))$ are introduced to ensure the orthogonality conditions in the second line above. %In order to simplify notation, we have chosen to absorb the term $- \p_t \bs H( \vec a(t), \vec v(t))$ from~\eqref{eq:hav-system} into the second term on the right involving $(\vec \lambda(t), \vec \mu(t))$. 
In Section~\ref{sec:cauchy-inf} we show (in Lemma~\ref{lem:kg-nonlin-eq-inf}) that to each pair of $C^1$ trajectories $(\vec a(t), \vec v(t))$ as in~\eqref{eq:traj-intro} we can associate a unique solution $(\vec \lambda(t), \vec \mu(t), \bs h(t))$ of~\eqref{eq:proj-intro} satisfying 
\EQ{  \label{eq:h-est-intro} 
\sup_{t \ge t_0} t \| \bs h(t) \|_{\E} < \delta \mand  \sup_{t \ge t_0} t^{2-} \| \bs h(t) \|_{\E}< \infty, 
}
for some sufficiently small $\delta>0$, and 
\begin{multline} \label{eq:param-est-intro} 
 \sup_{t \ge t_0} t^{3-} \Big(  \big| \lambda_k(t) + (-1)^k v_k(t) \big| \\+ \big| \mu_k(t) + (-1)^k 2 \kappa^2 \big( e^{-(a_{k+1}(t) - a_k(t))} - e^{-(a_{k}(t) - a_{k-1}(t))} \big)\big| \Big) < \infty. 
\end{multline} 
In light of~\eqref{eq:hav-system} the goal is now to show that there is a unique pair of trajectories $(\vec a(t), \vec v(t))$ as in~\eqref{eq:traj-intro} that satisfy 
\EQ{
 a_k(t) = - (-1)^k \lambda_k(t) , \quad  v_k(t) = - (-1)^k \mu_k(t) 
}
for each $k = 1, \dots, n$. In the terminology of Lyapunov-Schmidt reduction, this means solving the~\emph{bifurcation equations}, using as input the estimates~\eqref{eq:h-est-intro} and~\eqref{eq:param-est-intro}. This is accomplished in Section~\ref{sec:bifurcation} (see Proposition~\ref{prop:sol-bif}). 

This analysis associates to each $\vec a_0$ as in~\eqref{eq:a0-initial} a unique $n$-kink cluster with error $\bs h(t)$ small in the sense of~\eqref{eq:h-est-intro} and trajectories as in~\eqref{eq:traj-intro}. Coupling this qualified uniqueness statement with the estimates in Theorem~\ref{thm:asymptotics} we see that uniqueness extends to all kink $n$-clusters.  This argument is done carefully in Section~\ref{sec:proofs} along with the construction of the stable manifold $\calM_n$ and local stable manifold $\calM_{n, \loc}$. 

\subsection{Further discussion and related results}

There are a few related contexts in which to consider the results in Section~\ref{ssec:results}. First, from the point of view of mathematical physics, they describe the particle-like nature of solitons  -- the dynamics of kink $n$-clusters are captured by the $2n$ parameters giving the positions and momenta of the kinks and the evolutions of these  parameters are in turn governed by an approximate $n$-body law of motion (the attractive Toda system). Second, in  dispersive PDEs the problem of Soliton Resolution predicts that every finite energy solution to~\eqref{eq:csf} eventually decouples into a multi-kink configuration and a radiative wave. Within this framework, one can distinguish two types of nonlinear interactions to be studied separately: soliton-radiation interaction and soliton-soliton interaction. We focus on the latter, restricting to the strongly-interacting regime. Lastly, viewing the results under the lens of dynamical systems, this work develops an invariant manifold theory of multi-solitons.

The particle-like character of solitons is a well-known phenomenon, see \cite[Chapter 1]{MS}
for a historical account.
The question of justification that the positions of solitons satisfy
an approximate $n$-body law of motion was considered for instance in \cite{Stuart,GuSi06,DuMa,OvSi}.

In their work on blow-up for nonlinear waves, Merle and Zaag~\cite{MeZa12-AJM}
obtained a system of ODEs with exponential terms like in \eqref{eq:attractive-toda},
but which was a \emph{gradient flow} and not an $n$-body problem.
The dynamical behavior of solutions of this system was described by C\^ote and Zaag \cite{CoteZaag}.

In relation with our proof of Theorem~\ref{thm:any-position},
we note that Brouwer's theorem was previously used in constructions of multi-solitons,
but for a rather different purpose, namely in order to avoid the growth of linear unstable modes, see \cite{CMM11, CoteMunoz}.

Let us stress again that the main object of our study are
solutions approaching multi-soliton configurations in the strong energy norm, in other words we address the question of interaction of solitons \emph{in the absence of radiation}. The first construction of a two-soliton solution with trajectories having asymptotically vanishing velocities was obtained by Krieger, Martel and Rapha\"el \cite{KrMaRa09},
see also \cite{MaRa18, Vinh17, J-19-twobub, J-17-twobubnls, J-18p-gkdv, JL2-20p} for other constructions
and \cite{WadOhk} for related computations in the completely integrable setting.
Allowing for a radiation term seems to be currently out of reach for the classical scalar fields considered here,
the question of the asymptotic stability of the kink being
still unresolved, see for example~\cite{DelortMasmoudi, KMM, GermainPusateri, HN2, CLL, LuhrmannSchlag} for recent results
on this and related problems. 

A detailed understanding of soliton interactions, particularly the description of an ``ejection'' mechanism for solutions that leave a neighborhood of multi-solitons was used in our work on the Soliton Resolution problem for energy critical waves~\cite{JL1-18, JL5, JL6, JL7}. 
Determining universal profiles of soliton collapse is crucial to 
 several earlier works on dispersive equations, which have been important  inspirations. 
We mention the study of center-stable manifolds of ground states for various nonlinear wave equations, see for instance the works of Nakanishi and Schlag \cite{NaSc11-1,NaSc11-2}, and Krieger, Nakanishi, and Schlag \cite{KrNaSc15},
as well as the earlier work by Duyckaerts and Merle \cite{DM08}. We also mention the classification results for solutions near the ground state soliton for the gKdV equation by Martel, Merle, and Rapha\"el~\cite{MMR14-1, MMR15-2, MMR15-3}, particularly their description of the ``exit'' regime for solutions that leave a neighborhood of the soliton by remaining close to the minimal mass blow-up solution.  A distinction between the present work and these earlier landmarks  is that here we deal with multi-solitons.

Finally, we emphasize that our definition of kink clusters
concerns \emph{only one time direction}, and our study
does not address the question of the behaviour of kink clusters as $t \to {-}\infty$, which goes by the name of the \emph{kink collision problem}.
We refer to
\cite{KevrekidisEtc}
for an overview,
and to \cite{Abdon22p2, Abdon22p3} for recent rigorous results
in the case of the $\phi^6$ model.

% The development of invariant manifold theory for dispersive PDEs in other contexts has been an important inspiration. We mention the study of centre-stable manifolds of ground states for various nonlinear dispersive equations, by Nakanishi and Schlag, and Krieger, Nakanishi, and Schlag, see for instance \cite{NaSc11-1,NaSc11-2,KrNaSc15},
%as well as work of Duyckaerts and Merle \cite{DM08}. We also mention the classification results for solutions near the ground state soliton for the gKdV equation by Martel, Merle and Rapha\"el~\cite{MMR14-1, MMR15-2, MMR15-3}, particularly their description of the ``exit'' regime for solutions that leave a neighborhood of the soliton by remaining close to the minimal mass blow-up solution.  
%

%[Can we cite Abdon’s papers here too?, perhaps mentioning the collision problem in relation to soliton-soliton interactions…]

%\subsection{Acknowledgments}
%\label{ssec:merci}

\subsection{Notation}
\label{ssec:notation}
Even if $v(x)$ is a function of one variable $x$, we often write $\partial_x v(x)$
instead of $v'(x)$ to denote the derivative. The prime notation is only used
for the time derivative of a function of one variable $t$
and for the derivative of the potential $U$.

If $\vec a, \vec b \in \bR^n$, then $\vec a\cdot \vec b \coloneqq \sum_{k=1}^n a_k b_k$.
If $u$ and $v$ are (real-valued) functions, then $\la u, v\ra \coloneqq \int_{-\infty}^\infty u(x)v(x)\ud x$.

Boldface is used for pairs of values
(which will usually be a pair of functions
forming an element of the phase space).
A small arrow above a letter indicates a vector
with any finite number of components
(which will usually be the number of kinks
or the number of kinks diminished by 1).
If $\bs u = (u, \dot u) \in L^2(\bR)\times L^2(\bR)$ and $\bs v = (v, \dot v) \in L^2(\bR)\times L^2(\bR)$, we write
$\la \bs u, \bs v\ra \coloneqq \la u, v\ra + \la \dot u, \dot v\ra$.

We will have to manipulate finite sequences and sums. In order to make the formulas reasonably compact,
we need to introduce appropriate notation, some of which is not completely standard.
If $\vec w = (w_j)_j$ is a vector and $f: \bR \to \bR$ is a function, we denote $f(\vec w)$ the vector with components $f(w_j)$,
for example $\eee^{\vec w}$ will denote the vector $(\eee^{w_j})_j$, $\vec w\,^2$ the vector $(w_j^2)_j$ and $\log\vec w$ the vector $(\log w_j)_j$.
If $\vec v = (v_j)_j$ is another vector, we denote $\vec w\vec v \coloneqq (w_jv_j)_j$.
We also denote $w_{\min} \coloneqq \min_j w_j$.

When we write $\simeq$, $\lesssim$ or $\gtrsim$, it should be understood that the constant is allowed to depend only on $n$.
Our use of the symbol $\sh f \sim f$ is non-standard and indicates that $\sh f - f$ is a negligible
quantity (the meaning of ``negligible'' will be specified in each case),
without requiring that $\sh f / f$ be close to $1$.

The energy space is denoted $\cE \coloneqq H^1(\bR) \times L^2(\bR)$.
We also use local energies and energy norms defined as follows.
For $-\infty \leq x_0 < x_0' \leq \infty$ and $\phi_0: [x_0, x_0']\to \bR$, we denote
\begin{align}
E_p(\phi_0; x_0, x_0') &\coloneqq \int_{x_0}^{x_0'}\Big(\frac 12(\partial_x \phi_0(x))^2 + U(\phi_0)\Big)\ud x, \\
E(\bs \phi_0; x_0, x_0') &\coloneqq \int_{x_0}^{x_0'}\Big(\frac 12 (\dot \phi_0(x))^2 + \frac 12(\partial_x \phi_0(x))^2 + U(\phi_0)\Big)\ud x, \\
\|\phi_0\|_{H^1(x_0, x_0')}^2 &\coloneqq \int_{x_0}^{x_0'}\big((\partial_x \phi_0(x))^2 + \phi_0(x)^2\big)\ud x, \\
\|\bs \phi_0\|_{\cE(x_0, x_0')}^2 &\coloneqq \int_{x_0}^{x_0'}\big((\dot \phi_0(x))^2 + (\partial_x \phi_0(x))^2 + \phi_0(x)^2\big)\ud x.
\end{align}

The open ball of center $c$ and radius $r$ in a normed space $A$ is denoted $B_A(c, r)$.
%and $\cE^m \coloneqq H^m(\bR) \times H^{m-1}(\bR)$ for $m \in \{1, 2, \ldots\}$ (the energy space of order $m$).
We denote $\vD$ and $\vD^2$ the first and second Fr\'echet derivatives of a functional.

We write $x_+ \coloneqq \max(0, x)$.

We take $\chi: \bR \to [0, 1]$ to be a decreasing $C^\infty$ function
such that $\chi(x) = 1$ for $x \leq \frac 13$
and $\chi(x) = 0$ for $x \geq \frac 23$.

Proofs end with the sign $\boxvoid$. Statements given without proof end with the sign $\boxslash$.

\section{Kinks and interactions between them}
\label{sec:static}
\subsection{Stationary solutions}
A stationary field $\phi(t, x) = \psi(x)$ is a solution of \eqref{eq:csf} if and only if
\begin{equation}
\label{eq:psi4}
\partial_x^2\psi(x) = U'(\psi(x)),\qquad\text{for all }x\in \bR.
\end{equation}
We seek solutions of \eqref{eq:psi4} having finite potential energy $E_p(\psi)$.
Since $U(\psi) \geq 0$ for $\psi \in \bR$, the condition $E_p(\psi) < \infty$ implies
\begin{align}
\label{eq:psi4-H1}
&\int_{-\infty}^{+\infty}\frac 12 (\partial_x \psi(x))^2 \ud x < \infty, \\
\label{eq:psi4-U}
&\int_{-\infty}^{+\infty}U(\psi(x)) \ud x < \infty.
\end{align}
From \eqref{eq:psi4-H1} we have $\psi \in C(\bR)$,
so \eqref{eq:psi4} and $U \in C^\infty(\bR)$ yield $\psi \in C^\infty(\bR)$.
Multiplying \eqref{eq:psi4} by $\partial_x \psi$ we get
\begin{equation}
\partial_x\Big(\frac 12 (\partial_x \psi)^2 - U(\psi)\Big)
= \partial_x \psi\big(\partial_x^2 \psi - U'(\psi)\big) = 0,
\end{equation}
so $\frac 12 (\partial_x \psi(x))^2 - U(\psi(x)) = k$ is a constant.
But then \eqref{eq:psi4-H1} and \eqref{eq:psi4-U} imply $k = 0$.
We obtain first-order autonomous equations, called the Bogomolny equations,
\begin{equation}
\label{eq:bogom-eq}
\partial_x\psi(x) = \sqrt{2U(\psi(x))}\quad\text{or}\quad \partial_x\psi(x) = -\sqrt{2U(\psi(x))},\quad\text{for all }x \in \bR,
\end{equation}
which can be integrated in the standard way, see for instance \cite[Section 2]{JKL1}.
We conclude that:
\begin{itemize}
\item the only stationary solution of \eqref{eq:csf} belonging to $\cE_{1, 1}$ is the vacuum
state $\bs 1 \coloneqq (1, 0)$ and the only stationary solution of \eqref{eq:csf} belonging to $\cE_{-1, -1}$ is the vacuum
state $-\bs 1 \coloneqq (-1, 0)$,
\item the only stationary solutions of \eqref{eq:csf} belonging to $\cE_{-1, 1}$
are the translates of $(H, 0)$, where the function $H$ is defined by
\begin{equation}
\label{eq:H-def}
H(x) = G^{-1}(x), \qquad\text{with}\ \ G(\psi) \coloneqq \int_{0}^\psi \frac{\ud y}{\sqrt{2U(y)}}\ \ \text{for all }\psi \in ({-}1, 1),
\end{equation}
and the solutions of \eqref{eq:psi4} belonging to $\cE_{1, -1}$
are the translates of $(-H, 0)$.

\end{itemize}

The asymptotic behaviour of $H(x)$ for $|x|$ large is essential for our analysis.
We have the following result, which is a refinement of \cite[Proposition 2.1]{JKL1}.
\begin{proposition}
\label{prop:prop-H}
The function $H(x)$ defined by \eqref{eq:H-def} is odd, of class $C^\infty(\bR)$,
and there exist constants $\kappa > 0$ and $C > 0$ such that for all $x \geq 0$
\begin{equation}
\label{eq:H-asym-p}
\begin{aligned}
&\big|H(x)-1+ \kappa \eee^{- x}-\frac 16\kappa^2 U'''(1)\eee^{-2x}\big| \\
+ &\big|\partial_x H(x) -  \kappa \eee^{- x} + \frac 13\kappa^2 U'''(1)\eee^{-2x}\big| \\
+ &\big|\partial_x^2 H(x) + \kappa \eee^{- x} - \frac 23\kappa^2 U'''(1)\eee^{-2x}\big| \leq C\eee^{-3 x},
\end{aligned}
\end{equation}
and for all $x \leq 0$
\EQ{
\label{eq:H-asym-m}
&\big|H(x) +1- \kappa \eee^{ x}+\frac 16\kappa^2 U'''(1)\eee^{2x}\big| \\
+ &\big|\partial_x H(x) -  \kappa \eee^{ x} + \frac 13\kappa^2 U'''(1)\eee^{2x}\big| \\
+ &\big|\partial_x^2 H(x) - \kappa \eee^{ x} + \frac 23\kappa^2 U'''(1)\eee^{2x}\big| \leq C\eee^{3 x},
}

\end{proposition}
\begin{proof}
We only prove \eqref{eq:H-asym-p}, which gives the asymptotic behavior as $x \to \infty$, since \eqref{eq:H-asym-m} then follows from the oddness of $H$. Recall from~\cite[Proposition 2.1]{JKL1} that there exists a constant $C_1>0$ so that for $\kappa>0$ defined by 
\EQ{
\kappa \coloneqq  \exp \Big(  \int_0^1 \big( \frac{1}{ \sqrt{ 2 U(y)}} - \frac{1}{1-y} \big) \,\ud y \Big)
}
we have, 
\EQ{
\big|H(x)-1+ \kappa \eee^{- x}\big| \le C_1 e^{-2x}. 
}
Define $z$ by the formula 
\EQ{ \label{eq:z-def-21} 
\frac{1 - H(x)}{\kappa} =   e^{-x}(1 -\frac{\kappa U'''(1)}{6}  e^{-z}).
}
From the above we have
\EQ{ \label{eq:z-bound-21} 
e^{-z} \le C_1 \frac{ 6}{\kappa^2} \abs{U'''(1)}^{-1} e^{-x}. 
}
Using the fourth order Taylor expansion of $U(y)$ about $ y =1$ gives, 
\EQ{
\frac{1}{\sqrt{2U(y)}} - \frac{1}{1-y} - \frac{1}{6} U'''(1) = O(1-y) \mas y \nearrow 1 
}
and hence from~\eqref{eq:H-def} and~\eqref{eq:kappa-def}, 
\EQ{
G( \psi) &= - \log( 1- \psi) + \frac{1}{6} U'''(1) \int_0^\psi \ud y + \int_0^1 \Big( \frac{1}{\sqrt{2U(y)}} - \frac{1}{1-y} - \frac{1}{6} U'''(1) \Big) \, \ud y \\
&\quad - \int_\psi^1 \Big( \frac{1}{\sqrt{2U(y)}} - \frac{1}{1-y} - \frac{1}{6} U'''(1) \Big) \, \ud y\\
& = - \log\Big( \frac{1- \psi}{\kappa}\Big)  - \frac{\kappa}{6} U'''(1) \Big( \frac{1- \psi}{\kappa}  \Big) + O( (1- \psi)^2). 
}
Since $H(x) \coloneqq G^{-1}(x)$ we arrive at 
\EQ{
x = - \log\Big( \frac{1- H(x)}{\kappa}\Big) - \frac{\kappa}{6} U'''(1) \Big( \frac{1-H(x)}{\kappa}  \Big) + O( (1- H(x))^2). 
}
Using~\eqref{eq:z-def-21} in the previous line yields, 
\EQ{
 \log\Big(1 -\frac{\kappa U'''(1)}{6}  e^{-z}\Big) = \frac{\kappa U'''(1)}{6}  e^{-x}\Big(1 -\frac{\kappa U'''(1)}{6}  e^{-z}\Big) + O(e^{-2x}) .
}
From~\eqref{eq:z-bound-21} we then arrive at the estimate, 
\EQ{
| e^{-z} - e^{-x} | \lesssim e^{-2x}
}
provided $x>0$ is sufficiently large. 
Hence, 
\EQ{ \label{eq:H-exp} 
\big|H(x) +1- \kappa \eee^{ x}+\frac 16\kappa^2 U'''(1)\eee^{2x}\big| \lesssim e^{-x} | e^{-z}- e^{-x}| \lesssim e^{-3x}
}
 as claimed. The bound for $\p_x H(x)$ follows from~\eqref{eq:bogom}, the bound~\eqref{eq:H-exp}, and the expansion $\sqrt{2 U(y)} = (1 - y) - \frac{1}{6} U'''(1) (1-y)^2 + O( |1-y|^3)$. The bound for $\partial_x^2 H$ uses~\eqref{eq:psi4},~\eqref{eq:H-exp},  and the expansion $U'(y) = -(1-y) + \frac{1}{2} U'''(1) (1-y)^2 + O( |1-y|^3)$. 
 \end{proof}

%The asymptotic behaviour of $H(x)$ for $|x|$ large is essential for our analysis.
%We have the following result, see \cite[Proposition 2.1]{JKL1}.
%\begin{proposition}
%\label{prop:prop-H}
%The function $H(x)$ defined by \eqref{eq:H-def} is odd, of class $C^\infty(\bR)$
%and there exist constants $\kappa > 0$ and $C > 0$ such that for all $x \geq 0$
%\begin{equation}
%\label{eq:H-asym-p}
%\begin{aligned}
%\big|H(x)-1+ \kappa \eee^{- x}\big| 
%+ \big|\partial_x H(x) -  \kappa \eee^{- x}\big| 
%+ \big|\partial_x^2 H(x) + \kappa \eee^{- x}\big| \leq C\eee^{-2 x},
%\end{aligned}
%\end{equation}
%and for all $x \leq 0$
%\begin{gather}
%\label{eq:H-asym-m}
%\big|H(x)+1- \kappa \eee^{ x}\big| + \big|\partial_x H(x) -  \kappa \eee^{ x}\big| + \big|\partial_x^2 H(x) -  \kappa \eee^{ x}\big| \leq C\eee^{2 x}.
%\end{gather}
%\vspace{-1cm}\qedno
%\end{proposition}

We denote
\begin{align}
%\label{eq:reduced-metric}
M \coloneqq \|\partial_x H\|_{L^2}^2 = 2\int_0^1\sqrt{2U(y)}\ud y
= 2\int_{-\infty}^\infty U(H(x))\ud x = E_p(H),
\label{eq:M-def}
\end{align}
all these equalities following from \eqref{eq:bogom-eq} and the change of variable $y = H(x)$.
In the context of Special Relativity,
one can think of $M$ as the (rest) mass of the kink.

We will also use the fact that
\begin{align}
\label{eq:reduced-force}
\int_{-\infty}^\infty \partial_x H(x)\left(U''(H(x)) - 1\right)\eee^{x} \ud x = -2 \kappa,
\end{align}
and we observe  the equipartition of the potential energy of the kink:
\begin{equation}
\label{eq:equipartition}
\int_{-\infty}^\infty \frac 12 (\partial_x H(x))^2 \ud x = \frac 12 M = \int_{-\infty}^\infty U(H(x))\ud x.
\end{equation}
These computations are checked in~\cite[Section 2.1]{JKL1}. 

Finally, we recall the following \emph{Bogomolny trick} from \cite{Bogom76}.
If $\phi_0(x_0) \leq \phi_0(x_0')$ (understood as limits if $x_0$ or $x_0'$ is infinite), then
\begin{equation}
\label{eq:bogom}
\begin{aligned}
E_p(\phi_0; x_0, x_0') &= \frac 12\int_{x_0}^{x_0'}\Big(\big(\partial_x \phi_0 - \sqrt{2U(\phi_0)}\big)^2 + 2\phi_0'\sqrt{2U(\phi_0)}\Big)\ud x \\
&= \int_{\phi_0(x_0)}^{\phi_0(x_0')}\sqrt{2U(y)}\ud y + \frac 12\int_{x_0}^{x_0'}\big(\partial_x \phi_0 - \sqrt{2U(\phi_0)}\big)^2\ud x.
\end{aligned}
\end{equation}
Analogously, if $\phi_0(x_0) \geq \phi_0(x_0')$, then
\begin{equation}
\label{eq:bogom-2}
E_p(\phi_0; x_0, x_0') = \int_{\phi_0(x_0')}^{\phi_0(x_0)}\sqrt{2U(y)}\ud y + \frac 12\int_{x_0}^{x_0'}\big(\partial_x \phi_0 + \sqrt{2U(\phi_0)}\big)^2\ud x.
\end{equation}
Hence, restrictions of kinks and antikinks to (bounded or unbounded) intervals are minimisers of the potential
energy among all the functions connecting two given values in $(-1, 1)$.

\subsection{Traveling kinks and their energies}
\label{ssec:traveling}

The energy functional $E$ is smooth on each energy sector. Its Fr\'echet derivatives are given by
\begin{align}
\vD E(\bs\phi)(\bs h) &= \int_{-\infty}^\infty \big[\dot \phi(x)\dot h(x) +({-}\partial_x^2\phi(x) + U'(\phi(x))h(x)\big]\ud x, \\
\vD^2 E(\bs\phi)(\bs h_1, \bs h_2) &= \int_{-\infty}^\infty \big[\dot h_1(x)\dot h_2(x) + \partial_x h_1(x) \partial_x h_2(x) + U''(\phi(x))h_1(x)h_2(x)\big]\ud x, \\
\vD^k E(\bs\phi)(\bs h_1, \ldots, \bs h_k) &= \int_{-\infty}^\infty U^{(k)}(\phi(x))\prod_{j=1}^k h_j(x)\ud x,\qquad\text{for all }k \geq 3.
\end{align}
We write $\vD^2 E(\bs \phi)\bs h_1 \coloneqq  (-\partial_x^2 h_1 + U''(\phi)h_1, \dot h_1)$,
so that $\vD^2 E(\bs \phi)(\bs h_1, \bs h_2) = \la \vD^2 E(\bs \phi)\bs h_1, \bs h_2\ra$.

We gather below the explicit expressions for $\bs H(a, v)$, its derivatives, along with~$\bs \alpha( a, v), \bs \beta(a, v)$ defined in~\eqref{eq:alpha-beta-def}, 
although we will hardly use them:
\EQ{ \label{eq:dHalbe-form} 
\partial_a \bs H(0, v; x) &= \big({-}\gamma_v \partial_x H(\gamma_v x ), \gamma_v^2 v \partial_x^2 H(\gamma_v x )\big), \\
\partial_v \bs H(0, v; x) &= \big(\gamma_v^3 v x \partial_x H(\gamma_v x), {-}\gamma_v^3 \partial_x H(\gamma_v x )
- \gamma_v^4 v^2 x \partial_x^2 H(\gamma_v x)\big), \\
\bs \alpha(0, v; x) &= \big(\gamma_v^2 v \partial_x^2 H(\gamma_v x ),
\gamma_v \partial_x H(\gamma_v x )\big), \\
\bs \beta(0, v; x) &= \big({-}\gamma_v^3 \partial_x H(\gamma_v x )
- \gamma_v^4 v^2 x \partial_x^2 H(\gamma_v x),
{-}\gamma_v^3 v x \partial_x H(\gamma_v x)\big).
}
We have set $a = 0$, the general formulas being obtained
by translating the variable $x$.

We  note the identities, 
\EQ{ \label{eq:ab-dxda} 
\p_x \bs \al ( a, v) = - \p_a \bs \al(a, v), \quad \p_x \bs \beta ( a, v) = - \p_a \bs \beta(a, v). 
}

Since $\bs J$ is skew-symmetric, we have
\begin{equation}
\label{eq:identities-1}
\la \bs \alpha(a, v), \partial_a \bs H(a, v)\ra =
\la \bs \beta(a, v), \partial_v \bs H(a, v)\ra = 0.
\end{equation}
Differentiating \eqref{eq:momentum-Hv} with respect to $v$, we obtain
\begin{equation}
\label{eq:matrix-coef}
\begin{aligned}
\gamma_v^3 M &= \partial_v P(\bs H(a, v)) = {-}\frac 12 \la \partial_x \partial_v \bs H(a, v), \bs J\bs H(a, v)\ra
- \frac 12 \la \partial_x \bs H(a, v), \bs J\partial_v \bs H(a, v)\ra \\
&= {-}\la \bs J\partial_v \bs H(a, v), \partial_x \bs H(a, v)\ra 
= \la \bs \beta(a, v), \partial_a \bs H(a, v)\ra
= -\la \bs \alpha(a, v), \partial_v \bs H(a, v)\ra.
\end{aligned}
\end{equation}

If we multiply both sides of \eqref{eq:Hv-id} from the left by ${-}\bs J$, we get
\begin{equation}
\label{eq:alpha-id-0}
\vD E(\bs H(a, v)) = v\bs J\partial_x \bs H(a, v) = -v\bs\alpha(a, v),
\end{equation}
and differentiation with respect to $a$ yields
\begin{equation}
\label{eq:alpha-id}
- \vD^2 E(\bs H(a, v))(\bs J \bs \alpha(a, v)) = \vD^2 E(\bs H(a, v))\partial_a \bs H(a, v) = v\partial_x \bs \alpha(a, v).
\end{equation}
If we differentiate \eqref{eq:alpha-id-0} with respect to $v$, we get
\begin{equation}
\label{eq:beta-id}
-\vD^2 E(\bs H(a, v))(\bs J \bs \beta(a, v)) = \vD^2 E(\bs H(a, v))\partial_v \bs H(a, v) = -\bs \alpha(a, v) + v\partial_x \bs\beta(a, v).
\end{equation}

\subsection{Interaction of the kinks}
Our next goal is to compute the potential energy and the interaction forces
 for a multikink configuration, which we recall is denoted by 

\begin{align}
\bs H(\vec a, \vec v; x) \coloneqq  \bs 1 + \sum_{k=1}^{n} (-1)^k (\bs H_k(x) + \bs 1),
\end{align}
where $\vec a = (a_1, \ldots, a_n) \in \bR^n$ are the positions of the transitions,
$\vec v = (v_1, \ldots, v_n) \in (-1, 1)^n$ are the Lorentz parameters and $\bs H_k \coloneqq \bs H(a_k, v_k)$.
%For example, if $n = 0$, then $\vec a, \vec v$ are empty vector and $\bs H(\vec a, \vec v)$ is the vacuum $\bs 1$.
%If $n = 1$, $\vec a = (a_1)$ and $\vec v = (v_1)$, then $\bs H(\vec a, \vec v) = -\bs H(a_1, v_1)$ is an antikink.
%If $n = 2$, $\vec a = (a_1, a_2)$ with $a_2 - a_1 \gg 1$ and $\vec v = (v_1, v_2)$, then
%$\bs H(\vec a, \vec v) = \bs 1 - \bs H(a_1, v_1) + \bs H(a_2, v_2)$ has the shape of an antikink near $x = a_1$,
%and of a kink near $x = a_2$. These are the \emph{kink-antikink pairs}, which we studied with Kowalczyk in \cite{JKL1}.

We also recall  that for $k \in \{1, \ldots, n\}$ we have
\begin{equation}
\partial_{a_k}\bs H(\vec a, \vec v) = (-1)^k\partial_{a_k} \bs H_k, \qquad \partial_{v_k}\bs H(\vec a, \vec v) = (-1)^k\partial_{v_k} \bs H_k.
\end{equation}
We will tend to use the notation on the right hand side of these equalities in computations.

%We study the strongly-interacting regime, so will always assume that Lorentz paramters $\vec v$ are small, and the positions of the kinks are well-separated. We weight these requirements as follows. 
%\begin{definition}[Weight of modulation parameters]
%For all $(\vec a, \vec v) \in \bR^n\times \bR^n$, we set
%\begin{equation}
%\label{eq:rhoa-def}
%\rho(\vec a, \vec v) \coloneqq  \sum_{k=1}^{n-1} \eee^{-(a_{k+1} - a_k)} + \sum_{k=1}^n v_k^2.
%\end{equation}
Recalling the definition of $\rho(\vec a, \vec v)$ from~\eqref{eq:rhoa-def}, for $(\vec a, \vec v): I \to \bR^n\times \bR^n$, we define $\rho : I \to (0, \infty)$ by
\begin{equation}
\label{eq:rhot-def}
\rho(t) \coloneqq  \rho(\vec a(t), \vec v(t)) = \sum_{k=1}^{n-1} \eee^{-(a_{k+1}(t) - a_k(t))} + \sum_{k=1}^n v_k(t)^2.
\end{equation}
%\end{definition}

In what follows we will always assume that $\rho(\vec a, \vec v) \le \eta_0$, where $\eta_0>0$ is a small number to be eventually  fixed. We also introduce the following shorthand notation. 
Given a vector $\vec a \in \R^2$, with $a_1 \le a_2 \le \dots \le a_n$ we denote 
\EQ{ \label{eq:y-def1} 
\vec y &= (y_1, y_2, \dots, y_{n-1}) \in \R^{n-1}, \quad y_k \coloneqq  a_{k+1}- a_k\\
y_0 &\coloneqq  +\infty, \quad y_n \coloneqq  + \infty, \quad y_{\min}\coloneqq  \min_{k \in \{1, \dots, n-1\}} y_k .
}
Given $\vec v \in \R^n$ with $v_j \in (-1, 1)$ for each $j$ we set 
\EQ{
v_{\min} \coloneqq  \min_{k \in\{1, \dots, n\}} |v_k|
} 
and note that $\gamma_{v_{\min}} \ge 1$. 

We next define the interaction force between the kinks for solutions near a multikink configuration. Using  \eqref{eq:alpha-id-0} we compute
\EQ{ \label{eq:DE-expansion}
\vD E(\bs H(\vec a, \vec v))  &= \sum_{k =1}^n (-1)^k \vD E( \bs H_k) + \vD E( \bs H( \vec a, \vec v) ) - \sum_{k =1}^n (-1)^k \vD E( \bs H_k) \\
%& = - \sum_{k =1}^n (-1)^k v_k \bs \al_k +  \Big(\vD E( \bs H( \vec a, \vec v) ) - \sum_{k =1}^n (-1)^k \vD E( \bs H_k)\Big) \\
& =  - \sum_{k =1}^n (-1)^k v_k \bs \al_k  +  \pmat{ U'( H( \vec a, \vec v)) - \sum_{k=1}^n (-1)^kU'( H_k) \\ 0} .
}
Letting $\eta_0 <1$, let $\bs f \in C^\infty( \R^n \times B_{\R^n}(\vec 0, \sqrt{\eta_0}))$ be given by
\EQ{ \label{eq:Phi-def}
\bs f(\vec a,  \vec v) = \vD E(\bs H(\vec a, \vec v)) + \sum_{k=1}^n (-1)^kv_k \bs\alpha_k =  \pmat{ U'( H( \vec a, \vec v)) - \ds{\sum_{k=1}^n} (-1)^kU'( H_k) \\ 0}.
}
The $k$th interaction force is the defined as
\EQ{ \label{eq:F_k-def} 
F_k(\vec a, \vec v) &\coloneqq  (-1)^k\la \partial_{x} \bs H_k, \, \bs f( \vec a,  \vec v) \ra= -(-1)^k\la \partial_{a_k} \bs H_k, \, \bs  f( \vec a,  \vec v) \ra  = -  (-1)^k \la \bs \al_k, \,\bs  J  \bs f(\vec a, \vec v)\ra \\
&= (-1)^k \big\la \p_{x} H( a_k, v_k) , \, U'( H( \vec a, \vec v)) - \sum_{k=1}^n (-1)^kU'( H_k) \big\ra .
}
Analogously as in \eqref{eq:notation-v-0}, if the velocities are fixed to be zero, we set
\begin{equation} \label{eq:F_k-v-0}
F_k(\vec a) \coloneqq F_k(\vec a, \vec 0).
\end{equation}
Observe that $F_k(\vec a)=-\p_{a_k} E_{\up}( \bs H(\vec a))$, which is consistent with the analogy made by treating $\vec a = (a_1, \dots, a_n)$ as point masses. We also have $\sum_{k =1}^n F_k(\vec a) = 0$, due to the translation invariance of $E_p$, which we can view as a manifestation of  Newton's third  law. Since we have introduced the (small) Lorentz parameter $\vec v$ into the multikink ansatz, these identities no longer hold exactly, rather only approximately. 
%. However, since we restrict to the small $\vec v$ regime, we can think of this analogy with Newton's  laws holding in an approximate  sense. 
Indeed,
we observe the identity
\EQ{ \label{eq:sumF_k} 
\sum_{k =1}^n F_k(\vec a, \vec v)  =  -\sum_{k=1}^n \sum_{j=1}^n (-1)^k(-1)^j \la \p_{a_k} \bs H_k, \, v_j \bs \al_j\ra.
}
To see this, note that by the translation invariance of the total energy
\EQ{
0 = \sum_{k=1}^n \p_{a_k} E( \bs H( \vec a, \vec v)) %= \sum_{k=1}^n \la  \p_{a_k} \bs H(\vec a, \vec v), \vD E( \bs H(\vec a, \vec v) ) \ra 
= \sum_{k=1}^n (-1)^k \la \p_{a_k} \bs H_k, \, \vD E( \bs H(\vec a, \vec v) ) \ra .
}
Rearranging the above and using~\eqref{eq:Phi-def} and~\eqref{eq:F_k-def} gives
\EQ{ \label{eq:sum-F-id} 
\sum_{k =1}^n F_k(\vec a, \vec v)  =  -\sum_{k=1}^n (-1)^k \la \p_{a_k} \bs H_k, \, \bs f(\vec a, \vec v) \ra =  \sum_{k=1}^n \sum_{j=1}^n (-1)^k(-1)^j \la \p_{a_k} \bs H_k, \, v_j \bs \al_j\ra  , 
}
as claimed. The following lemma quantifies the smallness of $\sum_{j=1}^{n} F_k (\vec a, \vec v)$.

\begin{lemma} \label{lem:sumF} There exist constants $\eta_0>0, C>0$ so that for every $(\vec a, \vec v) \in \R^n \times \R^n$ with $\rho(\vec a, \vec v) \le \eta_0$, 
\EQ{
\big| \sum_{k =1}^n F_k(\vec a, \vec v) \big| \le C | \vec v | y_{\min} e^{- \gamma_{v_{\min}} y_{\min}}.
}

\end{lemma}

Lemma~\ref{lem:sumF} is a straightforward consequence of the identity~\eqref{eq:sum-F-id} and the following lemma, proved in~\cite{JKL1}. We also use this next lemma in subsequent results estimating the size of interactions. 
\begin{lemma}
\label{lem:exp-cross-term}\emph{\cite[Proof of Lemma 2.5]{JKL1}}
For any $x_1 < x_2$ and $\alpha, \beta > 0$ with $\alpha \neq \beta$ the following bound holds:
\begin{equation}
\int_{\bR}\eee^{-\alpha(x - x_1)_+}\eee^{-\beta(x_2 - x)_+}\ud x \lesssim_{\alpha, \beta} \eee^{-\min(\alpha, \beta)(x_2 - x_1)}.
\end{equation}
For any $\alpha > 0$, the following bound holds:
\begin{equation}
\int_{\bR}\eee^{-\alpha(x - x_1)_+}\eee^{-\alpha(x_2 - x)_+}\ud x \lesssim_{\alpha} (1 + x_2 - x_1)\eee^{-\alpha(x_2 - x_1)}.
\end{equation} \qedno
\end{lemma}
%\begin{proof}
%See~\cite[Proof of Lemma 2.5]{JKL1}.  
%\end{proof}

\begin{proof}[Proof of Lemma~\ref{lem:sumF}] 
First observe that by~\eqref{eq:identities-1} the terms with $j=k$ on the right side of~\eqref{eq:sumF_k} vanish. Then, by Proposition~\eqref{prop:prop-H} and Lemma~\ref{lem:exp-cross-term} we have, for each $j \neq k$,  
\EQ{ 
|\la \p_{a_k} \bs H_k, \, v_j \bs \al_j\ra| \lesssim | \vec v|^2 y_{\min} e^{- \gamma_{v_{\min}} y_{\min}}
}
which gives the lemma. 
\end{proof}

\begin{lemma}
Denote by $f\in C^\infty(\R^n)$ the function
\begin{equation}
f(w_1, \ldots, w_n) \coloneqq  U'\Big( 1 + \sum_{k=1}^n(-1)^k(w_k + 1) \Big) - \sum_{k=1}^n (-1)^k U'(w_k).
\end{equation}
There exists $C$ such that for all $\vec w \in [-2, 2]^n$ and $k \in \{1, \ldots, n\}$
\begin{equation}
\label{eq:Phi-taylor}
\begin{aligned}
&\big|f(\vec w) + (-1)^{k}(U''(w_k) - 1)((1+w_{k+1}) - (1 - w_{k-1}))\big| \leq \\
&\qquad C\big(\max_{j < k-1}|1-w_j| + \max_{j > k+1}|1+w_j| + \\
&\qquad\qquad+ |1-w_k||1+w_{k+1}|^2 +
|1+w_k||1-w_{k-1}|^2 + |1+w_{k+1}||1-w_{k-1}|\big),
\end{aligned}
\end{equation}
where by convention $w_0 \coloneqq 1$, $w_{n+1} \coloneqq -1$ and $\max_{j < 0}|1-w_j| = \max_{j > n+1}|1+w_j|= 0$.
\end{lemma}
\begin{proof}
Let $\vec v \in \bR^n$ be given by $v_j \coloneqq 1$ for $j < k$, $v_j \coloneqq -1$ for $j > k$ and $v_k \coloneqq w_k$.
The Taylor formula yields
\begin{equation}
\label{eq:Phi-taylor-1}
\begin{aligned}
f(\vec w) &= f(\vec v) + (\vec w - \vec v)\cdot\grad f(\vec v) \\
&+ \sum_{i, j = 1}^n (w_i - v_i)(w_j - v_j)\int_0^1 (1-t)\partial_{i}\partial_{j}f((1-t)\vec v + t\vec w)\ud t.
\end{aligned}
\end{equation}
We compute and estimate all the terms, calling a quantity ``negligible''
if its absolute value is smaller than the right hand side of \eqref{eq:Phi-taylor}.
For all $i, j \in \{1, \ldots, n\}$ and $\vec u \in \bR^n$, we have
\begin{align}
\partial_{j}f(\vec u) &= (-1)^j\Big(U''\Big(1 + \sum_{\ell=1}^n (-1)^\ell(u_\ell + 1)\Big) - U''(u_j)\Big), \\
\partial_{i}\partial_{j}f(\vec u) &= (-1)^{i+j}U'''\Big(1 + \sum_{\ell=1}^n (-1)^\ell(u_\ell + 1)\Big), \qquad\text{if }i \neq j, \\
\partial_{j}^2 f(\vec u) &= U'''\Big(1 + \sum_{\ell=1}^n (-1)^\ell(u_\ell + 1)\Big) - (-1)^jU'''(u_j).
\label{eq:dj2Phi}
\end{align}
Observe that
\begin{equation}
1 + \sum_{\ell=1}^n (-1)^\ell (v_\ell + 1) = 1 + 2\sum_{\ell=1}^{k-1}(-1)^\ell + (-1)^k(w_k + 1) = (-1)^k w_k,
\end{equation}
in particular, since $U''$ is even and $U''(1) = 1$, for all $j \neq k$ we have $\partial_{j}f(\vec v) = (-1)^j(U''(w_k) - 1)$.
We thus obtain
\begin{equation}
\begin{aligned}
(\vec w - \vec v)\cdot\grad f(\vec v) &\sim (w_{k-1}-1)\partial_{{k-1}}f(\vec v) + (w_{k+1}+1)\partial_{{k+1}}f(\vec v) \\
&= {-}(-1)^k(U''(w_k) - 1)((1+w_{k+1}) - (1 - w_{k-1})).
\end{aligned}
\end{equation}

Remains the second line of \eqref{eq:Phi-taylor-1}. The terms with $i \notin \{k-1, k+1\}$
or $j \notin \{k-1, k+1\}$ are clearly negligible,
as are the terms $(i, j) \in \{(k-1, k+1), (k+1, k-1)\}$,
hence it suffices to consider the cases $(i, j) \in \{(k-1, k-1), (k+1, k+1)\}$.
Since the latter is analogous to the former,
we only consider $i = j = k-1$,
in particular we assume $k \geq 2$.

Using \eqref{eq:dj2Phi}, the fact that $U'''$ is locally Lipschitz, and writing
\begin{equation}
1 + \sum_{\ell=1}^n(-1)^\ell(u_\ell + 1)
= \sum_{\ell=1}^{k-2}(-1)^\ell(u_\ell-1) + (-1)^{k-1}u_{k-1}
+ \sum_{\ell = k}^{n}(-1)^\ell(u_\ell + 1),
\end{equation}
we obtain
\begin{equation}
|\partial_{k-1}^2 f(\vec u)| \lesssim \sum_{\ell = 1}^{k-2}|1-u_\ell| + \sum_{\ell = k}^n |1 + u_\ell|.
\end{equation}
Setting $\vec u \coloneqq (1-t)\vec v + t\vec w$
and multiplying by $(w_{k-1} - v_{k-1})^2 = (1-w_{k-1})^2$,
we obtain a negligible term as claimed.
\end{proof}

%We denote
%\begin{equation}
%\label{eq:y-first-def}
%\begin{gathered}
%\vec y = (y_1, \ldots, y_{n-1}), \quad y_k \coloneqq a_{k+1} - a_k, \\
%y_0 \coloneqq +\infty, \quad y_n \coloneqq +\infty,\quad y_{\min} \coloneqq \min_{1 \leq k \leq n-1}y_k.
%\end{gathered}
%\end{equation}

\begin{lemma}
\label{lem:U-pyth}
There exists $C$ such that for all $\vec w \in [-2, 2]^n$
\begin{gather}
\label{eq:U-pyth}
\Big|U\Big( 1 + \sum_{k=1}^n(-1)^k(w_k + 1) \Big) - \sum_{k=1}^n U(w_k) \Big| \leq C \max_{1 \leq i < j \leq n}|1 - w_i||1 + w_j|, \\
\label{eq:U'-pyth}
\Big|U'\Big( 1 + \sum_{k=1}^n(-1)^k(w_k + 1) \Big) - \sum_{k=1}^n (-1)^k U'(w_k) \Big| \leq C \max_{1 \leq i < j \leq n}|1 - w_i||1 + w_j|.
\end{gather}
\end{lemma}
\begin{proof}
We proceed by induction with respect to $n$. For $n = 1$, both sides of both estimates equal $0$.

Let $n > 1$ and set
\begin{equation}
v \coloneqq 1 + \sum_{k=1}^{n-1}(-1)^k(w_k + 1) = \sum_{k=1}^{n-1}(-1)^k(w_k - 1) - (-1)^n.
\end{equation}
Integrating the bound $|U'(w + (-1)^n(w_n + 1)) - U'(w)| \lesssim |w_n +1|$
for $w$ between $-(-1)^n$ and $v$, and using $U({-}(-1)^n) = 0$ as well as $U((-1)^n w_n) = U(w_n)$, we get
\begin{equation}
\begin{aligned}
\label{eq:U-pyth-step}
|U(v + (-1)^n(w_n + 1)) - U(w_n) - U(v)| &\leq C|(-1)^n+v||1+w_n| \\
&\leq C|1+w_n|\sum_{k=1}^{n-1}|1-w_k|,
\end{aligned}
\end{equation}
which finishes the induction step for \eqref{eq:U-pyth}.

Integrating the bound $|U''(w + (-1)^n(w_n + 1)) - U''(w)| \lesssim |w_n +1|$
for $w$ between $-(-1)^n$ and $v$, and using $U'({-}(-1)^n) = 0$ as well as $U'((-1)^n w_n) = (-1)^nU'(w_n)$, we get
\begin{equation}
\begin{aligned}
\label{eq:U-pyth-step-bis}
|U'(v + (-1)^n(w_n + 1)) - (-1)^nU'(w_n) - U'(v)| &\leq C|(-1)^n+v||1+w_n| \\
&\leq C|1+w_n|\sum_{k=1}^{n-1}|1-w_k|,
\end{aligned}
\end{equation}
which finishes the induction step for \eqref{eq:U'-pyth}. 
\end{proof}

\begin{lemma}  \label{lem:U-bounds} 
There exists constants $C>0, \eta_0>0$ such that for every $(\vec a, \vec v) \in \R^n$ with $\rho(\vec a, \vec v) \le \eta_0$ 
\EQ{ \label{eq:dU-bounds} 
\| U'( H( \vec a, \vec v)) - \sum_{k=1}^n (-1)^k U'( H_k) \|_{H^1} \le C  \sqrt{y_{\min}}e^{-y_{\min}}
}
and  
\EQ{ \label{d2U-bounds}
\max_{1 \le k \le n} \| \partial_{a_k} H_k (U''( H( \vec a, \vec v)) -  U''( H_k)) \|_{H^1} &\le C  \sqrt{y_{\min}}e^{-y_{\min}} \\
\max_{1 \le k \le n} \| \partial_{v_k} H_k (U''( H( \vec a, \vec v)) -  U''( H_k)) \|_{H^1} &\le  C\abs{\vec v}\sqrt{y_{\min}}e^{-y_{\min}}.
}
\end{lemma} 

\begin{remark} \label{lem:sizeDEp}Using the identity~\eqref{eq:DE-expansion} we see that that above implies the estimates
\EQ{
\label{eq:DEH}
\|\vD E(\bs H(\vec a, \vec v)) + \sum_{k=1}^n (-1)^kv_k \bs\alpha_k\|_{H^1 \times H^2} \leq C \sqrt{y_{\min}}\eee^{-y_{\min}}, \\
%\|\vD E(\bs H(\vec a, \vec v)) + \sum_{k=1}^n (-1)^kv_k \bs\alpha_k\|_{H^1 \times H^2} \leq C \sqrt{y_{\min}}\eee^{-y_{\min}}
}  
or
\begin{equation}
\label{eq:DEH-2}
\|\bs J\vD E(\bs H(\vec a, \vec v)) - \sum_{k=1}^n (-1)^k v_k \partial_{a}\bs H_k\|_\cE \leq C \sqrt{y_{\min}}\eee^{-y_{\min}}.
\end{equation}
Also, setting $\vec v = 0$ in Lemma~\ref{lem:U-bounds} yields the estimates
\begin{align}
\|\vD E_{\up}(H(\vec a))\|_{L^2} \leq C\sqrt{y_\tx{min}}\eee^{-y_\tx{min}},
\label{eq:sizeDEp} \\
\max_{1 \leq k \leq n} \big\|\partial_x H(\cdot - a_k)\big(U''(H(\vec a)) - U''(H(\cdot - a_k))\big)\big\|_{L^2} \leq C\sqrt{y_\tx{min}}\eee^{-y_\tx{min}}.
\label{eq:size-diff-pot}
\end{align}
for every increasing $n$-tuple $\vec a$.
\end{remark} 

%\begin{lemma} \Red{The later references to this lemma should be changed to reference \eqref{dU-bounds} and~\eqref{eq:d2U-bounds} from Lemma~\ref{lem:U-bounds}}
%
%\end{lemma}
%

\begin{proof}[Proof of Lemma~\ref{lem:U-bounds}] %\red{fix} Since $\partial_x^2 H = U'(H)$, 
We see that \eqref{eq:U'-pyth} yields
\begin{equation}
%\Big| \partial_x^2 H(\vec a) - U'(H(\vec a))\Big| = 
\Big| \sum_{k=1}^n (-1)^k U'(H_k) - U'(H(\vec a, \vec v))\Big| \lesssim \max_{1 \leq i < j \leq n}
|1 - H_i||1 + H_j|.
\end{equation}
After taking the square, integrating over $\bR$ and applying Lemma~\ref{lem:exp-cross-term}, we obtain the $L^2$ bound in \eqref{eq:sizeDEp} since $\gamma_{v_{\min}} \ge 1$.

In order to prove the $L^2$ bound in the first estimate in \eqref{d2U-bounds}, we write
\begin{equation}
H(\vec a, \vec v) = \sum_{\ell=1}^{k-1}(-1)^\ell(H_\ell-1) + (-1)^{k}H_k
+ \sum_{\ell = k+1}^{n}(-1)^\ell(H_\ell + 1).
\end{equation}
Since $U''$ is locally Lipschitz and $|\partial_x H| \lesssim \min(|1 - H|, |1+H|)$, we obtain
\begin{equation}
\big| \partial_x H_k \big(U''(H(\vec a, \vec v)) - U''(H_k)\big)\big| \lesssim \max_{1 \leq \ell < k}
|1 - H_\ell||1 + H_k| + \max_{k < \ell \leq n}|1 - H_k||1 + H_\ell|,
\end{equation}
and we conclude as above.

To conclude the $H^1$ component of these estimates we note  the computations 
\EQ{
\p_x \big(U'( H( \vec a, \vec v)) - \sum_{k=1}^n (-1)^k U'( H_k)\big) = \sum_{k=1}^n (-1)^k \p_x H_k ( U''(H(\vec a, \vec v)) - U''(H_k))
}
and  
\EQ{
\p_x \big( \partial_{a_k} H_k (U''( H( \vec a, \vec v)) -  U''( H_k)) \big) &= \p_x^2 H_k (U''( H( \vec a, \vec v)) -  U''( H_k)) \\
&\quad  - (\p_x H_k)^2\Big( (-1)^k U'''(H(\vec a, \vec v)) - U'''(H_k)\Big) \\
&\quad - \sum_{j\neq k}(-1)^j (\p_x H_k) (\p_x H_j)U'''(H(\vec a, \vec v)).
}
and proceed as above. The second estimate in~\eqref{d2U-bounds} is similar. 
\end{proof}

\begin{lemma} 
\label{lem:Fz}
There exists $C, \eta_0>0$ such that for every $(\vec a, \vec v)$ such that $\rho(\vec a, \vec v) \le \eta_0$
%increasing $n$-tuple $\vec a$ and $k \in \{1, \ldots, n\}$ and each $\abs{ \vec v} \le \eta_0$, 
\begin{align}
\label{eq:Fz}
\big|F_k(\vec a, \vec v) - 2\kappa^2 (\eee^{-y_k} - \eee^{-y_{k-1}})\big| \leq Cy_{\min}\eee^{-2 y_{\min}} + C| v|^2 y_{\min}\eee^{- y_{\min}}, %\\
%\label{eq:dFz}
%\big|F'(z) + 2\|\partial_x H\|_{L^2}^{-2}\kappa^2 \eee^{-z}\big| \leq Cz \eee^{-2z},
\end{align}
where $\kappa$ is defined in Proposition~\ref{prop:prop-H}. 
\end{lemma}

\begin{proof}%[Proof of Lemma~\ref{lem:Fz}] 
We first show that 
\EQ{ \label{eq:Fkv-Fk0}
\big|F_k(\vec a, \vec v) - F_k( \vec a, \vec 0) \big| \lesssim | v|^2 y_{\min}\eee^{- y_{\min}}.
}
To see this, we write 
\EQ{
F_k(\vec a, \vec v) - F_k( \vec a, \vec 0)&= (-1)^k \la \p_{a}  H(a_k, 0) - \p_a  H_k, \,  f( \vec a, \vec v) \ra \\
&\quad + (-1)^k \la  H( a_k, 0) , \,  f( \vec a, \vec 0) - f(\vec a, \vec v) \ra
}
where we are writing $\bs f( \vec a, \vec v) = ( f(\vec a, \vec v), 0)$; see~\eqref{eq:Phi-def}.
For the first term on the right we use the estimate %\red{WHAT NORMS?; CHECK BOLDFACE}
\EQ{
\| \p_a H(a_k, 0) - \p_a  H(a_k, v_k) \|_{L^2}&= \| \int_0^{v_k} \p_v \p_a H( a_k, v) \, \ud v \|_{L^2}   \lesssim | \vec v|^2 \Big( \| x \p_x^2 H \|_{L^2} + \|\p_x H \|_{L^2} \Big)
%& \lesssim |\vec v|^2 \sup_{v \in [0, v_k]}  |(\p_x^2 H)( \gamma_v( \cdot- a_k)| + | (\p_x H)( \gamma_v(\cdot - a_k))|),
}
which, using~\eqref{eq:dU-bounds} and the definition of $\bs f(\vec a, \vec v)$, gives
\EQ{
|  \la \p_{a}  H(a_k, 0) - \p_a  H_k, \,  f( \vec a, \vec v) \ra | \lesssim |\vec v|^2 \sqrt{y_{\min}} e^{- y_{\min}}.
}
For the second term, given a vector $\vec v \in \R^n$, we introduce the temporary notation $\vec v(j) = (0, \dots, 0, v_{j+1}, \dots v_n)$ with the conventions $\vec v(0) = \vec v$ and $\vec v(n) = \vec 0$. and write 
\EQ{
 f( \vec a, \vec 0) -  f(\vec a, \vec v) = \sum_{ j =0}^n  f( \vec a, \vec v(j)) - f( \vec a, \vec v(j+1)).
}
For each of the terms on the right we have 
\EQ{
&\| f( \vec a, \vec v(j)) -  f( \vec a, \vec v(j+1))\|_{L^2} = \Big\| \int_0^{v_{j+1}} \p_{v_{j+1}} f( \vec a, (0, \dots, v, v_{j+2}, \dots, v_n)) \, \ud v \Big\|_{L^2} \\
& \lesssim  \int_0^{v_{j+1}} \Big\| \p_{v} H( a_{j+1}, v) \Big( U''( H( \vec a,(0, \dots, v, v_{j+2}, \dots, v_n)) - U''(a_{j+1}, v) \Big)\Big\|_{L^2} \, \ud v  .
%& \lesssim |v|^2 \sup_{v \in [0, v_{j+1}]} \Big| (\p_x H)( \gamma_{v}(x-a)) \Big( U''( H( \vec a, (0, \dots, v, v_{j+2}, \dots, v_n))) - U''(a_{j+1}, v) \Big) \, \ud v \Big|
}
Using~\eqref{d2U-bounds} and the above we conclude that 
\EQ{
| \la  H( a_k, 0) , \, f( \vec a, \vec 0) -  f(\vec a, \vec v) \ra| \lesssim |v|^2 \sqrt{y_{\min}} e^{-y_{\min}}, 
}
giving~\eqref{eq:Fkv-Fk0}. We have thus reduced to the case $\vec v = \vec 0$, which we will assume for the remainder of the computations in the proof. 
By Proposition~\ref{prop:prop-H} and Lemma~\ref{lem:exp-cross-term}, we have
\begin{equation}
\label{eq:F-cross-est}
\begin{aligned}
\int_{-\infty}^{\infty} &|\partial_x H_k|\big(\max_{j < k-1}|1-H_j| + \max_{j > k+1}|1+H_j| + |1-H_k||1+H_{k+1}|^2 +\\
&\qquad +|1+H_k||1-H_{k-1}|^2 + |1+H_{k+1}||1-H_{k-1}|\big)\ud x \lesssim y_{\min}\eee^{-2y_{\min}}.
\end{aligned}
\end{equation}
Thus, \eqref{eq:F_k-def} and \eqref{eq:Phi-taylor} yield
\begin{equation}
\label{eq:F-cross-est-1}
\Big|F_k(\vec a, \vec 0) + \int_{-\infty}^{\infty}\partial_x H_k(U''(H_k) - 1)((1 + H_{k+1}) - (1 - H_{k-1}))\ud x \Big| \lesssim y_{\min}\eee^{-2y_{\min}}.
\end{equation}
Applying again Proposition~\ref{prop:prop-H}, we have
\begin{equation}
|1 + H_{k+1} - \kappa \eee^{x - a_{k+1}}| \lesssim
\begin{cases}
\eee^{-2(a_{k+1} - x)}\qquad&\text{if }x \leq a_{k+1} \\
\eee^{x - a_{k+1}}\qquad&\text{if }x \geq a_{k+1},
\end{cases}
\end{equation}
thus, taking into account that $|\partial_x H_k(x)| + |U''(H_k(x)) - 1| \lesssim \eee^{-|x - a_k|}$,
\begin{equation}
\begin{aligned}
&\bigg|\int_{-\infty}^\infty \partial_x H_k(U''(H_k) - 1)(1+H_{k+1})\ud x - 
\kappa\int_{-\infty}^\infty \partial_x H_k(U''(H_k) - 1)\eee^{x - a_{k+1}}\ud x\bigg| \lesssim \\
&\qquad\lesssim \int_{-\infty}^{a_{k+1}} \eee^{-2|x - a_k|}\eee^{-2(a_{k+1} - x)}\ud x + \int_{a_{k+1}}^\infty \eee^{-2(x-a_k)}\eee^{x - a_{k+1}}\ud x.
\end{aligned}
\end{equation}
The first integral on the right hand side is $\lesssim y_{\min}\eee^{-2y_{\min}}$ by Lemma~\ref{lem:exp-cross-term},
and the second equals $\eee^{-2(a_{k+1} - a_k)}$.
Hence, \eqref{eq:reduced-force} yields
\begin{equation}
\bigg|\int_{-\infty}^\infty \partial_x H_k(U''(H_k) - 1)(1+H_{k+1})\ud x + 2\kappa^2 \eee^{-(a_{k+1} - a_k)}\bigg| \lesssim y_{\min}\eee^{-2y_{\min}}.
\end{equation}
Similarly, 
\begin{equation}
\bigg|\int_{-\infty}^\infty \partial_x H_k(U''(H_k) - 1)(1-H_{k-1})\ud x + 2\kappa^2 \eee^{-(a_{k} - a_{k-1})}\bigg| \lesssim y_{\min}\eee^{-2y_{\min}}.
\end{equation}
These two bounds, together with \eqref{eq:F-cross-est-1}, yield \eqref{eq:Fz}.
\end{proof}
\begin{remark}
We see from Lemma~\ref{lem:Fz} that the interaction between consecutive kink and antikink is attractive.
\end{remark}

We turn to computing the first nontrivial term in the expansion of the energy of a multi-kink configuration. 
\begin{lemma}
\label{lem:interactions}  
There exists $C > 0, \eta_0>0$ such that for every $(\vec a, \vec v)$ with $\rho(\vec a, \vec v) \le \eta_0$
\begin{align}
\Big| E( \bs H(\vec a, \vec v)) - Mn - \frac{1}{2} M \sum_{k =1}^n v_k^2  +2\kappa^2 \sum_{k=1}^{n-1} \eee^{-y_k} \Big|  
  \lesssim y_{\min}\eee^{- 2y_{\min}} + y_{\min}\abs{v}^4.
%\bigg|E_{\up}(H(\vec a, \vec v)) - nM + 2\kappa^2\sum_{k=1}^{n-1}\eee^{-y_k}\bigg| &\leq C y_{\min}\eee^{-2y_{\min}} + C | v|^2 y_{\min} e^{- y_{\min}} + C|v|^4.
\label{eq:EpHX-a}
\end{align}
\end{lemma}

\begin{proof}
First we compute the kinetic energy. 
\EQ{
E_{\uk}( \dot { H}(\vec a, \vec v)) &= E_{\uk}\big( \sum_{k =1}^n (-1)^k (-v_k \gamma_{v_k} \p_x H( \gamma_{v_k}( x - a_k)) )\big) \\
&= \frac{1}{2} \int_{-\infty}^\infty \Big( \sum_{k =1}^n (-1)^k (-v_k \gamma_{v_k} \p_x H( \gamma_{v_k}( x - a_k)) )\Big)^2 \, \ud x.
}
By Lemma~\ref{lem:exp-cross-term}, the cross terms (when $j \neq k$) satisfy 
\EQ{
 \Big| v_j v_k  \gamma_{v_j} \gamma_{v_k} \int_{-\infty}^\infty  \p_x H( \gamma_{v_k}( x - a_k))  \p_x H( \gamma_{v_j}( x - a_j)) \, \ud x \Big|  \lesssim |v|^2 y_{\min} e^{- \gamma_{v_{\min}} y_{\min}} ,
}
which is negligible since $\gamma_{v_{\min}} \ge 1$. The terms when $j =k$ then satisfy 
\EQ{
\frac{1}{2} v_k^2 \gamma_{v_k}^2 \int_{-\infty}^\infty  \Big(\p_x H( \gamma_{v_k}( x - a_k))\Big)^2 \, \ud x =\frac{1}{2} v_k^2 \gamma_{v_k}   M %= v_k^2 \frac{M}{2} + O( | v|^4) 
}
from which we see that the kinetic energy satisfies 
\EQ{
\Big| E_{\uk}( \dot { H}(\vec a, \vec v)) - \frac{1}{2} M \sum_{k=1}^n v_k^2 \gamma_{v_k}| \lesssim |v|^2 y_{\min} e^{- \gamma_{v_{\min}} y_{\min}} .
}
From \eqref{eq:U-pyth} and Lemma~\ref{lem:exp-cross-term}, we have
\begin{equation}
\int_{-\infty}^\infty \Big| U(H(\vec a, \vec v)) - \sum_{k=1}^n U(H_k)\Big|\ud x \lesssim y_{\min}\eee^{- \gamma_{v_{\min}} y_{\min}}.
\end{equation}
Invoking again Lemma~\ref{lem:exp-cross-term}, we also have
\begin{multline}
\label{eq:EpHX-a-1}
\int_{-\infty}^\infty \Big| \frac 12(\partial_x H(\vec a, \vec v))^2 - \frac 12\sum_{k=1}^n (\partial_x H( a_k, v_k))^2\Big|\ud x \\\leq \int_{-\infty}^\infty\sum_{1 \leq i < j \leq n}|\partial_x H_i||\partial_x H_j|\ud x \lesssim y_{\min}\eee^{- \gamma_{v_{\min}} y_{\min}}.
\end{multline}
Combining these two bounds, with the identity 
\EQ{
E_{\up}( H( a,v)) =  \gamma_v M - \frac{1}{2} v^2 \gamma_v M ,
}
we obtain
\begin{equation}
\big|E_{\up}(H(\vec a, \vec v)) -  M \sum_{k =1}^n \gamma_{v_k} + \frac{1}{2} M \sum_{k =1}^n v_k^2 \gamma_k\big| \lesssim y_{\min}\eee^{- \gamma_{v_{\min}}y_{\min}}.
\end{equation}
Combining with the kinetic energy arrive at the estimate
\EQ{
\Big| E( \bs H(\vec a, \vec v)) - M \sum_{k =1}^n \gamma_{v_k} \Big|   \lesssim y_{\min}\eee^{- \gamma_{v_{\min}}y_{\min}} .
}
For all $s \geq 0$, set $\vec a(s) \coloneqq (a_1 + s, a_2 + 2s, \ldots, a_n + ns)$.
Applying the last estimate with $\vec a(s)$ instead of $\vec a$, we get
\begin{equation}
\lim_{s \to \infty}E(\bs H(\vec a(s), \vec v)) =  M \sum_{k =1}^n \gamma_{v_k}.
\end{equation}
By the Chain Rule, \eqref{eq:F_k-def}, \eqref{eq:Fz}, and noting the orthogonality $0=\la \bs \al(a, v), \, \p_a \bs H(a, v)\ra$ we have
\EQ{
\dd s E(\bs H(\vec a(s), \vec v)) &= \la \vD E( \bs H(\vec a(s), \vec v)) , \, \sum_{k =1}^n (-1)^k k \p_{a_k} \bs H(a_k, v_k) \ra  \\
& = \sum_{k =1}^n (-1)^k k \la \bs f( \vec a(s), \vec v) , \,  \p_{a} \bs H(a_k + ks, v_k) \ra  \\
&\quad - \sum_{k =1}^n  \sum_{j =1}^n (-1)^k (-1)^j k v_j \la \bs \al( a_j + j s), \,  \p_{a} \bs H( a_k + k s, v_k) \ra  \\
& =-  \sum_{k =1}^n k F_k( \vec a(s), \vec v) + O\big( | \vec v|^2 (y_{\min}+s)\eee^{-(y_{\min}+s)}\big) \\
& = 2\kappa^2 \sum_{k=1}^{n-1} \eee^{-y_k-s} + O\big( | \vec v|^2 (y_{\min}+s)\eee^{-(y_{\min}+s)}\big) + O\big((y_{\min}+s)\eee^{-2(y_{\min}+s)}\big).
}
An integration in $s$ yields 
\EQ{ 
\Big| E( \bs H(\vec a, \vec v)) - M \sum_{k =1}^n \gamma_{v_k}  +2\kappa^2 \sum_{k=1}^{n-1} \eee^{-y_k} \Big|  
  \lesssim y_{\min}\eee^{- 2y_{\min}} + y_{\min}\abs{v}^4.
}
Lastly, noting the small $v$ estimate $\gamma_v = 1 + \frac{1}{2} v^2 + O( v^4)$, we obtain 
\EQ{
\Big| E( \bs H(\vec a, \vec v)) - Mn - \frac{1}{2} M \sum_{k =1}^n v_k^2  +2\kappa^2 \sum_{k=1}^{n-1} \eee^{-y_k} \Big|  
  \lesssim y_{\min}\eee^{- 2y_{\min}} + y_{\min}\abs{v}^4,
}
as claimed. 
%\begin{equation}
%\dd s E_p(H(\vec a(s))) = {-}\sum_{k=1}^n k F_k(\vec a(s)) = 2\kappa^2 \sum_{k=1}^{n-1} \eee^{-y_k-s} + O\big((y_{\min}+s)\eee^{-2(y_{\min}+s)}\big).
%\end{equation}
%An integration in $s$ yields \eqref{eq:EpHX-a}.
\end{proof}

\subsection{Taylor expansions}
We recall a few estimates based on the Taylor expansion of the nonlinearity from~\cite{JL9}. 
\begin{lemma} There exist $C, \eta > 0$ such that for all $-1-\eta \leq w, \sh w \leq 1+\eta$ and $|g| + |\sh g| \leq \eta$
the following bounds hold:
\begin{gather}
\label{eq:U-taylor}
\big|U(w + g) - U(w) - U'(w)g\big| \leq Cg^2, \\
\label{eq:U-taylor-2}
\big|U(w + g) - U(w) - U'(w)g - U''(w)g^2/2\big| \leq Cg^3, \\
\label{eq:U'-taylor}
\big|U'(w + g) - U'(w) - U''(w)g\big| \leq Cg^2, \\
\label{eq:U''-taylor}
\big|U''(w + g) - U''(w) - U'''(w)g\big| \leq Cg^2,
\end{gather}
\begin{equation}
\label{eq:U'-taylor-diff}
\begin{aligned}
\big|\big(U'(\sh w + \sh g) - U'(\sh w) - U''(\sh w)\sh g\big) -
\big(U'(w + g) - U'(w) - U''(w)g\big)\big| \\
\leq C\big(|\sh g| + |g|\big)\big(\big|\sh g - g\big| + |\sh w - w|\big(|\sh g| + |g|)\big).
\end{aligned}
\end{equation}
\end{lemma}
\begin{proof}
Bounds \eqref{eq:U-taylor}, \eqref{eq:U-taylor-2}, \eqref{eq:U'-taylor} and \eqref{eq:U''-taylor} are clear,
so we are left with \eqref{eq:U'-taylor-diff}.
The Taylor formula yields
\begin{align}
U'(w + g) - U'(w) - U''(w)g &= \int_0^g s U'''(w+g-s)\ud s, \\
U'(\sh w + \sh g) - U'(\sh w) - U''(\sh w)\sh g &= \int_0^{\sh g} s U'''(\sh w+\sh g-s)\ud s.
\end{align}
We observe that, since $U'''$ is locally Lipschitz,
\begin{equation}
\begin{aligned}
\bigg|\int_0^{\sh g} s U'''(\sh w+\sh g-s)\ud s - \int_0^{\sh g} s U'''(w+g-s)\ud s\bigg| &\lesssim \bigg|\int_0^{\sh g} s|\sh w + \sh g - w - g|\ud s\bigg| \\
&\lesssim (\sh g)^2(|\sh w - w| + |\sh g - g|).
\end{aligned}
\end{equation}
We also have
\begin{equation}
\bigg|\int_0^{\sh g} s U'''(w+g-s)\ud s - \int_0^{g} s U'''(w+g-s)\ud s\bigg| = \bigg|\int_g^{\sh g}s U'''(w + g - s)\ud s\bigg|
\lesssim \bigg|\int_g^{\sh g}|s|\ud s\bigg|.
\end{equation}
If $g$ and $\sh g$ have the same sign, then the last integral equals $\frac 12 |(\sh g)^2 - g^2| = \frac 12 (|\sh g| + |g|)|\sh g - g|$.
If $g$ and $\sh g$ have opposite signs, then we obtain $\frac 12 ((\sh g)^2 + g^2) \leq \frac 12 (|\sh g| + |g|)^2 = \frac 12 (|\sh g| + |g|)|\sh g - g|$. This proves \eqref{eq:U'-taylor-diff}.
\end{proof}

\subsection{Schr\"odinger operator with multiple potentials}
\label{ssec:schrod}
We define
\begin{equation}
L \coloneqq \vD^2 E_p(H) = -\partial_x^2 + U''(H) = -\partial_x^2 + 1 + (U''(H) - 1).
\end{equation}
Differentiating $\partial_x^2 H(x - a) = U'(H(x-a))$ with respect to $a$ we obtain
\begin{equation}
\label{eq:Lc-ker}
\big({-}\partial_x^2 + U''(H(\cdot - a))\big)\partial_x H(\cdot - a) = 0,
\end{equation}
in particular for $a = 0$ we have $L(\partial_x H) = 0$.
Since $\partial_x H$ is a positive function, $0$ is a simple eigenvalue of $L$,
which leads to the following coercivity estimate
(see \cite[Lemma 2.3]{JKL1} for the exact same formulation, as well as \cite{HPW82} for a similar result).
\begin{lemma}
There exist $\nu, C > 0$ such that for all $h \in H^1(\bR)$ the following inequality holds:
\begin{equation}
\label{eq:vLv-coer}
\la h, Lh\ra \geq \nu \|h\|_{H^1}^2 -C \la \partial_x H, h\ra^2.
\end{equation}\vspace{-1cm}\qedno
\end{lemma}

For any increasing $n$-tuple $\vec a$, and $\vec v \in B_{\R^n}(0, 1)$, 
we denote $L(\vec a, \vec v) \coloneqq \vD^2 E_p(H(\vec a, \vec v))$,
which is the Schr\"odinger operator on $L^2(\bR)$ given by
\begin{equation}
\label{eq:schrod-op}
(L(\vec a, \vec v)h)(x) \coloneqq {-}\partial_x^2 h(x) + U''(H(\vec a, \vec v; x))h(x).
\end{equation}
\begin{lemma}
\label{lem:D2H}
There exist $y_0, \nu, C > 0$ such that the following holds.
Let $\vec a \in \bR^n$ satisfy $a_{k+1} - a_k \geq y_0$ for all $k \in \{1, \ldots, n-1\}$
and let $h \in H^1(\bR)$. Then
\begin{equation}
\label{eq:D2H-coer}
\begin{aligned}
\la h, L(\vec a, \vec 0) h\ra \geq \nu \|h\|_{H^1}^2 -C\sum_{k=1}^n \la\partial_x H_k, h\ra^2.
\end{aligned}
\end{equation}
\end{lemma}

\begin{proposition} 
\label{prop:coer} 
There exist $\nu, C, \eta_0 > 0$ such that for all $(\vec a, \vec v)$ satisfying
$\rho(\vec a, \vec v) \leq \eta_0$ and $\bs h_0 \in \cE$ the following inequality holds:
\begin{equation}
\label{eq:coer}
\la \bs h_0, D^2 E(\bs H(\vec a, \vec v))\bs h_0\ra \geq \nu \|\bs h_0\|_\cE^2 -
C\sum_{k=1}^n\la \bs\beta_k, \bs h_0\ra^2.
\end{equation} 
\end{proposition}

%\EQ{ \label{eq:dHalbe-form} 
%\partial_a \bs H(0, v; x) &= \big({-}\gamma_v \partial_x H(\gamma_v x ), \gamma_v^2 v \partial_x^2 H(\gamma_v x )\big), \\
%\partial_v \bs H(0, v; x) &= \big(\gamma_v^3 v x \partial_x H(\gamma_v x), {-}\gamma_v^3 \partial_x H(\gamma_v x )
%- \gamma_v^4 v^2 x \partial_x^2 H(\gamma_v x)\big), \\
%\bs \alpha(0, v; x) &= \big(\gamma_v^2 v \partial_x^2 H(\gamma_v x ),
%\gamma_v \partial_x H(\gamma_v x )\big), \\
%\bs \beta(0, v; x) &= \big({-}\gamma_v^3 \partial_x H(\gamma_v x )
%- \gamma_v^4 v^2 x \partial_x^2 H(\gamma_v x),
%{-}\gamma_v^3 v x \partial_x H(\gamma_v x)\big).
%}

We first prove Lemma~\ref{lem:D2H} and then deduce Proposition~\ref{prop:coer} as a quick consequence.  

\begin{proof}[Proof of Lemma~\ref{lem:D2H}] To simplify notation, by analogy with \eqref{eq:notation-v-0} we write $L(\vec a) \coloneqq L(\vec a, \vec 0)$. 
We adapt \cite[Proof of Lemma 2.4]{JKL1}.
We set
\begin{equation}
\begin{aligned}
\chi_1(x) &\coloneqq \chi\Big(\frac{x - a_1}{a_2 - a_1}\Big), \\
\chi_k(x) &\coloneqq \chi\Big(\frac{x - a_k}{a_{k+1} - a_k}\Big)
- \chi\Big(\frac{x - a_{k-1}}{a_k - a_{k-1}}\Big),\qquad\text{for }k \in \{2, \ldots, n-1\}, \\
\chi_n(x) &\coloneqq 1 - \chi\Big(\frac{x - a_{n-1}}{a_n - a_{n-1}}\Big)
\end{aligned}
\end{equation}
and we let
\begin{equation}
h_k \coloneqq \chi_k h, \qquad\text{for }k \in \{1, \ldots, n\}.
\end{equation}
We have $h = \sum_{k=1}^n h_k$, hence
\begin{equation}
\la h, L(\vec a) h\ra = \sum_{k=1}^n \la h_k, L(\vec a) h_k\ra + 2\sum_{1 \leq i < j \leq n}\la h_i, L(\vec a) h_j\ra,
\end{equation}
so it suffices to prove that
\begin{gather}
\label{eq:coer-1}
\la h_k, L(\vec a) h_k\ra \geq \nu\|h_k\|_{H^1}^2 -  C\la\partial_x H_k, h_k\ra^2
-{o(1)\|h\|_{H^1}^2}, \\
\label{eq:coer-2}
\la h_i, L(\vec a) h_j\ra {\geq {-}o(1)\|h\|_{H^1}^2} \qquad \text{whenever }i \neq j, \\
\label{eq:coer-3}
\big|\la\partial_x H_k, h_k\ra^2 - \la \partial_x H_k, h\ra^2\big| {\leq o(1)\|h\|_{H^1}^2},
\end{gather}
{where $\nu>0$ is the constant in \eqref{eq:vLv-coer} and $o(1)\to 0$ as $y_0\to \infty$.}

We first prove \eqref{eq:coer-1}.
Without loss of generality we can assume $a_k = 0$. We then have
\begin{equation}
L(\vec a) = L + V, \qquad V \coloneqq U''(H(\vec a)) - U''(H),
\end{equation}
thus
\begin{equation}
\la h_k, L(\vec a) h_k\ra = \la h_k, L h_k\ra + \la h_k, Vh_k\ra \geq \nu \|h_k\|_{H^1}^2 - C\la \partial_x H, h_k\ra^2
+ \int_{-\infty}^\infty \chi_k^2 V h_k^2\ud x.
\end{equation}
We only need to check that $\|\chi_k^2 V\|_{L^\infty} \ll 1$. If $x \geq \frac 23 a_{k+1}$
or $x \leq \frac 23 a_{k-1}$, then $\chi_k(x) = 0$.
If $\frac 23 a_{k-1} \leq x \leq \frac 23 a_{k+1}$, then $|1 - H_j(x)| \ll 1$
for all $j < k$ and $|1 + H_j(x)| \ll 1$ for all $j > k$, hence
\begin{equation}
\big|H(\vec a; x) - (-1)^k H_k(x)\big| = \Big|\sum_{j=1}^{k-1}(-1)^j(H_j(x) - 1) + \sum_{j=k+1}^n(-1)^j(H_j(x) + 1)\Big| \ll 1,
\end{equation}
which implies $|V(x)| \ll 1$.

Next, we show \eqref{eq:coer-2}.
Observe that
\begin{equation}
\label{eq:chiichij}
\chi_i\chi_jU''(H(\vec a; x)) \geq 0, \qquad \text{for all }x \in \bR.
\end{equation}
Indeed, if $j \neq i+1$, then $\chi_i(x) \chi_j(x) = 0$ for all $x$.
If $j = i+1$, then $\chi_i(x) \chi_j(x) \neq 0$ only if $\frac 23 a_i + \frac 13 a_{i+1} \leq x \leq \frac 13 a_i + \frac 23 a_{i+1}$, which implies $|1 - H_\ell(x)| \ll 1$ for all $\ell \leq i$
and $|1 + H_\ell(x)| \ll 1$ for all $\ell > i$, hence $|H(\vec a; x) - (-1)^i| \ll 1$,
in particular $U''(H(\vec a; x)) > 0$.
Using \eqref{eq:chiichij} and the fact that $\|\partial_x \chi_k\|_{L^\infty} \ll 1$
for all $k$, we obtain
\[
\begin{aligned}
\la h_i, L(\vec a) h_j\ra &= \int_{-\infty}^\infty \Big(\partial_x(\chi_i h)\partial_x(\chi_j h)
+U''(H(\vec a)) \chi_i\chi_j h^2\Big)\ud x\\
&\geq \int_{-\infty}^\infty\chi_i\chi_j (\partial_x h)^2\ud x - o(1)\|h\|_{H^1}^2\geq {-}o(1)\|h\|_{H^1}^2.
\end{aligned}
\]

Finally, we have
\begin{equation}
\begin{aligned}
\big|\la\partial_x H_k, h_k\ra^2 - \la \partial_x H_k, h\ra^2\big| &\leq (\|h_k\|_{L^2} + \|h\|_{L^2})
\big|\la\partial_x H_k, h_k - h\ra\big| \\ &\leq 2\|h\|_{L^2}^2\sum_{j \neq k}\|\chi_j\partial_x H_k\|_{L^2} \leq o(1)\|h\|_{L^2}^2,
\end{aligned}
\end{equation}
hence \eqref{eq:coer-3} follows.
\end{proof}

\begin{proof}[Proof of Proposition~\ref{prop:coer}]
Set 
\EQ{
V( \vec a, \vec v) \coloneqq  U''( H( \vec a, \vec v)) - U''( H( \vec a)). 
}
In view of the definitions of $\bs \al_k, \bs \beta_k$ (see~\eqref{eq:dHalbe-form}) and Lemma~\ref{lem:D2H} it suffices to show that for  $| \vec v |$ sufficiently small we have 
\EQ{ \label{eq:V-bound-coer} 
|\la V( \vec a, \vec v) h, \, h \ra| \le o(1) \| h \|_{H^1}^2.  
}
Using the same partition of unity $\chi_k$ from the proof of Lemma~\ref{lem:D2H} we write $V(\vec a, \vec v) = \sum_{k=1}^n V_k(\vec a, \vec v)$ where $V_k(\vec a, \vec v; x) = \chi_k(x) V(\vec a, \vec v; x)$. Similarly to the proof of Lemma~\ref{lem:D2H} we have $\| V_k(\vec a, \vec v) - (U''( H(  a_k, v_k)) - U''( H( \cdot - a_k))) \|_{L^\infty} \ll 1$ for each $k$. The bound~\eqref{eq:V-bound-coer} then follows from noting that $\|U''( H(  a_k, v_k)) - U''( H( \cdot - a_k) \|_{L^\infty} \ll 1$ for each $k$, provided $|v|$ is taken sufficiently small. 
\end{proof}

To each pair $(\vec a, \vec v) \in \R^n \times \R^n$ we associate the subspace $\Pi(\vec a, \vec v) \subset \E$ defined by
\EQ{
\Pi(\vec a, \vec v)\coloneqq  \{ \bs h \in \E \mid  0 = \la \bs \al_k, \bs h \ra = \la \bs \be_k, \bs h \ra , \quad \forall k \in \{1, \dots, n\}\}. 
}
In other words, $\Pi(\vec a, \vec v)$ is the orthogonal complement  in  $\E$ of the subspace $\calN(\vec a, \vec v)\coloneqq \textrm{span} \{\bs \al_1, \dots, \bs\al_n, \bs \be_1, \dots \bs \beta_n\}$. 
We note that $\Pi(\vec a, \vec v)$ is a Hilbert space, and its continuous dual $\Pi(\vec a, \vec v)^*$ can be identified with the quotient space $\E  \slash  \calN(\vec a, \vec v)$. We denote the elements of $\E  \slash  \calN(\vec a, \vec v)$ by $[\bs f]$.  

For each $\ell, k \in \{1, \dots, n\}$ we define a $2\times 2$ matrix
\EQ{
\calM_{\ell, k} \coloneqq  \pmat{ \la  \p_{v_{\ell}} \bs H_\ell, \, \bs \al_k, \ra & \la \p_{v_\ell} \bs H_{\ell}, \, \bs \beta_k \ra \\ \la  \p_{a_{\ell}} \bs H_\ell, \, \bs \al_k, \ra & \la \p_{a_\ell} \bs H_{\ell}, \, \bs \beta_k \ra},
}
and the $2n \times 2n$ matrix $\calM(\vec a, \vec v)$ by the formulas
\EQ{ \label{eq:Mav-def} 
m_{2\ell-1, 2k-1}  &\coloneqq  \la \bs \p_{v_{\ell}} \bs H_{\ell} , \, \bs \alpha_k\ra \\
m_{2\ell -1, 2k}  & \coloneqq  \la \bs \p_{v_{\ell}} \bs H_{\ell} , \, \bs \be_k\ra\\
m_{2\ell, 2k-1} & \coloneqq  \la \bs \p_{a_{\ell}} \bs H_{\ell} , \, \bs \al_k\ra\\
m_{2 \ell, 2k} & \coloneqq \la \bs \p_{a_{\ell}} \bs H_{\ell} , \, \bs \be_k\ra,
}
and we note the identities
\EQ{ \label{eq:Mav-diag}
m_{2k-1, 2k-1} &= \la \bs \p_{v_{k}} \bs H_{k} , \, \bs \alpha_k\ra= -\gamma_{v_k}^3 M, \\
 m_{2k, 2k} &=  \la \bs \p_{a_{k}} \bs H_{k} , \, \bs \be_k\ra= \gamma_{v_k}^3 M, \\
   m_{2k -1, 2k} &= \la \bs \p_{v_{k}} \bs H_{k} , \, \bs \be_k\ra = 0 \\
    m_{2k, 2k-1} &= \la \bs \p_{a_{k}} \bs H_{k} , \, \bs \al_k\ra= 0.
}
and for $\ell < k$ the estimates, 
\EQ{ \label{eq:Mav-cross} 
|m_{2\ell-1, 2k-1}| \le C  y_{\min} e^{-y_{\min}} .% \rho(\vec a, \vec v)
}
Given vectors $\vec \lambda, \vec \mu \in \R^n$, we use the non-standard notation $(\vec \lam, \vec \mu)^{\intercal}$ to denote the column vector $(\lam_1, \mu_1, \lam_2, \mu_2, \dots, \lam_n, \mu_n)^{\intercal} \in \R^{2n}$, and then $M(\vec a, \vec v)$ acts on $(\vec \lam, \vec\mu)^{\intercal}$ by matrix multiplication.  

Because of the formulas~\eqref{eq:Mav-diag} and the estimates~\eqref{eq:Mav-cross}, the matrix $M(\vec a, \vec v)$ is diagonally dominant and hence invertible provided that $\rho(\vec a, \vec v)$ is sufficiently small. An immediate consequence is that for small enough $\rho(\vec a, \vec v)$
\EQ{ \label{eq:Nab-H} 
\calN(\vec a, \vec v)  = \textrm{span}\{ \bs \p_{a_1} \bs H_1, \dots, \partial_{a_n} \bs H_n, \partial_{v_1} \bs H_1, \dots, \partial_{v_n} \bs H_n\},
}
meaning that to each $(\vec \lam, \vec \mu)^{\intercal} \in  \R^{2n}$ there is a unique $(\vec c, \vec d)^{\intercal} \in \R^{2n}$ such that 
\EQ{ \label{eq:ab-HH} 
\sum_{k =1}^n ( \lam_k \bs \al_k + \mu_k \bs \beta_k) = \sum_{k =1}^n ( c_k \partial_{v_k} \bs H_k + d_k \partial_{a_k} \bs H_k). 
}

\begin{lemma}
\label{lem:static-ls}
There exist $C, \eta_0 > 0$ such that the following holds.
Let $(\vec a, \vec v) \in \bR^n \times \bR^n$ satisfy $\rho(\vec a, \vec v) \leq \eta_0$
and let $\bs f \in \cE$.
\begin{enumerate}[(i)]
\item\label{it:static-ls-i}
There exists unique $(\bs h, \vec \lambda, \vec\mu) \in \Pi(\vec a, \vec v) \times \bR^n \times \bR^n$ such that
\begin{equation}
\label{eq:static-ls}
\vD^2 E(\bs H(\vec a, \vec v))\bs h = \bs f + \sum_{k=1}^n (\lambda_k \bs \alpha_k + \mu_k \bs \beta_k).
\end{equation}
It satisfies the estimates
\EQ{ \label{eq:h-est-lm} 
\| \bs h \|_{\E} \le C \| \bs f \|_{\E}, 
}
and for each $k \in \{1, \dots, n\}$
\EQ{ \label{eq:lam-mu-est-lm} 
 \big| M \gamma_{v_k}^3 \lam_k - \la \p_{v_k} \bs H_{k} , \, \bs f\ra \big|  &\le C  \Big( (\sqrt{y_{\min}} e^{- y_{\min}}+ | \vec v|) \| \bs h \|_{\E}  +  y_{\min} e^{-y_{\min}} \| \bs f \|_{\E} \Big)  \\
 \big| M \gamma_{v_k}^3 \mu_k + \la \p_{a_k} \bs H_{k} , \, \bs f\ra \big|  &\le C \Big( (\sqrt{y_{\min}} e^{- y_{\min}}+ | \vec v|) \| \bs h \|_{\E}  +  y_{\min} e^{-y_{\min}} \| \bs f \|_{\E} \Big) 
}
\item\label{it:static-ls-ii}
For any $(\vec \ell, \vec m) \in \bR^n \times \bR^n$ there exists unique $(\bs h, \vec \lambda, \vec\mu) \in \cE \times \bR^n \times \bR^n$ such that
$\la {\bs \alpha}_k, \bs h\ra = \ell_k$ and $\la {\bs \beta}_k, \bs h\ra = m_k$ for all $k \in \{1, \ldots, n\}$, and
\begin{equation}
\label{eq:static-ls-ii}
\vD^2 E(\bs H(\vec a, \vec v))\bs h = \bs f + \sum_{k=1}^n (\lambda_k \bs \alpha_k + \mu_k \bs \beta_k ).
\end{equation}
It satisfies estimates
\EQ{ \label{eq:h-est-lm-ii} 
\| \bs h \|_{\E} \le C \Big( \| \bs f \|_{\E} + \Big( \sum_{k =1}^n (\ell_k^2 +m_k^2)\Big)^{\frac{1}{2}} \Big) 
}
and for each $k \in \{1, \dots, n\}$
\EQ{ \label{eq:lam-mu-est-lm-ii} 
|\gamma_{v_k}^3 M \lam_k - \la \partial_{v_k}\bs H_k, \bs f\ra - \ell_k| &\le C (\sqrt{y_{\min}} e^{- y_{\min}}+ | \vec v|) \| \bs h \|_{\E}  \\
&\quad +  C y_{\min} e^{-y_{\min}} \Big( \| \bs f \|_{\E} + \Big( \sum_{k =1}^n (\ell_k^2 +m_k^2)\Big)^{\frac{1}{2}} \Big)   \\
| \gamma_{v_k}^3 M \mu_k + \la \partial_{a_k}\bs H_k, \bs f\ra| & \le C (\sqrt{y_{\min}} e^{- y_{\min}}+ | \vec v|) \| \bs h \|_{\E}  \\
&\quad +  C y_{\min} e^{-y_{\min}} \Big( \| \bs f \|_{\E} + \Big( \sum_{k =1}^n (\ell_k^2 +m_k^2)\Big)^{\frac{1}{2}} \Big). 
}
%and, 
%\EQ{
%|\vec \lam| + | \vec \mu|  \lesssim  \| \bs f \|_{\E} +  \Big( \sum_{k =1}^n (\ell_k^2 +m_k^2)\Big)^{\frac{1}{2}} 
%}
\end{enumerate}
\end{lemma}
\begin{remark}
Part \ref{it:static-ls-i} of Lemma~\ref{lem:static-ls} is a special case of part \ref{it:static-ls-ii}, but we state it separately for later reference.
\end{remark}
\begin{proof}[Proof of Lemma~\ref{lem:static-ls}]

Let $\eta_0>0$ be as in Proposition~\ref{prop:coer} and fix $(\vec a, \vec v) \in \R^n \times \R^n$ so that $\rho(\vec a, \vec v) \le \eta_0$. Let $B:\Pi( \vec a, \vec b) \times \Pi( \vec a, \vec b) \to \R$ denote the bilinear form
\EQ{
B( \bs h, \bs g) \coloneqq  \la \vD^2 E(\bs H(\vec a, \vec v))\bs h, \, \bs g\ra. 
}
Integration by parts yields the bound $|B( \bs h, \bs g)| \lesssim \| \bs h\|_{\E} \| \bs g \|_{\E}$ and Proposition~\ref{prop:coer} gives that $B$ is strongly coercive on $\Pi(\vec a, \vec v)$, i.e., $B(\bs h, \bs h) \ge  \nu \| \bs h \|_{\E}^2$ for all $\bs h \in \Pi(\vec a, \vec v)$. Fix $\bs f \in \E$  and let $[\bs f]$ denote the corresponding equivalence class of elements in $ \E \slash \calN(\vec a, \vec v)$. By the Lax-Milgram theorem~\cite{Lax} there is a unique $\bs h \in \Pi(\vec a, \vec v)$ such that 
\EQ{
B( \bs h, \bs g) = \la \bs f , \, \bs g\ra
}
for all $ \bs g \in \Pi(\vec a, \vec b)$, and moreover, $\bs h$ satisfies the estimate, 
\EQ{ \label{eq:h-est-lm-i} 
\| \bs h \|_{\E} \le \frac{1}{\nu}   \inf_{\vec \lambda, \vec \mu} \big\|  \bs f + \sum_{k =1}^n ( \lam_k \bs \al_k + \mu_k \bs \beta_k) \big\|_{\E} .
}
Since $\la \vD^2 E(\bs H(\vec a, \vec v))\bs h - \bs f, \, \bs g\ra = 0$ for all $\bs g \in \Pi(\vec a, \vec v)$ we can find $(\vec \lam, \vec \mu)  \in \times \R^n \times \R^n$ such that, 
\EQ{
\vD^2 E(\bs H(\vec a, \vec v))\bs h = \bs f + \sum_{k=1}^n (\lambda_k \bs \alpha_k + \mu_k \bs \beta_k).
}
To prove the triplet $( \bs h, \vec \lam, \vec \mu)$ is unique in $\Pi(\vec a, \vec v) \times \R^n \times \R^n$ it suffices to show that $(\vec \lam, \vec \mu) = ( \vec 0, \vec 0)$ when $\bs h \in \Pi( \vec a, \vec v) = \bs 0$ and $\bs f  = \bs 0$. To see this, we take the inner product of the \eqref{eq:static-ls}
with $\partial_{a_k}\bs H_k$ and $\partial_{v_k}\bs H_k$,
using \eqref{eq:matrix-coef}, \eqref{eq:alpha-id} and \eqref{eq:beta-id} to obtain the system, 
\EQ{ \label{eq:lam-mu-sys-lm} 
\la v_k\partial_x \bs\beta_k, \bs h\ra &+ \la (\vD^2 E( \bs H(\vec a, \vec v)) - \vD^2 E(\bs H_k))\partial_{v_k} \bs H_k , \, \bs h\ra  = \la \partial_{v_k}\bs H_k, \bs f\ra \\
& \quad + \sum_{\ell = 1}^n \big(\lam_{\ell} \la \bs \al_\ell, \partial_{v_k} \bs H_k\ra + \mu_{\ell} \la \bs \beta_\ell, \partial_{v_k} \bs H_k \ra \big), \\
\la v_k\partial_x \bs\alpha_k, \bs h\ra &+ (\vD^2 E( \bs H(\vec a, \vec v)) - \vD^2 E(\bs H_k))\partial_{a_k} \bs H_k , \, \bs h\ra= \la \partial_{a_k}\bs H_k, \bs f\ra
 \\
 &\quad + \sum_{\ell = 1}^n \big(\lam_{\ell} \la \bs \al_\ell, \partial_{a_k} \bs H_k\ra + \mu_{\ell} \la \bs \beta_\ell, \partial_{a_k} \bs H_k \ra \big)
}
which, using~\eqref{eq:Mav-def}, and the invertibility of $\calM(\vec a, \vec v)$ yields, 
\EQ{ \label{eq:lam-mu-sys-ell} 
 ( \vec \lam, \vec \mu)^{\intercal}  = \calM(\vec a, \vec v)^{-1} \pmat{ -\la  \partial_{v_1} \bs H_1, \bs f \ra + \la v_1\partial_x \bs\beta_1, \bs h\ra +  \la (U''( H(\vec a, \vec v)) - U''(H_1)) \partial_{v_1} H_1, \, h\ra  \\ -\la  \partial_{a_1} \bs H_1, \bs f \ra + \la v_1\partial_x \bs\alpha_1, \bs h\ra +  \la (U''( H(\vec a, \vec v)) - U''(H_1)) \partial_{a_1} H_1, \, h\ra  \\  \dots \\ \dots \\ -\la  \partial_{v_n} \bs H_n, \bs f \ra + \la v_n\partial_x \bs\beta_n, \bs h\ra  +  \la (U''( H(\vec a, \vec v)) - U''(H_n)) \partial_{v_n} H_n, \, h\ra \\ -\la  \partial_{a_n} \bs H_n, \bs f \ra + \la v_n\partial_x \bs\alpha_n, \bs h\ra +  \la (U''( H(\vec a, \vec v)) - U''(H_n)) \partial_{a_n} H_n, \, h\ra },
} 
which shows that  $(\vec \lam, \vec \mu)$ are determined uniquely by $\bs h, \bs f$ and equal to $(\vec 0, \vec 0)$ when $\bs h = \bs 0$ and $\bs f = \bs 0$. From~\eqref{eq:lam-mu-sys-ell} and~\eqref{eq:h-est-lm} we obtain the preliminary estimates, 
\EQ{ \label{eq:lam-mu-prelim} 
|\vec \lam| + | \vec \mu| \lesssim  \| \bs f \|_{\E} .
}
For the more refined estimates in~\eqref{eq:lam-mu-est-lm}  we record the expressions 
\EQ{ \label{eq:lam-mu-exp} 
\gamma_{v_k}^3 M \lam_k - \la \partial_{v_k}\bs H_k, \bs f\ra &= - \la v_k\partial_x \bs\beta_k, \bs h\ra  - \la (U''( H(\vec a, \vec v)) - U''(H_k)) \partial_{v_k} H_k, \, h\ra \\
&\quad +  \sum_{\ell  \neq k} \big(\lam_{\ell} \la \bs \al_\ell, \partial_{v_k} \bs H_k\ra + \mu_{\ell} \la \bs \beta_\ell, \partial_{v_k} \bs H_k \ra \big) \\
\gamma_{v_k}^3 M \mu_k + \la \partial_{a_k}\bs H_k, \bs f\ra &= \la v_k\partial_x \bs\alpha_k, \bs h\ra +  \la (U''( H(\vec a, \vec v)) - U''(H_k)) \partial_{a_k} H_k, \, h\ra \\
&\quad  - \sum_{\ell  \neq k} \big(\lam_{\ell} \la \bs \al_\ell, \partial_{a_k} \bs H_k\ra + \mu_{\ell} \la \bs \beta_\ell, \partial_{a_k} \bs H_k \ra \big),
}
which follow directly from~\eqref{eq:lam-mu-sys-lm} using~\eqref{eq:matrix-coef}. Using~\eqref{eq:lam-mu-prelim} together with Lemma~\ref{lem:exp-cross-term} to estimate the cross terms gives ~\eqref{eq:lam-mu-est-lm}. 
%\Red{show more for the estimates} 

For part~\ref{it:static-ls-ii} we note that by~\eqref{eq:ab-HH} any $\bs h \in \E$ admits a decomposition
\EQ{ \label{eq:h-decomp-lm} 
\bs h = \bs h^{\perp} + \sum_{k =1}^n ( c_k \partial_{v_k} \bs H_k + d_k \partial_{a_k} \bs H_k)
}
for some $\bs h^{\perp} \in \Pi(\vec a, \vec v)$ and some $(\vec c, \vec d) \in \R^{n} \times \R^n$. Fixing $(\vec c, \vec d)^{\intercal} \in \R^{2n}$ by the formula
\EQ{ \label{eq:cd-def} 
(\vec c, \vec d)^{\intercal}\coloneqq  \calM(\vec a, \vec v)^{-1} ( \vec \ell, \vec m)^{\intercal} 
}
we see that part~\ref{it:static-ls-ii} is equivalent to finding unique $(\bs h^{\perp}, \vec \lam, \vec \mu) \in \Pi(\vec a, \vec v) \times \R^n, \R^n$ such that
\EQ{
\vD^2 E( \bs H(\vec a, \vec v)) \bs h^{\perp} =   \bs{\ti f} +  \sum_{k=1}^n \big((\lambda_k + c_k) \bs \alpha_k + \mu_k \bs \beta_k \big)
}
for 
\EQ{
\bs{\ti f}&\coloneqq  %\bs f -  \vD^2 E( \bs H(\vec a, \vec v))\Big(\sum_{k =1}^n ( c_k \partial_{v_k} \bs H_k + d_k \partial_{a_k} \bs H_k)\Big),
\bs f-\sum_{k =1}^n  \big( c_k v_k \partial_{x} \bs \beta_k + d_k v_k  \p_x  \bs \alpha_k\big) \\
&\quad -  \sum_{k =1}^n \pmat{ \big( U''( H(\vec a, \vec v)) - U''(H_k)\big)( c_k \partial_{v_k}  H_k + d_k \partial_{a_k}  H_k) \\ 0}
} 
as we note the formula
\EQ{
\vD^2 E( \bs H(\vec a, \vec v))&\Big(\sum_{k =1}^n ( c_k \partial_{v_k} \bs H_k + d_k \partial_{a_k} \bs H_k)\Big)  = \sum_{k =1}^n \vD^2 E( \bs H_k)( c_k \partial_{v_k} \bs H_k + d_k \partial_{a_k} \bs H_k) \\
& \quad  + \sum_{k =1}^n \Big(( \vD^2 E( \bs H(\vec a, \vec v)) - \vD^2E( \bs H_k)( c_k \partial_{v_k} \bs H_k + d_k \partial_{a_k} \bs H_k)\Big)\\
& = - \sum_{k =1}^n c_k \bs \al_k + \sum_{k =1}^n  \big( c_k v_k \partial_{x} \bs \beta_k + d_k v_k  \p_x  \bs \alpha_k\big)\\
&\quad +  \sum_{k =1}^n \pmat{ \big( U''( H(\vec a, \vec v)) - U''(H_k)\big)( c_k \partial_{v_k}  H_k + d_k \partial_{a_k}  H_k) \\ 0}.
}
The existence and uniqueness of $(\bs h^{\perp}, \vec \lam, \vec \mu)$ follows by an application of part~\ref{it:static-ls-i} and then $\bs h$ is given by the formula~\eqref{eq:h-decomp-lm} with $(\vec c, \vec d)$ uniquely determined  by~\eqref{eq:cd-def}. 
Using~\eqref{eq:coer} we obtain the following estimates for $\bs h$, 
\EQ{
\| \bs h \|_{\E} \lesssim \| \bs f \|_{\E} + \sum_{k =1}^n ( | \lambda_k|| \ell_k| + | \mu_k| | m_k|) + \Big( \sum_{k =1}^n (\ell_k^2 +m_k^2)\Big)^{\frac{1}{2}} 
}
and arguing similarly as in part~\ref{it:static-ls-i} we obtain the preliminary estimate 
\EQ{
|\vec \lam| + | \vec \mu|  \lesssim  \| \bs f \|_{\E} + \| \bs h \|_{\E}.
}
Using this in the previous line gives, 
\EQ{
\| \bs h \|_{\E} \lesssim \| \bs f \|_{\E} + \Big( \sum_{k =1}^n (\ell_k^2 +m_k^2)\Big)^{\frac{1}{2}} 
}
and thus
\EQ{ \label{eq:lam-mu-est-prelim-ii} 
|\vec \lam| + | \vec \mu|  \lesssim  \| \bs f \|_{\E} +  \Big( \sum_{k =1}^n (\ell_k^2 +m_k^2)\Big)^{\frac{1}{2}} .
}
We refine the estimates on $\vec \lam$ and $\vec \mu$, by taking the inner product of~\eqref{eq:static-ls-ii} with $\p_{v_k} \bs H$ and $\p_{a_k} \bs H$, and using ~\eqref{eq:alpha-id} and~\eqref{eq:beta-id} to obtain, 
\EQ{
\gamma_{v_k}^3 M \lam_k - \la \partial_{v_k}\bs H_k, \bs f\ra - \ell_k &= - \la v_k\partial_x \bs\beta_k, \bs h\ra  - \la (U''( H(\vec a, \vec v)) - U''(H_k)) \partial_{v_k} H_k, \, h\ra \\
&\quad +  \sum_{\ell  \neq k} \big(\lam_{\ell} \la \bs \al_\ell, \partial_{v_k} \bs H_k\ra + \mu_{\ell} \la \bs \beta_\ell, \partial_{v_k} \bs H_k \ra \big) \\
\gamma_{v_k}^3 M \mu_k + \la \partial_{a_k}\bs H_k, \bs f\ra &= \la v_k\partial_x \bs\alpha_k, \bs h\ra +  \la (U''( H(\vec a, \vec v)) - U''(H_k)) \partial_{a_k} H_k, \, h\ra \\
&\quad  - \sum_{\ell  \neq k} \big(\lam_{\ell} \la \bs \al_\ell, \partial_{a_k} \bs H_k\ra + \mu_{\ell} \la \bs \beta_\ell, \partial_{a_k} \bs H_k \ra \big),
}
from which the estimates~\eqref{eq:lam-mu-est-lm-ii} follow from the same argument used to prove~\eqref{eq:lam-mu-est-lm}. Note above the presence of the term $\ell_k$ on the left which is due to the expression~\eqref{eq:beta-id}. 
\end{proof}
\begin{lemma}\label{lem:diff-param}  There exists a constant $C>0$ with the following property. Let $\eta_0 > 0$ be given by Lemma~\ref{lem:static-ls} and let $\Om \subset \{ (\vec a, \vec v) \in \R^n \times \R^n  \mid \rho(\vec a, \vec v) \le \eta_0\}$ be an open region. Suppose that $\Om \ni (\vec a, \vec v) \mapsto \bs f( \vec a,  \vec v) \in \E$ is (twice) differentiable with respect to $a_k, v_k$ for each $k \in \{1, \dots, n\}$, i.e.,  $\partial_{a_k} \bs f, \p_{v_k} \bs f \in \E$ for each $k$ (same for all the second partials). For each $(\vec a, \vec v) \in \Omega$ let  $(\bs h(\vec a, \vec v), \vec \lam(\vec a, \vec v), \vec \mu(\vec a, \vec v)) \in \Pi(\vec a, \vec v) \times \R^n \times R^n$ be the unique solution to~\eqref{eq:static-ls} given by Lemma~\ref{lem:static-ls}. Then the mapping $\Omega \ni (\vec a, \vec v) \mapsto (\bs h(\vec a, \vec v), \vec \lam(\vec a, \vec v), \vec \mu(\vec a, \vec v))$ is (twice) differentiable with respect to $a_k, v_k$ for each $k$, and satisfies the estimates
\EQ{ \label{eq:h-first} 
\|\partial_{a_k}\bs h\|_\cE + \|\partial_{v_k}\bs h\|_\cE &\leq C \big( \| \partial_{a_k} \bs f \|_{\E} + \| \partial_{v_k} \bs f \|_{\E} + \| \bs f \|_{\E} \big), \\
| \partial_{a_k}\lambda_j| + | \partial_{a_k}\mu_j| &\le C\big( \| \partial_{a_k} \bs f \|_{\E}  + \| \bs f \|_{\E} \big), \\
| \partial_{v_k}\lambda_j| + | \partial_{v_k}\mu_j| &\le C\big( \| \partial_{v_k} \bs f \|_{\E}  + \| \bs f \|_{\E} \big), %\\
%|\partial_{a_j}\lambda_k -  \delta_{jk} | &\leq C \dots \\
%|\partial_{v_j}\mu_k - \p_j F_k| &\leq C \dots
}
and in the case of twice differentiability 
\begin{multline}  \label{eq:h-second} 
\|\partial_{a_j} \partial_{a_k}\bs h\|_\cE + \|\partial_{a_j} \partial_{v_k}\bs h\|_\cE + \|\partial_{v_j} \partial_{v_k}\bs h\|_\cE \\
\leq C \big(  \| \partial_{a_j} \partial_{a_k} \bs f \|_{\E} +  \| \partial_{a_j} \partial_{v_k} \bs f \|_{\E}+  \| \partial_{v_j} \partial_{v_k} \bs f \|_{\E}  + \|  \partial_{a_j} \bs f \|_{\E} + \| \partial_{v_j} \bs f \|_{\E} +\|  \partial_{a_k} \bs f \|_{\E} + \| \partial_{v_k} \bs f \|_{\E} + \| \bs f \|_{\E} \big),
\end{multline} 
for each $j, k$. 
Let  $(\vec a, \vec v), (\sh{\vec a}, \sh{\vec v}) \in  \Om$, and denote by $(\bs h, \vec \lam, \vec \mu)$ the solution to~\eqref{eq:static-ls} associated to  $(\vec a, \vec v)$ and let $(\sh{\bs h}, \sh{\vec \lam},  \sh{\vec \mu})$ the solution to~\eqref{eq:static-ls} associated to  $(\sh{\vec a}, \sh{\vec v})$. Writing $\bs f  = \bs f(\vec a, \vec v)$ and $\sh{\bs f} = \bs f(\sh{\vec a}, \sh{\vec v})$, and $\sh {\bs {H}}_k = \bs H(\sh{\vec a}, \sh{\vec v})$, we have
\EQ{ \label{eq:h-diff-est} 
\| \sh{\bs h} -{\bs h} \|_{\E} \le C \Big(( \| \sh{\bs f}\|_{\E} + \| \bs f \|_{\E} )( | \sh{\vec a} -  \vec a| + | \sh{\vec v} - \vec v|)   + \| \sh{\bs f }- \bs f \|_{\E}  \Big), 
%\|\partial_{a_k} \lambda_j - \partial_{a_k} \sh{\bs h} \|_{\E} \le C \big(     \\
}
\begin{multline} \label{eq:lam-diff} 
\sup_{1 \le k \le n} \big| (  \sh \lam_k -  \lam_k ) - \big((\gamma_{\sh v_k}^{3} M)^{-1}\la \partial_{v_k}\sh{\bs H}_k, \sh{\bs f}\ra- (\gamma_{v_k}^3 M)^{-1}\la \partial_{v_k}\bs H_k, \bs f\ra\big)\big|  \\
\le C | \sh{\vec a} - \vec a|( | \sh{\vec v} | + | \vec v| + y_{\min} e^{-y_{\min}})( \| \bs f \|_{\E} + \| \sh{ \bs f }\|_{\E})  \\+ 
C  | \sh{\vec v} - \vec v| ( \| \bs f \|_{\E} + \| \sh{ \bs f }\|_{\E}) +C \| \sh{\bs f} - \bs f \|_{\E} (  | \sh{\vec v} | + | \vec v| + y_{\min} e^{-y_{\min}}),
\end{multline}
and, 
\begin{multline} \label{eq:mu-diff}
\sup_{1 \le k \le n} | (\sh \mu_k - \mu_k) + \big((\gamma_{\sh v_k}^{3} M)^{-1}\la \partial_{a_k}\sh{\bs H}_k, \sh{\bs f}\ra- (\gamma_{v_k}^3 M)^{-1}\la \partial_{a_k}\bs H_k, \bs f\ra\big)  )  | \\
  \le C | \sh{\vec a} - \vec a|( | \sh{\vec v} | + | \vec v| + y_{\min} e^{-y_{\min}})( \| \bs f \|_{\E} + \| \sh{ \bs f }\|_{\E})  \\+ 
C  | \sh{\vec v} - \vec v| ( \| \bs f \|_{\E} + \| \sh{ \bs f }\|_{\E}) +C \| \sh{\bs f} - \bs f \|_{\E} (  | \sh{\vec v} | + | \vec v| + y_{\min} e^{-y_{\min}}). 
\end{multline}
Finally, for each $k \in \{1, \dots, n\}$ 
\EQ{\label{eq:dh-diff} 
\| \p_{a_k} \sh{\bs h} - \p_{a_k} \bs h \|_{\E} &\le C( \| \sh{\bs f}\|_{\E} + \| \bs f \|_{\E}  + \| \p_{a_k} \sh{\bs f}\|_{\E} + \| \p_{a_k}\bs f \|_{\E} )( | \sh{\vec a} -  \vec a| + | \sh{\vec v} - \vec v|)   \\
&\quad + \| \sh{\bs f }- \bs f \|_{\E}  + \| \p_{a_k} \sh{\bs f }- \p_{a_k} \bs f \|_{\E}, \\
\| \p_{v_k} \sh{\bs h} - \p_{v_k} \bs h \|_{\E} &\le C( \| \sh{\bs f}\|_{\E} + \| \bs f \|_{\E}  + \| \p_{v_k} \sh{\bs f}\|_{\E} + \| \p_{v_k}\bs f \|_{\E} )( | \sh{\vec a} -  \vec a| + | \sh{\vec v} - \vec v|)   \\
&\quad + \| \sh{\bs f }- \bs f \|_{\E}  + \| \p_{v_k} \sh{\bs f }- \p_{v_k} \bs f \|_{\E}. 
}
\end{lemma}

\begin{proof}
 We first consider the difference estimates for solutions associated to different $(\vec a, \vec v)$.  Let  $(\vec a, \vec v), (\sh{\vec a}, \sh{\vec v}) \in  \Om$, let $(\bs h, \vec \lam, \vec \mu)$ denote the solution to~\eqref{eq:static-ls} associated to  $(\vec a, \vec v)$, and let $(\sh{\bs h}, \sh{\vec \lam},  \sh{\vec \mu})$ the solution to~\eqref{eq:static-ls} associated to  $(\sh{\vec a}, \sh{\vec v})$. We write $\sh{\bs f}= \bs f(\sh{\vec a}, \sh{\vec v})$. Define
\EQ{
\fl{\bs h} &\coloneqq  \sh{\bs h} - \bs h, \quad 
\fl{\vec{\lam}} \coloneqq \sh{\vec  \lambda} - \vec \lambda , \quad 
\fl{\vec \mu}  \coloneqq  \sh{\vec \mu} - \vec \mu .
}
%It suffices to show that
%\EQ{ \label{eq:cont-param} 
%\lim_{(\sh{\vec{ a}}, \sh{\vec{ v}}) \to (\vec a, \vec v)}(\fl{\bs h}, \fl{\vec \lam}, \fl{\vec \mu} )  = (\bs 0, \vec 0, \vec 0)
%} 
By Lemma~\ref{lem:static-ls}, $(\fl{\bs h}, \fl{\vec \lam}, \fl{\vec \mu} )$
%for each $(\vec { a}, \vec{\ti v}) \in \Omega$, $(\bs g( \vec a, \vec v, \vec{\ti a} , \vec{\ti v}),\vec  \Lambda_1( \vec a, \vec v, \vec{\ti a} , \vec{\ti v}), \vec \Lambda_2( \vec a, \vec v, \vec{\ti a} , \vec{\ti v}))$ 
uniquely solves 
\EQ{ \label{eq:fl-h-eqn} 
\vD^2 E( \bs H(\vec a, \vec v)) \fl{\bs h} = \fl{\bs f} + \sum_{k =1}^n \fl \lam_k \bs \al_k + \fl \mu_k \bs \beta_k,
}
where 
\EQ{
\fl{\bs f} &= \big( \vD^2 E ( \bs H(\vec a, \vec v)) - \vD^2 E (\bs H( \sh{\vec a}, \sh{\vec v}))\big) \sh{ \bs h}  +  \big(\sh{\bs f} - \bs f\big) \\
&\quad  + \sum_{j=1}^n \Big(\sh \lambda_j\big(\sh{\bs \al}_j - \bs{ \al}_j \big) + \sh \mu_j\big(  \sh{\bs \be}_j - \bs {\be}_j\big)\Big), 
}
and 
\EQ{
\la \fl{\bs h} , \, \bs \al_k \ra   &= \la \sh{\bs h}, \bs \al_k - \sh{\bs \al}_k\ra = \fl \ell_k \\
\la \fl{\bs h} , \, \bs \beta_k \ra   &= \la \sh{\bs h}, \bs \beta_k - \sh{\bs \beta}_k\ra = \fl m_k
}
and where we have used the ad hoc notation $\sh{\bs{ \al}}_j = \bs J\p_{ a_j} \bs H(\sh a_j, \sh v_j)$ and $\sh{\bs{ \be}}_j = \bs J\p_{ v_j} \bs H(\sh a_j, \sh v_j)$.

%Then~\eqref{eq:cont-param} follows from the estimates~\eqref{eq:h-est-lm},~\eqref{eq:lam-mu-prelim}, and the continuity of $\bs f(\vec a, \vec v)$, $\bs \al_j$, and $\bs \beta_j$. 

The estimate~\eqref{eq:h-est-lm-ii} and ~\eqref{eq:lam-mu-est-prelim-ii} give
\EQ{ \label{eq:fl-est-static} 
\| \fl{\bs h} \|_{\E} + | \fl{\vec \lam} | + | \fl{\vec \mu}| \lesssim \| \fl{\bs f} \|_{\E} + \Big( \sum_{k =1}^n (\fl \ell_k)^2 + (\fl \mu_k)^2 \Big)^{\frac{1}{2}}
}
from which~\eqref{eq:h-diff-est} follows.
% The estimate~\eqref{eq:lam-mu-est-prelim-ii} gives the preliminary estimates, 
%\EQ{
%| \fl{\vec \lam} | + | \fl{\vec \mu}| \lesssim \| \fl{\bs f} \|_{\E} + \Big( \sum_{k =1}^n (\fl \ell_k)^2 + (\fl \mu_k)^2 \Big)^{\frac{1}{2}}
%}
%Next we consider the estimates~\eqref{eq:lam-mu-diff} on the parameters. 
Using~\eqref{eq:lam-mu-exp} we can take  the difference of the equations for $\sh\lam_k$ and $\lam_k$ to obtain
\EQ{ 
(  \sh \lam_k -  \lam_k ) &- \big((\gamma_{\sh v_k}^{3} M)^{-1}\la \partial_{\sh v_k}\sh{\bs H}_k, \sh{\bs f}\ra- (\gamma_{v_k}^3 M)^{-1}\la \partial_{v_k}\bs H_k, \bs f\ra\big) \\
&=\Big((\gamma_{\sh v_k}^{3} M)^{-1} - (\gamma_{v_k}^3 M)^{-1}\Big) \Big(  - \la v_k\partial_x \bs\beta_k, \bs h\ra  - \la (U''( H(\vec a, \vec v)) - U''(H_k)) \partial_{v_k} H_k, \, h\ra \\
&\quad +  \sum_{\ell  \neq k} \big(\lam_{\ell} \la \bs \al_\ell, \partial_{v_k} \bs H_k\ra + \mu_{\ell} \la \bs \beta_\ell, \partial_{v_k} \bs H_k \ra \big)\Big) \\
&\quad + (\gamma_{\sh v_k}^{3} M)^{-1} \Bigg[  \Big(  - \la \sh v_k\partial_x\sh{ \bs\beta}_k, \sh{\bs h}\ra  - \la (U''( H(\sh{\vec a}, \sh{\vec v})) - U''(\sh H_k)) \partial_{v_k} \sh{H}_k, \, \sh h\ra \\
&\quad +  \sum_{\ell  \neq k} \big(\sh \lam_{\ell} \la \sh{\bs \al}_\ell, \partial_{v_k} \sh{\bs H}_k\ra + \sh{\mu}_{\ell} \la \sh{ \bs \beta}_\ell, \partial_{v_k} \sh{\bs H}_k \ra \big)\Big)\\
 &\quad - \Big(  - \la v_k\partial_x \bs\beta_k, \bs h\ra  - \la (U''( H(\vec a, \vec v)) - U''(H_k)) \partial_{v_k} H_k, \, h\ra \\
&\quad +  \sum_{\ell  \neq k} \big(\lam_{\ell} \la \bs \al_\ell, \partial_{v_k} \bs H_k\ra + \mu_{\ell} \la \bs \beta_\ell, \partial_{v_k} \bs H_k \ra \big)\Big)\Bigg].
}
We call terms negligible if they are bounded by the right-hand-side of~\eqref{eq:lam-diff}. Since $\rho(\vec a, \vec v), \rho(\sh{ \vec a}, \sh{\vec v}) \le \de_0 \ll 1$ we have, 
\EQ{
|(\gamma_{\sh v_k}^{3} M)^{-1} - (\gamma_{v_k}^3 M)^{-1}| \lesssim | \sh{\vec v} - \vec v| ( |\sh{v}| + | \vec v |),
}
so the terms on the second and third line are negligible. We estimate the differences arising from the remaining terms as follows, 
\EQ{
|  \la v_k\partial_x \bs\beta_k, \bs h\ra - \la \sh v_k\partial_x\sh{ \bs\beta}_k, \sh{\bs h}\ra| &\lesssim \Big( | \sh {\vec v} - \vec v | + | \sh{\vec a} - \vec a |\big(( |\sh{\vec v}|  + | \vec v|)\big)\Big)\Big( \| \sh{ \bs h} \|_{\E} + \| \bs h \|_{\E} \Big)  \\
&\quad + ( |\sh{\vec v}|  + | \vec v|)\| \fl{\bs h} \|_{\E} 
%+ (| \sh{\vec a} - \vec a |+  | \sh {\vec v} - \vec v |) ( |\sh{\vec v}|  + | \vec v|)( \| \sh{ \bs h} \|_{\E} + \| \bs h \|_{\E} ) + 
}
and
\begin{multline} 
| \la (U''( H(\vec a, \vec v)) - U''(H_k)) \partial_{v_k} H_k, \, h\ra - \la (U''( H(\sh{\vec a}, \sh{\vec v})) - U''(\sh H_k)) \partial_{v_k} \sh{H}_k, \, \sh h\ra |  \\
\lesssim  \sqrt{y_{\min}} e^{-y_{\min}}\Big(  \|\fl{\bs{ h}} \|_{\E} +   ( \| \sh{ \bs h} \|_{\E} + \| \bs h \|_{\E} )( | \sh {\vec v} - \vec v | + | \sh{\vec a} - \vec a |)\Big)
\end{multline} 
and for all $\ell \neq k$, 
\EQ{
 |\big(\lam_{\ell} \la \bs \al_\ell, \partial_{v_k} \bs H_k\ra - \lam_{\ell} \la \bs \al_\ell, \partial_{v_k} \bs H_k\ra|&\lesssim  y_{\min}e^{-y_{\min}}\Big( | \fl{\vec \lam}|  + ( \| \sh{ \bs f} \|_{\E} + \| \bs f \|_{\E})( ( | \sh {\vec v} - \vec v | + | \sh{\vec a} - \vec a |)\Big)\\
 |\big(\mu_{\ell} \la \bs \be_\ell, \partial_{v_k} \bs H_k\ra - \mu_{\ell} \la \bs \beta_\ell, \partial_{v_k} \bs H_k\ra|&\lesssim  y_{\min}e^{-y_{\min}}\Big( | \fl{\vec \lam}|  + ( \| \sh{ \bs f} \|_{\E} + \| \bs f \|_{\E})( ( | \sh {\vec v} - \vec v | + | \sh{\vec a} - \vec a |)\Big).
}
An application of the estimates~\eqref{eq:fl-est-static} to the previous three displayed equations completes the proof of~\eqref{eq:lam-diff}. 
%\EQ{
%\gamma_{v_k}^3 M \mu_k + \la \partial_{a_k}\bs H_k, \bs f\ra &= \la v_k\partial_x \bs\alpha_k, \bs h\ra +  \la (U''( H(\vec a, \vec v)) - U''(H_k)) \partial_{a_k} H_k, \, h\ra \\
%&\quad  - \sum_{\ell  \neq k} \big(\lam_{\ell} \la \bs \al_\ell, \partial_{a_k} \bs H_k\ra + \mu_{\ell} \la \bs \beta_\ell, \partial_{a_k} \bs H_k \ra \big)
%}
An analogous argument yields the estimates~\eqref{eq:mu-diff} for $\sh{\vec \mu} - \vec \mu$.

We next turn to the differentiability of $(\bs h(\vec a, \vec v), \vec \lam(\vec a, \vec v), \vec \mu(\vec a,\vec v))$ with respect to parameters, noting that continuity with respect to parameters follows from the difference estimates proved above.   Let $(\delta\bs h, \delta\vec\lambda, \delta\vec\mu)$ be the solution of
\EQ{
\label{eq:delta-h}
\vD^2 E(\bs H(\vec a, \vec v))(\delta \bs h) &=  - (U'''( H(\vec a, \vec v))\partial_{a_k}H(\vec a, \vec v)h(\vec a, \vec v), 0) + \partial_{a_k}\bs f(\vec a, \vec v) \\
&\quad + \lambda_k(\vec a, \vec v) \partial_{a_k}\bs\alpha_k+\mu_k(\vec a, \vec v) \partial_{a_k}\bs\beta_k+\sum_{j=1}^n(\delta\lambda_j) \bs\alpha_j+(\delta\mu_j) \bs\beta_j
}
such that $\la \delta \bs h, \bs \alpha_j\ra = -\la \bs h(\vec a, \vec v), \partial_{a_k}\bs\alpha_j\ra$ and $\la \delta \bs h, \bs \beta_j\ra = -\la \bs h(\vec a, \vec v), \partial_{a_k}\bs\beta_j\ra$ for all $j$ (note that $\p_{a_k} \bs \al_j = \bs 0$ for $j \neq k$). This solution exists and is unique, by Lemma~\ref{lem:static-ls} \ref{it:static-ls-ii}.
We claim that
\begin{equation} \label{eq:dh-dlam-dmu} 
(\partial_{a_k} \bs h(\vec a, \vec v), \partial_{a_k} \vec \lambda(\vec a, \vec v), \partial_{a_k}\vec \mu(\vec a, \vec v))
= (\delta \bs h, \delta \vec\lambda, \delta \vec\mu).
\end{equation}
We introduce the shorthand notation $ \vec a(\eps) = (a_1, \dots, a_k + \eps, \dots a_n)$ and define
\EQ{
\bs q_\eps( \vec a, \vec v; x) &\coloneqq  \frac{ \bs h(\vec a(\eps), \vec v; x)- \bs h( \vec a, \vec v; x)  }{ \eps} - \de \bs h( x) \\
\vec \lambda_\eps( \vec a, \vec v) &\coloneqq  \frac{  \vec \lambda(\vec a(\eps), \vec v)-\vec \lambda(\vec a, \vec v) }{\eps} - \de \vec \lambda\\
\vec \mu_\eps(\vec a, \vec v) & \coloneqq  \frac{ \vec \mu(\vec a(\eps), \vec v) -\vec \mu(\vec a, \vec v) }{\eps} - \de \vec \mu
}
It suffices now to show that
\EQ{ \label{eq:q-eps-lim} 
\lim_{\eps \to 0}(\bs q_\eps( \vec a, \vec v), \vec \lambda_\eps( \vec a,  \vec v), \vec \mu_\eps( \vec a, \vec v)) = (\bs 0, \vec 0 , \vec 0).
}
Introducing the ad hoc notation $\bs \al_{k, \eps} = \bs J \p_{a} \bs H(a_k + \eps, v_k)$ and $\bs \be_{k, \eps} = \bs J \p_{v} \bs H(a_k + \eps, v_k)$, we see that $(\bs q_\eps( \vec a, \vec v), \vec \lam_\eps( \vec a, \vec v), \vec \mu_\eps( \vec a,  \vec v))$ uniquely solves~\eqref{eq:static-ls-ii} with 
\EQ{
\bs f_\eps &= \Big( \frac{  \bs f( \vec a(\eps), \vec v)-\bs f( \vec a, \vec v) }{\eps} - \p_{a_k} \bs f(\vec a, \vec v) \Big) - (U'''(  H(\vec a, \vec v)) \p_{a_k}  H(\vec a, \vec v), 0)\big( \bs h(\vec a, \vec v)  - \bs h(\vec a(\eps),\vec v) \big)\\
&\quad + \Big( \frac{   U''(H(\vec a(\eps), \vec v))-U''(H(\vec a, \vec v)) }{\eps} - U'''(H(\vec a, \vec v)) \p_{a_k} H(\vec a, \vec v), 0 \Big)  \bs h(\vec a(\eps), \vec v)\\
&\quad +   \lam_k(\vec a(\eps), \vec v) \Big(   \frac{ \bs \al_{k, \eps}- \bs \al_k }{\eps}-  \p_{a_k} \bs \al_k\Big)  +  \mu_k(\vec a(\eps), \vec v) \Big(   \frac{ \bs \be_{k, \eps} - \bs \be_k }{\eps}-  \p_{a_k} \bs \be_k\Big)\\ 
&\quad -  \p_{a_k} \bs \al_k \big( \lam_k(\vec a, \vec v) - \lam_k(\vec a(\eps), \vec v)\big) - \p_{a_k} \bs \be_k \big( \mu_k(\vec a, \vec v) - \mu_k(\vec a(\eps), \vec v)\big)  %\\
%&\quad + \sum_{j =1}^n \big(\lambda_{\eps, j} ( \vec a, \vec v) \bs \al(\vec a , \vec v) + \mu_{\eps, j}(\vec a, \vec v) \bs \be(\vec a, \vec v)\big)
}
and 
\EQ{
\la \bs q_\eps( \vec a, \vec v), \, \bs \al_k\ra &= -\la \bs h( \vec a(\eps), \vec v) , \, \frac{ \bs \al_{k, \eps} - \bs \al_k}{\eps} - \p_{a_k}\bs \al_k \ra   + \la \bs h(\vec a, \vec v) - \bs h(\vec a(\eps), \vec v), \, \p_{a_k} \bs \al_k\ra \\
\la \bs q_\eps( \vec a, \vec v), \, \bs \al_j\ra &= 0 \mif j \neq k \\
\la \bs q_\eps( \vec a, \vec v), \, \bs \be_k\ra &= -\la \bs h( \vec a(\eps), \vec v) , \, \frac{ \bs \be_{k, \eps} - \bs \be_k}{\eps} - \p_{a_k}\bs \be_k \ra   + \la \bs h(\vec a, \vec v) - \bs h(\vec a(\eps), \vec v), \, \p_{a_k} \bs \be_k\ra \\
\la \bs q_\eps( \vec a, \vec v), \, \bs \be_j\ra &= 0 \mif j \neq k.
}
Using the continuity of $(\bs h(\vec a, \vec v), \vec \lam(\vec a, \vec v), \mu(\vec a, \vec v))$ with respect to the parameters, the differentiability of $\bs f(\vec a, \vec v)$, the smoothness of $U$,  $\bs H(\vec a, \vec v)$, $\bs \al_k$ and $\bs \be_k$, we see that $\lim_{\eps \to 0} \| \bs f_\eps\|_{\E} = 0$. Using the continuity of $\bs h(\vec a, \vec v)$ with respect to parameters and the smoothness of $\bs \al_k, \bs \be_k$, we see that in addition
\EQ{
\lim_{  \eps  \to 0} \la \bs q_\eps( \vec a, \vec v), \, \bs \al_k\ra = \lim_{  \eps  \to 0} \la \bs q_\eps( \vec a, \vec v), \, \bs \be_k\ra = 0.
}
Then~\eqref{eq:q-eps-lim} follows from~\eqref{eq:h-est-lm-ii} and~\eqref{eq:lam-mu-est-prelim-ii}. Having established~\eqref{eq:dh-dlam-dmu} the estimates for $(\p_{a_k} \bs h,  \p_{a_k} \vec \lam, \p_{a_k} \vec \mu)$ now follow from~\eqref{eq:delta-h} and Lemma~\ref{lem:static-ls}~\ref{it:static-ls-ii}. 

The argument for $(\partial_{v_k}\bs h, \partial_{v_k}\vec \lambda, \partial_{v_k}\vec \mu)$ is identical. The estimates~\eqref{eq:h-second} for all the mixed second partials is similar.

Lastly, we consider the difference estimates for $\p_{a_k} \bs h$ and $\p_{v_k} \bs h$. Let $\p_{a_k} \fl{ \bs h} = \p_{a_k} \sh{\bs h} - \p_{a_k} \bs h$,  and $( \partial_{a_k} \fl{\vec \lam}, \partial_{a_k} \fl{ \vec \mu})$, which we can view as solving
\EQ{
\vD^2 E(\bs H(\vec a, \vec v))( \partial_{a_k} \fl{ \bs h}) &= (\vD^2 E(\bs H(\vec a, \vec v)) - \vD^2 E(\bs H(\sh{\vec a}, \sh{\vec v}))) \p_{a_k} \sh{ \bs h} \\
&\quad  + (U'''( H(\vec a, \vec v))\partial_{a_k}H(\vec a, \vec v)h, 0) - (U'''( H(\sh{\vec a}, \sh{\vec v}))\partial_{a_k}H(\sh{\vec a}, \sh{\vec v})\sh h, 0)\\
&\quad+  \partial_{a_k} \sh{ \bs f}  - \partial_{a_k}\bs f \\
&\quad + \sh \lambda_k \partial_{a_k}\sh{\bs\alpha_k}- \lambda_k \partial_{a_k}\bs\alpha_k +\sh \mu_k \partial_{a_k}\sh{\bs\beta_k}- \mu_k \partial_{a_k}\bs\beta_k \\
&\quad + \sum_{j =1}^n  \partial_{a_k} \sh \lam_j ( \sh{\bs \al_j} - { \bs \al_j}) + \partial_{a_k} \sh \mu_j ( \sh{ \bs \beta}_j  - \bs \beta_j) \\
&\quad  + \sum_{j=1}^n(\partial_{a_k} \fl \lam_j)  \bs\alpha_j+ (\p_{a_k} \fl \mu_j ) \bs\beta_j
}
with 
\EQ{
\la \partial_{a_k} \fl{ \bs h}, \bs \al_j\ra &= \la \bs h , \, \p_{a_k} \bs \al_j\ra - \la \sh{ \bs h}, \, \p_{a_k} \sh {\bs \al}_j\ra  + \la \p_{a_k} \sh {\bs h} , \, \sh{ \bs \al}_j - \bs \al_j \ra \\
\la \partial_{a_k} \fl{ \bs h}, \bs \be_j\ra &= \la \bs h , \, \p_{a_k} \bs \be_j\ra - \la \sh{ \bs h}, \, \p_{a_k} \sh {\bs \be}_j\ra  + \la \p_{a_k} \sh {\bs h} , \, \sh{ \bs \be}_j - \bs \be_j \ra 
} 
and the estimates for $\p_{a_k} \fl{ \bs h}$ in ~\eqref{eq:dh-diff} follow from Lemma~\ref{lem:static-ls} part~\ref{it:static-ls-ii}. The estimates for $\p_{v_k} \sh{ \bs h} - \p_{v_k} { \bs h}$ follow from the same argument. 
%Now that we have established the differentiability of $(\bs h(\vec a, \vec v), \vec\lam(\vec a, \vec v), \vec \mu(\vec a, \vec v))$ we prove the refined estimates~\eqref{eq:} for the derivatives of the parameters by differentiating the formula the formula~\eqref{eq:lam-mu-exp}. We obtain the expressions, 
%\EQ{
%\gamma_{v_k}^3 M (\p_{a_j}\lam_k) - \p_{a_j} \big(\la \partial_{v_k}\bs H_k, \bs f\ra\big) &= - \la v_k \p_{a_j}\partial_x \bs\beta_k, \bs h\ra -  \la v_k \partial_x \bs\beta_k, \p_{a_j} \bs h\ra\\
%&\quad  - \la ((U'''( H(\vec a, \vec v))(-1)^j  - U'''(H_k))\p_{a_j} H_j) \partial_{v_k} H_k, \, h\ra \\
%&\quad - \la (U''( H(\vec a, \vec v)) - U''(H_k)) \p_{a_j}\partial_{v_k} H_k, \, h\ra \\
%&\quad -  \la (U''( H(\vec a, \vec v)) - U''(H_k)) \partial_{v_k} H_k, \,\p_{a_j} h\ra\\
%&\quad +  \sum_{\ell  \neq k} \big(\lam_{\ell} \la \bs \al_\ell, \partial_{v_k} \bs H_k\ra + \mu_{\ell} \la \bs \beta_\ell, \partial_{v_k} \bs H_k \ra \big) 
%}
%\EQ{
%\gamma_{v_k}^3 M \mu_k + \la \partial_{a_k}\bs H_k, \bs f\ra &= \la v_k\partial_x \bs\alpha_k, \bs h\ra +  \la (U''( H(\vec a, \vec v)) - U''(H_k)) \partial_{a_k} H_k, \, h\ra \\
%&\quad  - \sum_{\ell  \neq k} \big(\lam_{\ell} \la \bs \al_\ell, \partial_{a_k} \bs H_k\ra + \mu_{\ell} \la \bs \beta_\ell, \partial_{a_k} \bs H_k \ra \big)
%}
\end{proof}

\section{Preliminaries on the Cauchy problem}
\label{sec:cauchy}
Since the local well-posedness theory for \eqref{eq:csf-2nd} presents no serious difficulties,
we only briefly resume it, leaving some details to the Reader.
\begin{definition}
If $I \subset \bR$ is an open interval, $\iota_-, \iota_+ \in \{-1, 1\}$ and $\bs\phi = (\phi, \dot \phi): I \times \bR \to \bR\times\bR$, then we say that $\bs\phi$ is a solution of \eqref{eq:csf} in the energy sector $\cE_{\iota_-, \iota_+}$ if $\bs \phi \in C(I; \cE_{\iota_-, \iota_+})$
and \eqref{eq:csf} holds in the sense of distributions.
\end{definition}

Choose any $(\xi, 0) \in \cE_{\iota_-, \iota_+}$ such that $\xi \in C^\infty$ and $\partial_x \xi$ is of compact support,
and write $(\phi, \dot\phi) = (\xi + \psi, \dot\psi)$. Then $\bs \phi \in C(I; \cE_{\iota_-, \iota_+})$ is equivalent to $\bs\psi \in C(I; \cE)$
and \eqref{eq:csf} becomes
\begin{equation}
\label{eq:cauchy-psi}
\dd t\begin{pmatrix} \psi \\ \dot\psi \end{pmatrix} = \begin{pmatrix}\dot \psi \\ \partial_x^2 \psi - \psi - \big(U'(\xi + \psi) - U'(\xi) - \psi\big)
+ \big(\partial_x^2 \xi - U'(\xi)\big)\end{pmatrix},
\end{equation}
which is the linear Klein-Gordon equation for $\psi$ with the forcing term ${-}\big(U'(\xi + \psi) - U'(\xi) - \psi\big) + \big(\partial_x^2 \xi - U'(\xi)\big)
\in C(I; \cE)$. By the uniqueness of weak solutions of linear wave equations, see \cite{Friedrichs54}, we have that a weak solution of \eqref{eq:cauchy-psi}
is in fact a strong solution given by the Duhamel formula.

If $\bs\phi_{1, 0}, \bs \phi_{2, 0}, \ldots$ and $\bs\phi_0$
are states of bounded energy, then we write
$\bs \phi_{m, 0} \to \bs\phi_0$ if $\bs \phi_{m, 0} - \bs\phi_0 \to 0$ in $\cE$, and
$\bs \phi_{m, 0} \wto \bs\phi_0$ if $\dot \phi_{m, 0} \wto \dot \phi_0$
in $L^2(\bR)$ and $\phi_{m, 0} - \phi_0 \to 0$ in $L^\infty_\tx{loc}(\bR)$. Observe that if $\bs\phi_{m, 0} - \bs \phi_0$ is bounded in $\cE$,
then $\bs \phi_{m, 0} \wto \bs\phi_0$ is equivalent to the weak convergence to $0$ in the Hilbert space $\cE$ of the sequence $\bs \phi_{m, 0} - \bs \phi_0$. However, our notion of weak convergence
is more general, in particular the topological class is not necessarily
preserved under weak limits as defined above.
A similar notion of weak convergence was used in Jia and Kenig \cite{JK}.
\begin{proposition}
\label{prop:cauchy}
\begin{enumerate}[(i)]
\item \label{it:cauchy-en}
For all $\iota_-, \iota_+ \in \{-1, 1\}$ and $\bs\phi_0 \in \cE_{\iota_-, \iota_+}$ and $t_0 \in \bR$,
there exists a unique solution $\bs\phi: \bR \to \cE_{\iota_-, \iota_+}$ of \eqref{eq:csf} such that $\bs\phi(t_0) = \bs\phi_0$.
The energy $E(\bs \phi(t))$ does not depend on $t$.
\item \label{it:cauchy-H2}
If $\partial_x \phi_0 \in H^1(\bR)$ and $\dot\phi_0 \in H^1(\bR)$, then $\partial_x \phi \in C(\bR; H^1(\bR)) \cap C^1(\bR; L^2(\bR))$, $\dot\phi \in C(\bR; H^1(\bR)) \cap C^1(\bR; L^2(\bR))$
and \eqref{eq:csf-1st} holds in the strong sense in $H^1(\bR) \times L^2(\bR)$.
\item \label{it:cauchy-strong}
Let $\bs \phi_{1,0}, \bs \phi_{2,0}, \ldots \in \cE_{\iota_-, \iota_+}$ and $\bs\phi_{m, 0} \to \bs\phi_0 \in \cE_{\iota_-, \iota_+}$.
If $(\bs \phi_m)_{m=1}^\infty$ and $\bs\phi$ are the solutions of \eqref{eq:csf} such that $\bs \phi_m(t_0) = \bs\phi_{m,0}$ and $\bs\phi(t_0) = \bs\phi_0$,
then $\bs \phi_m(t) \to \bs \phi(t)$ for all $t \in \bR$, the convergence being uniform on every bounded time interval.
\item \label{it:cauchy-speed}
If $\wt{\bs\phi}_0 \vert_{[x_1, x_2]} = {\bs\phi}_0 \vert_{[x_1, x_2]}$, then $\wt{\bs\phi}(t) \vert_{[x_1 + |t - t_0|, x_2 - |t - t_0|]} = {\bs\phi}(t) \vert_{[x_1 + |t - t_0|, x_2 - |t - t_0|]}$ for all $t \in \big[t_0 - \frac 12(x_2 - x_1), t_0 + \frac 12(x_2 - x_1)\big]$.
\item \label{it:cauchy-weak}
Let $\bs \phi_{1,0}, \bs \phi_{2,0}, \ldots \in \cE_{\iota_-, \iota_+}$ and $\bs\phi_{m, 0} \wto \bs\phi_0 \in \cE_{\iota_-, \iota_+}$.
If $(\bs \phi_m)_{m=1}^\infty$ and $\bs\phi$ are the solutions of \eqref{eq:csf} such that $\bs \phi_m(t_0) = \bs\phi_{m,0}$ and $\bs\phi(t_0) = \bs\phi_0$,
then $\bs \phi_m(t) \wto \bs \phi(t)$ for all $t \in \bR$.
\end{enumerate}
\end{proposition}
\begin{proof}
Statements \ref{it:cauchy-en}, \ref{it:cauchy-H2} and \ref{it:cauchy-strong} follow from energy estimates and Picard iteration, see for example \cite{GiVe85} or \cite[Section X.13]{ReedSimon}
for similar results.
Since the linear Klein-Gordon equation has propagation speed equal to 1, each Picard iteration satisfies the finite propagation speed property, and \ref{it:cauchy-speed} follows by passing to the limit.
We skip the details.

We sketch a proof of \ref{it:cauchy-weak}.

In the first step, we argue that it can be assumed without loss of generality that $\|\bs \phi_{m, 0} - \bs \phi_0\|_\cE$ is bounded and $\|\phi_{m, 0} - \phi_0\|_{L^\infty} \to 0$.
To this end, let $x_m \to -\infty$ and $x_m' \to \infty$ be such that
\begin{equation}
\lim_{m\to\infty}\sup_{x \in [x_m, x_m']}|\phi_{m, 0}(x) - \phi_0(x)| = 0.
\end{equation}
We define a new sequence $\wt \phi_{m, 0}: \bR \to \bR$
by the formula
\begin{equation}
\label{eq:weak-conv-cutoff}
\wt\phi_{m, 0}(x) \coloneqq \begin{cases}
\iota_- & \text{for all }x \leq x_m - 1, \\
(x_m-x)\iota_- + (1-x_m+x)\phi_{m, 0}(x_m)& \text{for all }x \in [x_m - 1, x_m], \\
\phi_{m, 0}(x) &\text{for all }x \in [x_m, x_m'], \\
(x - x_m')\iota_+ + (1-x+x_m')\phi_{m, 0}(x_m') & \text{for all }x \in [x_m', x_m' + 1], \\
\iota_+ & \text{for all }x \geq x_m' + 1.
\end{cases}
\end{equation}
Taking into account that
\begin{equation}
\lim_{m\to\infty}\sup_{x \leq x_m}\big(|\phi_0(x) - \iota_-|
+\sup_{x \geq x_m'}|\phi_0(x) - \iota_+|\big) = 0,
\end{equation}
we have that $\|\wt\phi_{m, 0} - \phi_0\|_{L^\infty} \to 0$.
In particular, if we take $R \gg 1$, then
\begin{equation}
\limsup_{m\to\infty}\sup_{x \leq -R}|\wt\phi_{m, 0}(x) - \iota_-| \ll 1,
\end{equation}
thus
\begin{equation}
\limsup_{m\to\infty}\int_{-\infty}^{-R}|\wt\phi_{m, 0}(x) - \iota_-|^2\ud x \lesssim \limsup_{m\to\infty}\int_{-\infty}^{-R}U(\wt\phi_{m, 0}(x))\ud x < \infty,
\end{equation}
and similarly for $x \geq R$. Hence, $\limsup_{m\to\infty}\|\wt \phi_{m, 0} - \phi_0\|_{H^1} < \infty$.

Let $\wt{\bs\phi}_m$ be the solution of \eqref{eq:csf} such that
$\wt{\bs \phi}_m(t_0) = (\wt \phi_{m, 0}, \dot \phi_{m, 0})$.
By property \ref{it:cauchy-speed}, it suffices to prove that
$\wt{\bs\phi}_m(t) \wto \bs\phi(t)$ for all $t$,
which finishes the first step. In the sequel, we write $\phi_{m, 0}$
instead of $\wt\phi_{m, 0}$.

We choose $\xi$ as in \eqref{eq:cauchy-psi} and write $\bs \phi_m = (\xi, 0) + \bs\psi_m$, $\bs\phi = (\xi, 0) + \bs\psi$.
It suffices to prove that, for every $t\in \bR$, any subsequence of $\bs\psi_m(t)$ has a subsequence weakly converging to $\bs\psi(t)$.

Let $\bs h_m$ be the solution of \eqref{eq:free-kg} with initial data $\bs h_m(t_0) = \bs\psi_m(t_0)$,
$\bs h$ the solution of the same problem with initial data $\bs h(t_0) = \bs\psi(t_0)$,
and $\bs g_m$ the solution of
\begin{equation}
\label{eq:cauchy-gm}
\dd t\begin{pmatrix} g_m \\ \dot g_m \end{pmatrix} = \begin{pmatrix}\dot g_m \\ \partial_x^2 g_m - g_m - \big(U'(\xi + \psi_m) - U'(\xi) - \psi_m\big)
+ \big(\partial_x^2 \xi - U'(\xi)\big)\end{pmatrix},
\end{equation}
with the initial data $\bs g_m(t_0) = 0$.
We thus have $\bs\psi_m = \bs h_m + \bs g_m$ for all $m$.

By the continuity of the free Klein-Gordon flow, we have $\bs h_m(t) \wto \bs h(t)$ for all $t \in \bR$.
By the energy estimates, $\bs g_m$ is bounded in $C^1(I; \cE)$ for any bounded open interval $I$.
Since the weak topology on bounded balls of $\cE$ is metrizable, the Arzel\`a-Ascoli theorem yields $\bs g \in C(\bR; \cE)$
and a subsequence of $(\bs g_m)_m$, which we still denote $(\bs g_m)_m$, such that $\bs g_m(t) \wto \bs g(t)$ for every $t \in \bR$.
Let $\wt{\bs \psi} \coloneqq \bs h + \bs g$, so that $\bs\psi_m(t) \wto \wt{\bs\psi}(t)$ for all $t \in \bR$.
In particular, $\psi_m(t, x) \to \wt\psi(t, x)$ for all $(t, x)$, hence by the dominated convergence theorem
\begin{equation}
U'(\xi + \psi_m) - U'(\xi) - \psi_m \to U'(\xi + \wt\psi) - U'(\xi) - \wt\psi
\end{equation}
in the sense of distributions. We can thus pass to the distributional limit in \eqref{eq:cauchy-psi} and conclude that $\wt{\bs \psi}$
is a solution of \eqref{eq:csf} with initial data $\wt{\bs \psi}(t_0) = \bs\psi(t_0)$.
By the uniqueness of weak solutions, $\wt{\bs\psi} = \bs\psi$.
\end{proof}
Finally, we have local stability of vacuum solutions.
\begin{lemma}
\label{lem:loc-vac-stab}
There exist $\eta_0, C_0 > 0$ having the following property.
If $\bs\iota \in \{\bs 1, -\bs 1\}$, $t_0 \in \bR$, $-\infty \leq x_1 + 1 < x_2 - 1 \leq \infty$, $E(\bs\phi_0) < 0$, $\|\bs \phi_0 - \bs\iota\|_{\cE(x_1, x_2)} \leq \eta_0$ and $\bs\phi$ is a solution of \eqref{eq:csf} such that $\bs\phi(0) = \bs\phi_0$, then for all $t \in \big[t_0 - \frac 12(x_2 - x_1)+1, t_0 + \frac 12(x_2 - x_1)-1\big]$
\begin{equation}
\label{eq:loc-vac-stab}
\|\bs \phi(t) - \bs \iota\|_{\cE(x_1 + |t - t_0|, x_2 - |t - t_0|)} \leq C_0 \|\bs \phi_0 - \bs\iota\|_{\cE(x_1, x_2)}.
\end{equation}
\end{lemma}
\begin{proof}[Sketch of a proof]
To fix ideas, assume $\iota = 1$ and $t > t_0 = 0$.
The positivity of $U$ and Green's formula in space-time imply that the function
\begin{equation}
\big[0, (x_2 - x_1)/2 - 1\big] \owns t \mapsto \int_{x_1 + t}^{x_2 - t}\Big( \frac 12(\dot \phi(t))^2 + \frac 12 (\partial_x \phi(t))^2 + U(\phi(t))\Big)\ud x
\end{equation}
is decreasing.
Set $\bs\psi(t) \coloneqq \bs\phi(t) - \bs 1$.
For $|\psi|$ small, we have $U(1 + \psi) \simeq \psi^2$.
By a continuity argument and the fact that $\|\psi\|_{L^\infty(I)}\lesssim \|\psi\|_{H^1(I)}$
on any interval $I$ of length $\geq 2$, we obtain
$\|\psi(t)\|_{L^\infty(x_1 + t, x_2 - t)} \lesssim \eta_0$ and \eqref{eq:loc-vac-stab}.
\end{proof}

\section{Estimates on the modulation parameters}
\label{sec:mod}
Our present goal is to reduce the motion of a kink cluster to a system of ordinary
differential equations with sufficiently small error terms.
\subsection{Basic modulation}
\label{ssec:basic-mod}
We consider a solution of \eqref{eq:csf} which is close, on some open time interval $I$, to multikink configurations.
It is natural to express the solution as the sum of a multikink and a small error,
which can be done in multiple ways. A unique choice of such a decomposition is obtained
by imposing specific \emph{orthogonality conditions}. In this section, we prove Lemma~\ref{lem:mod-av-intro}, which we state again, while including an additional estimate. 

If $(\vec a, \vec v) \in C^1(I; \bR^n \times \R^n)$ and, using the notation \eqref{eq:Ha-def}, we decompose
\begin{equation}
\label{eq:g-def}
\bs \phi(t) = \bs H(\vec a(t), \vec v(t)) + \bs g(t),
\end{equation}
then the Chain Rule implies that $\bs \phi$ solves \eqref{eq:csf} if and only if $\bs g$ solves
\begin{equation}
\label{eq:ls-g-eq}
\begin{aligned}
\partial_t \bs g(t) &= \bs J\vD E\big(\bs H(\vec a(t), \vec v(t)) + \bs g(t)\big)
- {\vec a \,}'(t)\cdot \partial_{\vec a}\bs H(\vec a(t), \vec v(t)) - {\vec v \,}'(t)\cdot \partial_{\vec v}\bs H(\vec a(t), \vec v(t)) 
\end{aligned}
\end{equation}

%Similarly as in the previous section, we write $\bs H_k(t, x) \coloneqq H(x - a_k(t))$.
We impose the orthogonality conditions, which we rewrite as
\begin{equation}
\label{eq:g-orth}
\la  \bs \alpha_k(t), \bs g(t)\ra =\la  \bs \beta_k(t), \bs g(t)\ra = 0, \qquad\text{for all }k \in \{1, \ldots, n\}\text{ and }t \in I.
\end{equation}
Indeed, we prove the following   ``static'' modulation lemma.
%\begin{definition}[Weight of modulation parameters]
%For all $\vec a \in \bR^n$, we set
%\begin{equation}
%\label{eq:rhoa-def}
%\rho(\vec a) \coloneqq \sum_{k=1}^{n-1} \eee^{-(a_{k+1} - a_k)}.
%\end{equation}
%For $\vec a: I \to \bR^n$, we define $\rho : I \to (0, \infty)$ by
%\begin{equation}
%\label{eq:rhot-def}
%\rho(t) \coloneqq \rho(\vec a(t)) = \sum_{k=1}^{n-1} \eee^{-(a_{k+1}(t) - a_k(t))}.
%\end{equation}\end{definition}
%We recall that for all $\bs\phi_0 \in \cE_{1, (-1)^n}$ we set
%\begin{equation}\label{eq:d-def-2}
%\bfd(\bs\phi_0) \coloneqq \inf_{\vec a \in \bR^n}\big(\|\bs \phi_0 - \bs H(\vec a)\|_\cE^2 + \rho(\vec a)\big),
%\end{equation}
%see Definition~\ref{def:d-def}.
%Since both $\bs \phi_0$ and $\bs H(\vec a)$ belong to $\cE_{1, (-1)^n}$,
%it follows that $\bs \phi_0 - \bs H(\vec a) \in \cE$, so that $\bfd(\bs \phi_0) < \infty$.
%
%We have the following ``static'' modulation lemma.

\begin{lemma} [Modulation parameters] \label{lem:static-mod}%\label{lem:mod-av}
There exist $\eta_0, \eta_1, C_0>0$ with the following properties. For all $\bs \phi_0 \in \E_{1, (-1)^n}$  such that 
\EQ{
\bfd( \bs \phi_0) < \eta_0,
}
 there exist unique $(\vec a, \vec v) = (\vec a( \bs \phi_0), \vec v(\bs \phi_0)) \in \R^n \times \R^n$  with $a_1 \le a_2 \le \dots \le a_n$ such that 
\EQ{\label{eq:dist-bd-eta1}
 \big\| \bs \phi_0 - \bs H( \vec a, \vec v) \big\|_{\E}^2 + \rho(\vec a, \vec v) < \eta_1, 
}
and 
\EQ{ \label{eq:static-orth}%\label{eq:ortho-av-d} 
\la \bs \phi_0 - \bs H(\vec a , \vec v), \, \bs \alpha(a_k, v_k) \ra = 
\la \bs \phi_0 - \bs H(\vec a, \vec v) ,\, \bs \beta(a_k, v_k) \ra = 0, 
}
for all $k  = 1, \dots, n$. Moreover, the pair $(\vec a, \vec v)$ satisfy the estimates
\EQ{ \label{eq:dist-bd-C0}%\label{eq:av-d} 
 \big\| \bs \phi_0- \bs H( \vec a, \vec v) \big\|_{\E}^2 &+ \rho(\vec a, \vec v) < C_0  \bfd( \bs \phi_0) .
} 
and there exists $\nu>0, C>0$ such that 
\begin{multline} \label{eq:g-coer-stat}
\frac 12 \| \dot \phi_0- \dot H( \vec a, \vec v)\|_{L^2}^2 + \frac{\nu}{4}\|\phi_0-  H( \vec a, \vec v)\|_{H^1}^2 + \frac{M}{2} \sum_{k =1}^n v_k^2  \\
 \le 2\kappa^2 \sum_{k=1}^{n-1}\eee^{-y_k} + Cy_{\min}\eee^{-2y_{\min}} + Cy_{\min}\abs{v}^4+ E(\bs \phi) - nM.
\end{multline} 
%\Red{[proof not adapted yet to include Lorentz parameter]
%\begin{align}
%%\label{eq:dist-bd-C0}
%%\|\bs \phi_0 - \bs H(\vec a)\|_\cE^2 + \rho(\vec a) &\leq C_0 \bfd(\bs \phi_0), \\
%\label{eq:g-coer-stat}
%\|\bs \phi - \bs H(\vec a, \vec v)\|_\cE^2 &\leq C_0(\rho(\vec a, \vec v) + E(\bs \phi) - nM).
%\end{align} }
%%Moreover, the map $\cE_{1, (-1)^n} \owns \bs\phi_0 \mapsto \vec a(\bs \phi_0) \in \bR^n$ is of class $C^1$.
Finally, the map $\E_{1, (-1)^n} \owns \bs \phi_0 \mapsto  (\vec a( \bs \phi_0), \vec v(\bs \phi_0)) \in \R^n \times \R^n$ is $C^1$. 

\end{lemma} 

%\begin{rem} 
%We will often use the following special case of~\eqref{g-coer-stat}: 
%\end{rem} 

\begin{proof} 
  Let $\eta  \coloneqq  \bfd( \bs \phi)$. We can choose a pair $(\vec b, \vec w)$ so that 
\EQ{
\eta \le  \big\| \bs \phi - \bs H( \vec b, \vec w) \big\|_{\E}^2 + \sum_{k =1}^{n-1}  e^{ - (b_{k+1} - b_k)}   
  + \sum_{k =1}^n | w_k |^2  \le 4 \eta . 
}
for each $k \in \{1, \dots, n\}$ we define the functions 
\EQ{
A_{2k-1}( \vec a, \vec v; \bs h) &= \la   \bs  \alpha(a_k , v_k) , \, \bs h + \bs H(\vec b, \vec w) - \bs H(\vec a, \vec v) \ra  \\
B_{2k} (\vec a, \vec v; \bs h) & = \la \bs \beta(a_k, v_k), \, \bs h + \bs H(\vec b, \vec w) - \bs H(\vec a, \vec v) \ra.
}
and for $k = 1, \dots, n$ define $ \vec \Gamma = ( \Gamma_1, \dots, \Gamma_{2n}): C^1( \R^n \times \R^n \times \E ; \R^n \times \R^n)$ by 
\EQ{
\Gamma_{\ell}( \vec a, \vec v; \bs h)\coloneqq   \begin{cases} A_{2k-1}( \vec a, \vec v; \bs h) \mif  \ell = 2k-1  \\ B_{2k} (\vec a, \vec v; \bs h) \mif \ell = 2k \end{cases} 
}
Using the usual shorthand $\bs \al_k , \bs \beta_k, \bs H_k$  we compute 
\EQ{ \label{eq:A-diff} 
\p_{a_j} A_{2k-1}( \vec a, \vec v; \bs h) &= \begin{cases} \la \p_{a} \bs \al_k  , \, \bs h + \bs H( \vec b, \vec w) - \bs H( \vec a, \vec v)\ra  \mif j = k \\ 
 -(-1)^j \la \bs \al_k , \,  \p_a \bs H_j\ra  \mif j \neq k\end{cases} \\
 \p_{v_j} A_{2k-1}( \vec a, \vec v; \bs h) &= \begin{cases} \la \p_{v} \bs \al_k  , \, \bs h + \bs H( \vec b, \vec w) - \bs H( \vec a, \vec v)\ra - (-1)^k \la \bs \al_k , \, \p_v \bs H_k \ra  \mif j = k \\ 
 -(-1)^j  \la \bs \al_k , \, \p_v \bs H_j\ra  \mif j \neq k\end{cases} 
 }
 and 
 \EQ{ \label{eq:B-diff} 
 \p_{a_j} B_{2k}( \vec a, \vec v; \bs h) &= \begin{cases} \la \p_{a} \bs \beta_k  , \, \bs h + \bs H( \vec b, \vec w) - \bs H( \vec a, \vec v)\ra - (-1)^k \la \bs \be_k , \, \p_a \bs H_k \ra \mif j = k \\ 
 - (-1)^j\la \bs \beta_k , \,  \p_a \bs H_j\ra \mif j \neq k\end{cases} \\
 \p_{v_j} B_{2k}( \vec a, \vec v; \bs h) &= \begin{cases} \la \p_{v} \bs \beta_k  , \, \bs h + \bs H( \vec b, \vec w) - \bs H( \vec a, \vec v)\ra   \mif j = k \\ 
 -(-1)^j  \la \bs \beta_k , \, \p_v \bs H_j\ra  \mif j \neq k\end{cases} 
 }
 We see that $\p_{(\vec a, \vec v)} \vec \Gamma( \vec b, \vec w; \bs h)$ is a $2n \times 2n$ invertible matrix with uniformly bounded inverse as long as $\bs h$, $| \vec a - \vec b|$ and $| \vec v - \vec w|$ are small enough. 
 
Let $\bs h_0 \coloneqq  \bs \phi - \bs H (\vec b, \vec w)$. Let $C_1>0$ be a constant to be fixed below and consider, the mapping $\vec \Phi: B_{\R^n}( \vec b, C_1\sqrt{ \eta_0}) \times B_{\R^n} (\vec w, C_1 \sqrt{\eta_0}) \to \R^n \times \R^n$ defined by 
 \EQ{
 \vec \Phi(\vec a, \vec v) : = (\vec a, \vec v) - \big[\p_{(\vec a, \vec v)} \vec \Gamma( \vec b, \vec w; \bs h_0)\big]^{-1}  \vec \Gamma(\vec a, \vec v; \bs h_0) . 
 }
 We set up a contraction mapping to find a fixed point of $\vec \Phi$. 
 Observe that 
 \EQ{ \label{eq:eta-away1} 
 |  \vec \Phi(\vec b, \vec w; \bs h) - (\vec b, \vec w)| \lesssim \| \bs h_0 \|_{\E}  = \| \bs \phi - \bs H(\vec b, \vec w) \|_{\E} \le 2\sqrt{ \eta} 
 }
 and hence $C_1$ can be chosen so that 
 \EQ{
 |  \vec \Phi(\vec b, \vec w; \bs h) - (\vec b, \vec w)|  \le \frac{1}{3} C_1 \sqrt{\eta} \le \frac{1}{3} C_1 \sqrt{\eta_0} . 
 }
 By the Fundamental Theorem of Calculus, 
 \EQ{
& \vec \Phi( \vec{\sh a}, \vec{\sh v} ) - \vec \Phi( \vec a, \vec v)\\
&  = \Big[ \Id -  \big[\p_{(\vec a, \vec v)} \vec \Gamma( (\vec b, \vec w); \bs h_0)\big]^{-1}\int_0^1 \p_{(\vec a, \vec v)} \vec \Gamma( s(\vec{\sh a}, \vec{ \sh v}) + (1-s)(\vec a, \vec v);  \bs h_0) \, \ud s \Big] \Big( (\vec{\sh a}, \vec {\sh v}) - (\vec a, \vec v) \Big) 
 }
 From~\eqref{eq:A-diff},~\eqref{eq:B-diff} we see that $\p_{(\vec a, \vec v)} \vec \Gamma((\vec a, \vec v) ; \bs h_0)$ is Lipshitz with respect to $(\vec a, \vec v)$ and hence, 
 \EQ{ \label{eq:contract-prop} 
 | \p_{(\vec a, \vec v)} \vec \Gamma( s(\vec{\sh a}, \vec{ \sh v}) + (1-s)(\vec a, \vec v);  \bs h_0) -  \p_{(\vec a, \vec v)} \vec \Gamma( (\vec b, \vec w);  \bs h_0)| \lesssim C_1 \sqrt{\eta_0}
 }
 and thus, 
 \EQ{
 | \vec \Phi( \vec{\sh a}, \vec{\sh v} ) - \vec \Phi( \vec a, \vec v)| \lesssim C_1 \sqrt{\eta_0} | (\vec{\sh a}, \vec{\sh v}) - (\vec a, \vec v)| \le \frac{1}{3} C_1 | (\vec{\sh a}, \vec{\sh v}) - (\vec a, \vec v)|
 }
 provided $\eta_0$ is chosen small enough. Hence $\vec \Phi$ is a strict contraction on the ball $B_{\R^n}(\vec b, C_1 \sqrt{\eta_0}) \times B_{\R^n}(\vec w, C_1 \sqrt{\eta_0})$. 
 
 Let $(\vec a, \vec v) = (\vec a( \bs \phi), \vec v(\bs \phi)$ be the unique fixed point of $\vec \Phi$ inside the ball  $B_{\R^n}(\vec b, C_1 \sqrt{\eta_0}) \times B_{\R^n}(\vec w, C_1 \sqrt{\eta_0})$. The orthogonality conditions~\eqref{eq:static-orth} hold for this choice of $(\vec a, \vec v)$ because of the definitions of $\vec \Gamma$ and $\vec \Phi$.  Next, using that $(\vec a, \vec v)$ is a fixed point of $\vec \Phi$, ~\eqref{eq:eta-away1}, and~\eqref{eq:contract-prop} we see that 
 \EQ{
 | (\vec a, \vec v) - (\vec b, \vec w)| = | \vec \Phi(\vec a, \vec v)  -\vec \Phi(b, w) + \vec \Phi(b, w) - (\vec b, \vec w)|  \le \frac{1}{3} | (\vec a, \vec v) - (\vec b, \vec w)| + \frac{1}{3} C_1 \sqrt{\eta}
 }
 and hence $ | (\vec a, \vec v) - (\vec b, \vec w)|  \lesssim \sqrt{\eta} = \sqrt{\bfd( \bs \phi)}$. The estimates~\eqref{eq:dist-bd-C0} readily follow. 
 
% [needs fixing from here on]We refer the reader to~\cite[Proof of Lemma 4.1]{JL9} for the remaining two claims, namely that there is no pair $(\vec a, \vec v)$ outside the ball $B_{\R^n}(\vec b, C_1 \sqrt{\eta_0}) \times B_{\R^n}(\vec w, C_1 \sqrt{\eta_0})$ satisfying~\eqref{eq:static-orth} and that the map $\E \ni \phi \mapsto (\vec a( \bs \phi), \vec v(\bs \phi)) \in \R^n \times \R^n$ is $C^1$. 
 
 We now prove \eqref{eq:g-coer-stat}. Set $\bs g \coloneqq \bs \phi - \bs H(\vec a, \vec v)$.
We have the Taylor expansion
\begin{multline} 
E(\bs H(\vec a, \vec v) + \bs g) = E(\bs H(\vec a, \vec v )) + \la \vD E(\bs H(\vec a, \vec v)), \bs g\ra + \frac 12 \la \vD^2 E(\bs H(\vec a, \vec v))\bs g, \bs g\ra \\
+ \int_{-\infty}^\infty \Big(U(H(\vec a, \vec v) + g) - U(H(\vec a, \vec v)) - U'(H(\vec a, \vec v))g - \frac 12 U''(H(\vec a, \vec v))g^2\Big)\ud x.
\end{multline} 
Lemma~\ref{lem:D2H} and \eqref{eq:static-orth} yield
\begin{equation}
\la \bs g, \vD^2 E(\bs H(\vec a, \vec v))\bs g\ra  = \|\dot g\|_{L^2}^2 + \la g, \vD^2 E_p(H(\vec a, \vec v)) g\ra \geq \|\dot g\|_{L^2}^2 + \nu\|g\|_{H^1}^2.
\end{equation}
The Sobolev embedding implies that the second line has absolute value $\lesssim \|g\|_{H^1}^3$, thus $ \leq \frac \nu 8 \|g\|_{H^1}^2$ if $\eta_0$ is small enough.
By~\eqref{eq:DEH} and~\eqref{eq:static-orth}, we also have
\begin{equation}
|\la \vD E(\bs H(\vec a, \vec v)), \bs g\ra|  \lesssim \sqrt{y_{\min}}\eee^{-y_{\min}}\|g\|_{L^2},
\end{equation}
thus $|\la \vD E(\bs H(\vec a, \vec v)), \bs g\ra| \leq \frac{\nu}{8}\|g\|_{H^1}^2 + Cy_{\min}\eee^{-2y_{\min}}$. 
We arrive at the estimate
\begin{equation}
\label{eq:g-coer-pres}
\frac 12 \|\dot g\|_{L^2}^2 + \frac{\nu}{4}\|g\|_{H^1}^2 \leq E(\bs H(\vec a, \vec v) + \bs g) - E(\bs H(\vec a, \vec v ))+ Cy_{\min}\eee^{-2y_{\min}}.
%& \le 2\kappa^2 \sum_{k=1}^{n-1}\eee^{-y_k} + Cy_{\min}\eee^{-2y_{\min}} + E(\bs H(\vec a) + \bs g) - nM,
\end{equation}
Lemma~\ref{lem:interactions} then implies
\begin{multline} 
\frac 12 \|\dot g\|_{L^2}^2 + \frac{\nu}{4}\|g\|_{H^1}^2 + \frac{M}{2} \sum_{k =1}^n v_k^2  \\
 \le 2\kappa^2 \sum_{k=1}^{n-1}\eee^{-y_k} + Cy_{\min}\eee^{-2y_{\min}} + y_{\min}\abs{v}^4+ E(\bs H(\vec a) + \bs g) - nM,
\end{multline} 
in particular \eqref{eq:g-coer-stat}.

Continuous differentiability of $\bs \phi_0 \mapsto (\vec a(\bs \phi_0), \vec v( \bs \phi_0))$ follows from $\vec \Phi \in C^1$, see \cite[Chapter 2, Theorem 2.2]{chow-hale}.

It remains to prove that, if $\eta_1$ is small enough,
then there is no $(\vec a, \vec v) \notin B_{\bR^{2n}}((\vec b, \vec w), C_1\sqrt{\eta_0})$ satisfying \eqref{eq:dist-bd-eta1}. Suppose the contrary. Then there exist $\bs \phi_{0, m}$, $(\vec b_m, \vec w_m)$ and
$(\vec a_m, v_m)$ such that
\begin{gather}
\|\bs \phi_{0, m} - \bs H(\vec b_m, \vec w_m)\|_\cE^2 + \rho(\vec b_m, \vec w_m) \leq 2\bfd(\bs\phi_{0, m})\qquad\text{for all }m, \\
\lim_{m\to\infty} \big(\|\bs\phi_{0, m} - \bs H(\vec a_m, \vec v_m)\|_\cE^2 + \rho(\vec a_m, \vec v_m)\big) = 0, \\
|\vec a_m - \vec b_m| + | \vec v_m - \vec w_m|\geq C_1\sqrt{\eta_0} \qquad\text{for all }m.
\end{gather}
In particular, the first two conditions above yield
\begin{equation}
\lim_{m\to\infty} \big(\|\bs H(\vec a_m, \vec v_m) - \bs H(\vec b_m, \vec w_m)\|_\cE^2 + \rho(\vec a_m, \vec v_m)+ \rho(\vec b_m, \vec w_m)\big) = 0,
\end{equation}
which contradicts the third condition. For the details of this last point, see for example~\cite[proof of Lemma 2.27]{JL6} for a similar argument. %\red{ANY SHORT ARGUMENT FOR THIS? See proof of Lemma~2.27 in [JL6].
%Show that $a_{m, 1} \leq b_{m, 1} + 1$, then that cut out the region of the first kink, and do induction on the number of kinks.}
\end{proof}

\begin{lemma}
\label{lem:basic-mod}
There exist $\eta_0, C_0, C_1 > 0$ such that the following is true.
Let $I \subset \bR$ be an open interval and let $\bs \phi: I \to \cE_{1, (-1)^n}$ be a solution of \eqref{eq:csf} such that
\begin{equation}
\bfd(\bs \phi(t)) \leq \eta_0, \qquad\text{for all }t \in I.
\end{equation}
Then there exist $\vec a \in C^1(I; \bR^n)$ and $\bs g \in C(I; \cE)$
satisfying \eqref{eq:g-def} and \eqref{eq:g-orth}.
In addition, %denoting
%\begin{equation}
%\label{eq:rhot-def}
%\rho(t) \coloneqq \rho(\vec a(t), E(\bs \phi)),
%\end{equation}
for all $t \in I$ such that
\begin{equation}
\label{eq:delta-geq-ass}
\bfd(\bs\phi(t)) \geq C_1(E(\bs \phi) - nM)
\end{equation}
the following bounds hold:
\begin{gather}
\label{eq:g-coer}
\|\bs g(t)\|_{\cE}^2 \leq C_0 \rho(t) \leq C_0^2 \bfd(\bs\phi(t)), \\
\label{eq:ak'-est}
\big|a_k'(t) - v_k(t)\big | + \big|v_k'(t)\big| \leq C_0\rho(t).
\end{gather}
Finally, if $\bs \phi_m$ is a sequence of solutions of \eqref{eq:csf} such that
\begin{equation}
\label{eq:phi-conv-en}
\lim_{m \to \infty} \|\bs \phi_m - \bs \phi\|_{C(I; \cE)} = 0,
\end{equation}
then the corresponding $\vec a_m$, $\vec v_m$ and $\bs g_m$ satisfy
\begin{equation}
\label{eq:mod-basic-conv}
\lim_{m \to \infty}\big(\|\vec a_m - \vec a\|_{C^1(I; \bR^n)}
+ \|\vec v_m - \vec v\|_{C^1(I; \bR^n)}
+ \|\bs g_m - \bs g\|_{C(I; \cE)}\big) = 0.
\end{equation}
\end{lemma}
\begin{remark}
In Sections~\ref{sec:n-body} and \ref{sec:any-position},
we will always have $E(\bs \phi) = nM$,
hence \eqref{eq:delta-geq-ass} will be automatically satisfied.
In Section~\ref{sec:profiles}, an additional argument will be necessary in order to ensure \eqref{eq:delta-geq-ass} on time intervals of interest, see the beginning of the proof of Theorem~\ref{thm:unstable}.
\end{remark}
\begin{proof}[Proof of Lemma~\ref{lem:basic-mod}]
\textbf{Step 1.} (Modulation equations.)
For every $t \in I$, let $\vec a(t) \coloneqq \vec a(\bs\phi(t))$ and $\vec v(t) \coloneqq \vec v(\bs\phi(t))$ be given by Lemma~\ref{lem:static-mod}, and let $\bs g(t)$ be given by \eqref{eq:g-def}.
The maps $t \mapsto \bs\phi(t) \in \cE_{1, (-1)^n}$, $\cE_{1, (-1)^n} \owns \bs\phi_0 \mapsto \vec a(\bs\phi_0)$ and $\cE_{1, (-1)^n} \owns \bs\phi_0 \mapsto \vec v(\bs\phi_0)$ are continuous, hence $\vec a \in C(I; \bR^n)$ and $\vec v \in C(I; \bR^n)$.
Moreover, \eqref{eq:phi-conv-en} implies
\begin{equation}
\label{eq:mod-basic-conv-0}
\lim_{m\to\infty}\big(\|\vec a_m - \vec a\|_{C(I; \bR^n)} +\|\vec v_m - \vec v\|_{C(I; \bR^n)} + \|\bs g_m - \bs g\|_{C(I; \cE)}\big) = 0.
\end{equation}

We claim that $\vec a \in C^1(I; \bR^n)$, $\vec v \in C^1(I; \bR^n)$ and that for any $k \in \{1, \ldots, n\}$ we have %\red{WHAT TO REPLACE THIS WITH}
 \EQ{\label{eq:ak-lin-syst}
& (a_k'- v_k)\Big(( (-1)^k\gamma_{v_k}^3 M -  \la \bs g, \,  \p_{a} \bs \beta_k \ra\Big) - v_k ' \la \bs g, \, \p_v \bs \beta_k \ra + \sum_{ j \neq k} (-1)^j \big( a_j'  \la \p_{a} \bs H_j, \,\bs \beta_k \ra  + v_j'  \la \p_{v} \bs H_j , \, \bs \beta_k \ra \big)  \\
& =- \sum_{j \neq k} v_j \la \bs \alpha_j, \, \p_v \bs H_k \ra   + \la U'( H(\vec a, \vec v)) - \sum_{j =1}^n (-1)^j U'(H_j), \, \p_v H_k \ra +  \la      g, \,  \big(U''( H(\vec a, \vec v))  - U''(  H_k)\big)\p_v  H_k \ra \\
& \quad + \la \bs J \big( \vD E(\bs H(\vec a, \vec v) + \bs g) - \vD E(\bs H(\vec a, \vec v) ) - \vD^2 E(\bs H(\vec a, \vec v))  \bs g \big), \, \bs \beta_k\ra .
}
and
\EQ{
\label{eq:vk-lin-syst}
 &(a_k' - v_k) \la \bs g, \,  \p_{a} \bs \alpha_k \ra  - v_k ' \Big( (-1)^k \gamma_{v_k}^3 M  - \la \bs g, \, \p_v \bs \alpha_k \ra \Big) + \sum_{ j \neq k} (-1)^j \big( a_j'  \la \p_{a} \bs H_j, \,\bs \al_k \ra  + v_j'  \la \p_{v} \bs H_j , \, \bs \al_k \ra \big) \\
 &= - (-1)^k F_k( \vec a, \vec v) + \la    \big(U''(H(\vec a, \vec v)) - U''( H_k)  \big)  \p_a  H_k , \, g \ra   -\sum_{j \neq k} (-1)^j v_j  \la \bs J  \bs \al_j, \, \bs \al_k \ra  \\
 &\quad + \la \bs J \big( \vD E(\bs H(\vec a, \vec v) + \bs g) - \vD E(\bs H(\vec a, \vec v) ) - \vD^2 E(\bs H(\vec a, \vec v))  \bs g \big), \, \bs \al_k\ra .
 }
%and 
%\EQ{
%&a_k' \la \bs g, \,  \p_{a} \bs \beta(a_k, v_k) \ra + v_k ' \la \bs g, \, \p_v \bs \beta(a_k, v_k) \ra \\
%&=\la  \bs  J \vD E(\bs H(\vec a, \vec v) + \bs g), \, \bs \beta_k \ra  -\sum_{ j =1}^n (-1)^j \big( a_j'  \la \p_{a} \bs H_j, \,\bs \beta_k \ra  + v_j'  \la \p_{v} \bs H_j , \, \bs \beta_k \ra \big)
%}
%\begin{equation}
%\label{eq:ak-lin-syst}
%\begin{aligned}
%((-1)^kM - \la \partial_x^2 H_k(t), g(t)\ra)a_k'(t) + \sum_{j \neq k}\la \partial_x H_k(t), \partial_x H_j(t)\ra a_j'(t) \\
%= {-}\la \partial_x H_k(t), \dot g(t)\ra.
%\end{aligned}
%\end{equation}

In order to justify~\eqref{eq:ak-lin-syst} and \eqref{eq:vk-lin-syst}, first assume that $\bs \phi \in C^1(I; \cE_{1, (-1)^n})$
and \eqref{eq:csf} holds in the strong sense. Since the map $\bs\phi_0 \mapsto \vec a(\bs\phi_0)$
is of class $C^1$, we obtain $\vec a\in C^1(I; \bR^n)$ and $\vec v\in C^1(I; \bR^n)$.
We pair \eqref{eq:ls-g-eq} with $\bs \al_k(t)$ and $\bs \beta_k(t)$  to obtain%and use~\eqref{eq:} to obtain
\EQ{ \label{eq:gav-paired} 
\la \p_t \bs g, \bs \al_k\ra &=  \la  \bs  J \vD E(\bs H(\vec a, \vec v) + \bs g), \, \bs \al_k \ra  -\sum_{ j =1}^n (-1)^j \big( a_j'  \la \p_{a} \bs H_j, \,\bs \al_k \ra  + v_j'  \la \p_{v} \bs H_j , \, \bs \al_k \ra \big) \\
\la \p_t \bs g, \bs \beta_k\ra &=  \la  \bs  J \vD E(\bs H(\vec a, \vec v) + \bs g), \, \bs \beta_k \ra  -\sum_{ j =1}^n (-1)^j \big( a_j'  \la \p_{a} \bs H_j, \,\bs \beta_k \ra  + v_j'  \la \p_{v} \bs H_j , \, \bs \beta_k \ra \big)
}
We expand the first term on the right of the first equation as 
\EQ{
 \la  \bs  J \vD E(\bs H(\vec a, \vec v) &+ \bs g), \, \bs \al_k \ra  = \la  \bs  J \vD^2 E(\bs H(\vec a, \vec v))  \bs g, \, \bs \al_k \ra +  \la \bs J \vD E(\bs H(\vec a, \vec v)), \, \bs \al_k \ra\\
 & \quad +\la \bs J \big( \vD E(\bs H(\vec a, \vec v) + \bs g) - \vD E(\bs H(\vec a, \vec v) ) - \vD^2 E(\bs H(\vec a, \vec v))  \bs g \big), \, \bs \al_k\ra 
}
For the first term on the right we use~\eqref{eq:alpha-id} to write 
\EQ{
\la  \bs  J \vD^2 E(\bs H(\vec a, \vec v))  \bs g, \, \bs \al_k \ra  &=   \la    \big(\vD^2 E(\bs H(\vec a, \vec v)) - \vD^2 E (\bs H_k))\big)  \bs g, \, \p_a \bs H_k \ra  + v_k \la \p_x \bs \al_k, \, \bs g \ra \\
& =   \la    \big(U''(H(\vec a, \vec v)) - U''( H_k)  \big)  \p_a  H_k , \, g \ra  + v_k \la \p_x \bs \al_k, \, \bs g \ra% \\
%&\quad - \la \vD^2 E( \bs H_k) \bs J \bs \al_k  , \, \bs g\ra  
}
For the second term we have 
\EQ{
 \la \bs J \vD E(\bs H(\vec a, \vec v)), \, \bs \al_k \ra &=  -\sum_{j \neq k} (-1)^j v_j  \la \bs J  \bs \al_j, \, \bs \al_k \ra +  \la  U'( H(\vec a, \vec v)) - \sum_{j =1}^n(-1)^j U'( H_j)  , \,  \p_a  H_k \ra \\
 & = -\sum_{j \neq k} (-1)^j v_j  \la \bs J  \bs \al_j, \, \bs \al_k \ra  - (-1)^k F_k( \vec a, \vec v)
}
For the left-hand side of~\eqref{eq:gav-paired} we use the relations 
\EQ{ \label{eq:diff-ortho-g} 
-a_k' \la \bs g, \,  \p_{a} \bs \alpha_k \ra - v_k ' \la \bs g, \, \p_v \bs \alpha_k \ra  &=\la  \p_t \bs g, \, \bs \alpha_k \ra \\
 - a_k' \la \bs g, \,  \p_{a} \bs \beta_k \ra - v_k ' \la \bs g, \, \p_v \bs \beta_k \ra  &=\la  \p_t \bs g, \, \bs \beta_k \ra,
}
 which come from differentiating in the time the orthogonality conditions~\eqref{eq:g-orth}. After rearranging, we obtain~\eqref{eq:vk-lin-syst}. 
 
Next,  we expand 
 \EQ{
  \la & \bs  J \vD E(\bs H(\vec a, \vec v) + \bs g), \, \bs \beta_k \ra  =\la  \bs  J \vD^2 E(\bs H(\vec a, \vec v))  \bs g, \, \bs \beta_k \ra +  \la \bs J \vD E(\bs H(\vec a, \vec v)), \, \bs \beta_k \ra\\
 & \quad +\la \bs J \big( \vD E(\bs H(\vec a, \vec v) + \bs g) - \vD E(\bs H(\vec a, \vec v) ) - \vD^2 E(\bs H(\vec a, \vec v))  \bs g \big), \, \bs \beta_k\ra %\\
 %&= \la  \bs  J \vD^2 E(\bs H_k) \bs g, \, \bs \beta_k \ra + \la  \bs  J \big(\vD^2 E(\bs H(\vec a, \vec v))  - \vD^2 E( \bs H_k)\big) \bs g, \, \bs \beta_k \ra
}
We apply~\eqref{eq:beta-id} for the first term which, using~\eqref{eq:g-orth} becomes 
\EQ{
\la  \bs  J \vD^2 E(\bs H(\vec a, \vec v))  \bs g, \, \bs \beta_k \ra &= \la  \bs  J \vD^2 E(\bs H_k) \bs g, \, \bs \beta_k \ra + \la      g, \,  \big(U''( H(\vec a, \vec v))  - U''(  H_k)\big)\p_v  H_k \ra\\
& = -v_k \la \p_a  \bs \beta_k,\, \bs g\ra + \la      g, \,  \big(U''( H(\vec a, \vec v))  - U''(  H_k)\big)\p_v  H_k \ra
}
For the second term we have 
\EQ{
\la \bs J \vD E(\bs H(\vec a, \vec v)), \, \bs \beta_k \ra &= (-1)^k \gamma_{v_k}^3 M v_k - \sum_{j \neq k} v_j \la \bs \alpha_j, \, \p_v \bs H_k \ra  \\
& \quad + \la U'( H(\vec a, \vec v)) - \sum_{j =1}^n (-1)^j U'(H_j), \, \p_v H_k \ra 
}
Using again~\eqref{eq:diff-ortho-g} for the left hand-side of~\eqref{eq:gav-paired} and rearranging we obtain~\eqref{eq:ak-lin-syst}. 
%\EQ{
%& (a_k'- v_k)\Big(( (-1)^k\gamma_{v_k}^3 M -  \la \bs g, \,  \p_{a} \bs \beta_k \ra\Big) - v_k ' \la \bs g, \, \p_v \bs \beta_k \ra + \sum_{ j \neq k} (-1)^j \big( a_j'  \la \p_{a} \bs H_j, \,\bs \beta_k \ra  + v_j'  \la \p_{v} \bs H_j , \, \bs \beta_k \ra \big)  \\
%& =- \sum_{j \neq k} v_j \la \bs \alpha_j, \, \p_v \bs H_k \ra   + \la U'( H(\vec a, \vec v)) - \sum_{j =1}^n (-1)^j U'(H_j), \, \p_v H_k \ra +  \la      g, \,  \big(U''( H(\vec a, \vec v))  - U''(  H_k)\big)\p_v  H_k \ra \\
%& \quad + \la \bs J \big( \vD E(\bs H(\vec a, \vec v) + \bs g) - \vD E(\bs H(\vec a, \vec v) ) - \vD^2 E(\bs H(\vec a, \vec v))  \bs g \big), \, \bs \beta_k\ra .
%}

%The first component of \eqref{eq:ls-g-eq} yields
%\begin{equation}
%\label{eq:g-1st}
%\partial_t g(t) = \dot g(t) + \sum_{k=1}^n (-1)^k a_k'(t)\partial_x H_k(t).
%\end{equation}
%Differentiating in time the relation \eqref{eq:g-orth}, we obtain \eqref{eq:ak-lin-syst}.

Consider now the general case (without the additional regularity assumptions).
By Proposition~\ref{prop:cauchy}, there exists a sequence $\bs\phi_m$ satisfying
\eqref{eq:phi-conv-en} and the additional regularity assumption stated above,
hence the corresponding parameters $(\vec a_m, \vec v_m)$ satisfy \eqref{eq:mod-basic-conv-0}.
They also satisfy \eqref{eq:ak-lin-syst} and~\eqref{eq:vk-lin-syst} with $(\vec a, \vec v)$ replaced by $(\vec a_m, \vec v_m)$,
which is a linear system for $\dd t\vec a_m(t)$ and $\dd t\vec v_m(t)$.
The non-diagonal terms of its matrix are
$ \ll \sqrt{\rho(\vec a_m, \vec v_m)}$; see~\eqref{eq:Mav-cross}.
It follows from \eqref{eq:g-coer} that the diagonal terms differ from $(-1)^k M$ by a quantity of order at most $\sqrt{\rho(\vec a_m, \vec v_m)}$; see~\eqref{eq:Mav-diag} and note that $\|\bs g_m \|_{\E} \lesssim \sqrt{ \rho(\vec a_m, \vec v_m)}$.
In particular, the matrix is uniformly non-degenerate, hence we have uniform convergence
of $\dd t\vec a_m$ and $\dd t\vec v_m$, which implies that $\vec a \in C^1$, $\vec v \in C^1$ and \eqref{eq:ak-lin-syst} holds.

\textbf{Step 2.} (Coercivity.)
If we choose $C_1 \coloneqq 2C_0$, where $C_1$ is the constant in \eqref{eq:delta-geq-ass}
and $C_0$ the constant in \eqref{eq:g-coer-stat}, then we have
\begin{equation}
\begin{aligned}
\|\bs g(t)\|_\cE^2 &\leq C_0\rho(\vec a(t), \vec v(t)) + \frac 12C_1\big(E(\bs \phi) - nM\big) \\
&\leq C_0 \rho(\vec a(t), \vec v(t)) + \frac 12 \bfd(\bs \phi(t)) \\
&\leq C_0 \rho(\vec a(t), \vec v(t)) + \frac 12 \big(\rho(\vec a(t)) + \|\bs g(t)\|_\cE^2\big),
\end{aligned}
\end{equation}
the last inequality resulting from the definition of $\bfd$.
This proves the first inequality in \eqref{eq:g-coer}, up to adjusting $C_0$.
The second inequality in \eqref{eq:g-coer} follows from \eqref{eq:dist-bd-C0}.

By standard estimates on the inverse of a diagonally dominant matrix, \eqref{eq:ak-lin-syst},~\eqref{eq:vk-lin-syst}
and \eqref{eq:g-coer} yield \eqref{eq:ak'-est}.

\textbf{Step 3.} (Continuous dependence of first derivatives.)
We have already observed that \eqref{eq:phi-conv-en} implies \eqref{eq:mod-basic-conv-0}.
It is clear that \eqref{eq:phi-conv-en} implies uniform convergence of the coefficients
of the system \eqref{eq:ak-lin-syst} and~\eqref{eq:vk-lin-syst}, hence uniform convergence of $\dd t\vec a_m$ and $\dd t \vec v_m$,
and we conlude that \eqref{eq:phi-conv-en} implies \eqref{eq:mod-basic-conv}.
\end{proof}
\begin{remark}
We refer to
\cite[Proposition 3]{GuSi06}, \cite[Proposition 3.1]{MeZa12} or \cite[Lemma 3.3]{J-18-nonexist} for results and proofs similar to Lemmas~\ref{lem:static-mod} and \ref{lem:basic-mod}.
\end{remark}
\begin{remark} 
\label{lem:impr-coer}
In the setting of Lemma~\ref{lem:basic-mod},
assume in addition that $E(\bs \phi) = nM$. Then, a direct consequence  of the estimate ~\eqref{eq:g-coer-stat} is the bound 
\begin{equation}
\label{eq:a'-est}
\|\dot g(t)\|_{L^2}^2 + \frac \nu2\|g(t)\|_{H^1}^2 + M|\vec v(t)|^2
\leq 4\kappa^2 \sum_{ k=1}^{n-1} e^{-( a_{k+1}(t) - a_k(t))} + C_0\rho(t)^{\frac 32},
\end{equation}
where $\nu > 0$ is the constant in Lemma~\ref{lem:D2H}.
This more precise coercivity bound  %\Red{This is now already done in Lemma~\ref{lem:static-mod}}
will be useful in Section~\ref{sec:n-body}. 
\end{remark}
%\begin{proof}
%From \eqref{eq:g-orth} and the Leibniz rule, we have
%\begin{equation}
%\la \partial_x H_k(t), \partial_t g(t)\ra = a_k'(t)\la \partial_x^2 H_k(t), g(t)\ra,
%\end{equation}
%thus \eqref{eq:g-coer} and \eqref{eq:ak'-est} yield
%\begin{equation}
%|\la \partial_x H_k(t), \partial_t g(t)\ra| \lesssim \rho(t).
%\end{equation}
%Using this bound, \eqref{eq:g-1st} yields
%\begin{equation}
%\label{eq:gdot-decomp}
%\bigg|\|\dot g(t)\|_{L^2}^2 - \|\partial_t g\|_{L^2}^2 - \Big\|\sum_{k=1}^n (-1)^k a_k'(t)\partial_x H_k(t)\Big\|_{L^2}^2\bigg| \lesssim \rho(t)^{\frac 32}.
%\end{equation}
%%
%%The bound \eqref{eq:ak'-est} yields $|a_k'(t)| \lesssim \sqrt{\rho(t)}$.
%%Multiplying \eqref{eq:ak'-est} by $a_k'(t)$ and summing with respect to $k$, we obtain
%%\begin{equation}
%%\label{eq:imp-a-1}
%%\Big|M|\vec a\,'(t)|^2 + \Big\la \sum_{k=1}^n (-1)^k a_k'(t)\partial_x H_k(t), \dot g(t)\Big\ra \Big| \lesssim \rho(t)^\frac 32.
%%\end{equation}
%Similarly as in \eqref{eq:EpHX-a-1}, we have
%\begin{equation}
%\int_{-\infty}^\infty\bigg(\bigg|\sum_{k=1}^n (-1)^k a_k'(t)\partial_x H_k(t)\bigg|^2 - \sum_{k=1}^n |a_k'(t)|^2|\partial_x H_k(t)|^2 \bigg)\ud x \lesssim y_{\min}\eee^{-2y_{\min}},
%\end{equation}
%thus we can rewrite \eqref{eq:gdot-decomp} as
%\begin{equation}
%\big|\|\dot g(t)\|_{L^2}^2 - \|\partial_t g\|_{L^2}^2 - M|\vec a\,'(t)|^2\big| \lesssim \rho(t)^{\frac 32}.
%\end{equation}
%Injecting this into \eqref{eq:g-coer-pres}, we obtain \eqref{eq:a'-est}.
%\end{proof}

\subsection{Refined modulation}
\label{ssec:ref-mod}
In order to proceed with the analysis of the dynamics of the modulations parameters,
we follow an idea used in a similar context in \cite{J-18p-gkdv} and introduce \emph{localized momenta}, see also \cite[Proposition 4.3]{RaSz11}.

Recall that $\chi \in C^\infty$ is a decreasing function such that $\chi(x) = 1$
for all $x \leq \frac 13$ and $\chi(x) = 0$ for all $x \geq \frac 23$.
\begin{definition}[Localized momenta]
\label{def:p}
Let $I$, $\bs \phi$, $\vec a$, $\vec v$ and $\bs g$ be as in Lemma~\ref{lem:basic-mod}.
We set
\begin{equation}
\begin{aligned}
\chi_1(t, x) &\coloneqq \chi\Big(\frac{x - a_1(t)}{a_2(t) - a_1(t)}\Big), \\
\chi_k(t, x) &\coloneqq \chi\Big(\frac{x - a_k(t)}{a_{k+1}(t) - a_k(t)}\Big)
- \chi\Big(\frac{x - a_{k-1}(t)}{a_k(t) - a_{k-1}(t)}\Big),\qquad\text{for }k \in \{2, \ldots, n-1\}, \\
\chi_n(t, x) &\coloneqq 1 - \chi\Big(\frac{x - a_{n-1}(t)}{a_n(t) - a_{n-1}(t)}\Big).
\end{aligned}
\end{equation}
We define $\vec p = (p_1, \ldots, p_n): I \to \bR^n$ by
\begin{equation}
\label{eq:p-def}
p_k(t) \coloneqq -\la (-1)^k\partial_x H_k(t) + \chi_k(t)\partial_x g(t), \dot \phi(t)\ra.
\end{equation}
\end{definition}
\begin{lemma}
\label{lem:ref-mod}
There exists $C_0$ such that, under the assumptions of Lemma~\ref{lem:basic-mod},
$\vec p \in C^1(I; \bR^n)$ and the following bounds hold for all $t \in I$ and $k \in \{1, \ldots, n\}$:
\begin{align}
\label{eq:ak'} |Mv_k(t) - p_k(t)| &\leq C_0 \rho(t), \\
\label{eq:pk'} |p_k'(t) - F_k(\vec a(t))| &\leq \frac{C_0\rho(t)}{-\log \rho(t)},
\end{align}
where $M$ and $F_k$ are defined by \eqref{eq:M-def} and \eqref{eq:F_k-v-0}.
\end{lemma}
\begin{remark}
If we think of $M$, $p_k$ and $F_k$ as the (rest) mass of the kink,
its momentum and the force acting on it, then \eqref{eq:ak'} and \eqref{eq:pk'}
yield (approximate) Newton's second law for the kink motion.
\end{remark}
\begin{proof}[Proof of Lemma~\ref{lem:ref-mod}]
%The bound \eqref{eq:ak'} follows from \eqref{eq:p-def}, \eqref{eq:ak'-est} and \eqref{eq:g-coer}.
We start with the bound~\eqref{eq:ak'}. To see this, we pair $\dot \phi(t)$ with $ \dot \alpha_k(t)$ giving 
\EQ{
 \la \dot \phi, \, \dot  \alpha_k \ra = \la \dot H( \vec a, \vec v) , \,  \dot \alpha_k \ra  + \la \dot g,\, \dot \alpha_k \ra 
 & =  \la \dot H( \vec a, \vec v) , \,  \dot \alpha_k \ra  - \la g,\, \alpha_k \ra,
}
due to the orthogonality $ \la \bs g , \, \bs \alpha_k \ra = 0$. We expand the right-hand side 
\EQ{
 \la \dot H( \vec a, \vec v) , \,  \dot \alpha_k \ra  - \la g,\, \alpha_k \ra&=(-1)^k\la \dot H_k , \,  \dot \alpha_k \ra + O( | \vec v| y_{\min} e^{-y_{\min}} + |v| \| \bs g \|_{\E})\\
 & = -(-1)^k v_k \gamma_{v_k} M +O( | \vec v| y_{\min} e^{-y_{\min}} + |v| \| \bs g \|_{\E}),
}
where the first line uses Lemma~\ref{lem:exp-cross-term} and the second the definition of $\dot H_k$ and $\dot \alpha_k$.  Using $ \gamma_{v_k} = 1+ O( |\vec v|^2)$, ~\eqref{eq:g-coer}, and 
\EQ{
|  \la \dot \phi, \, \dot  \alpha_k \ra - \la \dot \phi, \, \dot \alpha(a_k, 0) \ra| \lesssim \rho(\vec a, \vec v),
}
we arrive at the expression 
\EQ{
M v_k = -(-1)^k \la \dot \phi, \, \p_x H( \cdot - a) \ra + O( \rho(\vec a, \vec v)),
}
and~\eqref{eq:ak'} follows from the above, ~\eqref{eq:p-def}, and~\eqref{eq:g-coer}. 

It is clear that $\vec p \in C(I; \bR^n)$.
We claim that $\vec p \in C^1(I; \bR^n)$ and for any
$k \in \{1, \ldots, n\}$ we have
%\begin{equation}
%\label{eq:pk'-id}
%\begin{aligned}
%p_k'(t) &= (-1)^k a_k'(t)\int_{-\infty}^{\infty}(1 - \chi_k(t))\partial_x^2 H_k(t)\dot g(t)\ud x \\
%&- \sum_{j \neq k}(-1)^j a_j'(t)\int_{-\infty}^{\infty}\chi_k(t)\partial_x^2 H_j(t)\dot g(t)\ud x \\
%&+ \frac 12 \int_{-\infty}^\infty \partial_x \chi_k(t)\big((\dot g(t))^2 + (\partial_x g(t))^2\big)\ud x - \int_{-\infty}^\infty \partial_t \chi_k(t)\dot g(t)\partial_x g(t)\ud x \\
%&+ \int_{-\infty}^\infty \chi_k(t) \partial_x g(t)\big(U'(H(\vec a(t)) + g(t)) - \sum_{j}(-1)^j U'(H_j(t))\big)\ud x \\
%&+ (-1)^k \int_{-\infty}^\infty \partial_x H_k(t)\Big(U'(H(\vec a(t)) + g(t)) \\
%&\qquad\qquad\qquad\qquad\qquad- \sum_{j=1}^n (-1)^j U'(H_j(t)) - U''(H_k(t))g\Big)\ud x.
%\end{aligned}
%\end{equation}
\EQ{\label{eq:pk'-id}
&p_k' =  (-1)^k a_k' \la(1- \chi_k) \p_x^2 H_k, \, \dot \phi \ra + (-1)^k v_k'\la \p_v \partial_x H_k, \dot g\ra  \\
&\quad + \sum_{j \neq k} (-1)^j \Big \la \Big((a_j' - v_j) \p_x\p_a H_j + v_j' \p_x \p_v H_j \Big), \,\dot \phi\Big\ra + \frac{1}{2}v_k^2 \la \p_x\chi_k \p_x H_k ,\,  \p_x H_k \ra  \\
&\quad + \frac12 \la \p_x \chi_k\dot g, \,  \dot g  \ra  + \frac{1}{2} \la \p_x \chi \p_x g, \p_x g \ra +  \la \p_x \chi_k \dot g, \, \dot H(\vec a, \vec v) \ra- \la \partial_t \chi_k\partial_x g, \,\dot \phi\ra\\
&\quad + (-1)^k v_k \la \chi_k \p_x^2 H_k, \, \dot H( \vec a, \vec v) - \dot H_k \ra +\la \chi_k\dot g, \,  \p_x(\dot H(\vec a, \vec v)- \dot H_k \ra) \\
&\quad  -(-1)^k ( 1- \gamma_{v_{k}}^2)\big \la U''( H_k) \p_x H_k, \, g \big \ra  - (-1)^k  \big \la (U''(H( \vec a, \vec v) - U''(H_k)) \p_x H_k, \, g \big \ra \\
& \quad  +  \Big \la (-1)^k \partial_x H_k, \,\sum_{j =1}^n (-1)^j v_j \p_a \dot H_j \Big \ra+ \Big \la  \chi_k\partial_x g, \,\sum_{j =1}^n (-1)^j v_j \p_a \dot H_j \Big \ra \\
&\quad +\Big \la  \chi_k\partial_x g,\,\Big( U'(H(\vec a, \vec v) + g) - \sum_{j=1}^n (-1)^j U'(H_j) \Big) \Big \ra \\
&\quad + \Big \la (-1)^k\partial_x H_k, \,\Big( U'(H(\vec a, \vec v) + g) -  \sum_{j=1}^n (-1)^j U'(H_j) - U''(H(\vec a, \vec v)) g\Big)\Big \ra  .
%&\quad + \int_{-\infty}^\infty \chi_k(t) \partial_x g(t)\big(U'(H(\vec a(t)) + g(t)) - \sum_{j}(-1)^j U'(H_j(t))\big)\ud x \\
%&\quad+ (-1)^k \int_{-\infty}^\infty \partial_x H_k(t)\Big(U'(H(\vec a(t)) + g(t)) - \sum_{j=1}^n (-1)^j U'(H_j(t)) - U''(H_k(t))g\Big)\ud x.
}

%\EQ{
% \p_t \bs g &= \bs J \vD E( \bs H( \vec a, \vec v) + \bs g)- \p_t \bs H( \vec a, \vec v) \\
% & = \bs J \vD^2 E( \bs H( \vec a, \vec v)) \bs g + \bs J \Big( \vD E( \bs H( \vec a, \vec v) + \bs g) - \vD E( \bs H( \vec a, \vec v) ) - \vD^2 E( \bs H(\vec a, \vec v)) \bs g \Big) \\
% & \quad  - \Big( \p_t \bs H(\vec a, \vec v) - \bs J \vD E( \bs H(\vec a, \vec v)) \Big) \\
% & = \bs J \vD^2 E( \bs H( \vec a, \vec v)) \bs g + \bs J \Big( \vD E( \bs H( \vec a, \vec v) + \bs g) - \vD E( \bs H( \vec a, \vec v) ) - \vD^2 E( \bs H(\vec a, \vec v)) \bs g \Big)\\
% &\quad - \sum_{j =1}^n (-1)^j \Big( (a_j' - v_j) \p_a \bs H_j + v_j' \p_v \bs H_j\Big) - \pmat{ 0 \\ U'( H(\vec a, \vec v)) - \sum_{j=1}^n (-1)^j U'( H_j) } \\
%}

In order to justify \eqref{eq:pk'-id}, first assume that
$\partial_x g, \dot g \in C(I; H^1(\bR)) \cap C^1(I; L^2(\bR))$, 
that first component of \eqref{eq:ls-g-eq}, which is
\EQ{
\p_t g =  \dot g - \sum_{j =1}^n (-1)^j \Big((a_j' - v_j) \p_a H_j + v_j' \p_v H_j \Big)
}
 holds in the strong sense in $H^1(\bR)$,
and that the second component of \eqref{eq:ls-g-eq}, which can be expanded as 
\EQ{\label{eq:g-2nd}
\p_t \dot g &= \p_x^2 g - U''(H(\vec a, \vec v)) g  - \Big( U'(H(\vec a, \vec v) + g) - U'(H(\vec a, \vec v)) - U''(H(\vec a, \vec v)) g\Big) \\
&\quad - \Big( U' (H(\vec a, \vec v)) - \sum_{j=1}^n (-1)^j U'(H_j) \Big) - \sum_{j =1}^n (-1)^j ( a_j' -v_j) \p_a \dot H_j + v_j' \p_v \dot H_j)
}
%\begin{equation}
%\label{eq:g-2nd}
%\partial_t \dot g = \partial_x^2 g - U'(H(\vec a, \vec v) + g)+ \sum_{j=1}^n (-1)^j U'(H_j)  - \sum_{j =1}^n (-1)^j ( a_j' -v_j) \p_a \dot H_j + v_j' \p_v \dot H_j)
%\end{equation}
holds in the strong sense in $L^2(\bR)$. Using the above we also record the expression
\EQ{
\p_t \dot \phi &= \p_x^2 g - U''(H(\vec a, \vec v)) g  - \Big( U'(H(\vec a, \vec v) + g) - U'(H(\vec a, \vec v)) - U''(H(\vec a, \vec v)) g\Big) \\
&\quad - \Big( U' (H(\vec a, \vec v)) - \sum_{j=1}^n (-1)^j U'(H_j) \Big) + \sum_{j =1}^n (-1)^j v_j \p_a \dot H_j .
}
From \eqref{eq:p-def} and the Leibniz rule, we obtain
\EQ{
\label{eq:pk'-1}
p_k' &= (-1)^k a_k'\la \partial_x^2 H_k, \dot \phi\ra +(-1)^k v_k'\la \p_v \partial_x H_k, \dot g\ra - \la \partial_t \chi_k\partial_x g, \,\dot \phi\ra \\
&- \Big\la \chi_k\partial_x\Big(\dot g - \sum_{j =1}^n (-1)^j \Big((a_j' - v_j) \p_a H_j + v_j' \p_v H_j \Big), \,\dot \phi\Big\ra \\
%&- \Big\la (-1)^k\partial_x H_k + \chi_k\partial_x g, \,  \p_t \dot H(\vec a, \vec v )\Big\ra \\
& - \Big \la (-1)^k\partial_x H_k, \, \p_x^2 g - U''(H(\vec a, \vec v)) g \Big\ra -  \Big \la \chi_k\partial_x g,\, \p_x^2 g \Big\ra \\
&+ \Big \la (-1)^k\partial_x H_k, \,\Big( U'(H(\vec a, \vec v) + g) -  \sum_{j=1}^n (-1)^j U'(H_j) - U''(H(\vec a, \vec v)) g\Big)\Big \ra  \\
&+\Big \la  \chi_k\partial_x g,\,\Big( U'(H(\vec a, \vec v) + g) - \sum_{j=1}^n (-1)^j U'(H_j) \Big) \Big \ra  \\
%& + \Big \la (-1)^k \partial_x H_k, \, \Big( U' (H(\vec a, \vec v)) - \sum_{j=1}^n (-1)^j U'(H_j) \Big) \Big \ra  \\
%&+ \Big \la  \chi_k\partial_x g, \, \Big( - \sum_{j=1}^n (-1)^j U'(H_j) \Big) \Big \ra  \\
& +  \Big \la (-1)^k \partial_x H_k, \,\sum_{j =1}^n (-1)^j v_j \p_a \dot H_j \Big \ra+ \Big \la  \chi_k\partial_x g, \,\sum_{j =1}^n (-1)^j v_j \p_a \dot H_j \Big \ra .
%&- \Big\la (-1)^k\partial_x H_k + \chi_k\partial_x g,  \\
%&\qquad\qquad\partial_x^2 g - U'(H(\vec a, \vec v) + g)+ \sum_{j=1}^n (-1)^j U'(H_j) + \sum_{j =1}^n(-1)^j v_j \p_a \dot H_j)\Big\ra.
}
In the second line we expand a few of the terms as follows. First, we integrate by parts to obtain
\EQ{
- \la \chi_k\partial_x\dot g, \,  \dot \phi  \ra &= -(-1)^kv_k\la \chi_k\dot g, \, \p_x^2 H_k\ra + \la \chi_k\partial_x\dot g, \, \p_x( \dot H(\vec a, \vec v)- \dot H_k) \ra + \frac12 \la \p_x\chi_k\dot g, \,  \dot g  \ra \\
& \quad  + \la \p_x \chi_k \dot g, \, \dot H(\vec a, \vec v) \ra .
}
Then note that 
\EQ{
(-1)^k(a_k' - v_k) \la \chi_k\partial_x \p_a H_k, \, \dot \phi \ra &= - (-1)^k a_k' \la \chi_k \p_x^2 H_k, \, \dot \phi \ra +(-1)^k v_k \la \chi_k \p_x^2 H_k ,\,  \dot g \ra \\
&+ \frac{1}{2}v_k^2 \la \p_x\chi_k \p_x H_k ,\,  \p_x H_k \ra +  (-1)^k v_k \la \chi_k \p_x^2 H_k, \, \dot H( \vec a, \vec v) - \dot H_k \ra ,
}
where in the last line we used that $\la \p_x^2 H_k, \, \p_x H_k \ra = 0$.  
%We observe that an integration by parts yields, 
%\EQ{
%\la \chi_k(t)\partial_x \dot g(t), \dot g(t)\ra &= -\int_{-\infty}^\infty \frac 12 \partial_x \chi_k(t)(\dot g(t))^2\ud x,
%}
Thus the first line added to the second line reduces to 
\begin{multline} 
 (-1)^k a_k' \la(1- \chi_k) \p_x^2 H_k, \, \dot \phi \ra + (-1)^k v_k'\la \p_v \partial_x H_k, \dot g\ra - \la \partial_t \chi_k\partial_x g, \,\dot \phi\ra +  \la \p_x \chi_k \dot g, \, \dot H(\vec a, \vec v) \ra \\
+\la \chi_k\dot g, \, \p_x( \dot H(\vec a, \vec v)- \dot H_k) \ra + \frac12 \la \p_x \chi_k\dot g, \,  \dot g  \ra + v_k^2 \la (1-\chi_k) \p_x^2 H_k ,\,  \p_x H_k \ra \\
+  (-1)^k v_k \la \chi_k \p_x^2 H_k, \, \dot H( \vec a, \vec v) - \dot H_k \ra  + \sum_{j \neq k} (-1)^j \Big((a_j' - v_j) \p_x\p_a H_j + v_j' \p_x \p_v H_j \Big), \,\dot \phi\Big\ra .
\end{multline} 
%\begin{equation}
%\label{eq:pk'-1}
%\begin{aligned}
%p_k'(t) &= (-1)^k a_k'(t)\la \partial_x^2 H_k(t), \dot g(t)\ra 
%- \la \partial_t \chi_k(t)\partial_x g(t), \dot g(t)\ra \\
%&- \Big\la \chi_k(t)\partial_x\Big(\dot g(t) + \sum_{j=1}^n (-1)^j a_j'(t)\partial_x H_j(t)\Big), \dot g(t)\Big\ra \\
%&- \Big\la (-1)^k\partial_x H_k(t) + \chi_k(t)\partial_x g(t), \\
%&\qquad\qquad\partial_x^2 g(t) - U'(H(\vec a(t)) + g(t))+ \sum_{j=1}^n (-1)^j U'(H_j(t))\Big\ra.
%\end{aligned}
%\end{equation}
Next, consider the third line of~\eqref{eq:pk'-1}. We use \eqref{eq:Lc-ker} to obtain
\EQ{
 - \Big \la (-1)^k\partial_x H_k,\, \p_x^2 g - U''(H(\vec a, \vec v)) g \Big\ra &= -(-1)^k ( 1- \gamma_{v_{k}}^2)\big \la U''( H_k) \p_x H_k, \, g \big \ra \\
 &\quad - (-1)^k  \big \la (U''(H( \vec a, \vec v) - U''(H_k)) \p_x H_k, \, g \big \ra 
}
and \begin{equation}
\begin{aligned}
%\la \chi_k(t)\partial_x \dot g(t), \dot g(t)\ra &= -\int_{-\infty}^\infty \frac 12 \partial_x \chi_k(t)(\dot g(t))^2\ud x, \\
\la \chi_k(t)\partial_x g(t), \partial_x^2 g(t)\ra &= -\int_{-\infty}^\infty \frac 12 \partial_x \chi_k(t)(\partial_x g(t))^2\ud x.
\end{aligned}
\end{equation}
%and that \eqref{eq:Lc-ker} implies
%\begin{equation}
%\la \partial_x H_k(t), \partial_x^2 g(t)\ra = \gamma_{v_k}^2 \int_{-\infty}^\infty \partial_x H_k(t) U''(H_k(t))g(t)\ud x.
%\end{equation}
Inserting these relations into \eqref{eq:pk'-1} and rearranging
the terms, we obtain \eqref{eq:pk'-id}.

Consider now the general case (without the additional regularity assumptions). By Proposition~\ref{prop:cauchy},
there exists a sequence of solutions $\bs\phi_m$ satisfying \eqref{eq:phi-conv-en} and the additional regularity assumptions stated above.
Let $\vec p_m$ be the corresponding localized momentum.
By \eqref{eq:mod-basic-conv} and \eqref{eq:pk'-id},
the sequence $\big(\vec p_m\big)_m$ converges in $C^1$.

It remains to prove \eqref{eq:pk'}. In the computation below, we call a term ``negligible''
if its absolute value is smaller than the right hand side of \eqref{eq:pk'}.

By Proposition~\ref{prop:prop-H} and the definition of $\chi_k$, we have
\begin{equation}
\begin{aligned}
\int_{-\infty}^\infty \big((1 - \chi_k(t))\partial_x^2 H_k(t)\big)^2\ud x &\leq
2\int_{\frac 13 y_{\min}(t)}^{\infty}(\partial_x^2 H)^2\ud x \lesssim \eee^{-\frac 23 y_{\min}(t)} \lesssim \rho(t)^\frac 23.
\end{aligned}
\end{equation}
By a similar computation, for all $j \neq k$ we have
\begin{equation}
\begin{aligned}
\int_{-\infty}^\infty \big(\chi_k(t)\partial_x^2 H_j(t)\big)^2\ud x \lesssim \rho(t)^\frac 23.
\end{aligned}
\end{equation}
Applying \eqref{eq:g-coer}, \eqref{eq:ak'-est}, the bound~$\| \dot \phi(t) \|_{L^2} \lesssim \sqrt{\rho(t)}$ and the Cauchy-Schwarz inequality,
we obtain that the first line and first term of the second line of \eqref{eq:pk'-id} are at most of order $\rho(t)^\frac 43$,
hence negligible. Next, we observe that $\|\partial_x \chi_k\|_{L^\infty} \lesssim y_{\min}^{-1}
\lesssim ({-}\log \rho(t))^{-1}$. % hence the first integral of the third line of \eqref{eq:pk'-id}
%is negligible. 
And, the Chain Rule and \eqref{eq:ak'-est} yield $\|\partial_t \chi_k\|_{L^\infty}
\lesssim ({-}\log \rho(t))^{-1}\sqrt{\rho(t)}$, hence the last term of the second line along with the third line of~\eqref{eq:pk'-id} is 
% the second integral is 
 negligible as well. The fourth line is seen to be negligible after an application of Cauchy-Schwarz and Lemma~\ref{lem:exp-cross-term}. Similarly, the fifth line is negligible. 

By Lemma~\ref{lem:sizeDEp} and the Cauchy-Schwarz inequality, the fourth line of \eqref{eq:pk'-id}
differs by a term of order $\sqrt{y_{\min}(t)}\eee^{-\frac 32 y_{\min}(t)} \ll \rho(t)({-}\log \rho(t))^{-1}$ from
\begin{equation}
\label{eq:pk'-4-alt}
\int_{-\infty}^\infty \chi_k(t) \partial_x g(t)\big(U'(H(\vec a(t), \vec v(t)) + g(t)) - U'(H(\vec a(t), \vec v(t)))\big)\ud x.
\end{equation}
We now transform the last integral of the right hand side of \eqref{eq:pk'-id}. By \eqref{eq:size-diff-pot}, we have
\begin{equation}
\int_{-\infty}^\infty \partial_x H_k(t) U''(H_k(t)) g(t) \ud x \sim \int_{-\infty}^\infty \partial_x H_k(t) U''(H(\vec a(t), \vec v(t))) g(t) \ud x.
\end{equation}
Recalling the definition of $F_k(\vec a, \vec v)$, see \eqref{eq:F_k-def}, we thus obtain
that the last integral of the right hand side of \eqref{eq:pk'-id} differs by a negligible term from
\begin{equation}
\label{eq:line-56}
\begin{aligned}
F_k(\vec a(t), \vec v(t)) + (-1)^k\int_{-\infty}^\infty \partial_x H_k(t)\big(&U'(H(\vec a(t), \vec v(t)) + g(t))
 \\ &- U'(H(\vec a(t), \vec v(t))) - U''(H(\vec a(t), \vec v(t)))g(t)\big)\ud x.
 \end{aligned}
\end{equation}
Applying the Taylor formula pointwise and using the fact that
$$\|\chi_k(t)\partial_x H(\vec a(t), \vec v(t)) - (-1)^k \partial_x H_k(t)\|_{L^\infty} \lesssim \eee^{-\frac 13 y_{\min}(t)},$$
we see that up to negligible terms, in \eqref{eq:line-56} we can replace
$(-1)^k\partial_x H_k(t)$ by $\chi_k(t)\partial_x H(\vec a(t), \vec v(t))$.
Recalling \eqref{eq:pk'-4-alt}, we thus obtain
\begin{equation}
\begin{aligned}
&p_k'(t) \sim F_k(\vec a(t), \vec v(t)) \\&+ \int_{-\infty}^\infty \chi_k(t) \partial_x \big(U(H(\vec a(t), \vec v(t)) + g(t)) - U(H(\vec a(t), \vec v(t))) - U'(H(\vec a(t), \vec v(t)))g(t)\big)\ud x,
\end{aligned}
\end{equation}
and an integration by parts shows that the second line is negligible. Finally, we recall from~\eqref{eq:F_k-v-0} and~\eqref{eq:Fkv-Fk0} that $F_k(\vec a,  \vec v)$ and $F_k(\vec a)$ differ by a negligible term. 
\end{proof}
\section{Characterizations of kink clusters}
\label{sec:equiv-def}
\subsection{Kink clusters approach multikink configurations}
\label{ssec:close-to-H}
We now give a proof of Proposition~\ref{prop:close-to-H}.
%, which states that kink clusters in the sense of Definition~\ref{def:cluster} can equivalently be characterised as solutions
%whose distance to the set of multikink configurations decays to zero in time.
The essential ingredient is the following lemma
yielding the strong convergence of minimising sequences,
similar to the results in \cite[Appendix A]{Abdon22p1}.
\begin{lemma}
\label{lem:min-conv}
\begin{enumerate}[(i)]
\item\label{it:min-conv-i}
If the sequences $x_m \in \bR$ and $\phi_m: ({-}\infty, x_m] \to \bR$ satisfy
\begin{equation}
\begin{gathered}
\lim_{x \to -\infty} \phi_m(x) = -1\quad\text{for all }m, \qquad
\lim_{m\to\infty} E_p(\phi_m; -\infty, x_m) = 0,
\end{gathered}
\end{equation}
then $\lim_{m\to\infty}\|\phi_m + 1\|_{H^1({-}\infty, x_m)} = 0$.
\item\label{it:min-conv-ii}
If the sequences $x_m' \in \bR$ and $\phi_m: [x_m', \infty) \to \bR$ satisfy
\begin{equation}
\begin{gathered}
\lim_{x \to \infty} \phi_m(x) = 1\quad\text{for all }m, \qquad
\lim_{m\to\infty} E_p(\phi_m; x_m', \infty) = 0,
\end{gathered}
\end{equation}
then $\lim_{m\to\infty}\|\phi_m - 1\|_{H^1(x_m', \infty)} = 0$.
\item\label{it:min-conv-iii}
If the sequence $\phi_m : \bR \to \bR$ satisfies
\begin{equation}
\begin{gathered}
\phi_m(0) = 0, \quad\lim_{x \to -\infty} \phi_m(x) = -1, \quad \lim_{x \to \infty} \phi_m(x) = 1\quad\text{for all }m, \\ \lim_{m \to \infty}E_p(\phi_m) = M,
\end{gathered}
\end{equation}
then $\lim_{m\to \infty}\|\phi_m - H\|_{H^1} = 0$.
\item\label{it:min-conv-iv}
If the sequences $x_m < 0$, $x_m' > 0$ and $\phi_m : [x_m, x_m'] \to \bR$ satisfy
\begin{equation}
\begin{gathered}
\phi_m(0) = 0, \quad \lim_{m \to -\infty} \phi_m(x_m) = -1, \quad \lim_{m \to \infty}\phi_m(x_m') = 1\quad\text{for all }m, \\ \lim_{m \to \infty}E_p(\phi_m; x_m, x_m') = M,
\end{gathered}
\end{equation}
then $\lim_{m \to \infty} x_m = -\infty$, $\lim_{m\to \infty}x_m' = \infty$ and
$\lim_{m \to \infty}\|\phi_m - H\|_{H^1(x_m, x_m')} = 0$.
\end{enumerate}
\end{lemma}
\begin{proof}
It follows from \eqref{eq:bogom} and \eqref{eq:bogom-2} that
\begin{equation}
\lim_{m\to \infty}\sup_{x \leq x_m}\int_{-1}^{\phi_m(x)}\sqrt{2U(y)}\ud y = 0,
\end{equation}
hence
\begin{equation}
\lim_{m\to \infty}\sup_{x \leq x_m}|\phi_m(x) + 1| = 0.
\end{equation}
Since $U''(-1) = 1$, for all $m$ large enough we have
\begin{equation}
U(\phi_m(x)) \geq \frac 14(\phi_m(x) + 1)^2\qquad\text{for all }x \leq x_m,
\end{equation}
which proves part \ref{it:min-conv-i}.

Part \ref{it:min-conv-ii} follows from \ref{it:min-conv-i} by symmetry.

In part \ref{it:min-conv-iii},
it suffices to prove that the conclusion holds for a subsequence of any subsequence.
We can thus assume that $\partial_x \phi_m \wto \partial_x \phi_0$ in $L^2(\bR)$
and $\phi_m \to \phi_0$ uniformly on every bounded interval.

From \eqref{eq:bogom}, we have
\begin{equation}
\label{eq:bogom-conv}
\lim_{m\to \infty} \int_{-\infty}^\infty\big(\partial_x \phi_m - \sqrt{2U(\phi_m)}\big)^2\ud x = 0.
\end{equation}
Restricting to bounded intervals and passing to the limit, we obtain $\partial_x \phi_0(x) = \sqrt{2U(\phi_0(x))}$ for all $x \in \bR$, hence $\phi_0 = H$.
Using again \eqref{eq:bogom-conv} restricted to bounded intervals,
we obtain $\lim_{m\to \infty}\|\phi_m - H\|_{H^1(-R, R)} = 0$ for every $R > 0$,
in particular $\lim_{m \to \infty} |E_p(\phi_m; -R, R) - E_p(H; -R, R)| = 0$.
Hence, there exists a sequence $R_m \to \infty$ such that
\begin{equation}
\label{eq:Rm-conv}
\lim_{m\to \infty}\big( \|\phi_m - H\|_{H^1(-R_m, R_m)} + |E_p(\phi_m; -R_m, R_m) - E_p(H; -R_m, R_m)|\big) = 0.
\end{equation}
Since $E_p(H; -R_m, R_m) \to M$ and $E_p(\phi_m) \to M$, we obtain
\begin{equation}
\lim_{m\to \infty}\big( E_p(\phi_m; -\infty, -R_m) + E_p(\phi_m; R_m, \infty)\big) = 0.
\end{equation}
Applying parts \ref{it:min-conv-i} and \ref{it:min-conv-ii}, we obtain
\begin{equation}
\lim_{m\to\infty}\big(\|\phi_m +1\|_{H^1(-\infty, -R_m)} + \|\phi_m - 1\|_{H^1(R_m, \infty)}\big) = 0.
\end{equation}
It is clear that $\lim_{m\to\infty}\big(\|H +1\|_{H^1(-\infty, -R_m)} + \|H - 1\|_{H^1(R_m, \infty)}\big) = 0$, thus \eqref{eq:Rm-conv} yields the conclusion.

In order to prove part \ref{it:min-conv-iv}, we define a new sequence $\wt \phi_m: \bR \to \bR$
by the formula, similar to \eqref{eq:weak-conv-cutoff},
\begin{equation}
\wt\phi_m(x) \coloneqq \begin{cases}
-1 & \text{for all }x \leq x_m - 1, \\
-(x_m-x) + (1-x_m+x)\phi_m(x_m)& \text{for all }x \in [x_m - 1, x_m], \\
\phi_m(x) &\text{for all }x \in [x_m, x_m'], \\
(x - x_m') + (1-x+x_m')\phi_m(x_m') & \text{for all }x \in [x_m', x_m' + 1], \\
1 & \text{for all }x \geq x_m' + 1.
\end{cases}
\end{equation}
Then $\wt \phi_m$ satisfies the assumptions of part \ref{it:min-conv-iii}
and we obtain $\lim_{m\to \infty} \|\wt \phi_m - H\|_{H^1} = 0$.
In particular, $\wt\phi_m \to H$ uniformly, hence $H(x_m)\to -1$ and $H(x_m') \to 1$,
implying $x_m \to -\infty$ and $x_m' \to \infty$.
\end{proof}
\begin{proof}[Proof of Proposition~\ref{prop:close-to-H}]
If $\bs \phi$ is a solution of \eqref{eq:csf} satisfying $\lim_{t \to \infty} \bfd( \bs \phi(t)) = 0$,
then Lemma~\ref{lem:interactions} yields $E(\bs \phi) = nM$.
Moreover, if $x_0(t), x_1(t), \ldots, x_n(t)$ are any functions such that
$$
\lim_{t\to\infty}\big(a_k(t) - x_{k-1}(t)\big) = \lim_{t\to\infty}\big(x_k(t) - a_{k}(t)\big) = \infty \qquad\text{for all }k \in \{1, \ldots, n\}
$$
then we have $\lim_{t\to\infty}\phi(t, x_k(t)) = (-1)^k$ for all $k \in \{0, 1, \ldots, n\}$.

In the opposite direction, assume that $\bs \phi$ is a kink cluster according to Definition~\ref{def:cluster}.
For all $k \in \{1, \ldots, n\}$ and $t$ sufficiently large, let $a_k(t) \in (x_{k-1}(t), x_k(t))$
be such that $\phi(t, a_k(t)) = 0$. From \eqref{eq:bogom} and \eqref{eq:bogom-2}, we obtain
\begin{equation}
\liminf_{t \to \infty} E_p(\phi(t); x_{k-1}(t), x_k(t)) \geq M, \qquad\text{for all }k \in \{1, \ldots, n\}.
\end{equation}
Thus, the condition $E(\bs \phi) \leq nM$ implies
\begin{gather}
\lim_{t \to \infty} \|\partial_t \phi(t)\|_{L^2} = 0, \\
\lim_{t \to \infty} E_p(\phi(t); x_{k-1}(t), x_k(t)) = M, \qquad\text{for all }k \in \{1, \ldots, n\}, \\
\lim_{t \to \infty} \big(E_p(\phi(t); -\infty, x_0(t)) + E_p(\phi(t); x_n(t), \infty)\big) = 0.
\end{gather}
Applying Lemma~\ref{lem:min-conv}, we obtain
\begin{gather}
\lim_{t \to \infty}(a_k(t) - x_{k-1}(t)) = \lim_{t \to \infty}(x_k(t) - a_k(t)) = \infty \qquad\text{for all }k \in \{1, \ldots, n\}
\end{gather}
and
\begin{equation}
\begin{aligned}
&\lim_{t\to \infty}\Big(\|\phi(t)\|_{H^1(-\infty, x_0(t))} + \|\phi(t)\|_{H^1(x_n(t),\infty)} \\
&\qquad+ \sum_{k=1}^n \|\phi(t) - (-1)^k H(\cdot - a_k(t))\|_{H^1(x_{k-1}(t), x_k(t))}\Big) = 0,
\end{aligned}
\end{equation}
thus
\begin{equation}
\lim_{t\to \infty}\big(\|\bs\phi(t) - \bs H(\vec a(t))\|_\cE^2 + \rho(\vec a(t))\big) = 0.\qedhere
\end{equation}
In particular, fixing $\vec v = 0$ in \eqref{eq:bfd-def}, we obtain $\lim_{t\to\infty}\bfd(\bs\phi(t)) = 0$.

%\vspace{-3em}
\end{proof}
\subsection{Asymptotically static solutions are kink clusters}
\label{ssec:asym-stat}
The present section is devoted to a proof of Proposition~\ref{prop:asym-stat}, which is inspired by some of the arguments in \cite{Cote15, JK}.
Recall that we restrict our attention to the case $U(\phi) = \frac 18(1-\phi^2)^2$,
so that $\pm \bs 1$ and $\pm \bs H(\cdot -a )$ are the only static states of \eqref{eq:csf}.
\begin{lemma}
\label{lem:kink-or-antikink}
Let $\bs\phi$ be a finite-energy solution of \eqref{eq:csf} such that $\lim_{t\to\infty}\|\partial_t \phi(t)\|_{L^2}^2 = 0$.
If $t_m \to \infty$ and $(a_m)_m$ is a sequence of real numbers such that
$\bs \phi(t_m, \cdot + a_m) \wto \bs \phi_0$ as $m \to \infty$,
then $\bs \phi_0$ is a static state.
\end{lemma}
\begin{proof}
Let $\wt{\bs \phi}$ be the solution of \eqref{eq:csf} for the initial data
$\wt{\bs\phi}(0) = \bs\phi_0$.
By Proposition~\ref{prop:cauchy} \ref{it:cauchy-weak}, for all $s \in [0, 1]$
we have $\partial_t \phi(t_m + s, \cdot + a_m) \wto \partial_t \wt\phi(s)$
in $L^2(\bR)$, thus for all $s \in [0, 1]$ we have
\begin{equation}
\|\partial_t \wt\phi(s)\|_{L^2} \leq \liminf_{m\to\infty}\|\partial_t \phi(t_m + s, \cdot + a_m)\|_{L^2} = \liminf_{m\to\infty}\|\partial_t \phi(t_m + s)\|_{L^2} = 0,
\end{equation}
implying that $\bs\phi_0$ is a static state.
\end{proof}

\begin{lemma}
\label{lem:asym-stat-Linf-conv}
Let $\bs\phi$ be a finite-energy solution of \eqref{eq:csf} such that
$\lim_{t\to\infty}\|\partial_t \phi(t)\|_{L^2}^2 = 0$
and let $t_m \to \infty$.
After extraction of a subsequence, there exist
$n \in \{0, 1, \ldots\}$, $\iota \in \{{-}1, 1\}$ and
$\vec a_{m} \in \bR^n$ for $m \in \{1, 2, \ldots\}$ such that
\begin{equation}
\label{eq:asym-stat-Linf-conv}
\lim_{m\to\infty}\big(\|\phi(t_m) - \iota H(\vec a_m)\|_{L^\infty}^2
+ \rho(\vec a_m)\big) = 0.
\end{equation}
\end{lemma}
\begin{proof}
Let $n$ be the maximal natural number such that there exist
a subsequence of $t_m$, and sequences $a_{m,1}, \ldots, a_{m,n}$
satisfying
$\lim_{m\to\infty}(a_{m,k+1} - a_{m,k}) = \infty$ for all $k \in \{1, \ldots, n-1\}$ and
\begin{equation}
\bs\phi(t_m, \cdot + a_{m, k}) \wto \bs\phi_k\neq \pm\bs 1\qquad\text{for all }k.
\end{equation}
By Lemma~\ref{lem:kink-or-antikink}, each $\bs \phi_k$ is a kink or an antikink. Upon adjusting the sequence $a_{m, k}$, we can assume
$\bs \phi_k = \pm \bs H$ for all $k$, and by an appropriate choice of $\iota$
we can reduce to $\bs\phi_1 = {-}\bs H$.

We claim that $\bs\phi_k = (-1)^k \bs H$ for all $k \in \{1, \ldots, n\}$.
This is true for $k = 1$. Suppose $k$ is such that $\bs \phi_k = \bs \phi_{k+1}$. Then there exists a sequence $x_m$ such that
$\lim_{m\to\infty}(x_m - a_{m, k}) = \lim_{m\to\infty}(a_{m, k+1} - x_m) = \infty$ and $\phi(t_m, x_m) = 0$ for all $m$.
Extracting a subsequence, we can assume that $\phi(t_m, \cdot + x_m)$
converges locally uniformly. The limit cannot be a vacuum,
contradicting the maximality of $n$ and finishing the proof of the claim.

Suppose there exists a sequence $x_m$ such that
\begin{equation}
\limsup_{m\to \infty} |\phi(t_m, x_m) - H(\vec a_m; x_m)| > 0.
\end{equation}
We then have $|a_{m, k} - x_m| \to \infty$,
thus $\bs \phi(t_m, \cdot + x_m)$ has (after extraction of a subsequence)
a non-constant weak limit, contradicting the maximality of $n$.
\end{proof}

The last auxiliary result which we need is the following
description of local in time stability of multikink configurations.
\begin{lemma}\label{lem:mkink-stab}
Let $n \in \{0, 1, \ldots\}$ and $\vec a_m \in \bR^n$ for $m \in \{1, 2, \ldots\}$ be such that $\lim_{m\to \infty}\rho(\vec a_m) = 0$.
Let $\bs g_{m, 0} \in \cE$ for $m \in \{1, 2, \ldots\}$
and let $\bs\phi_m: [-1, 1] \to \cE_{1, (-1)^n}$
be the solution of \eqref{eq:csf} for the initial data $\bs\phi_m(0) = \bs H(\vec a_m) + \bs g_{m, 0}$.
\begin{enumerate}[(i)]
\item \label{it:strong-mkink-stab}
If $\lim_{m\to\infty}\|\bs g_{m, 0}\|_\cE = 0$, then
\begin{equation}
\label{eq:loc-stab-1}
\lim_{m\to\infty}\sup_{t \in [-1, 1]}\|\bs\phi_m(t) - \bs H(\vec a_m)\|_{\cE} = 0.
\end{equation}
\item \label{it:weak-mkink-stab}
If $\limsup_{m\to\infty}\|g_{m, 0}\|_{H^1} < \infty$,
$\lim_{m\to\infty}\|\dot g_{m, 0}\|_{L^2} = 0$
and $\lim_{m\to\infty}\|g_{m, 0}\|_{L^\infty} = 0$, then
%\begin{equation}
%\label{eq:loc-stab-2}
%\lim_{m\to\infty}\sup_{t \in [-1, 1]}\|\phi_m(t) - H(\vec a_m)\|_{L^\infty} = 0
%\end{equation}
%and
\begin{equation}
\label{eq:loc-stab-3}
\lim_{m\to\infty}\sup_{t \in [-1, 1]}\|\bs\phi_m(t) - \bs H(\vec a_m)
-\bs g_{m, \lin}(t)\|_{\cE} = 0,
\end{equation}
where $\bs g_{m, \lin}(t)$ is the solution of the free linear Klein-Gordon
equation for the initial data $\bs g_{m, \lin}(0) = \bs g_{m, 0}$.
\end{enumerate}
\end{lemma}
\begin{proof}
Part \ref{it:strong-mkink-stab} follows from part \ref{it:weak-mkink-stab} and the fact that $\|\bs g_{m, \lin}(t)\|_\cE = \|\bs g_{m, 0}\|_\cE$ for all $t$.

%Suppose \eqref{eq:loc-stab-2} is false,
%so that, after taking a subsequence, there exist $\epsilon > 0$ and
%$(t_m, x_m) \in [-1, 1]\times \bR$ such that $t_m \to t_0 \in [-1, 1]$ and
%\begin{equation}
%|\phi_m(t_m, x_m) - H(\vec a_m, x_m)| \geq \epsilon.
%\end{equation}
%The assumptions of part \ref{it:weak-mkink-stab} imply
%$\bs \phi_m(0, \cdot + x_m) \wto \bs H(\vec a_m, \cdot + x_m)$,
%thus 

%Consider $\bs\psi_{m, 0}(x) \coloneqq \bs\phi_m(t_m - t_0, x + x_m)$,
%and let $\bs\psi_m$ be the solution of \eqref{eq:csf} such that
%$\bs\psi_m(0) = \bs\psi_{m, 0}$. 

%and \eqref{eq:loc-stab-2} is obtained by applying the weak continuity
%property to space translations of the sequence.

%Using \eqref{eq:loc-stab-2} and an energy estimate, \eqref{eq:loc-stab-3}
%follows.

In order to prove \eqref{eq:loc-stab-3}, we set $\bs g_m(t)\coloneqq  \bs\phi_m(t) - \bs H(\vec a_m)$ and $\bs h_m(t)\coloneqq \bs g_m(t)  -\bs g_{m, \lin}(t)$. Then $\bs h_m(t) = (h_m(t), \partial_t h_m(t))$ solves the equation
$\partial_t^2 h_m - \partial_x^2 h_m + h_m = f_m$, where
\begin{equation}
\begin{aligned}
\label{eq:fm-gronwall}
f_m(t) \coloneqq &- \big( U'(H(\vec a_m) + g_m(t)) - U'(H(\vec a_m)) - U''(H(\vec a_m)) g_m(t) \big)  \\
& - \Big( U'(H(\vec a_m)) - \sum_{k =1}^{n} (-1)^k U'( H( \cdot - a_{k, m})) \Big) \\
& - \big( U''(H(\vec a_m)) -1\big) g_m(t).
\end{aligned}
\end{equation}
Since $\bs h_m(0) = 0$, the standard energy estimate yields
\begin{equation}
\|\bs h_m(t)\|_\cE \leq \bigg|\int_0^t \|f_m(s)\|_{L^2}\ud s \bigg|\qquad\text{for all }t \in [-1, 1].
\end{equation}
By Gronwall's inequality, it suffices to verify that
\begin{equation}
\label{eq:fm-gronwall-bd}
\|f_m(t)\|_{L^2} \leq C\|\bs h_m(t)\|_{\cE} + \epsilon_m,\qquad\text{
with }\lim_{m\to\infty}\epsilon_m = 0.
\end{equation}
The first line of \eqref{eq:fm-gronwall} satisfies
\begin{align} 
&\| U'(H(\vec a_m) + g_m(t)) - U'(H(\vec a_m)) - U''(H(\vec a_m)) g_m(t) \|_{L^2} \\
&\qquad\lesssim \min(\|g_m(t)\|_{L^2}, \| g_m(t)^2 \|_{L^2}).
\end{align}
If $\|\bs h_m(t)\|_{\cE} \geq 1$, then
$$\|g_m(t)\|_{L^2} \leq \|h_m(t)\|_{L^2} + \|g_{m,\lin}(t)\|_{L^2} \leq \|\bs h_m(t)\|_\cE + \|\bs g_{m, 0}\|_\cE \lesssim \|\bs h_m(t)\|_\cE.$$
If $\|\bs h_m(t)\|_\cE \leq 1$, then
\begin{align}
\|g_m(t)^2\|_{L^2} &\leq \| g_m(t) \|_{L^\infty} \| g_m(t) \|_{L^2} \\
&\leq
(\|\bs h_m(t)\|_\cE + \|g_{m,\lin}(t)\|_{L^\infty})
(\|\bs h_m(t)\|_\cE + \|\bs g_{m,0}\|_\cE) \\
& \lesssim \|\bs h_m(t)\|_\cE + \|g_{m,\lin}(t)\|_{L^\infty}.
\end{align} 
The term on the second line of \eqref{eq:fm-gronwall} is equal to $-\textrm{D} E_p( H(\vec a_m))$, thus, using \eqref{eq:sizeDEp}, satisfies
\begin{align} 
\Big\| U'(H(\vec a_m)) - \sum_{k =1}^{n} (-1)^k U'( H( \cdot - a_{k, m})) \Big\|_{L^2} \lesssim \sqrt{{-}\log \rho(\vec a_m)} \rho(\vec a_m). 
\end{align} 
Finally, the third line of \eqref{eq:fm-gronwall} satisfies
\begin{align} 
\big\| \big( U''(H(\vec a_m)) -1\big) g_m(t) \big\|_{L^2} \leq \|U''(H(\vec a_m)) -1 \|_{L^2} \|g_m(t) \|_{L^{\infty}} \lesssim \|g_m(t) \|_{L^{\infty}}.
\end{align}
Gathering the estimates above and using the fact that,
by explicit formulas for the free Klein-Gordon equation,
we have
$\lim_{m\to\infty}\sup_{t\in[-1, 1]}\|g_{m,\lin}(t)\|_{L^\infty} = 0$, 
we obtain \eqref{eq:fm-gronwall-bd}.
\end{proof}
\begin{proof}[Proof of Proposition~\ref{prop:asym-stat}]
Without loss of generality, assume
\begin{equation}
\label{eq:asym-stat-class}
\lim_{x \to -\infty}\phi(t, x) = 1\qquad\text{for all }t.
\end{equation}
Suppose that either $n \coloneqq M^{-1}E(\bs \phi) \notin \bN$, or
there exists a sequence $t_m \to \infty$ such that
\begin{equation}
\label{eq:asym-stat-contr}
\liminf_{m \to \infty}\bfd(\bs \phi(t_m)) > 0,
\end{equation}
where $\bfd$ defined by \eqref{eq:bfd-def} is the distance
to the set of $n$-kink configurations.
In both cases, for any $n' \in \bN$ and $\vec a_m \in \bR^{n'}$ we have
\begin{equation}
\label{eq:asym-stat-g-lbd}
\liminf_{m \to \infty}\big(\|\bs \phi(t_m) - \bs H(\vec a_m)\|_\cE^2
+ \rho(\vec a_m)\big) > 0.
\end{equation}
By Lemma~\ref{lem:asym-stat-Linf-conv}, after extracting a subsequence,
we can assume that \eqref{eq:asym-stat-Linf-conv} holds
for some $n' \in \bN$ and $\vec a_m \in \bR^{n'}$.
We have $\iota = 1$ due to \eqref{eq:asym-stat-class}.
We decompose
\begin{equation}
\bs \phi(t_m + s) = \bs H(\vec a_m) + \bs g_m(s), \qquad s \in [-1, 1].
\end{equation}
We claim that
\begin{equation}
\label{eq:g-large-en}
\liminf_{m\to\infty}\inf_{s \in [-1, 1]}\|\bs g_m(s)\|_\cE > 0.
\end{equation}
Otherwise, by Lemma~\ref{lem:mkink-stab} \ref{it:strong-mkink-stab},
we would have $\lim_{m\to\infty}\|\bs g_m(0)\|_\cE = 0$,
contradicting \eqref{eq:asym-stat-g-lbd}.

Let $\bs g_{m, \lin}: [-1, 1] \to \cE$ be the solution of the free
linear Klein-Gordon equation for the initial data $\bs g_{m, \lin}(0)
= \bs g_m(0)$. By Lemma~\ref{lem:mkink-stab} \ref{it:weak-mkink-stab},
we have
\begin{equation}
\lim_{m\to\infty}\sup_{s\in [-1, 1]}\|\bs g_m(s) - \bs g_{m, \lin}(s)\|_\cE = 0,
\end{equation}
thus
\begin{gather}
\lim_{m\to\infty}\sup_{s\in[-1, 1]}\|\partial_t g_{m, \lin}(s)\|_{L^2} = 0, \\
\liminf_{m\to\infty}\inf_{s\in[-1, 1]}\|g_{m, \lin}(s)\|_{H^1} > 0,
\end{gather}
the last inequality following from \eqref{eq:g-large-en}.
It is well-known that such behaviour is impossible
for the free Klein-Gordon equation.
For instance, it contradicts the identity
\begin{equation}
\begin{aligned}
&\int_{-\infty}^\infty \big(\partial_t g_{m, \lin}(1, x)g_{m, \lin}(1, x) - \partial_t g_{m, \lin}(-1, x)g_{m, \lin}(-1, x)\big)\ud x \\
&\qquad =\int_{-1}^1 \int_{-\infty}^\infty \big((\partial_t g_{m, \lin}(t, x))^2 - (\partial_x g_{m, \lin}(t, x))^2 - g_{m, \lin}(t, x)^2\big)\ud x\ud t,
\end{aligned}
\end{equation}
obtained by multiplying the equation by $g_{m, \lin}$ and integrating in space-time.
Hence, \eqref{eq:asym-stat-contr} is impossible.
\end{proof}

\section{Main order asymptotic behavior of kink clusters}
\label{sec:n-body}
The current section is devoted to the analysis of the approximate $n$-body problem
associated with a kink cluster motion, obtained in Lemma~\ref{lem:ref-mod}.
\subsection{Linear algebra results}
\label{ssec:lin-alg}
We gather here a few auxiliary results needed later.
\begin{definition}
\label{def:laplacian}
For any $n \in \{1, 2, \ldots\}$, the (positive) discrete
Dirichlet Laplacian $\Delta^{(n)} = (\Delta_{jk}^{(n)})_{j, k=1}^{n-1} \in \bR^{(n-1)\times (n-1)}$ is defined by
$\Delta^{(n)}_{kk} = 2$ for $k = 1, \ldots, n-1$,
$\Delta^{(n)}_{k, k+1} = \Delta^{(n)}_{k+1, k} = -1$
for $k = 1, \ldots, n-2$ and $\Delta^{(n)}_{j, k} = 0$ for $|j - k| \geq 2$.
\end{definition}
We denote
\begin{equation}
\label{eq:1-and-sig-def}
\vec 1 \coloneqq (1, \ldots, 1) \in \bR^{n-1}, \qquad
\vec \sigma = (\sigma_1, \ldots, \sigma_{n-1}), \ \sigma_k \coloneqq \frac{k(n-k)}{2}.
\end{equation}
We have
\begin{equation}
\Delta^{(n)}\vec\sigma = \Big( k(n-k) - \frac 12((k-1)(n-k+1) + (k+1)(n-k-1)) \Big)_{k = 1}^{n-1} = \vec 1.
\end{equation}
We also denote
\begin{align}
\Pi &\coloneqq \{\vec z\in \bR^{n-1}: \vec \sigma\cdot \vec z = 0\}, \\
\mu_0 &\coloneqq \frac{1}{\vec \sigma\cdot \vec 1}
= \frac{12}{(n+1)n(n-1)}, \\
P_\sigma \vec y &\coloneqq \vec y - \frac{\vec \sigma\cdot\vec y}{\|\vec \sigma\|^2}\vec \sigma \quad\text{(the orthogonal projection of $\vec y \in \bR^{n-1}$ on $\Pi$)}, \\
P_1 \vec y &\coloneqq \vec y - \mu_0(\vec\sigma\cdot\vec y)\vec 1\quad\text{(the projection of $\vec y \in \bR^{n-1}$ on $\Pi$ along the direction $\vec 1$)}.
\end{align}
Observe that
\begin{equation}
P_1 \Delta^{(n)}\vec y = \Delta^{(n)} \vec y - \mu_0 (\vec\sigma\cdot (\Delta^{(n)}\vec y))\vec 1 = \Delta^{(n)}\vec y - \mu_0 (\vec 1\cdot \vec y) \vec 1.
\end{equation}
\begin{lemma}
\label{lem:matrix-D}
The matrix $\Delta^{(n)}$ is positive definite. There exists $\mu_1 > 0$ such that for all $\vec y \in \bR^{n-1}$
\begin{equation}
\label{eq:matrix-D}
\vec y\cdot (P_1 \Delta^{(n)}\vec y) = \vec y\cdot(\Delta^{(n)}\vec y) - \mu_0(\vec 1\cdot\vec y)^2 \geq \mu_1 | P_\sigma \vec y |^2.
\end{equation}
\end{lemma}
\begin{proof}
Since $\vec \sigma$ and $\vec 1$ are not orthogonal, it suffices to prove \eqref{eq:matrix-D}.

Consider the matrix $\wt \Delta = (\wt \Delta_{jk})$ given by $\wt \Delta_{jk} \coloneqq 2\delta_{jk} + \mu_0 - \Delta_{jk}$.
The desired inequality is equivalent to
\begin{equation}
\label{eq:perron}
\vec y \cdot (\wt \Delta\vec y) \leq 2|\vec y|^2 - \mu_1 | P_\sigma \vec y |^2.
\end{equation}
We have
\begin{equation}
\wt \Delta\vec \sigma = 2\vec\sigma + \mu_0(\vec 1\cdot \vec \sigma)\vec 1 - \Delta^{(n)}\vec \sigma = 2\vec\sigma.
\end{equation}
The matrix $\wt \Delta$ and the vector $\vec \sigma$ have strictly positive entries.
By the Perron-Frobenius theorem, the largest eigenvalue of $\wt \Delta$ equals $2$ and is simple, which implies \eqref{eq:perron}.
\end{proof}

\begin{lemma}
\label{lem:zcr} \begin{enumerate}[(i)]
\item \label{it:zcr-1}
The function
\begin{equation}
\Pi \owns \vec z \mapsto \vec 1\cdot \eee^{-\vec z} \in (0, \infty)
\end{equation}
has a unique critical point $\vec z_\tx{cr}$, which is its global minimum.
\item\label{it:zcr-2}
There exist $C > 0$ and $\delta_0 > 0$ such that for all $\vec z \in \Pi$ and $\delta \leq \delta_0$
\begin{equation}
|P_\sigma \eee^{-\vec z}| \leq \delta\vec 1 \cdot\eee^{-\vec z}\quad\Rightarrow\quad|\vec z - \vec z_\tx{cr}| \leq C\delta.
\end{equation}
\item\label{it:zcr-3}
There exists $C > 0$ such that for all $\vec z \in \Pi$
\begin{equation}
C^{-1} |P_\sigma \eee^{-\vec z}|^2 \leq \vec 1\cdot\eee^{-\vec z}\big(\vec 1\cdot \eee^{-\vec z} - \vec 1\cdot \eee^{-\vec z_\tx{cr}}\big) \leq C |P_\sigma \eee^{-\vec z}|^2.
\end{equation}
\end{enumerate}
\end{lemma}
\begin{proof}
The function $\vec z \mapsto  \vec 1\cdot \eee^{-\vec z}$,
defined on the hyperplane $\Pi$, is strictly convex and tends to $\infty$ as $|\vec z| \to \infty$, thus it has no critical points other than its global minimum.
By the Lagrange multiplier method, this unique critical point $\vec z_\tx{cr}$ is determined by
\begin{equation}
\label{eq:zcr-lagr}
\eee^{-\vec z_\tx{cr}} = \lambda_\tx{cr}\vec \sigma \quad\Leftrightarrow\quad \vec z_\tx{cr} = -\log(\lambda_\tx{cr}\vec \sigma)
= -\log(\lambda_\tx{cr})\vec 1 - \log\vec\sigma,
\end{equation}
where $\lambda_\tx{cr} \coloneqq \exp(-\mu_0\vec\sigma\cdot\log\vec\sigma)$ is the unique value for which $\vec z_\tx{cr} \in \Pi$.

We can rewrite the condition $|P_\sigma \eee^{-\vec z}| \leq\delta \vec 1\cdot\eee^{-\vec z}$ as follows:
\begin{equation}
\eee^{-\vec z} = \lambda \vec \sigma + \vec u, \qquad \lambda \in \bR,\ \vec u \in \Pi,\ |\vec u| \leq \delta \vec 1\cdot\eee^{-\vec z}.
\end{equation}
We thus have $\lambda = |\vec\sigma|^{-2} \vec\sigma\cdot \eee^{-\vec z}$, which implies $\lambda \simeq \vec 1\cdot \eee^{-\vec z}$, so
\begin{equation}
|{-}\vec z - \log(\lambda\vec \sigma)| = |\log(\lambda \vec \sigma + \vec u) - \log(\lambda\vec \sigma)| \lesssim \frac{|\vec u|}{\lambda} \lesssim \delta.
\end{equation}
By \eqref{eq:zcr-lagr} and the triangle inequality, we obtain
\begin{equation}
|\vec z - \vec z_\tx{cr} + (\log\lambda-\log\lambda_\tx{cr})\vec 1| \lesssim \delta.
\end{equation}
Taking the inner product with $\vec\sigma$, we get $|\log\lambda-\log\lambda_\tx{cr}| \lesssim \delta$, which proves \ref{it:zcr-2}.

Finally, we prove \ref{it:zcr-3}.
If $|\vec z - \vec z_\tx{cr}| \leq 1$, then $\vec 1\cdot \eee^{-\vec z} \simeq 1$
and $\vec 1\cdot \eee^{-\vec z} - \vec 1\cdot\eee^{-\vec z_\tx{cr}} \simeq |\vec z - \vec z_\tx{cr}|^2$, thus it suffices to verify that $|P_\sigma \eee^{-\vec z}| \simeq |z - z_\tx{cr}|$.
The inequality $\lesssim$ follows from $P_\sigma \eee^{-\vec z_\tx{cr}} = 0$
and the mean value theorem. The inequality $\gtrsim$ follows from \ref{it:zcr-2}.
Suppose now that $|\vec z - \vec z_\tx{cr}| \geq 1$. We then have
$\vec 1\cdot\eee^{-\vec z} - \vec 1 \cdot\eee^{-\vec z_\tx{cr}} \simeq \vec 1\cdot\eee^{-\vec z}$
and, invoking again \ref{it:zcr-2}, $|P_\sigma \eee^{-\vec z}| \simeq \vec 1\cdot\eee^{-\vec z}$.
\end{proof}

\subsection{Analysis of the $n$-body problem}
\label{ssec:analysis-mod}
Let $\vec y(t) = (y_1(t), \ldots, y_{n-1}(t))$ be defined by
\begin{equation}
\label{eq:y-def}
y_k(t) \coloneqq a_{k+1}(t) - a_k(t), \qquad k \in \{1, \ldots, n-1\}
\end{equation}
and let
\begin{equation}
y_{\min}(t) \coloneqq \min_{1\leq k<n} y_k(t).
\end{equation}
We denote
\begin{equation}
\label{eq:rho-tilde}
\wt \rho(t) \coloneqq \rho(\vec a(t)) = \sum_{k=1}^{n-1}\eee^{-y_k(t)},
\end{equation}
which (up to a constant) is the potential energy of the ODE system \eqref{eq:attractive-toda}.
From \eqref{eq:a'-est}, we have
\begin{equation}
\label{eq:g-by-rhotilde}
\|\bs g(t)\|_\cE^2 + |\vec v(t)|^2  \lesssim \wt \rho(t) \simeq \eee^{-y_{\min}(t)}.
\end{equation}
Let $\vec q(t) = (q_1(t), \ldots, q_{n-1}(t))$
be defined by
\begin{equation}
\label{eq:q-def}
q_k(t) \coloneqq M^{-1}\big(p_{k+1}(t) - p_{k}(t)\big).
\end{equation}
By \eqref{eq:ak'-est} and.\eqref{eq:ak'}, we have
\begin{equation}
\label{eq:y'-ineq}
|{\vec y \,}'(t) - \vec q(t)| \lesssim \eee^{-y_{\min}(t)}.
\end{equation}
By \eqref{eq:ak'} and \eqref{eq:g-by-rhotilde}, we have
\begin{equation}
\label{eq:q-ineq}
|\vec q(t)| \lesssim \eee^{-\frac 12 y_{\min}(t)}.
\end{equation}
Let
\begin{equation}
\label{eq:A-def}
A\coloneqq \frac{\kappa\sqrt 2}{\sqrt M}.
\end{equation}
By Lemma~\ref{lem:Fz} and Lemma~\ref{lem:ref-mod}, we have
\begin{equation}
\label{eq:q'-ineq}
{\vec q\,}'(t) = -A^2\Delta^{(n)}\eee^{-\vec y(t)} + O(y_{\min}(t)^{-1} \eee^{-y_{\min}(t)}),
\end{equation}
where $\Delta^{(n)}$ is the matrix of the (positive) discrete Dirichlet Laplacian, see Definition~\ref{def:laplacian},
and we denote $\eee^{-\vec y(t)} \coloneqq (\eee^{-y_1(t)}, \ldots, \eee^{-y_{n-1}(t)})$
as explained in Section~\ref{ssec:notation}.

%We have the following ``ejection lemma'' \red{in the spirit of etc.}
%\begin{lemma}
%There exist $C_1 > 0$ such that the following is true.
%Let $t_0 \in \bR$, let $\bs \phi$ be a solution of \eqref{eq:csf} such that
%$\bfd(\bs \phi(t_0)) < \eta_0$ and let $\rho(t)$ be given by \eqref{eq:rhot-def}
%whenever $\bfd(\bs \phi(t)) < \eta_0$.
%\begin{itemize}
%\item
%If $\rho'(t_0) \geq 0$, then there exists $t_1 \in \big[t_0, t_0 + C_1\rho(t_0)^{-\frac 12}\big]$ such that $\bfd(\bs \phi(t_1)) = \eta_0$.
%\item
%If $\rho'(t_0) \leq 0$, then there exists $t_1 \in \big[t_0 - C_1\rho(t_0)^{-\frac 12}, t_0\big]$ such that $\bfd(\bs \phi(t_1)) = \eta_0$.
%\end{itemize}
%\end{lemma}
%\begin{proof}
%The second statement follows from the first, by the time reversal symmetry.
%
%Suppose the conclusion fails. We thus have well-defined modulation parameters for all
%$t \in \big[t_0, t_0 + C_1\rho(t_0)^{-\frac 12}\big]$. Consider the function
%\begin{equation}
%f(t) \coloneqq
%\end{equation}
%\end{proof}
\begin{lemma}
\label{lem:ymin-asym}
If $n > 1$ and $\bs\phi$ is a kink $n$-cluster, then
\begin{equation}
\sup_{t \geq 0}|y_{\min}(t) - 2\log t| < \infty.
\end{equation}
\end{lemma}
\begin{proof}
We first prove that $\sup_{t \geq T_0} y_{\min}(t) - 2\log t < \infty$.
By \eqref{eq:g-coer}, we have $|y_k'(t)| \lesssim \eee^{-\frac 12 y_{\min}(t)}$ for all $k \in \{1, \ldots, n\}$.
Therefore, $y_{\min}$ is a locally Lipschitz function, satisfying
\begin{equation}
\label{eq:ymin'}
|y_{\min}'(t)| \lesssim \eee^{-\frac 12 y_{\min}(t)},\qquad\text{for almost all }t \geq 0.
\end{equation}
By the Chain Rule, $|(\eee^{\frac 12 y_{\min}(t)})'| \lesssim 1$ almost everywhere.
After integration, we obtain $y_{\min}(t) - 2\log t \lesssim 1$.

Now we prove that $\sup_{t \geq 0} y_{\min}(t) - 2\log t > -\infty$.
Let $c_0 > 0$ be small enough and consider the function
\begin{equation}
\beta(t) \coloneqq \vec q(t)\cdot \eee^{-\vec y(t)}- c_0\eee^{-\frac 32 y_{\min}(t)} = \sum_{k=1}^{n-1} q_k(t) \eee^{-y_k(t)} - c_0\eee^{-\frac 32 y_{\min}(t)}.
\end{equation}
Using \eqref{eq:y'-ineq} and \eqref{eq:q'-ineq} and $|q_k(t)| \lesssim \eee^{-\frac 12 y_{\min}(t)} \ll y_{\min}^{-1}$, we obtain
\begin{equation}
\begin{aligned}
\beta'(t) = &-\sum_{k=1}^{n-1} q_k(t)^2 \eee^{-y_k(t)}- A^2 \eee^{-\vec y(t)}\cdot (\Delta^{(n)}\eee^{-\vec y(t)}) \\
&+ \frac 32c_0y_{\min}'(t)\eee^{-\frac 32 y_{\min}(t)}
+ O(y_{\min}(t)^{-1}\eee^{-2y_{\min}(t)}).
\end{aligned}
\end{equation}
By Lemma~\ref{lem:matrix-D}, there is $c_1 > 0$ such that for all $t$ we have
$ \eee^{-\vec y(t)}\cdot(\Delta^{(n)}\eee^{-\vec y(t)}) \geq c_1\eee^{-2y_{\min}(t)}$.
If we recall \eqref{eq:ymin'}, we obtain that $\beta$ is a decreasing function
if $c_0$ is sufficiently small. By assumption, $\lim_{t\to\infty}y_{\min}(t) = \infty$, which, together with \eqref{eq:q-ineq}, implies $\lim_{t\to\infty}\beta(t) = 0$,
hence $\beta(t) \geq 0$ for all $t \geq T_0$, in other words
\begin{equation}
\label{eq:1stderlbound}
 \vec q(t)\cdot \eee^{-\vec y(t)} \geq c_0\eee^{-\frac 32 y_{\min}(t)}.
\end{equation}
Consider now the function $\widetilde\rho(t)$ defined by \eqref{eq:rho-tilde}.
From \eqref{eq:y'-ineq} and \eqref{eq:1stderlbound} we get, perhaps modifying $c_0$,
\begin{equation}
\label{eq:rho'-diff-ineq}
\widetilde\rho\,'(t) = -\vec q(t)\cdot \eee^{-\vec y(t)}
+ O(\eee^{-2y_{\min}}) \leq -c_0 \widetilde\rho(t)^\frac 32 \ \Leftrightarrow \ \dd t\big((\widetilde\rho(t))^{-\frac 12}\big) \geq \frac{c_0}{2},
\end{equation}
so that, after integrating in time, $\widetilde\rho(t) \lesssim t^{-2}$, yielding the conclusion.
\end{proof}
We have thus obtained upper and lower bounds on the minimal weighted distance between neighboring kinks.
Our next goal is to obtain an upper bound on the \emph{maximal} distance.
From the lower bound, by integrating the modulation inequalities we get $|\vec y(t)| \lesssim \log t$.
However, because of exponential interactions, our analysis requires an estimate up to an \emph{additive} constant.
\begin{remark}
At this point, by means of energy estimates, we could improve bounds on the error to $\|\bs g(t)\|_\cE \leq t^{-2 + \epsilon}$ for $t \gg 1$.
However, such information does not seem to trivialise the analysis of the ODE,
which is not immediate even in the absence of \emph{any} error terms.
\end{remark}
\begin{lemma}
\label{lem:U-lbound}
If $\bs\phi$ is a kink $n$-cluster and $\widetilde\rho$ is defined by \eqref{eq:rho-tilde}, then
\begin{equation}
\label{eq:rho-lbound}
\limsup_{t\to \infty}\,\frac 12A^2t^2 \widetilde\rho(t) \geq \vec 1\cdot \vec \sigma.
\end{equation}
\end{lemma}
\begin{proof}
The idea is to estimate from below the kinetic energy.

From the lower bound on the distance between the kinks, there exists $C$ such that for all $k \in \{1, \ldots, n \}$ and $t \geq T_0$ we have
\begin{equation}
|a_{n+1 - k}(t) - a_k(t)| \geq 2|n+1-2k|\log t - C,
\end{equation}
in particular
\begin{equation}
\label{eq:U-lbound-1}
\int_{T_0}^t \big(|a_{n+1 - k}'(s)| + |a_k'(s)|\big)\ud s \geq 2|n+1-2k|\log t - C.
\end{equation}
By Lemma~\ref{lem:ref-mod} and Lemma~\ref{lem:ymin-asym}, we have
\begin{equation}
\int_{T_0}^\infty |M^{-1}\vec p(s) - {\vec a \,}'(s)|\ud s < \infty,
\end{equation}
thus \eqref{eq:U-lbound-1} yields
\begin{equation}
M^{-1}\int_{T_0}^t (|p_k(s)| + |p_{n+1-k}(s)|)\ud s \geq 2|n+1-2k|\log t - C.
\end{equation}
Multiplying by $|n+1 - 2k|$ and taking the sum in $k$, we obtain
\begin{equation}
M^{-1}\sum_{k=1}^n \int_{T_0}^t |(n+1-2k)p_k(s)| \geq \bigg(\sum_{k=1}^n(n+1 - 2k)^2\bigg)\log t - C,
\end{equation}
hence, as $t \to \infty$, the quantity
\begin{equation}
\int_{T_0}^t\Big(s\sum_{k=1}^n|(n+1-2k)p_k(s)| - M\sum_{k=1}^n(n+1 - 2k)^2\Big)\frac{\vd s}{s}
\end{equation}
is bounded from below, which implies in particular
\begin{equation}
\limsup_{t\to\infty}\,t\sum_{k=1}^n |(n+1-2k)p_k(t)| \geq M\sum_{k=1}^n(n+1 - 2k)^2.
\end{equation}
The Cauchy-Schwarz inequality yields
\begin{equation}
\label{eq:p3-est}
\begin{aligned}
&\frac{1}{M^2}\limsup_{t\to \infty}\, t^2 \sum_{k=1}^n |p_k(t)|^2\sum_{k=1}^n(n+1-2k)^2 \geq \\
&\frac{1}{M^2}\bigg(\limsup_{t\to\infty}\,t\sum_{k=1}^n |(n+1-2k)p_k(t)|\bigg)^2
\geq \Big( \sum_{k=1}^n(n+1-2k)^2 \Big)^2,
\end{aligned}
\end{equation}
thus
\begin{equation}
\label{eq:p2-est}
\frac{1}{4M^2}\limsup_{t\to \infty}\, t^2 \sum_{k=1}^n |p_k(t)|^2
\geq \frac 14\sum_{k=1}^n(n+1 - 2k)^2 = \frac{(n-1)n(n+1)}{12} = \vec 1\cdot \vec \sigma.
\end{equation}
By \eqref{eq:a'-est} and \eqref{eq:ak'}, we have
\begin{equation}
\frac{1}{4M^2}|\vec p(t)|^2 \leq \big(\kappa^2M^{-1} + o(1)\big) \widetilde\rho(t) = \Big(\frac 12 A^2 + o(1)\Big)\widetilde\rho(t),
\end{equation}
thus \eqref{eq:p2-est} implies \eqref{eq:rho-lbound}.
\end{proof}

\begin{lemma}
\label{lem:y-bound-seq}
If $\bs\phi$ is a kink $n$-cluster, then there exists an increasing sequence $(t_m)_{m=1}^\infty$,
$\lim_{m\to\infty} t_m = \infty$, such that
\begin{equation}
\label{eq:y-bound-seq}
\lim_{m\to\infty}\big|\vec y(t_m) - \big(2\log(At_m)\vec 1 - \log(2\vec\sigma)\big)\big| = 0.
\end{equation}
\end{lemma}
\begin{proof}
Set $\conj\ell \coloneqq \limsup_{t\to \infty} \frac 12 A^2 t^2 \widetilde\rho(t)$.
Lemma~\ref{lem:U-lbound} yields $\conj\ell \geq \vec 1\cdot \vec \sigma$
and Lemma~\ref{lem:ymin-asym} yields $\conj\ell < \infty$.
Let $t_m$ be an increasing sequence such that $\lim_{m\to \infty} t_m = \infty$ and
\begin{equation}
\lim_{m\to \infty} \frac 12A^2 t_m^2 \widetilde\rho(t_m) = \conj\ell.
\end{equation}
We will prove that $(t_m)$ satisfies \eqref{eq:y-bound-seq}, in particular $\conj\ell = \vec 1\cdot\vec \sigma$.
In order to make the formulas shorter, we denote $\vec y_m \coloneqq \vec y(t_m)$ and $\vec q_m \coloneqq \vec q(t_m)$.

Let $0 < \epsilon \ll 1$. The idea is to deduce \eqref{eq:y-bound-seq}
by analysing the evolution of $\vec y(t)$ for $(1-\epsilon)t_m \leq t \leq (1+\epsilon)t_m$.
In the computation which follows, the asymptotic notation $o$, $O$, $\ll$ etc.
is used for claims about the asymptotic behaviour of various quantities for $m$ large enough and $\epsilon$ small enough.

Lemma~\ref{lem:ymin-asym} yields $\widetilde\rho(t) \lesssim t^{-2}$, hence $|{\vec y \,}'(t)| \lesssim t^{-1}$. We thus obtain
\begin{equation}
\label{eq:y-bound-eps}
|\vec y(t) - \vec y_m| \lesssim \epsilon,\qquad \text{for all }(1-\epsilon)t_m \leq t \leq (1+\epsilon)t_m,
\end{equation}
which implies
\begin{equation}
\label{eq:ey-bound-eps}
|\eee^{-\vec y(t)} - \eee^{-\vec y_m}| \lesssim \epsilon t_m^{-2},\qquad \text{for all }(1-\epsilon)t_m \leq t \leq (1+\epsilon)t_m.
\end{equation}
Using this bound and integrating \eqref{eq:q'-ineq} in time, we get
\begin{equation}
\vec q(t) = \vec q_m - (t-t_m)A^2 \Delta^{(n)} \eee^{-\vec y_m} + o(\epsilon t_m^{-1}),\qquad \text{for all }(1-\epsilon)t_m \leq t \leq (1+\epsilon)t_m.
\end{equation}
Integrating \eqref{eq:y'-ineq}, we get
\begin{equation}
\begin{aligned}
\vec y(t) &= \vec y_m + \int_{t_m}^t \vec q(\tau)\ud \tau + O(\epsilon t_m^{-1}) \\
&= \vec y_m + \int_{t_m}^t \big(\vec q_m - (\tau-t_m)A^2 \Delta^{(n)} \eee^{-\vec y_m}\big)\ud\tau + o(\epsilon^2) + O(\epsilon t_m^{-1}) \\
&= \vec y_m + (t - t_m)\vec q_m - \frac 12 (t - t_m)^2 A^2 \Delta^{(n)} \eee^{-\vec y_m} + o(\epsilon^2) + O(\epsilon t_m^{-1}).
\end{aligned}
\end{equation}
By taking $m$ large enough, we can have $t_m^{-1} \ll \epsilon$. The bound above allows us to compute the asymptotic expansion
of $t^2 \widetilde\rho(t)$ up to order $o(\epsilon^2)$. Recalling from Section~\ref{ssec:notation} our notation for component-wise operations on vectors, we can write
\begin{equation}
\eee^{-\vec y(t)} = \eee^{-\vec y_m}\Big(1 - (t-t_m)\vec q_m+ \frac 12(t-t_m)^2 \big((\vec q_m)^2 + A^2 \Delta^{(n)} \eee^{-\vec y_m}\big)\Big) + o(\epsilon^2 t_m^{-2}),
\end{equation}
thus
\begin{equation}
\begin{aligned}
\widetilde\rho(t) &= \widetilde\rho(t_m) - (t - t_m) \vec q_m\cdot\eee^{-\vec y_m} \\
&+ \frac 12 (t-t_m)^2\big( (\vec q_m)^2\cdot \eee^{-\vec y_m}
+ A^2 \eee^{-\vec y_m}\cdot \Delta^{(n)} \eee^{-\vec y_m}\big) + o(\epsilon^2 t_m^{-2}).
\end{aligned}
\end{equation}
In particular, we have
\begin{equation}
\begin{aligned}
&((1\pm\epsilon)t_m)^2 \widetilde\rho((1+\epsilon)t_m) = t_m^2 \widetilde\rho(t_m) \pm \epsilon\big(2t_m^2 \widetilde\rho(t_m) - t_m^3\vec q_m\cdot \eee^{-\vec y_m}\big) \\
&\quad+\epsilon^2 \Big(t_m^2 \widetilde\rho(t_m) - 2t_m^3 \vec q_m\cdot \eee^{-\vec y_m} +\frac 12 t_m^4\big( (\vec q_m)^2\cdot \eee^{-\vec y_m}
+ A^2 \eee^{-\vec y_m}\cdot \Delta^{(n)} \eee^{-\vec y_m}\big)\Big) + o(\epsilon^2).
\end{aligned}
\end{equation}
From this estimate and the definition of the sequence $(t_m)$ we deduce that
\begin{equation}
\label{eq:1st-ord-eps}
\lim_{m\to \infty} \big(2t_m^2 \widetilde\rho(t_m) - t_m^3 \vec q_m\cdot\eee^{-\vec y_m}\big) = 0
\end{equation}
and
\begin{equation}
\label{eq:2nd-ord-eps}
\limsup_{m\to\infty} \Big(t_m^2 \widetilde\rho(t_m) - 2t_m^3 \vec q_m\cdot\eee^{-\vec y_m} +\frac 12 t_m^4\big((\vec q_m)^2\cdot \eee^{-\vec y_m}
+ A^2\eee^{-\vec y_m}\cdot \Delta^{(n)} \eee^{-\vec y_m}\big)\Big) \leq 0.
\end{equation}
By the Cauchy-Schwarz inequality,
\begin{equation}
\big(t_m^3\vec q_m\cdot \eee^{-\vec y_m}\big)^2 \leq t_m^2\vec 1\cdot\eee^{-\vec y_m} t_m^4 (\vec q_m)^2\cdot \eee^{-\vec y_m}
= t_m^2 \widetilde\rho(t_m) t_m^4 (\vec q_m)^2\cdot \eee^{-\vec y_m},
\end{equation}
thus \eqref{eq:1st-ord-eps} yields
\begin{equation}
\label{eq:ord-eps-3}
\liminf_{m\to\infty} t_m^4  (\vec q_m)^2\cdot \eee^{-\vec y_m} \geq 4\lim_{m\to\infty}t_m^2 \widetilde\rho(t_m).
\end{equation}
Injecting this estimate to \eqref{eq:2nd-ord-eps} and using again \eqref{eq:1st-ord-eps}, we obtain
\begin{equation}
\limsup_{m\to\infty} \big(t_m^4 A^2 \eee^{-\vec y_m}\cdot \Delta^{(n)}\eee^{-\vec y_m}\big) \leq 2\lim_{m\to\infty}t_m^2 \widetilde\rho(t_m),
\end{equation}
hence, by Lemma~\ref{lem:matrix-D},
\begin{equation}
\limsup_{m\to\infty}\big(A^2 \mu_0 (t_m^2\widetilde\rho(t_m))^2 + A^2 \mu_1|t_m^2 P_\sigma \eee^{-\vec y_m}|^2\big) \leq 2\lim_{m\to\infty}t_m^2 \widetilde\rho(t_m).
\end{equation}
Recalling that $\lim_{m\to\infty} t_m^2 \widetilde\rho(t_m) \geq 2A^{-2}\mu_0^{-1}$, we obtain that in the last inequality
there is in fact equality and, additionally, $\lim_{m\to\infty}|t_m^2 P_\sigma \eee^{-\vec y_m}| = 0$.
In other words, $\eee^{-\vec y_m} = \lambda_m\vec\sigma + o(t_m^{-2})$. The coefficient $\lambda_m$ is determined by
\begin{equation}
2(At_m)^{-2}\mu_0^{-1} + o(t_m^{-2}) = \widetilde\rho(t_m) = \vec 1\cdot \eee^{-\vec y_m} = \lambda_m \vec 1\cdot \vec \sigma + o(t_m^{-2}),
\end{equation}
thus $\lambda_m = 2(At_m)^{-2} + o(t_m^{-2})$, so $\eee^{-\vec y_m} = 2(At_m)^{-2}\vec\sigma + o(t_m^{-2})$,
which yields \eqref{eq:y-bound-seq} after taking logarithms.
\end{proof}
\begin{remark}
From the proof, one can see that $\lim_{m\to \infty}|t_m\vec q(t_m) - 2\times \vec 1| = 0$.
Below, we conclude the modulation analysis without using this information.
\end{remark}
%Our objective is now to improve the sequential convergence in Lemma~\ref{lem:y-bound-seq} to continuous time convergence.

We decompose
\begin{equation}
\label{eq:yq-syst}
\begin{gathered}
\vec y(t) = r(t)\vec 1 + \vec z(t),\qquad \vec z(t) \coloneqq P_1 \vec y(t), \\
\vec q(t) = b(t)\vec 1 + \vec w(t), \qquad \vec w(t) \coloneqq P_1 \vec q(t).
\end{gathered}
\end{equation}
We thus have
\begin{equation}
\label{eq:y'q'-syst}
\begin{gathered}
{\vec y \,}'(t) = r'(t)\vec 1 + {\vec z \,}'(t),\qquad {\vec z \,}'(t) = P_1\vec y'(t), \\
{\vec q \,}'(t) = b'(t)\vec 1 + {\vec w \,}'(t), \qquad {\vec w \,}'(t) = P_1 {\vec q \,}'(t).
\end{gathered}
\end{equation}
We denote $z_{\min}(t) \coloneqq \min_{1 \leq k \leq n-1}z_j(t)$. We note that $z_{\min}(t) \leq 0$ and $y_{\min}(t) = r(t) + z_{\min}(t)$.
Observe also that $|z_{\min}(t)| \simeq |\vec z(t)|$, since both quantities are norms on the hyperplane $\Pi$, and that $\eee^{-z_{\min}(t)} \simeq \vec 1\cdot \eee^{-\vec z(t)}$. Hence,
Lemma~\ref{lem:zcr} \ref{it:zcr-3} yields
\begin{equation}
\label{eq:Psigexp-equiv}
\begin{aligned}
\eee^{-r(t)}|P_\sigma \eee^{-\vec z(t)}|^2 &\simeq \eee^{z_{\min}(t) - y_{\min}(t)}\vec 1\cdot\eee^{-\vec z(t)}(\vec 1\cdot \eee^{-\vec z(t)} - \vec 1 \cdot \eee^{-\vec z_\tx{cr}}) \\
&\simeq \eee^{-y_{\min}(t)}(\vec 1\cdot \eee^{-\vec z(t)} - \vec 1 \cdot \eee^{-\vec z_\tx{cr}}).
\end{aligned}
\end{equation}
Taking the inner product of \eqref{eq:yq-syst} and \eqref{eq:y'q'-syst} with $\vec\sigma$, and using \eqref{eq:y'-ineq}--\eqref{eq:q'-ineq},
we obtain
\begin{equation}
\label{eq:rrho-syst}
\begin{aligned}
|r'(t) - b(t)| &\lesssim \eee^{-r(t)}\eee^{-z_{\min}(t)}, \\
|b'(t) + \mu_0 A^2 \eee^{-r(t)}  \vec 1\cdot \eee^{-\vec z(t)}| &\lesssim y_{\min}(t)^{-1}\eee^{-r(t)}\eee^{-z_{\min}(t)}.
\end{aligned}
\end{equation}
From this, again using \eqref{eq:y'-ineq}--\eqref{eq:q'-ineq}, we deduce
\begin{equation}
\label{eq:zw-syst}
\begin{aligned}
|{\vec z \,}'(t) - \vec w(t)| &\lesssim \eee^{-r(t)}\eee^{-z_{\min}(t)}, \\
|{\vec w \,}'(t) + \eee^{-r(t)}A^2P_1 \Delta^{(n)}\eee^{-\vec z(t)}| &\lesssim y_{\min}(t)^{-1}\eee^{-r(t)}\eee^{-z_{\min}(t)}.
\end{aligned}
\end{equation}
In Lemma~\ref{lem:y-bound-seq}, we proved that $\liminf_{t\to\infty}|\vec z(t) - \vec z_\tx{cr}| = 0$.
Our next goal is to improve this information to continuous time convergence, in other words we prove a no-return lemma.
\begin{proposition}
If $\bs\phi$ is a kink $n$-cluster, then
\begin{equation}
\label{eq:mod-conclusion}
\lim_{t \to \infty}\big(\big|\vec y(t) - \big(2\log(At)\vec 1 - \log(2\vec\sigma)\big)\big| + \big| t\vec q(t) - 2\times \vec 1 \big| \big) = 0.
\end{equation}
\end{proposition}
\begin{proof}
\textbf{Step 1.}
Consider the functions
\begin{equation}
\xi(t) \coloneqq \vec 1\cdot \eee^{-\vec z(t)} - \vec 1\cdot \eee^{-\vec z_\tx{cr}}, \qquad \zeta(t) \coloneqq \vec w(t)\cdot \eee^{-\vec z(t)} - c_0 t^{-1}\xi(t),
\end{equation}
where $c_0 > 0$ will be chosen below.
We know that $\liminf_{t\to\infty} \xi(t) = 0$, and we will improve this to $\lim_{t\to\infty}\xi(t) = 0$.

We claim that for every $\epsilon > 0$ there exist $c_0 = c_0(\epsilon) > 0$ and $t_0 = t_0(\epsilon)$ such that for all $t \geq t_0$
\begin{equation}
\xi(t) \geq \epsilon\ \text{and}\ \zeta(t) \leq 0 \ \Rightarrow \ \zeta'(t) \leq 0.
\end{equation}
Indeed, applying successively \eqref{eq:zw-syst}, \eqref{eq:matrix-D} and
\eqref{eq:Psigexp-equiv}, we get
\begin{equation}
\begin{aligned}
&\dd t \vec w(t)\cdot \eee^{-\vec z(t)} \\
&\ \leq - \vec w(t)^2\cdot \eee^{-\vec z(t)} - \eee^{-r(t)}A^2  \eee^{-\vec z(t)}\cdot (P_1 \Delta^{(n)}\eee^{-\vec z(t)}) + o(\eee^{-r(t)}\eee^{-2z_{\min}(t)}) \\
&\ \lesssim -\eee^{-r(t)}|P_\sigma \eee^{-\vec z(t)}|^2 + o(\eee^{-r(t)}\eee^{-2z_{\min}(t)})
\lesssim -\eee^{-y_{\min}(t)}(\xi(t) + o(\eee^{-z_{\min}(t)})).
\end{aligned}
\end{equation}
Now observe that $\xi(t) \geq \epsilon$ implies $\xi(t) \geq c_1\eee^{-z_{\min}(t)}$,
where $c_1 = c_1(\epsilon) > 0$. Hence, for $t$ large enough
\begin{equation}
\label{eq:wez-deriv}
\dd t \vec w(t)\cdot \eee^{-\vec z(t)}\lesssim -\eee^{-y_{\min}(t)}\xi(t) \lesssim -t^{-2}\xi(t).
\end{equation}

We have $-(c_0t^{-1}\xi(t))' = c_0t^{-2}\xi(t) - c_0t^{-1}\xi'(t)$.
If $c_0$ is small enough, then \eqref{eq:wez-deriv} allows to absorb the first term.
Regarding the second term, \eqref{eq:zw-syst}
and the assumption $\zeta(t) \leq 0$ yield
\begin{equation}
-\xi'(t) = \vec w(t)\cdot \eee^{-\vec z(t)} + O(\eee^{-r(t)}\eee^{-2z_{\min}(t)})
\leq c_0 t^{-1}\xi(t) + O(t^{-2}\xi(t)).
\end{equation}
For $c_0$ small enough and $t$ large enough, the term ${-}c_0t^{-1}\xi'(t)$ can thus be absorbed as well,
which finishes Step 1.

\textbf{Step 2.}
We are ready to prove that
\begin{equation}
\label{eq:z-conv}
\lim_{t\to\infty} \vec z(t) = \vec z_\tx{cr}.
\end{equation}
Suppose this is false, so there exists $\epsilon_0 \in (0, 1)$ and a sequence $t_m \to \infty$ such that
\begin{equation}
\xi(t_m) = \epsilon_0, \qquad \xi'(t_m) \geq 0
\end{equation}
(this is not the same sequence as in Lemma~\ref{lem:y-bound-seq},
but we denote it by the same symbol in order to simplify the notation).
Take $0 < \epsilon \ll \epsilon_0$, $c_0 = c_0(\epsilon)$ and $m$ large (depending on $\epsilon$).
Using \eqref{eq:zw-syst}, we have
\begin{equation}
\zeta(t_m) = -\xi'(t_m)- c_0 t_m^{-1}\xi(t_m) + O(\eee^{-r(t_m)}\eee^{-2z_{\min}(t_m)}) \lesssim -c_0\epsilon_0 t_m^{-1}.
\end{equation}
Let $\tau_m$ be the first time $\tau_m \geq t_m$ such that $\xi(\tau_m) = \epsilon$.
On the time interval $[t_m, \tau_m]$ we have $\xi(t) \geq \epsilon$, so Step 2. implies that $\zeta$ is decreasing,
provided $m$ is large enough, thus $\zeta(\tau_m) \lesssim -c_0\epsilon_0 t_m^{-1} \leq -c_0\epsilon_0 \tau_m^{-1}$.
On the other hand, $\xi'(\tau_m) \leq 0$, which, by the same computation as above, leads to $\zeta(\tau_m) \gtrsim -c_0\epsilon \tau_m^{-1}$, a contradiction.

\textbf{Step 3.}
We claim that
\begin{equation}
\label{eq:w-conv}
\lim_{t \to \infty} t|\vec w(t)| = 0.
\end{equation}
We use a Tauberian argument. From \eqref{eq:zw-syst}, we have $|{\vec z \,}'(t) - \vec w(t)|
+ |{\vec w \,}'(t)| \leq C t^{-2}$ for all $t$.
If there existed $\epsilon > 0$, $k \in \{1, \ldots, n-1\}$ and a sequence $t_m \to \infty$ such that $w_k(t_m) \geq \epsilon t_m^{-1}$ for all $m$,
then we would have $w_k(t) \geq \frac 12\epsilon t_m^{-1}$ for all $t \in \big[t_m, \big(1 + \frac{\epsilon}{2C}\big)t_m\big]$, hence for $m$ large enough $z_k'(t) \geq \frac 14\epsilon t_m^{-1}$ for all $t \in \big[t_m, \big(1 + \frac{\epsilon}{2C}\big)t_m\big]$,
contradicting the convergence of $z_k(t)$ as $t \to \infty$. The case $w_k(t_m) \leq {-}\epsilon t_m^{-1}$ is similar.

From \eqref{eq:zcr-lagr}, we have $\mu_0 \vec 1\cdot \eee^{-\vec z_\tx{cr}} = \lambda_\tx{cr}$.
Thus, \eqref{eq:rrho-syst} and \eqref{eq:z-conv} yield
\begin{equation}
\label{eq:rrho-syst-2}
|r'(t) - b(t)| \lesssim \eee^{-r(t)}, \qquad |b'(t) + A^2\lambda_\tx{cr} \eee^{-r(t)}| \ll \eee^{-r(t)}.
\end{equation}
From Lemma~\ref{lem:ymin-asym}, \eqref{eq:yq-syst} and \eqref{eq:z-conv} we have $\sup_{t\geq 0}|r(t) - 2\log t| < \infty$. Furthermore, \eqref{eq:q-ineq} and \eqref{eq:w-conv} yield
$\sup_{t\geq 0}t|b(t)| < \infty$. Hence,
\begin{equation}
\begin{aligned}
\dd t\Big( \frac 12 b(t)^2 - A^2\lambda_\tx{cr}\eee^{-r(t)} \Big)
&= b(t)\big(b'(t) + A^2 \lambda_\tx{cr}\eee^{-r(t)}\big) \\ 
&+ A^2\lambda_\tx{cr}(r'(t) - b(t))\eee^{-r(t)} = o(t^{-3}),
\end{aligned}
\end{equation}
and an integration in $t$ leads to
\begin{equation}
\label{eq:r-1st-ord}
\big(b(t) - A\sqrt{2\lambda_\tx{cr}}\eee^{-\frac 12 r(t)}\big)\big(b(t) + A\sqrt{2\lambda_\tx{cr}}\eee^{-\frac 12 r(t)}\big) =  b(t)^2 - 2A^2\lambda_\tx{cr}\eee^{-r(t)} = o(t^{-2}).
\end{equation}
The second bound in \eqref{eq:rrho-syst-2} yields
$b'(t) < 0$ for all $t$ large enough. Since $\lim_{t\to \infty} b(t) = 0$,
we have $b(t) > 0$ for all $t$ large enough, hence \eqref{eq:rrho-syst-2} and \eqref{eq:r-1st-ord} yield
\begin{equation}
\label{eq:r'-est}
r'(t) - A\sqrt{2\lambda_\tx{cr}}\eee^{-\frac 12 r(t)} = 
b(t) - A\sqrt{2\lambda_\tx{cr}}\eee^{-\frac 12 r(t)} + o(t^{-1}) = o(t^{-1}),
\end{equation}
which implies
\begin{equation}
\dd t \big(\eee^{\frac 12 r(t)}\big) = A\sqrt{\frac{\lambda_\tx{cr}}{2}} + o(1),
\end{equation}
thus
\begin{equation}
\eee^{\frac 12 r(t)} = t A\sqrt{\frac{\lambda_\tx{cr}}{2}}\big(1  + o(1) \big)
\end{equation}
and after taking the logarithm we finally obtain
\begin{equation}
r(t) = 2\log(At) + \log\big(\frac 12 \lambda_\tx{cr}\big) + o(1).
\end{equation}
Invoking again \eqref{eq:r'-est}, we also have $b(t) = 2t^{-1} + o(1)$.
Hence, \eqref{eq:mod-conclusion} follows from \eqref{eq:yq-syst}, \eqref{eq:zcr-lagr},
\eqref{eq:z-conv} and \eqref{eq:w-conv}.
\end{proof}
\begin{proof}[Proof of Theorem~\ref{thm:asymptotics}]
The bound on $a_{k+1}(t) - a_k(t)$ follows directly
from \eqref{eq:mod-conclusion}, and the definitions of $A$ and $\vec\sigma$.

Let $\conj a(t) \coloneqq \frac 1n \sum_{j=1}^n a_j(t)$.
Then for all $k \in \{1, \ldots, n\}$, \eqref{eq:mod-conclusion} implies
\begin{equation}
\label{eq:ak'-final}
a_k'(t) = \conj a\,'(t) + \frac 1n \sum_j (a_k'(t) - a_j'(t))
= \conj a\,'(t) + (2k - n - 1)t^{-1} + o(t^{-1}),
\end{equation}
hence
\begin{equation}
|\vec a\,'(t)|^2 = (\conj a\,'(t))^2 + t^{-2}\sum_{k=1}^n(n+1-2k)^2
+ o(t^{-2}).
\end{equation}
An explicit computation yields
\begin{equation}
4\kappa^2 \widetilde\rho(t) = Mt^{-2}\sum_{k=1}^n(n+1-2k)^2 + o(t^{-2}).
\end{equation}
Applying \eqref{eq:a'-est}, we obtain
\begin{equation}
\lim_{t\to\infty}t\big(\|\partial_t g(t)\|_{L^2} + \|g(t)\|_{H^1} + |\conj a\,'(t)|\big) = 0.
\end{equation}
The required bound on $a_k'(t)$ follows from the last inequality and \eqref{eq:ak'-final}.
\end{proof}

\section{Existence of a kink $n$-cluster for prescribed initial positions}
\label{sec:any-position}
The present section is devoted to a proof of Theorem~\ref{thm:any-position}.
The case $n = 1$ is clear, hence we assume $n > 1$.
We will use the following consequence of Brouwer's fixed point theorem,
known as the Poincar\'e--Miranda theorem.
\begin{theorem}\cite{Miranda}
\label{thm:miranda}
Let $L_1 < L_2$, $\vec y = (y_1, \ldots, y_d) \in \bR^d$ and $\vec \Psi = (\Psi_1, \ldots, \Psi_d) : [L_1, L_2]^d \to \bR^d$ be a continuous map such that
for all $k \in \{1, \ldots, d\}$ the following conditions are satisfied:
\begin{itemize}
\item $\Psi_k(\vec x) \leq y_k$ for all $\vec x = (x_1, \ldots, x_d) \in [L_1, L_2]^d$ such that $x_k = L_1$,
\item $\Psi_k(\vec x) \geq y_k$ for all $\vec x = (x_1, \ldots, x_d) \in [L_1, L_2]^d$ such that $x_k = L_2$.
\end{itemize}
Then there exists $\vec x \in [L_1, L_2]^d$ such that $\vec \Psi(\vec x) = \vec y$.
\end{theorem}
\begin{remark}
The result is often stated with $L_1 = -1$, $L_2 = 1$ and $\vec y = 0$,
which can be achieved by a straightforward change of coordinates.
\end{remark}

\begin{lemma}
\label{lem:data-at-T}
There exist $L_0, C_0 > 0$ such that the following is true.
Let $L \geq L_0$ and $\vec a_0 = (a_{0, 1}, \ldots, a_{0, n}) \in \bR^n$ be such that $a_{0, k+1} - a_{0, k}
\geq L$ for all $k \in \{1, \ldots, n-1\}$.
For any $T \geq 0$ there exists $(\vec a_T, \vec v_T) \in \bR^n \times \bR^n$
such that the solution $\bs \phi$ of \eqref{eq:csf} with
\begin{equation}
\label{eq:phiT}
\bs \phi(T) = \bs H(\vec a_T, \vec v_T)
\end{equation}
has energy $nM$ and satisfies $\bs \phi(t) = \bs H(\vec a(t), \vec v(t)) + \bs g(t)$ for all $t \in [0, T]$, where
\begin{enumerate}[(i)]
\item\label{it:data-at-T-1}
$\bs g(t)$ satisfies \eqref{eq:g-orth},
%\item $a_{k+1}(t) - a_k(t) \geq L - C_0$ for all $t \in [0, T]$,
\item\label{it:data-at-T-2}
$\rho(t) \leq C_0/(e^L + t^2)$ for all $t \in [0, T]$, where $\rho(t)$ is defined by \eqref{eq:rhot-def},
\item\label{it:data-at-T-3}
$\vec a(0) = \vec a_0$.
\end{enumerate}
\end{lemma}
\begin{proof}
\textbf{Step 1.} (Preliminary observations.)
In order to have $\vec a(0) = \vec a_0$,
it suffices to guarantee that $a_{k+1}(0) - a_{k}(0) = a_{0, k+1} - a_{0, k}$,
and apply a translation if needed.
Set $y_{0, k} \coloneqq a_{0, k+1} - a_{0, k}$ and let $0 < L_1 < L_2$ be chosen later
(depending on $T$ and $\vec y_0$).
The idea is to construct an appropriate function $\vec\Psi: [L_1, L_2]^{n-1} \to \bR^{n-1}$ satisfying the assumptions of Theorem~\ref{thm:miranda}.

For given $\vec y_T \in [L_1, L_2]^{n-1}$, let $\vec a_T \in \bR^n$ be such that $a_{T, k+1} - a_{T, k} = y_{T, k}$ for all $k \in \{1, \ldots, n-1\}$
and $\sum_{k=1}^n a_{T, k} = 0$
(the last condition is a matter of choice, we could just as well impose $a_{T, 1} = 0$
or any condition of similar kind), so that $\vec a_T = \vec a_T(\vec y_T)$ is continuous.
Let $\bs \phi = \bs \phi(T, \vec a_T, \vec v_T)$ be the solution of \eqref{eq:csf}
for the data \eqref{eq:phiT} at time $T$, and let $T_0 \geq 0$ be the minimal time
such that $\bfd(\bs \phi(t)) \leq \eta_0$ for all $t \in [T_0, T]$ (we set $T_0 = 0$
if $\bfd(\bs \phi(t)) \leq \eta_0$ for all $t \in [0, T]$).
We thus have well-defined modulation parameters $\vec a(t)$ and $\vec v(t)$ for all $t \in [T_0, T]$.
We define $\vec p(t)$, $\vec y(t)$, $y_{\min}(t)$, $\widetilde\rho(t)$ and $\vec q(t)$ as in Sections~\ref{sec:mod} and \ref{sec:n-body}.

\noindent
\textbf{Step 2.} (Choice of $\vec v_T$.)
Let $y_{T, \min} \coloneqq \min_{1 \leq k < n} y_{T, k}$.
We define $\vec v_T = \vec v_T(\vec a_T)$ as the unique vector in $\bR^n$ such that
\begin{enumerate}[(i)]
\item \label{it:choice-vT-sum}
$\sum_{k=1}^n v_{T, k} = 0$,
\item \label{it:choice-vT-lambda}
there exists $\lambda > 0$ such that $v_{T, k+1} - v_{T, k} = \lambda\sqrt{\rho(\vec a_T)}$ for all $k \in \{1, \ldots, n-1\}$,
\item \label{it:choice-vT-en}
$E\big(\bs H(\vec a_T, \vec v_T)\big) = nM$.
\end{enumerate}
Conditions \ref{it:choice-vT-sum} and \ref{it:choice-vT-lambda} are equivalent to
\begin{equation}
\label{eq:vT-form}
v_{T, k} = \frac{2k-n-1}{2}\lambda\sqrt{\rho(\vec a_T)}, \qquad\text{for all }k \in \{1, \ldots, n\}.
\end{equation}
Consider the auxiliary function
\begin{equation}
f(\lambda) \coloneqq E(\bs H(\vec a_T, \vec v_T)),
\end{equation}
with $\vec v_T$ given by \eqref{eq:vT-form}.
It follows from Lemma~\ref{lem:interactions} that there exist constants $\lambda_1, \lambda_2 > 0$ depending only on $n$ such that
\begin{equation}
\label{eq:sign-aux-f}
f(\lambda_1 ) < nM\qquad f(\lambda_2) > nM.
\end{equation}
We claim that for $\lambda \in [\lambda_1, \lambda_2] $ we have
\begin{equation}
f'(\lambda) \gtrsim \rho(\vec a_T).
\end{equation}
Indeed, we have $\partial_\lambda v_{T, k} = \lambda^{-1} v_{T, k}$, hence
\begin{equation}
\begin{aligned}
f'(\lambda) &= \la \vD E(\bs H(\vec a_T, \vec v_T)), \partial_\lambda \bs H(\vec a_T, \vec v_T)\ra \\
&= \lambda^{-1} \Big\la \vD E(\bs H(\vec a_T, \vec v_T)), \sum_{k=1}^n (-1)^kv_{T, k}\partial_v \bs H(a_{T, k}, v_{T, k})\Big\ra,
\end{aligned}
\end{equation}
and it suffices to use \eqref{eq:DEH} and \eqref{eq:matrix-coef}.
We conclude that there exists unique $\lambda \in [\lambda_1, \lambda_2]$ such that $f(\lambda) = nM$.
Moreover, $\lambda$ is continuous with respect to $\vec a_T$.

\noindent
\textbf{Step 3.} (Definition and continuity of the exit time.)
%Consider the function $\rho(t) = \vec 1\cdot \eee^{-\vec y(t)}$, see \eqref{eq:rhot-def}.
Note that $\widetilde\rho(t) \leq n\eee^{-L_1}$ if $y_k(t) \geq L_1$ for all $k$.
We claim that there exists $c_0 > 0$ depending only on $n$ such that
\begin{equation}
\label{eq:sec5-rho'}
\widetilde\rho\,'(t) \leq {-}c_0\widetilde\rho(t)^\frac 32 \qquad\text{for all }t \in [T_0, T].
\end{equation}
Let $\beta$ be defined as in the proof of Lemma~\ref{lem:ymin-asym}.
The same argument shows that $\beta(t)$ is decreasing for $t \in [T_0, T]$ if $c_0$ is small enough.

We claim that $\beta(T) > 0$, which will imply that $\beta(t) > 0$ for all $t \in [T_0, T]$.
The proof of \eqref{eq:rho'-diff-ineq} then applies without changes.
By the definition of $\beta$, it suffices to verify that
\begin{equation}
\label{eq:qkT-pos}
q_k(T) \gtrsim \sqrt{\rho(\vec a_T)}\qquad\text{for all }k \in \{1, \ldots, n-1\}.
\end{equation}
%We have $\dot g(T) = {-}\sum_{k=1}^n (-1)^k v_{T, k}\partial_x H(\cdot - a_{T, k})$.
But from \eqref{eq:ak'} we obtain $|q_k(T) - (v_{T, k+1} - v_{T, k})| \lesssim \rho(\vec a_T)$,
which implies \eqref{eq:qkT-pos} since $v_{T, k+1} - v_{T, k} = \lambda\sqrt{\rho(\vec a_T)}$ and $\lambda \simeq 1$.

Let $T_1 \in [T_0, T]$ be the minimal time such that $\widetilde\rho(t) \leq 2n\eee^{-L_1}$
for all $t \in [T_1, T]$. Suppose that the last inequality holds for all $t \in [T_0, T]$.
Then in particular $\widetilde\rho(T_0) \leq 2n\eee^{-L_1}$,
and Lemma~\ref{lem:basic-mod} yields $\bfd(\bs\phi(T_0)) < \eta_0$ if $L_1$ is large enough,
thus $T_0 = 0$. In this case, we set $T_1 \coloneqq 0$ as well.

We claim that, for $T$ fixed, $T_1 = T_1(\vec a_T)$ is a continuous function.
Let $\vec a_{T, m} \to \vec a_T$ and consider the corresponding sequence of solutions constructed above.

Assume first that $T_1(\vec a_T) > 0$, thus $\widetilde\rho(T_1) = 2n\eee^{-L_1}$.
Let $\epsilon > 0$. Since $\widetilde\rho$ is strictly decreasing by \eqref{eq:sec5-rho'},
we have $\widetilde\rho(t) < 2n\eee^{-L_1}$ for all $t \in [T_1 + \epsilon, T]$.
By the continuity of the flow, the same inequality holds for the $m$-th solution of the sequence if $m$ is large enough,
hence $T_1(\vec a_{T, m}) \leq T_1 + \epsilon$ if $m$ is large enough.
Similarly, $\widetilde\rho(T_1 - \epsilon) > 2n\eee^{-L_1}$, implying $T_1(\vec a_{T, m}) \geq T_1 - \epsilon$ for $m$ large enough.

Now assume that $T_1(\vec a_T) = T_0(\vec a_T) = 0$, thus $\widetilde\rho(0) \leq 2n\eee^{-L_1}$.
Since $\widetilde\rho$ is strictly decreasing, for any $\epsilon > 0$ we have $\widetilde\rho(t) < 2n\eee^{-L_1}$ for all $t \in [\epsilon, T]$, and again the continuity of the flow yields $T_1(\vec a_{T, m}) \leq \epsilon$ for $m$ large enough.

\noindent
\textbf{Step 4.} (Application of the Poincar\'e--Miranda theorem.)
We set
\begin{equation}
\vec\Psi(\vec y_T) \coloneqq \vec y(T_1).
\end{equation}
The preceding step together with the continuity of the flow imply that $\vec\Psi$ is continuous.
Let $k \in \{1, \ldots, n-1\}$ be such that $y_{T, k} = L_2$.
If $L_2$ is sufficiently large, then $\Psi_k(\vec y_T) \geq y_{0, k}$
(it suffices to integrate in time the bound $|y_k'(t)| \lesssim \sqrt{\widetilde\rho(t)} \lesssim \eee^{-\frac 12 L_1}$).

Assume now that $y_{T, k} = L_1$, thus $\widetilde\rho(T) \geq \eee^{-L_1}$.
By \eqref{eq:sec5-rho'}, we have
\begin{equation}
\label{eq:deriv-rho-12}
\dd t\big(\widetilde\rho(t)^{-\frac 12}\big) \gtrsim 1, \qquad\text{for all }t \in [T_1, T].
\end{equation}
Since $\widetilde\rho(T_1) \leq 2n\eee^{-L_1}$, we obtain
\begin{equation}
T - T_1 \lesssim \eee^{L_1/2}\big(1 - (2n)^{-\frac 12}\big).
\end{equation}
But we also have $|y_k'(t)| \lesssim \eee^{-\frac 12 L_1}$ for all $t \in [T_1, T]$,
hence $|y_k(T_1) - y_k(T)| \lesssim 1$, and it suffices to let $L_1 = L - C$ for $C$ sufficiently large.

By Theorem~\ref{thm:miranda}, there exists $\vec y_T \in [L_1, L_2]^{n-1}$ such that $\vec\Psi(\vec y_T) = \vec y_0$.
In particular, $\widetilde\rho(T_1) \lesssim \eee^{-L} = \eee^{-C}\eee^{-L_1} \ll 2n\eee^{-L_1}$ if $C$ is large.
We thus have $T_0 = T_1 = 0$ and $\vec y(0) = \vec y_0$.

\noindent
\textbf{Step 5.} (Estimates on the solution.)
By \eqref{eq:deriv-rho-12}, we have
\begin{equation}
\widetilde\rho(t)^{-\frac 12} - \widetilde\rho(0)^{-\frac 12} \gtrsim t, \qquad\text{for all }t \in [0, T].
\end{equation}
Since $\widetilde\rho(0) \lesssim \eee^{-L}$, we obtain $\widetilde\rho(t) \lesssim (e^L + t^2)^{-1}$.
%An application of \eqref{eq:g-coer} finishes the proof.
\end{proof}

\begin{proof}[Proof of Theorem~\ref{thm:any-position}]
Let $T_m$ be an increasing sequence tending to $\infty$ and let
\begin{equation}
\bs \phi_m(t) = \bs H(\vec a_m(t), \vec v_m(t)) + \bs g_m(t)
\end{equation}
be the solution given by Lemma~\ref{lem:data-at-T} for $T = T_m$.
After extraction of a subsequence, we can assume that $\bs g_m(0) \wto \bs g_0 \in \cE$
and, using the Arzel\`a-Ascoli theorem, $(\vec a_m, \vec v_m) \to (\vec a, \vec v)$ locally uniformly.
Let $\vec v_0 \coloneqq \vec v(0)$ and let $\bs\phi$ be the solution of \eqref{eq:csf} such that $\bs\phi(0) = \bs H(\vec a_0, \vec v_0) + \bs g_0$.
We verify that $\bs \phi$ is the desired kink cluster.

%By Lemma~\ref{lem:data-at-T} \ref{it:data-at-T-3}, for all $m$ we have $\vec a_m(0) = \vec a_0$, hence
%$\la \partial_x H(\cdot - a_{0, k}), g_m(0)\ra = 0$. Taking the weak limit, we get \eqref{eq:g0-orth}.
For all $m$ and $k \in \{1, \ldots, n\}$ we have
\begin{equation}
\la \bs\alpha(a_{m, k}(0), v_{m, k}(0)), \bs g_m(0)\ra = \la \bs\beta(a_{m, k}(0), v_{m, k}(0)), \bs g_m(0)\ra = 0.
\end{equation}
Passing to the weak limit, we obtain
\begin{equation}
\la \bs\alpha(a_{0, k}, v_{0, k}), \bs g_0\ra = \la \bs\beta(a_{0, k}, v_{0, k}), \bs g_0\ra = 0,
\end{equation}
hence $\vec a(\bs \phi(0)) = \vec a_0$.

Fix $t \geq 0$. Since $\bs \phi_m(0) \wto \bs\phi(0)$, Proposition~\ref{prop:cauchy}
implies $\bs \phi_m(t) \wto \bs \phi(t)$, hence 
\begin{equation}
\bs g_m(t)  = \bs\phi_m(t) - \bs H(\vec a_m(t), \vec v_m) \wto \bs \phi(t) - \bs H(\vec a(t), \vec v(t)).
\end{equation}
From \eqref{eq:g-coer} and Lemma~\ref{lem:data-at-T} \ref{it:data-at-T-2}, we obtain
\begin{equation}
\|\bs g_m(t)\|^2 \leq C_0/(\eee^L + t^2) \qquad\text{for all }m,
\end{equation}
hence, by the weak compactness of closed balls in separable Hilbert spaces,
\begin{equation}
\|\bs \phi(t) - \bs H(\vec a(t), \vec v(t))\|_\cE^2 \leq C_0/(\eee^L + t^2).
\end{equation}
We also have $\rho(\vec a(t)) = \lim_{m\to\infty} \rho(\vec a_m(t)) \leq C_0/(\eee^L + t^2)$,
yielding the required bound on $\bfd(\bs\phi(t))$ by Remark~\ref{lem:impr-coer}.
\end{proof}

\section{Kink clusters as profiles of kink collapse}
\label{sec:profiles}
In this final section, we prove Theorem~\ref{thm:unstable}.
In the proof, we will need the following ``Fatou property''.
\begin{lemma}
\label{lem:fatou}
There exist $\eta_0, y_0 > 0$ such that the following holds. Let $\vec a_m = (a_{m, 1}, \ldots, a_{m, n}) \in \bR^n$ and $\vec v_m = (v_{m, 1}, \ldots, v_{m, n}) \in \bR^n$
for all $m \in \{1, 2, \ldots\}$ be such that $a_{m, k+1} - a_{m, k} \geq y_0$ for all $m$ and $k$, and $\lim_{m \to \infty} \vec a_m = \vec a \in \bR^n$, $\lim_{m \to \infty} \vec v_m = \vec v \in \bR^n$.
Let $\bs g_m \in \cE$, $\|\bs g_m\|_\cE \leq \eta_0$, $\bs g_m$ satisfy the orthogonality conditions \eqref{eq:g-orth}
with $(\vec a, \vec v, \bs g)$ replaced by $(\vec a_m, \vec v_m, \bs g_m)$, and assume that $\bs g_m \wto \bs g \in \cE$. Then
\begin{equation}
\label{eq:fatou}
E(\bs H(\vec a, \vec v) + \bs g) \leq \liminf_{m \to \infty}E(\bs H(\vec a_m, \vec v_m) + \bs g_m),
\end{equation}
and $E(\bs H(\vec a, \vec v) + \bs g) = \lim_{m \to \infty}E(\bs H(\vec a_m, \vec v_m) + \bs g_m)$ if and only if $\bs g_m \to \bs g$ strongly in $\cE$.
\end{lemma}
\begin{proof}
Let $\Pi \subset \cE$ be the subspace of codimension $2n$ defined by \eqref{eq:g-orth}.
Let $\wt{\bs g}_m$ be the orthogonal projection of $\bs g_m$ on $\Pi$.
Then $\lim_{m \to \infty}\|\wt{\bs g}_m - \bs g_m\|_\cE = 0$, hence we are reduced to the situation where $\vec a_m = \vec a$ and $\vec v_m = \vec v$ for all $m$.
To simplify the notation, we write $\bs g_m$ instead of $\wt{\bs g}_m$.

We have the Taylor expansion around $\bs H(\vec a, \vec v) + \bs g$:
\begin{equation}
\begin{aligned}
E(\bs H(\vec a, \vec v) + \bs g_m) &= E(\bs H(\vec a, \vec v) + \bs g) + \la \vD E(\bs H(\vec a, \vec v) + \bs g), \bs g_m - \bs g\ra \\
&+ \frac 12 \la \vD^2 E(\bs H(\vec a, \vec v) + \bs g)(\bs g_m - \bs g), \bs g_m - \bs g\ra + O(\|\bs g_m - \bs g\|_\cE^3).
\end{aligned}
\end{equation}
The second term of the right hand side converges to 0 by assumption.
The second line is $\gtrsim \|\bs g_m - \bs g\|_\cE^2$, by the coercivity estimate in Lemma~\ref{lem:D2H} and smallness of $\|\bs g_m\|_\cE$.
\end{proof}
\begin{proof}[Proof of Theorem~\ref{thm:unstable}]
%\textbf{Step 1.} (Construction of the profiles.)
Let $\vec a_m: [0, T_m] \to \bR^n$ and $\vec v_m: [0, T_m] \to \bR^n$ be the modulation parameters corresponding to $\bs\phi_m$.
%By Lemma~\ref{lem:basic-mod}, we have $\bfd(\bs\phi_m(t)) \leq C_0(\rho(\vec a_m(t)) + E(\bs \phi_m) - nE_p(H))$.
Let $C_1$ be the constant in \eqref{eq:delta-geq-ass}.
Since $\lim_{m\to\infty}E(\bs \phi_m) = nM$,
there exists a sequence $\wt T_m \in [0, T_m]$
so that $\bfd(\bs\phi_m(t)) \geq C_1(E(\bs \phi_m) - nM)$ for all $m$
and $t \in [0, \wt T_m]$, but still $\lim_{m\to\infty}\bfd(\bs\phi_m(\wt T_m)) = 0$.
This condition allows to apply Lemma~\ref{lem:basic-mod} even though we do not assume $E(\bs \phi_m) = nM$.
Below, we write $T_m$ instead of $\wt T_m$.

Let $y_{m, k} \coloneqq a_{m, k+1} - a_{m, k}$ and $\wt\rho_m(t) \coloneqq \rho(\vec a_m(t))$,
as in Sections~\ref{sec:mod} and \ref{sec:n-body}.
Let $\beta_m: [0, T_m] \to \bR$ be the function defined as in the proof of Lemma~\ref{lem:ymin-asym}, corresponding to the solution $\bs \phi_m$. By the same argument, $\beta(t)$
is decreasing for $t \in [0, T_m]$, hence
\begin{equation}
\beta(t) \geq -\epsilon_m, \qquad\text{for all }t \in [0, T_m],
\end{equation}
where $\epsilon_m \coloneqq \beta(T_m) \to 0$ as $m \to \infty$.

Fix $t > 0$. For $m$ large enough and all $\tau \in [0, t]$,
we have $|{\wt\rho_m}\,\!\!\!\!'\,(\tau)| \lesssim \wt\rho_m(\tau)^\frac 32$, hence $\big|\dd t(\wt\rho_m(\tau)^{-\frac 12})\big| \lesssim 1$. Since $\wt\rho_m(0) \gtrsim \eta$, we obtain $\wt\rho_m(\tau)^{-\frac 12} \lesssim \eta^{-\frac 12} + \tau$, hence
\begin{equation}
\wt\rho_m(\tau) \gtrsim \frac{1}{\eta^{-1} + t^2}\qquad\text{for all }m\text{ large enough and }\tau \in [0, t].
\end{equation}
The proof of \eqref{eq:rho'-diff-ineq} yields
\begin{equation}
\dd t\big((\wt\rho_m(\tau))^{-\frac 12}\big) - \frac{c_0}{2} \gtrsim {-}\epsilon_m\wt\rho_m(\tau)^{-\frac 32} \gtrsim {-}\epsilon_m (\eta^{-\frac 32} + t^3).
\end{equation}
Integrating in time, we deduce that there is a constant $C_2$ such that
\begin{equation}
\label{eq:profiles-uni-d}
\limsup_{m\to\infty} \wt\rho_m(t) \leq \frac{C_2}{\eta^{-1} + t^2}.
\end{equation}

Upon extracting a subsequence, we can assume that $\lim_{m \to \infty} (a_{m, k+1}(0) - a_{m, k}(0)) \in \bR \cup \{\infty\}$ exists for all $k$.
Observe that at least one of these limits has to be finite, due to the bound \eqref{eq:g-coer} and the fact that $\bfd(\bs \phi_m(0)) = \eta$ for all $m$.
We set $n^{(0)} \coloneqq 0$ and define inductively
\begin{equation}
n^{(j)} \coloneqq \max \big\{k : \lim_{m\to \infty} a_{m, k}(0) - a_{m, n^{(j-1)}+1}(0) < \infty \big\},
\end{equation}
until we reach $n^{(\ell)} = n$ for some $\ell$. We set $X_m^{(j)} \coloneqq a_{m, n^{(j)}}(0)$,
so that $\lim_{m \to \infty} \big(a_{m, n^{(j-1)} + k}(0) - X_m^{(j)}\big) \in \bR$
for all $k \in \big\{1, \ldots, n^{(j)} - n^{(j-1)}\big\}$.
By the definition of $n^{(j)}$, conclusion (ii) of the theorem holds.

Again extracting a subsequence and applying the Arzel\`a-Ascoli theorem,
we can assume that $a_{m, n^{(j-1)} + k}(t) - X_m^{(j)}$
and $v_{m, n^{(j-1)} + k}(t)$ converge uniformly
on every bounded time interval. Hence, we can define
\begin{gather}
\vec a\,^{(j)}: [0, \infty) \to \bR^{n^{(j)} - n^{(j-1)}}, \qquad a\,^{(j)}_k(t) \coloneqq \lim_{m \to \infty} \big(a_{m, n^{(j-1)} + k}(t) - X_m^{(j)}\big).
\end{gather}
and
\begin{gather}
\vec v\,^{(j)}: [0, \infty) \to \bR^{n^{(j)} - n^{(j-1)}}, \qquad v\,^{(j)}_k(t) \coloneqq \lim_{m \to \infty} v_{m, n^{(j-1)} + k}(t).
\end{gather}
Note that $a\,^{(j)}_{n^{(j)} - n^{(j-1)}}(0) = 0$ for all $j$.
From the bound \eqref{eq:profiles-uni-d} we get $\rho\big(\vec a\,^{(j)}\big) \lesssim (\eta^{-1} + t^2)^{-1}$.

For $0 \leq j \leq \ell$, let $\iota^{(j)} \coloneqq (-1)^{n^{(j)}}$, $\bs\iota^{(j)} \coloneqq (\iota^{(j)}, 0)$ and for $1 \leq j < \ell$, let $I_m^{(j)} \coloneqq \big[\frac 89 X_m^{(j)} + \frac 19 X_m^{(j+1)},
\frac 19 X_m^{(j)} + \frac 89 X_m^{(j+1)}\big]$.
For fixed $t > 0$, we have $\lim_{m\to \infty}\|H(\vec a_m(t), \vec v_m(t)) - \iota^{(j)}\|_{H^1(I_m^{(j)})} = 0$,
thus \eqref{eq:profiles-uni-d} yields
\begin{equation}
\label{eq:Im-small-en}
\limsup_{m\to \infty} \|\bs\phi_m(t) - \bs\iota^{(j)}\|_{\cE(I_m^{(j)})}^2 \lesssim (\eta^{-1} + t^2)^{-1}.
\end{equation}
Let $J_m^{(j)} \coloneqq \big[\frac 56 X_m^{(j)} + \frac 16 X_m^{(j+1)},
\frac 16 X_m^{(j)} + \frac 56 X_m^{(j+1)}\big]$.
By \eqref{eq:Im-small-en} and Lemma~\ref{lem:loc-vac-stab}, we obtain
\begin{equation}
\label{eq:Jm-small-en}
\limsup_{m\to \infty} \|\bs\phi_m(0) - \bs\iota^{(j)}\|_{\cE(J_m^{(j)})}^2 \lesssim (\eta^{-1} + t^2)^{-1}.
\end{equation}
Letting $t \to \infty$, we get
\begin{equation}
\label{eq:Jm-small-en-0}
\limsup_{m\to \infty} \|\bs\phi_m(0) - \bs\iota^{(j)}\|_{\cE(J_m^{(j)})}^2 = 0.
\end{equation}

Next, we divide the initial data $\bs\phi_m(0)$ into $\ell$ regions in the following way. We set
\begin{equation}
\begin{aligned}
\chi^{(1)}_m(x) &\coloneqq \chi\Big(\frac{x - X^{(1)}_m}{X^{(2)}_m - X^{(1)}_m}\Big), \\
\chi^{(j)}_m(x) &\coloneqq \chi\Big(\frac{x - X^{(j)}_m}{X^{(j+1)}_m - X^{(j)}_m}\Big)
- \chi\Big(\frac{x - X^{(j-1)}_m}{X^{(j)}_m - X^{(j-1)}_m}\Big),\qquad\text{for }j \in \{2, \ldots, \ell-1\}, \\
\chi^{(\ell)}_m(x) &\coloneqq 1 - \chi\Big(\frac{x - X^{(\ell)}_m}{X^{(\ell)} - X^{(\ell-1)}_m}\Big)
\end{aligned}
\end{equation}
(in the case $\ell = 1$, we set $\chi^{(1)}_m(x) \coloneqq 1$).
%We also set $\chi_m^{(0)}(x) \coloneqq 0$ and $\chi_m^{(\ell + 1)}(x) \coloneqq 0$.
We now define
\begin{align}
\label{eq:phim0j-def}
\bs\phi^{(j)}_{m, 0} &\coloneqq \bs\iota^{(j-1)}\sum_{i = 1}^{j-1}\chi^{(i)}_m + \chi^{(j)}_m \bs \phi_m(0)
+ \bs\iota^{(j)}\sum_{i=j+1}^{\ell}\chi^{(i)}_m, &\text{for }j \in \{1, \ldots, \ell\}
\end{align}
(note that the two sums above are telescopic sums).
From \eqref{eq:Jm-small-en-0}, we deduce
\begin{equation}
\label{eq:phim0j-out}
\begin{gathered}
\lim_{m\to\infty} \|\bs\phi_{m, 0}^{(j)} - \bs\iota^{(j-1)}\|_{\cE(-\infty,\frac 56 X_m^{(j)}+ \frac 16 X_m^{(j-1)})} = 0, \\
\lim_{m\to\infty} \|\bs\phi_{m, 0}^{(j)} - \bs\iota^{(j)}\|_{\cE(\frac 56 X_m^{(j)} + \frac 16 X_m^{(j+1)}, \infty)} = 0.
\end{gathered}
\end{equation}

%We have $\bs\phi_m^{(j)} : [0, T_m] \to \cE_{\iota^{(j-1)}, \iota^{(j)}}$.
%For fixed $t \geq 0$, the triangle inequality yields the estimate
%\begin{equation}
%\begin{aligned}
%&\frac 12 \big\|\iota^{(j-1)}\bs\phi_m^{(j)}\big(t, \cdot + X_m^{(j)}\big) - \bs H\big(\vec a\,^{(j)}(t)\big) \big\|_\cE^2 \leq \| \chi^{(j)}_m (\bs \phi_m - \bs H(\vec a_m(t))) \|_\cE^2 \\
%&\qquad+\big\|\iota^{(j-1)}\chi^{(j-1)}_m + \chi_m^{(j)}\bs H(\vec a_m(t), \cdot + X_m^{(j)}) + \iota^{(j)}\chi^{(j+1)}_m - \iota^{(j-1)} \bs H\big(\vec a\,^{(j)}(t)\big) \big\|_\cE^2.
%\end{aligned}
%\end{equation}
%The exponential decay of $H$ and the definition of $\vec a\,^{(j)}$ imply that
%the second line converges to $0$ as $m \to \infty$. From \eqref{eq:profiles-uni-d},
%we thus obtain
%\begin{equation}
%\label{eq:profiles-phim-bound}
%\limsup_{m \to \infty}\big\|\iota^{(j-1)}\bs\phi_m^{(j)}\big(t, \cdot + X_m^{(j)}\big) - \bs H\big(\vec a\,^{(j)}(t)\big) \big\|_\cE^2 \lesssim \frac{1}{\eta^{-1} + t^2}.
%\end{equation}
Let $\wt {\bs \phi}_m^{(j)}$ be the solution of \eqref{eq:csf} for the initial data
$\wt{\bs \phi}_m^{(j)}(0) = \bs\phi_{m, 0}^{(j)}$.
Let
\begin{equation}
\begin{aligned}
K_m^{(1)} &\coloneqq \big({-}\infty, \frac 79 X_m^{(1)} + \frac 29 X_m^{(2)}\big], \\
K_m^{(j)} &\coloneqq \big[\frac 79 X_m^{(j)} + \frac 29 X_m^{(j-1)}, \frac 79 X_m^{(j)} + \frac 29 X_m^{(j+1)}\big]\qquad\text{for }j \in \{2, \ldots, \ell -1\}, \\
K_m^{(\ell)} &\coloneqq \big[\frac 79 X_m^{(\ell)} + \frac 29 X_m^{(\ell-1)}, \infty\big).
\end{aligned}
\end{equation}
By the finite speed of propagation, see Proposition~\ref{prop:cauchy} \ref{it:cauchy-speed},
and the definition of $\bs \phi_m^{(j)}$, we have
\begin{equation}
\wt{\bs \phi}_m^{(j)}(t)\vert_{K_m^{(j)}} = \bs\phi_m(t)\vert_{K_m^{(j)}}\qquad
\text{for any given }t \geq 0\text{ and }m\text{ large enough},
\end{equation}
hence \eqref{eq:profiles-uni-d} and \eqref{eq:g-coer-stat} imply
\begin{equation}
\label{eq:phimjt-in}
\lim_{m\to\infty} \|\wt{\bs\phi}_{m}^{(j)}(t) - \iota^{(j-1)}\bs H(\vec a\,^{(j)}(t), \vec v\,^{(j)}(t); \cdot - X_m^{(j)})\|_{\cE(K_m^{(j)})} \lesssim \frac{1}{\eta^{-1} + t^2}.
\end{equation}
For any given $t \geq 0$, Lemma~\ref{lem:loc-vac-stab} and \eqref{eq:phim0j-out} yield
\begin{equation}
\label{eq:phimjt-out}
\begin{gathered}
\lim_{m\to\infty} \|\wt{\bs\phi}_{m}^{(j)}(t) - \bs\iota^{(j-1)}\|_{\cE(-\infty,\frac 79 X_m^{(j)}+ \frac 29 X_m^{(j-1)})} = 0, \\
\lim_{m\to\infty} \|\wt{\bs\phi}_{m}^{(j)}(t) - \bs\iota^{(j)}\|_{\cE(\frac 79 X_m^{(j)} + \frac 29 X_m^{(j+1)}, \infty)} = 0.
\end{gathered}
\end{equation}
Invoking \eqref{eq:phimjt-in}, we obtain that for all $t \geq 0$
\begin{equation}
\label{eq:profiles-wtphim-bound}
\limsup_{m \to \infty}\big\|\iota^{(j-1)}\wt{\bs\phi}_m^{(j)}\big(t, \cdot + X_m^{(j)}\big) - \bs H\big(\vec a\,^{(j)}(t), \vec v\,^{(j)}(t)\big) \big\|_\cE^2 \lesssim \frac{1}{\eta^{-1} + t^2}.
\end{equation}

After extraction of a subsequence, we can assume that, for all $j \in \{1, \ldots, \ell\}$,
\begin{equation}
\label{eq:P0j-def}
\iota^{(j-1)}\bs \phi_{m, 0}^{(j)}\big(\cdot + X_m^{(j)}\big) \wto \bs P_0^{(j)} \in \cE_{1, \iota^{(j-1)}\iota^{(j)}}.
\end{equation}
%
%\textbf{Step 2.} (The profiles are kink clusters.)
Let $\bs P^{(j)}$ be the solution of \eqref{eq:csf} such that $\bs P^{(j)}(0) = \bs P^{(j)}_0$.
By Proposition~\ref{prop:cauchy}
\ref{it:cauchy-weak} and \eqref{eq:profiles-wtphim-bound}, $\bs P^{(j)}$ is a kink cluster,
so conclusion (i) of the theorem holds.
In particular, $E(\bs P^{(j)}) = (n^{(j)} - n^{(j-1)})M$.
From \eqref{eq:Jm-small-en} and \eqref{eq:phim0j-def}, we have
\begin{equation}
\lim_{m\to\infty}\sum_{j=1}^\ell E(\bs\phi_{m, 0}^{(j)}) = \lim_{m\to\infty}E(\bs\phi_m) =
nM = \sum_{j=1}^\ell E(\bs P^{(j)}).
\end{equation}
By the last part of Lemma~\ref{lem:fatou}, we thus have strong convergence in \eqref{eq:P0j-def}
for all $j \in \{1, \ldots, \ell\}$.
Elementary algebra yields
\begin{equation}
\bs\phi_m(0) = \bs 1 + \sum_{j=1}^\ell \big(\bs\phi_{m, 0}^{(j)} - \bs \iota^{(j-1)}\big),
\end{equation}
implying conclusion (iii) of the theorem.
\end{proof}

\section{Well-posedness of the linearized problem}
\label{sec:cauchy-fin}
If $I \subset \bR$ is a time interval and $X$ a normed space, we denote by
$C(I; X)$ the space of continuous functions $I \to X$,
and $C_0(I; X)$ its (non-closed) subspace of functions with compact support $J \subset I$.
If $I$ is open and $k \in \{0, 1, \ldots, \infty\}$, we denote by
$C^k(I; X)$ the space of $k$ times continuously differentiable functions $I \to X$,
and $C_0^k(I; X)$ its subspace of functions with compact support $J \subset I$.
If $X$ is a Banach space,
we denote by $L^p(I; X)$ the usual Lebesgue space of $X$-valued functions on $I$,
see \cite[III.3]{dunford-schwartz}.
If $1 \leq p < \infty$, then $C_0(I; X)$ is dense in $L^p(I; X)$, see \cite[IV.8.19]{dunford-schwartz}.
It follows that if $Y \subset X$ is a dense subspace, then $C_0(I; Y)$ is dense in $L^p(I; X)$,
since any element of $C_0(I; X)$ can be approximated by $Y$-valued piecewise affine functions.

We denote the free energy
\begin{equation}
E_0(\bs h_0) \coloneqq  \int_{-\infty}^\infty \Big(\frac 12\dot h_0(x)^2 + \frac 12(\partial_x h_0(x))^2
+ \frac 12 h_0(x)^2 \Big)\ud x.
\end{equation}
The Klein-Gordon equation $\partial_t^2 h(t, x)
 = \partial_x^2 h(t, x) - h(t, x)$
 can be rewritten in the Hamiltonian form as
 \begin{equation}
 \partial_t \bs h(t) = \bs J\vD E_0(\bs h(t)) = \bs J\vD^2 E_0(0)\bs h(t).
 \end{equation}
 We denote the free propagator by $\bs S_0(t)$, which is given by 
 \EQ{
 \bs S_0(t) = e^{t J \vD E_0} =  \pmat{ \cos (t \omega) &  \omega^{-1} \sin (t \omega) \\ -\omega \sin (t \omega)& \cos (t \omega)}
 }
 where we have used the short-hand notation $\om = \sqrt{1 + | \p_x^2|}$ above. 
 \begin{definition}
 
 Let $I \subset \bR$ be an open interval, $T_0 \in I$,
 $\bs h_0 \in \cE$ and $\bs f \in L^1(I; \cE)$.
 We say that $\bs h \in C(I; \cE)$ is the solution of the Cauchy problem
 \begin{align}
 \label{eq:kg-free-lin-eq}
 &\partial_t \bs h(t) = \bs J\vD E_0(\bs h(t)) + \bs f(t), \\
 \label{eq:kg-free-lin-init}
 &\bs h(T_0) = \bs h_0
 \end{align}
 if $\bs h$ is given by the Duhamel formula 
 \EQ{ \label{eq:duhamel} 
 \bs h(t) = \bs S_0(t) \bs h_0 + \int_0^t \bs S_0(t-s) \bs f(s) \, \ud s. 
 }
 \end{definition}
 \begin{remark}
 It can be proved that the definition above is equivalent
 to the notion of a weak solution, see \cite{Friedrichs54}
 for more general results, but %\red{most likely}
 we will not need this fact.
 \end{remark}

\begin{proposition}
\label{prop:kg-free-lin}
Let $I \subset \bR$ be an open interval, $T_0 \in I$,
$\bs h_0 \in \cE$, $\bs f \in L^1(I; \cE)$
and let $\bs h \in C(I; \cE)$ be the solution of
\eqref{eq:kg-free-lin-eq}--\eqref{eq:kg-free-lin-init}.
For all $t\in I$ the following \emph{energy estimate} holds:
\begin{equation}
\label{eq:kg-free-lin-en-est}
\|\bs h(t)\|_\cE \leq \|\bs h_0\|_\cE + \bigg|\int_{T_0}^t \|\bs f(s)\|_{\cE}\ud s\bigg|.
\end{equation}
\end{proposition}
\begin{proof}
This follows from the Duhamel formula \eqref{eq:duhamel}.
\end{proof}

In the remaining part of this section $I$ will denote an open time interval and $(\vec a, \vec v) \in C^1(I; \bR^n \times \bR^n)$ are fixed
modulation parameters such that
\EQ{
\rho(\vec a(t), \vec v(t)) \le \eta_0 \quad \forall \, t \in I, 
}
where $\eta_0$ is as in Proposition~\ref{prop:coer}. 

 \begin{definition}
 Let $I \subset \bR$ be an open interval, $T_0 \in I$,
 $\bs h_0 \in \cE$ and $\bs f \in L^1(I; \cE)$.
 We say that $\bs h \in C(I; \cE)$ is the solution of the Cauchy problem
 \begin{align}
 \label{eq:kg-lin-eq}
 &\partial_t \bs h(t) = \bs J\vD^2 E(\bs H(\vec a(t), \vec v(t)))\bs h(t) + \bs f(t), \\
 \label{eq:kg-lin-init}
 &\bs h(T_0) = \bs h_0
 \end{align}
 if $\bs h$ is the solution of \eqref{eq:kg-free-lin-eq}--\eqref{eq:kg-free-lin-init} with $\bs f$ replaced by
 \begin{equation}
 \wt{\bs f}(t) \coloneqq  \bs J(\vD^2 E(\bs H(\vec a(t), \vec v(t))) - \vD^2 E_0(0))\bs h(t) + \bs f(t).
 \end{equation}
 \end{definition}
\begin{lemma}
\label{lem:kg-lin}
Let $I \subset \bR$ be an open interval and $T_0 \in I$.
For every $\bs h_0 \in \cE$ and $\bs f \in L^1(I; \cE)$
the problem \eqref{eq:kg-lin-eq}--\eqref{eq:kg-lin-init}
has a unique solution $\bs h \in C(I; \cE)$.

%For all $t \in I$ the following \emph{energy estimate} holds:
%\begin{equation}
%\|\bs h(t)\|_\cE
%\end{equation}

If $\bs \xi \in C^1(I; L^2\times H^1)$, then for all $t \in I$
\begin{equation}
\label{eq:dt-prod}
\dd t\la \bs \xi(t), \bs h(t)\ra =
\la \partial_t \bs \xi(t) - \vD^2 E(\bs H(\vec a(t), \vec v(t)))\bs J\bs \xi(t), \bs h(t)\ra + \la \bs \xi(t), \bs f(t)\ra.
\end{equation}
\end{lemma}
\begin{proof}
We use the Picard iteration scheme.
We set $\bs h^{(0)}(t) \coloneqq  0$ for all $t \in I$,
and define by induction $\bs h^{(n)}$ for all $n \geq 1$
as the solution of the equation
\begin{equation}
\partial_t \bs h^{(n)}(t) = \bs J\vD^2 E_0(0)\bs h^{(n)}(t)
+ \bs J\big( \vD^2 E(\bs H(\vec a(t), \vec v(t))) - \vD^2 E_0(0) \big)\bs h^{(n-1)}(t) + \bs f(t)
\end{equation}
with the initial data $\bs h^{(n)}(T_0) = \bs h_0$,
given by Proposition~\ref{prop:kg-free-lin}.

Observing that $\bs J\big( \vD^2 E(\bs H(\vec a(t), \vec v(t))) - \vD^2 E_0(0) \big)$ is a uniformly (with respect to $t \in I$) bounded operator $\cE \to \cE$
(in fact, it is even regularizing, with a gain of one derivative)
and applying \eqref{eq:kg-free-lin-en-est}, we obtain for all $t \in I$
and $n \in \{1, \ldots\}$
\begin{equation}
\|\bs h^{(n+1)}(t) - \bs h^{(n)}(t)\|_\cE
\leq C\bigg|\int_{T_0}^t \|\bs h^{(n)}(s) - \bs h^{(n-1)}(s)\|_\cE \ud s\bigg|.
\end{equation}
Hence, by straightforward induction,
\begin{equation}
\|\bs h^{(n+1)}(t) - \bs h^{(n)}(t)\|_\cE
\leq \frac{C^n|t - T_0|^{n-1}}{(n-1)!}\bigg|\int_{T_0}^t \|\bs h^{(1)}(s)\|_\cE \ud s\bigg|,
\end{equation}
which implies that $\bs h^{(n)}$ converges in $L_\tx{loc}^\infty(I; \cE)$
to some $\bs h: I \to \cE$.
By the continuity statement in Proposition~\ref{prop:kg-free-lin},
we obtain that $\bs h$ is a solution of~\eqref{eq:kg-lin-eq}--\eqref{eq:kg-lin-init}. 

If $\wt{\bs h}$ is another solution, then \eqref{eq:kg-free-lin-en-est} yields
\begin{equation}
\|\wt{\bs h}(t) - \bs h(t)\|_\cE \leq C\bigg|\int_{T_0}^t \|\wt{\bs h}(s) - \bs h(s)\|_\cE \ud s\bigg|,
\end{equation}
thus $\wt{\bs h} = \bs h$ by the Gronwall's inequality. The identity~\eqref{eq:dt-prod} follows using the Duhamel formula~\eqref{eq:duhamel} for $\bs h(t)$ together with an approximation argument. 
\end{proof}
\begin{remark}
In order for \eqref{eq:dt-prod} to hold, it would be sufficient
to assume that $\bs \xi \in C(I; L^2 \times H^1)$
and $\partial_t \bs \xi \in C(I; H^{-1}\times L^2)$.
\end{remark}

\begin{definition}
Let $\wt \cE \in \{\cE_{-,-}, \cE_{-,+}, \cE_{+,-}, \cE_{+,+}\}$
and $\bs \phi_0 \in \wt \cE$.
We say that $\bs \phi: I \to \wt \cE$ is a solution of \eqref{eq:csf} if $\bs h(t) \coloneqq  \bs \phi(t) - \bs \phi_0$
is a solution of \eqref{eq:kg-free-lin-eq}--\eqref{eq:kg-free-lin-init} with $\bs h_0 = 0$ and
\begin{equation}
\bs f(t) \coloneqq  \bs J\vD E( \bs \phi_0 + \bs h(t)) - \bs J\vD E_0(\bs h(t))
\end{equation}
\end{definition}
\begin{proposition}
For all $\bs \phi_0 \in \wt \cE$ the problem
\eqref{eq:csf}
has a unique solution.
It is global in time and depends continuously on the initial data.
\end{proposition}
\begin{proof}
This follows from the usual Picard iteration scheme.
\end{proof}

%\red{The matrix for computing $(\vec \lambda, \vec \mu)$.}
%We define the matrix $\cM(\vec a, \vec v) = (m_{kl}) \in \bR^{2n\times 2n}$
%by the following formulas:
%\begin{align}
%m_{2k-1, 2l-1} \coloneqq  \ldots
%\end{align}
%We have computed in Section~\ref{ssec:travelling}
%the terms with $k = l$. The remaining terms are
%exponentially small with respect to the distance between the kinks.

Recall the matrix $\calM(\vec a, \vec v)$ defined in~\eqref{eq:Mav-def} and denote by $\calM(\vec a, \vec v)^{\intercal}$ its transpose. We recall again for convenience the unconventional notation 
\EQ{
\pmat{\vec \mu \\ \vec \lam} \coloneqq  \pmat{ \mu_1 \\ \lam_1 \\ \dots \\ \mu_n \\  \lam_n}.
}
\begin{lemma}
\label{lem:kg-lin-orth}
Let $I \subset \bR$ be an open interval and $T_0 \in I$.
For every $\bs h_0 \in \cE$ such that
\EQ{ \label{eq:orth-in-data} 
\la \bs\alpha_k(T_0), \bs h_0\ra = 
\la \bs\beta_k(T_0), \bs h_0\ra = 0\qquad\text{for all }k \in \{1, \ldots, n\}
}
and $\bs f \in L^1(I; \cE)$
there exist unique $\vec \lambda, \vec \mu \in L^1(I; \bR^n)$
such that the solution $\bs h \in C(I; \cE)$ of the problem
\eqref{eq:kg-lin-eq}--\eqref{eq:kg-lin-init}
with $\bs f(t)$ replaced by
\EQ{ \label{eq:new-f} 
\bs f(t) + \sum_{k=1}^n(\lambda_k(t) \partial_{a_k}\bs H_k(t)
+ \mu_k(t) \partial_{v_k}\bs H_k(t))
}
satisfies for all $t \in I$ and $k \in \{1, \ldots, n\}$
\begin{equation}
\label{eq:h-orth}
\la \bs\alpha_k(t), \bs h(t)\ra = 
\la \bs\beta_k(t), \bs h(t)\ra = 0.
\end{equation}
Moreover, there is the relation %\red{Introduce appropriate notation for the matrix and correct the formula below.}
%\begin{equation}
%\label{eq:val-lambda-mu}
%\begin{pmatrix}
%\vec\lambda(t) \\ \vec\mu(t)
%\end{pmatrix} =
%{-}\begin{pmatrix}
%\la \bs\alpha_k(t), \partial_a \bs H\ra & \la \bs\alpha, \partial_v \bs H\ra \\
%\la \bs\beta, \partial_a \bs H\ra & \la \bs \beta, \partial_v \bs H\ra
%\end{pmatrix}^{-1}
%\begin{pmatrix}
%\la \partial_t \bs\alpha - \vD^2 E(\bs H)\bs J\bs \alpha, \bs h(t)\ra
%+ \la \bs \alpha, \bs f\ra \\
%\la \partial_t \bs\beta - \vD^2 E(\bs H)\bs J\bs \beta, \bs h(t)\ra
%+ \la \bs \beta, \bs f\ra\end{pmatrix}.
%\end{equation}
\EQ{
\label{eq:val-lambda-mu}
\begin{pmatrix}
\vec\mu \\ \vec\lam
\end{pmatrix} = 
- \big[\calM(\vec a, \vec v)^{-1} \big]^{\intercal} \pmat{ \la \bs h , \, \p_t \bs \al_1 - \vD^2 E( \bs H(\vec a, \vec v)) \bs J \bs \alpha_1 \ra + \la \bs f, \, \bs \al_1\ra \\  \la \bs h , \, \p_t \bs \be_1 - \vD^2 E( \bs H(\vec a, \vec v)) \bs J \bs \be_1 \ra + \la \bs f, \, \bs \be_1\ra \\ \dots \\ \dots \\  \la \bs h , \, \p_t \bs \al_n - \vD^2 E( \bs H(\vec a, \vec v)) \bs J \bs \alpha_n \ra + \la \bs f, \, \bs \al_n\ra \\ \la \bs h , \, \p_t \bs \be_n - \vD^2 E( \bs H(\vec a, \vec v)) \bs J \bs \beta_n \ra + \la \bs f, \, \bs \beta_n\ra }.
}
\end{lemma}

\begin{proof}
We follow the Picard iteration scheme.
We set $\vec\lambda^{(0)}(t) = \vec\mu^{(0)}(t) = 0$ for all $t \in I$,
and define inductively for all $\ell \geq 1$ the function $\bs h^{(\ell)}(t)$
as the solution of the problem \eqref{eq:kg-lin-eq}--\eqref{eq:kg-lin-init} with $\bs f(t)$ replaced by
\begin{equation}
\bs f(t) + \sum_{k=1}^n\big(\lambda^{(\ell-1)}_k(t)\partial_{a_k}\bs H_k(t)
+ \mu^{(\ell-1)}_k(t) \partial_{v_k}\bs H_k(t)\big),
\end{equation}
and
$(\lambda^{(\ell)}(t), \vec\mu^{(\ell)}(t))$ as the solution
of the linear system
\EQ{
\begin{pmatrix}
\vec\mu^{(\ell)}(t) \\ \vec\lam^{(\ell)}(t)
\end{pmatrix} = 
- \big[\calM(\vec a(t), \vec v(t))^{-1} \big]^{\intercal} \pmat{ \la \bs h^{(\ell)}(t) , \, \p_t \bs \al_1(t) - \vD^2 E( \bs H(\vec a(t), \vec v(t))) J \bs \alpha_1(t) \ra + \la \bs f, \, \bs \al_1(t)\ra \\  \la \bs h^{(\ell)}(t)  , \, \p_t \bs \be_1(t) - \vD^2 E( \bs H(\vec a(t), \vec v(t))) J \bs \be_1(t) \ra + \la \bs f, \, \bs \be_1(t)\ra \\ \dots \\ \dots \\  \la \bs h^{(\ell)}(t)  , \, \p_t \bs \al_n(t) - \vD^2 E( \bs H(\vec a(t), \vec v(t))) J \bs \alpha_n(t) \ra + \la \bs f, \, \bs \al_n(t)\ra \\ \la \bs h^{(\ell)}(t)  , \, \p_t \bs \be_n(t) - \vD^2 E( \bs H(\vec a(t), \vec v(t))) J \bs \beta_n(t) \ra + \la \bs f, \, \bs \beta_n(t)\ra }
}
%\begin{equation}
%\begin{pmatrix}
%\vec\lambda^{(n)} \\ \vec\mu^{(n)}
%\end{pmatrix} =
%{-}\begin{pmatrix}
%\la \bs\alpha, \partial_a \bs H\ra & \la \bs\alpha, \partial_v \bs H\ra \\
%\la \bs\beta, \partial_a \bs H\ra & \la \bs \beta, \partial_v \bs H\ra
%\end{pmatrix}^{-1}
%\begin{pmatrix}
%\la \partial_t \bs\alpha - \vD^2 E(\bs H)\bs J\bs \alpha, \bs h^{(n)}\ra
%+ \la \bs \alpha, \bs f\ra \\
%\la \partial_t \bs\beta - \vD^2 E(\bs H)\bs J\bs \beta, \bs h^{(n)}\ra
%+ \la \bs \beta, \bs f\ra\end{pmatrix}.
%\end{equation}
We then have
\begin{equation}
|\vec\lambda^{(\ell+1)}(t) - \vec\lambda^{(\ell)}(t)| +
|\vec\mu^{(\ell+1)}(t) - \vec\mu^{(\ell)}(t)| \lesssim \|\bs h^{(\ell+1)}(t) - \bs h^{(\ell)}(t)\|_\cE
\end{equation}
and
\begin{equation}
\|\bs h^{(\ell+1)}(t) - \bs h^{(\ell)}(t)\|_\cE \lesssim \bigg|\int_{T_0}^t\big(|\vec \lambda^{(\ell)}(s) - \vec\lambda^{(\ell-1)}(s)|
+|\vec \mu^{(\ell)}(s) - \vec\mu^{(\ell-1)}(s)|\big)\ud s\bigg|.
\end{equation}
We conclude similarly as in the proof of Lemma~\ref{lem:kg-lin} above
that $(\vec\lambda^{(\ell)}, \vec\mu^{(\ell)}, \bs h^{(\ell)})$ converge, yielding in particular the relationship~\eqref{eq:val-lambda-mu}. Using the identity~\eqref{eq:dt-prod} with $\bs \xi = \bs \al_k, \bs \beta_k$ and $\bs f$ replaced by~\eqref{eq:new-f} we see that~\eqref{eq:val-lambda-mu} ensures that 
\EQ{ \label{eq:dt-orth} 
\frac{\ud}{\ud t} \la \bs \al_k(t), \, \bs h(t)\ra = \frac{\ud}{\ud t} \la \bs \beta_k(t), \, \bs h(t)\ra = 0
}
for each $t \in I$ and each $k$. Together with~\eqref{eq:orth-in-data} the above gives~\eqref{eq:h-orth}.  

In order to check the uniqueness of $(\vec \lambda, \vec \mu)$, it suffices to prove that if $\bs h_0 = 0$, $\bs f = 0$
and $\bs h$ satisfies \eqref{eq:h-orth} for all $t$, then $\vec \lambda = \vec \mu = 0$.
From~\eqref{eq:val-lambda-mu}, we obtain
\begin{equation}
|\vec \lambda(t)| + |\vec \mu(t)| \lesssim \|\bs h(t)\|_\cE
\end{equation}
and from the estimate~\eqref{eq:kg-free-lin-en-est} and the fact that $\bs J\big( \vD^2 E(\bs H(\vec a(t), \vec v(t))) - \vD^2 E_0(0) \big)$ is a uniformly (with respect to $t \in I$) bounded operator $\cE \to \cE$, we obtain
\begin{equation}
\|\bs h(t)\|_\cE \lesssim \bigg|\int_{T_0}^t \big( \| \bs h(s) \|_{\E} + |\vec \lambda(s)|
+|\vec \mu(s)|\big)\ud s\bigg|,
\end{equation}
hence $\bs h(t) = 0$ and $\vec\lambda(t) = \vec\mu(t) = 0$ for all $t$ by Gronwall's inequality.
\end{proof}

%\red{This is an attempt of rewriting the virial correction computation
%in some ``functional analytic'' way; unclear if this is of any use.}

We now turn to the question of various refined energy estimates. To motivate the estimates, we first examine the free evolution~\eqref{eq:kg-free-lin-eq}--\eqref{eq:kg-free-lin-init}. 

\begin{remark}
We differentiate a localised momentum for the free evolution
and see what we obtain. Let $\chi$ be any function depending on $t$ and $x$, and set
\begin{equation}
Q(t; \bs h_1, \bs h_2) \coloneqq  {-}\frac 12 \big(\big\la\chi(t)\dot h_2, \partial_x h_1\big\ra
+ \big\la \chi(t)\dot h_1, \partial_x h_2\big\ra\big).
\end{equation}
Let $\bs h$ be a solution of the problem \eqref{eq:kg-free-lin-eq}--\eqref{eq:kg-free-lin-init}.
We then obtain
\begin{equation}
\begin{aligned}
\dd t Q(t; \bs h(t), \bs h(t)) &= \partial_t Q(t; \bs h(t), \bs h(t)) \\
&+ 2Q(t; \bs h(t), \bs f(t)) \\
& - \int_\bR\big(\chi(t)(\partial_x^2 h(t) - h(t))\partial_x h(t) + \chi(t)\dot h(t)\partial_x\dot h(t)\big)\ud x \\
&= {-}\int_\bR \partial_t \chi(t)\dot h(t) \partial_x h(t)\ud x \\
&{-} \int_\bR \big(\chi(t)\dot f(t)\partial_x h(t) + \chi(t)\dot h(t) \partial_x f(t)\big)\ud x \\
&+ \frac 12\int_\bR \partial_x \chi(t)\big((\partial_x h(t))^2 + (\dot h(t))^2 - h(t)^2\big)\ud x.
\end{aligned}
\end{equation}
\end{remark}

Next, we adapt the localized momentum to the linearization about a multikink, i.e.,~\eqref{eq:kg-lin-eq}--\eqref{eq:kg-lin-init}. For $k \in \{1, \ldots, n\}$, let $\chi_k$ be a bump function centered at $a_k(t)$, so that $\partial_t \chi_k$ is of size $\lesssim |\vec v(t)|$ and $\|\partial_x \chi_k\|_{L^\infty}
\lesssim y_\tx{min}(t)^{-1}$.
For all $k \in \{1, \ldots, n\}$ and $t \in I$, we set
\begin{equation}
\label{eq:loc-momentum}
Q_k(t; \bs h_1, \bs h_2) \coloneqq  {-}\frac 12 \big(\big\la\chi_k(t)\dot h_2, \partial_x h_1\big\ra
+ \big\la \chi_k(t)\dot h_1, \partial_x h_2\big\ra\big),
\end{equation}
which is a symmetric continuous bilinear function $\cE \times \cE \to \bR$.
\begin{lemma}
\label{lem:loc-momentum}
In the setting of Lemma~\ref{lem:kg-lin}, we have
\begin{equation}
\begin{aligned}
\dd tQ_k(t; \bs h(t), \bs h(t)) &= 2Q_k(t; \bs h(t), \bs f(t)) \\
&{-}\frac 12\vD^3 E(\bs H(\vec a(t), \vec v(t))(\chi_k(t)\partial_x \bs H(\vec a(t), \vec v(t)), \bs h(t), \bs h(t)) \\
&+ O((|\vec v(t)| + y_\tx{min}(t)^{-1})\|\bs h(t)\|_\cE^2).
\end{aligned}
\end{equation}
\end{lemma}
\begin{proof}
With respect to the computation in the Remark above, we have the additional term
\begin{equation}
\int_\bR \chi_k(t)U''(H(\vec a(t), \vec v(t)))h(t)\partial_x h(t)\ud x.
\end{equation}
Integrating by parts, we get a negligible term and
\begin{equation}
\begin{aligned}
&-\frac 12\int_\bR \chi_k(t)\partial_x H(\vec a(t), \vec v(t))U'''(H(\vec a(t), \vec v(t)))h(t)^2\ud x \\
&\qquad = {-}\frac 12\vD^3 E(\bs H(\vec a(t), \vec v(t))(\chi_k(t)\partial_x \bs H(\vec a(t), \vec v(t)), \bs h(t), \bs h(t)).
\end{aligned}
\end{equation}
as claimed. 
\end{proof}

\begin{lemma}
\label{lem:en-est-lin}
In the setting of Lemma~\ref{lem:kg-lin}, assume in addition that
\begin{equation}
\Big(\sum_{k=1}^n\big(|\la \bs \alpha_k(t), \bs h(t)\ra|^2+ |\la \bs\beta_k(t), \bs h(t)\ra|^2\big)\Big)^{\frac{1}{2}} \leq \delta(t),\qquad\text{for all }t \in I
\end{equation}
for some continuous function $\delta: I \to [0, \infty)$. Then for all $t \in I$
\begin{equation}
\begin{aligned}
\|\bs h(t)\|_\cE &\leq C\|\bs h_0\|_\cE\exp\bigg(C\bigg|\int_{T_0}^t k(s)\ud s\bigg|\bigg) \\
&+ C\bigg|\int_{T_0}^t (\|\bs f(s)\|_\cE + \delta(s))\exp\Big(C\Big|\int_{s}^t k(s')\ud s'\Big|\Big)\ud s\bigg|,
\end{aligned}
\end{equation}
where
\begin{equation}
 \label{eq:k-def} k(t) \coloneqq  |\vec a\,'(t) - \vec v(t)| + |\vec v\,'(t)| + |\vec v(t)|^2 + |\vec v(t)|y_\tx{min}(t)^{-1} + y_\tx{min}(t)\eee^{-y_\tx{min}(t)}.
\end{equation}
\end{lemma}

%We also need a refinement of the previous lemma pointed out to us by Y. Martel.  
%\begin{lemma}
%\label{lem:en-est-lin-trick}
%In the setting of Lemma~\ref{lem:kg-lin-orth}, assume in addition that $\p_t \bs f \in L^1(I; \E)$  and 
%%\begin{equation}
%%\Big(\sum_{k=1}^n\big(|\la \bs \alpha_k(t), \bs h(t)\ra|^2+ |\la \bs\beta_k(t), \bs h(t)\ra|^2\big)\Big)^{\frac{1}{2}} = 0,\qquad\text{for all }t \in I. 
%%\end{equation}
% Then for all $t \in I$ \Red{[need to change conclusion]}
%\begin{equation}
%\begin{aligned}
%\|\bs h(t)\|_\cE &\leq C\|\bs h_0\|_\cE\exp\bigg(C\bigg|\int_{T_0}^t k(s)\ud s\bigg|\bigg) \\
%&+ C\bigg|\int_{T_0}^t (\|\bs f(s)\|_\cE + \delta(s))\exp\Big(C\Big|\int_{s}^t k(s')\ud s'\Big|\Big)\ud s\bigg|,
%\end{aligned}
%\end{equation}
%where
%\begin{equation}
% \label{eq:k-def} k(t) \coloneqq  |\vec a\,'(t) - \vec v(t)| + |\vec v\,'(t)| + |\vec v(t)|^2 + |\vec v(t)|y_\tx{min}(t)^{-1} + y_\tx{min}(t)\eee^{-y_\tx{min}(t)}.
%\end{equation}
%\end{lemma}

\begin{proof}[Proof of Lemma~\ref{lem:en-est-lin}] 
First, observe that
\begin{equation}
	\begin{aligned}
&\dd t \frac 12\la \bs h(t), \vD^2 E(\bs H(\vec a(t), \vec v(t)))\bs h(t)\ra = \la \bs f(t), \vD^2 E(\bs H(\vec a(t), \vec v(t)))\bs h(t)\ra \\
&\quad + \frac 12 \vD^3 E(\bs H(\vec a(t), \vec v(t)))\Big(\sum_{k=1}^n (-1)^k \big(a_k'(t) \partial_{a}\bs H_k(t) + v_k'(t)\partial_{v}\bs H_k(t)\big), \bs h(t), \bs h(t)\Big).
\end{aligned}
\end{equation}
The most problematic term is the one containing $a_k'(t)$.
Because of this term, we introduce a correction to the energy and introduce the following \emph{modified energy functional}:
\begin{equation}
I(t) \coloneqq  \frac 12\la \bs h(t), \vD^2 E(\bs H(\vec a(t), \vec v(t)))\bs h(t)\ra - \sum_{k=1}^n v_k Q_k(t; \bs h(t), \bs h(t)),
\end{equation}
where $Q_k$ is defined by \eqref{eq:loc-momentum}. Now observe that
\begin{equation}
\|\chi_k(t)\partial_x \bs H(\vec a(t), \vec v(t)) + (-1)^k\partial_{a}\bs H_k(t)\|_\cE \lesssim \eee^{-c y_\tx{min}} \ll y_\tx{min}^{-1},
\end{equation}
thus Lemma~\ref{lem:loc-momentum} yields
\begin{equation}
\begin{aligned} \label{eq:I'} 
I'(t) &= \la \bs f(t), \vD^2 E(\bs H(\vec a(t), \vec v(t)))\bs h(t)\ra \\
&+ O\big((|\vec a\,'(t) - \vec v(t)| + |\vec v\,'(t)| + |\vec v(t)|^2 + |\vec v(t)|y_\tx{min}(t)^{-1})\|\bs h(t)\|_\cE^2\big) \\
&+ O\big(|\vec v(t)|\|\bs f(t)\|_\cE\|\bs h(t)\|_\cE\big).
\end{aligned}
\end{equation}

We set
\begin{equation}
\wt I(t) \coloneqq  I(t) + \sum_{k=1}^n\big(\la \bs \alpha_k(t), \bs h(t)\ra^2
+ \la \bs \beta_k(t), \bs h(t)\ra^2\big).
\end{equation}
By Proposition~\ref{prop:coer}, we have
\begin{equation}
\|\bs h(t)\|_\cE^2 \lesssim \wt I(t).
\end{equation}
Using the computation above and \eqref{eq:alpha-id}, \eqref{eq:beta-id}, we obtain
\begin{equation}
\begin{aligned}
\wt I'(t) &= \la \bs f(t), \vD^2 E(\bs H(\vec a(t), \vec v(t)))\bs h(t)\ra + O\big(k(t)\wt I(t)+\delta(t)\sqrt{\wt I(t)}\big),
\end{aligned}
\end{equation}
We thus have
\begin{equation}
|\wt I'(t)| \leq C\big(\|\bs f(t)\|_\cE\sqrt{\wt I(t)} + k(t)\wt I(t)\big),
\end{equation}
and the conclusion follows from Gronwall's inequality, see \cite[Lemma 3.3]{BaChDa11}.
\end{proof}

%\begin{proof}[Proof of Lemma~\ref{lem:en-est-lin-trick} ]
%We introduce another modified energy functional. Letting $I(t)$ be as in the proof of Lemma~\ref{lem:en-est-lin} we set 
%\EQ{
%\check I(t) \coloneqq  I(t) + \la  \bs J  \bs h (t), \, \bs f(t) \ra 
%}
%By Proposition~\ref{prop:coer} we have 
%\EQ{
%\|\bs h \|_{\E} \lesssim \check I(t) + \| \bs h \|_{\E} \| \bs f \|_{\E} 
%}
%Using~\eqref{eq:I'}, and recalling that $\bs h(t)$ satisfies the equation 
%\EQ{
%\p_t \bs h = \bs J \vD^2 E( \bs H( \vec a, \vec v) \bs h +  \bs f + \sum_{k} \lam_k \p_{a} \bs H_k + \mu_k  \p_v \bs H_k 
%}
%we compute 
%\EQ{
%\check I'(t) &= I'(t) + \la \bs J \p_t \bs h(t), \, \bs f(t) \ra + \la \bs J \bs h(t) , \, \p_t \bs f(t)\ra \\
%& = \la \bs f, \vD^2 E(\bs H(\vec a, \vec v))\bs h\ra  + \sum_k \lam_k \la \p_{a} \bs H_k ,\, \bs h \ra \\
%& \quad + O\big((|\vec a\,'(t) - \vec v(t)| + |\vec v\,'(t)| + |\vec v(t)|^2 + |\vec v(t)|y_\tx{min}(t)^{-1})\|\bs h(t)\|_\cE^2\big) \\
%&\quad + O\big(|\vec v(t)|\|\bs f(t)\|_\cE\|\bs h(t)\|_\cE\big) \\
%&\quad  + \la \bs J^2  \vD^2 E( \bs H( \vec a, \vec v)) \bs h , \, \bs f \ra  + \la \bs J \bs f, \, \bs f \ra + \sum_k \lam_k \la \bs J \p_{a} \bs H_k, \,  \bs f \ra + \mu_k \la \bs J \p_v \bs H_k , \, \bs f \ra  \\
%&\quad + \la \bs J \bs h, \, \p_t \bs f\ra \\
%& =  \la \bs J \bs h, \, \p_t \bs f\ra 
%}
%\EQ{
%\p_t \bs h = \bs J \vD^2 E( \bs H( \vec a, \vec v) \bs h +  \bs f + \sum_{k} \lam_k \p_{a} \bs H_k + \mu_k  \p_v \bs H_k 
%}
%\end{proof} 

\section{Analysis of the projected equation}
\label{sec:cauchy-inf}
In this section we study the so-called projected equation in the Lyapunov-Schmidt scheme. This amounts to fixing a pair of trajectories $(\vec a(t), \vec v(t))$ within a certain function class.  
The trajectories that we consider satisfy the following conditions.
We set
\begin{equation}
t_0 \coloneqq  \exp\big(y_\tx{min}(0)/2\big).
\end{equation}
We then require that for all $t \geq 0$
\begin{equation}
\label{eq:traj-cond}
|\vec a\,'(t) - \vec v(t)| + |\vec v\,'(t)| + |\vec v(t)|^2 + \exp({-}y_\tx{min}(t)) \lesssim (t_0+t)^{-2},
\end{equation}
with a constant independent of $y_\tx{min}(0)$.

Given $\gamma > 0$ and a normed space $E$, we define the norm $\|\cdot\|_{N_\gamma(E)}$
on the space of continuous functions $\bs f: [0, \infty) \to E$ by the formula
\begin{align}
\|\bs f\|_{N_\gamma(E)} &\coloneqq  \sup_{t \geq 0} \big[(t_0+t)^\gamma\|\bs f(t)\|_E\big].
\end{align}
We write $N_\gamma$ instead of $N_\gamma(\bR)$.

The first three lemmas below concern the linear equation 
\begin{equation}
 \label{eq:kg-lin-eq-inf}
\begin{aligned}
 \partial_t \bs h(t) &= \bs J\vD^2 E(\bs H(\vec a(t), \vec v(t)))\bs h(t)+\bs f(t) + \sum_{k=1}^n \lambda_k(t)\partial_{ a_k} \bs H_k(t)
+  \mu_k(t) \partial_{ v_k}\bs H_k(t),
 \end{aligned}
\end{equation}
with $\bs f \in N_{\gamma +1}(\E)$ for some $\gamma>0$ and trajectories $(\vec a(t), \vec v(t))$ as in~\eqref{eq:traj-cond}.

\begin{lemma}
\label{lem:en-est-lin-inf}
For all $\nu > 0$ there exists $C > 0$ such that the following is true.
Let $(\bs h, \vec \lam, \vec \mu)$ be a solution of \eqref{eq:kg-lin-eq-inf} on the time interval $I = [T, \infty)$
such that $\bs h \in N_\gamma(\cE)$ for some $\gamma > 0$ and some $T \ge 0$. 
Assume in addition that
\begin{equation}
\Big(\sum_{k=1}^n|\la \bs \alpha_k(t), \bs h(t)\ra|^2+ |\la \bs\beta_k(t), \bs h(t)\ra|^2\Big)^{\frac{1}{2}} \leq \delta(t),\qquad\text{for all }t \in I
\end{equation}
for some continuous function $\delta: I \to [0, \infty)$. Then
\begin{equation}
\|\bs h\|_{N_\nu(\cE)} \leq C(\|\bs f\|_{N_{\nu + 1}(\cE)} + \|\delta\|_{N_{\nu + 1}}).
\end{equation}
\end{lemma}
\begin{proof}
We apply Lemma~\ref{lem:en-est-lin} with $\bs h_0 = \bs 0$ and let $T_0 \to \infty$.
Observe that for fixed $t$ and all $T_0 \geq t$ we have, with $\epsilon$ arbitrarily small upon taking $\tau_0$ large enough,
\begin{equation}
\begin{aligned}
&\int_t^{T_0}(\|\bs f(s)\|_\cE + \delta(s))\exp\bigg(C\int_t^s k(s')\ud s'\bigg)\ud s \\
&\qquad \lesssim (\|\bs f\|_{N_{\nu + 1}(\cE)} + \|\delta\|_{N_{\nu + 1}})\int_t^{T_0}(t_0+s)^{-\nu-1}\exp\bigg(\epsilon\int_t^s \frac{1}{t_0+s'}\ud s'\bigg)\ud s \\
&\qquad = (\|\bs f\|_{N_{\nu + 1}(\cE)} + \|\delta\|_{N_{\nu + 1}})(t_0+t)^{-\epsilon}\int_t^{T_0}(t_0+s)^{-\nu+\epsilon-1}\ud s \\
&\qquad \lesssim (t_0 + t)^{-\nu}(\|\bs f\|_{N_{\nu + 1}(\cE)} + \|\delta\|_{N_{\nu + 1}}), 
\end{aligned}
\end{equation}
which gives the claimed bound. 
\end{proof}
\begin{remark}
We could perhaps obtain $\|\delta\|_{N_{\nu + 0.5}}$ instead of $\|\delta\|_{N_{\nu + 1}}$ on the right hand side, if we perform differently the energy estimate.
\end{remark}

\begin{lemma}
\label{lem:kg-lin-eq-inf}
For any $\gamma > 0$ there exists $\tau_0$ such that if $t_0 \geq \tau_0$, then
for any $(\vec a, \vec v)$ as in~\eqref{eq:traj-cond} and $\bs f \in N_{\gamma+1}(\cE)$
there exists a unique solution $(\bs h, \vec\lambda, \vec\mu)$ of the equation~\eqref{eq:kg-lin-eq-inf}
% \begin{equation}
% \label{eq:kg-lin-eq-inf}
%\begin{aligned}
% \partial_t \bs h(t) &= \bs J\vD^2 E(\bs H(\vec a(t), \vec v(t)))\bs h(t)+\bs f(t) \\
%&+ \vec \lambda(t)\cdot \partial_{\vec a}\bs H(\vec a(t), \vec v(t))
%+ \vec \mu(t)\cdot \partial_{\vec v}\bs H(\vec a(t), \vec v(t))
% \end{aligned}
%\end{equation}
such that $\bs h \in N_\gamma(\cE)$ and \eqref{eq:h-orth} holds for all $t \geq 0$.

Moreover, $\bs f\mapsto \bs h$ is a bounded linear operator $N_{\gamma+1}(\cE) \to N_\gamma(\cE)$
(uniformly with respect to $t_0 \geq \tau_0$).

Finally, for all $t \geq 0$  and all $\eps>0$ we have %\red{what are the correct signs?}
\begin{equation} \label{eq:lam-mu-lin-kg} 
\begin{aligned}
|M  \lambda_k(t) + \la \bs\beta_k(t), \bs f(t)\ra| \leq C_\eps(t_0+t)^{-2+\eps}(\|\bs h(t)\|_\cE + \|\bs f(t)\|_\cE), \\
|M \mu_k(t) - \la \bs\alpha_k(t), \bs f(t)\ra| \leq C_{\eps} (t_0+t)^{-2+ \eps}(\|\bs h(t)\|_\cE + \|\bs f(t)\|_\cE). \\
\end{aligned}
\end{equation}
\end{lemma}

In one instance we require a subtle refinement of Lemma~\ref{lem:kg-lin-eq-inf} under decay assumptions on $\p_t  \bs f$. This trick was pointed out to us by Yvan Martel. 

\begin{lemma} \label{lem:kg-lin-eq-trick} In the setting of Lemma~\ref{lem:kg-lin-eq-inf}, suppose in addition that $\p_t  \bs f  \in N_{\gamma+2}(\E)$. Then,  
%For any $\gamma > 0$ there exists $\tau_0$ such that if $t_0 \geq \tau_0$, then
%for any $(\vec a, \vec v)$ as in~\eqref{eq:traj-cond} and $\bs f \in N_{\gamma+1}(\cE)$, $\p_t  \bs f  \in N_{\gamma+2}(\E)$
%there exists a unique solution $(\bs h, \vec\lambda, \vec\mu)$ of the equation~\eqref{eq:kg-lin-eq-inf}
% \begin{equation}
% \label{eq:kg-lin-eq-inf}
%\begin{aligned}
% \partial_t \bs h(t) &= \bs J\vD^2 E(\bs H(\vec a(t), \vec v(t)))\bs h(t)+\bs f(t) \\
%&+ \vec \lambda(t)\cdot \partial_{\vec a}\bs H(\vec a(t), \vec v(t))
%+ \vec \mu(t)\cdot \partial_{\vec v}\bs H(\vec a(t), \vec v(t))
% \end{aligned}
%\end{equation}
%such that $\bs h \in N_{\gamma+1}(\cE)$ and \eqref{eq:h-orth} holds for all $t \geq t_0$. In particular, 
$\bs h$ satisfies the estimates 
\EQ{
\| \bs h \|_{N_{\gamma+1}( \E) } \le C_{\gamma} \Big( \| \bs f \|_{N_{\gamma+1}(\E)} + \| \p_t \bs f \|_{N_{\gamma +2}(\E)}\Big). 
}

%Moreover, $\bs f\mapsto \bs h$ is a bounded linear operator $N_{\gamma+1}(\cE) \to N_\gamma(\cE)$
%(uniformly with respect to $t_0 \geq \tau_0$).

%Finally, for all $t \geq 0$ \red{what are the correct signs?}
%\begin{equation}
%\begin{aligned}
%|M \gamma_{v_k}^3 \lambda_k(t) + \la \bs\beta_k(t), \bs f(t)\ra| \leq C(t_0+t)^{-2+}(\|\bs h(t)\|_\cE + \|\bs f(t)\|_\cE), \\
%|M \gamma_{v_k}^3 \mu_k(t) - \la \bs\alpha_k(t), \bs f(t)\ra| \leq C(t_0+t)^{-2+}(\|\bs h(t)\|_\cE + \|\bs f(t)\|_\cE). \\
%\end{aligned}
%\end{equation}

\end{lemma} 

\begin{proof}[Proof of Lemma~\ref{lem:kg-lin-eq-inf}] 
Fixing $\gamma>0$, we have $ \bs f \in L^1([t_0, \infty); \E) \subset  N_{\gamma +1}(\E)$. Then, letting $\bs h_0 = \bs 0$ and $T_0\to \infty$ in  Lemma~\ref{lem:kg-lin-orth} we obtain a unique triplet $(\bs h, \vec \lam, \vec \mu) \in C([t_0, \infty); \E) \times L^1([t_0, \infty)) \times L^1([t_0, \infty))$ solving~\eqref{eq:kg-lin-eq-inf} and such that \eqref{eq:h-orth} holds. Using Lemma~\ref{lem:en-est-lin} and arguing as in Lemma~\ref{lem:en-est-lin-inf} we obtain the estimate 
\EQ{
\| \bs h \|_{N_{\gamma}(\E)} \le C \| \bs f \|_{N_{\gamma+1}(\E)}. 
}
We next prove~\eqref{eq:lam-mu-lin-kg}.   In light of~\eqref{eq:ab-dxda}, ~\eqref{eq:alpha-id}, and~\eqref{eq:beta-id} we note the identities
\EQ{ \label{eq:dt-alpha-beta} 
\p_t \bs \al_k - \vD^2 E( \bs H_k) \bs J \bs \al_k &= (a_k' - v_k) \p_a \bs \alpha_k + v_k' \p_v \bs \alpha_k, \\
\p_t \bs \beta_k - \vD^2 E( \bs H_k) \bs J \bs \beta_k &= (a_k' - v_k) \p_a \bs \beta_k + v'_k \p_v \bs \beta_k -  \bs \al_k . 
}
Starting with the expression~\eqref{eq:val-lambda-mu}, using the boundedness of $\calM(\vec a, \vec v)^{-1}$ for trajectories as in~\eqref{eq:traj-cond}, the above with~\eqref{eq:h-orth}, and~\eqref{eq:traj-cond} we obtain the preliminary estimates
\EQ{ \label{eq:lam-mu-prelim1} 
| \vec \lam(t)| + | \vec \mu(t) |  \lesssim (t+ t_0)^{-2} \| \bs h(t) \|_{\E} + \| \bs f(t)\|_{\E}. 
}
To obtain~\eqref{eq:lam-mu-lin-kg} we pair the equation~\eqref{eq:kg-lin-eq-inf}  with $\bs \beta_k$ and $\bs \al_k$ yielding 
\EQ{
\la \p_t \bs h, \, \bs \beta_k\ra  &= \la \bs J \vD^2 E( \bs H( \vec a, \vec v)) \bs h, \, \bs \beta_k \ra + \la \bs f , \, \bs \beta_k \ra + \sum_j \lam_j \la \p_a \bs H_j ,\, \bs\beta_k\ra + \mu_j \la \p_v \bs H_j, \, \bs \beta_k \ra \\
\la \p_t \bs h, \, \bs \alpha_k\ra  &= \la \bs J \vD^2 E( \bs H( \vec a, \vec v)) \bs h, \, \bs \al_k \ra + \la \bs f , \, \bs \al_k \ra + \sum_j \lam_j \la \p_a \bs H_j ,\, \bs\al_k\ra + \mu_j \la \p_v \bs H_j, \, \bs \al_k \ra . 
}
Using~\eqref{eq:matrix-coef}~\eqref{eq:identities-1} and~\eqref{eq:dt-orth} we obtain the expressions 
\EQ{
M \gamma_{v_k}^3 \lam_k + \la \bs f , \, \bs \beta_k \ra &=  - \la \bs  h, \, \p_t \bs \beta_k - \vD^2 E( \bs H_k) \bs J \bs \beta_k \ra +  \sum_{j \neq k} \lam_j \la \p_a \bs H_j ,\, \bs\beta_k\ra + \mu_j \la \p_v \bs H_j, \, \bs \beta_k \ra \\
&\quad - \la \bs  h, \, \Big( \vD^2 E(\bs H_k) - \vD^2 E( \bs H( \vec a, \vec v))\Big)  \p_v \bs H_k \ra\\
M\gamma_{v_k}^3 \mu_k -  \la \bs f , \, \bs \al_k \ra & =  \la \bs h , \, \p_t \bs \al_k - \vD^2 E( \bs H_k) \bs J \bs \al_k \ra -  \sum_{j \neq k} \lam_j \la \p_a \bs H_j ,\, \bs\al_k\ra + \mu_j \la \p_v \bs H_j, \, \bs \al_k \ra\\
&\quad + \la \bs  h, \, \Big( \vD^2 E(\bs H_k) - \vD^2 E( \bs H( \vec a, \vec v))\Big)  \p_a \bs H_k \ra. 
}
Noting that $\gamma_{v_k}^{-3} -1 = O((t_0 + t)^{-2})$, using~\eqref{eq:dt-alpha-beta} and ~\eqref{eq:traj-cond} on the first terms on the right, Lemma~\ref{lem:exp-cross-term} to estimate the cross terms in the sum over $j \neq k$, and~\eqref{d2U-bounds} for the last terms in each expression  completes the proof of~\eqref{eq:lam-mu-lin-kg}. 
\end{proof}

\begin{proof}[Proof of Lemma~\ref{lem:kg-lin-eq-trick} ]
We introduce another modified energy functional. Letting $I(t)$ be as in the proof of Lemma~\ref{lem:en-est-lin} we set 
\EQ{
\check I(t) \coloneqq  I(t) + \la  \bs J  \bs h (t), \, \bs f(t) \ra 
}
By Proposition~\ref{prop:coer} we have 
\EQ{
\|\bs h \|_{\E}^2 \lesssim \check I(t) + \| \bs h \|_{\E} \| \bs f \|_{\E} 
}
Using~\eqref{eq:I'}, and recalling that $\bs h(t)$ satisfies the equation 
\EQ{
\p_t \bs h = \bs J \vD^2 E( \bs H( \vec a, \vec v) \bs h +  \bs f + \sum_{k} \lam_k \p_{a} \bs H_k + \mu_k  \p_v \bs H_k 
}
we compute 
\EQ{
\check I'(t) &= I'(t) + \la \bs J \p_t \bs h(t), \, \bs f(t) \ra + \la \bs J \bs h(t) , \, \p_t \bs f(t)\ra \\
%& = \la \bs f, \vD^2 E(\bs H(\vec a, \vec v))\bs h\ra    + O\big((|\vec a\,' - \vec v| + |\vec v\,'| + |\vec v|^2 + |\vec v|y_\tx{min}^{-1})\|\bs h\|_\cE^2\big)  + O\big(|\vec v|\|\bs f\|_\cE\|\bs h\|_\cE\big) \\
%&\quad  + \la \bs J^2  \vD^2 E( \bs H( \vec a, \vec v)) \bs h , \, \bs f \ra  + \la \bs J \bs f, \, \bs f \ra + \sum_k \lam_k \la \bs J \p_{a} \bs H_k, \,  \bs f \ra + \mu_k \la \bs J \p_v \bs H_k , \, \bs f \ra  + \la \bs J \bs h, \, \p_t \bs f\ra \\
& =  \la \bs J \bs h, \, \p_t \bs f\ra   + \sum_k \lam_k \la \bs \alpha_k, \,  \bs f \ra + \mu_k \la \bs \beta_k , \, \bs f \ra \\
&\quad + O\big((|\vec a\,' - \vec v| + |\vec v\,'| + |\vec v|^2 + |\vec v|y_\tx{min}^{-1})\|\bs h\|_\cE^2\big)  + O\big(|\vec v|\|\bs f\|_\cE\|\bs h\|_\cE\big).
}
Using~\eqref{eq:lam-mu-lin-kg} and the estimates~\eqref{eq:lam-mu-prelim1} we note the cancellation 
\EQ{
|\lam_k \la \bs \alpha_k, \,  \bs f \ra + \mu_k \la \bs \beta_k , \, \bs f \ra| &=  M^{-1}\big|  \lam_k( \la \bs \alpha_k, \,  \bs f \ra - M\mu_k)  + \mu_k (\la \bs \beta_k , \, \bs f \ra  + M\lambda_k) \Big| \\
&\lesssim (t_0 + t)^{-\frac{3}{2}} ( \| \bs h \|_{\E} +  \| \bs f \|_{\E}) \Big( (t+ t_0)^{-2} \| \bs h(t) \|_{\E} + \| \bs f(t)\|_{\E} \Big) \\
&\lesssim (t_0 + t)^{-\frac{3}{2}} ( \| \bs h \|_{\E}^2 +  \| \bs f \|_{\E}^2).
}
Using this in the previous line we are led to the estimate
\EQ{
| \check I'(t) | \lesssim \frac{1}{ (t_0 + t) \log t} \| \bs h(t) \|_{\E}^2 + \frac{1}{t_0+ t} \| \bs h(t) \|_{\E} \| \bs f(t) \|_{\E} +  \frac{1}{(t_0 + t)^{\frac{3}{2}}} \| \bs f (t) \|_{\E}^2 + \| \bs h(t) \|_{\E} \|\p_t \bs f \|_{\E} 
}
and the conclusion follows by integration. 
%\EQ{
%\p_t \bs h = \bs J \vD^2 E( \bs H( \vec a, \vec v) \bs h +  \bs f + \sum_{k} \lam_k \p_{a} \bs H_k + \mu_k  \p_v \bs H_k 
%}
\end{proof} 

%\Red{remember subtle cancellation that appears $\lam_k \mu_k -  \lam_k \mu_k$} 

Next, we consider the nonlinear problem for a fixed trajectory ~$(\vec a, \vec v)$ as in~\eqref{eq:traj-cond}.
\begin{lemma}[The nonlinear equation] 
\label{lem:kg-nonlin-eq-inf}
There exist $\delta, \tau_0 > 0$ such that for any $t_0 \geq \tau_0$
and $(\vec a, \vec v)$ satisfying \eqref{eq:traj-cond} there exists a unique solution $(\bs h, \vec\lambda, \vec\mu)$ of the equation
\EQ{
 \label{eq:kg-nonlin-eq-inf}
 \partial_t \bs h(t) &= \bs J\vD E(\bs H(\vec a(t), \vec v(t))+\bs h(t)) %- \p_t \bs H( \vec a(t), \vec v(t)) \\
 + \sum_{k=1}^n\big(\lambda_k(t) \partial_{a}\bs H_k(t)+ \mu_k(t)\partial_{v}\bs H_k(t)\big)
}
such that $\|\bs h\|_{N_1} \leq \delta$ and \eqref{eq:h-orth} holds for all $t \geq 0$.
This solution satisfies $\|\bs h\|_{N_\gamma} \leq C_\gamma$  and
%\begin{equation}
%\begin{aligned}
%|\vec \lambda(t) - \vec v(t)| &\leq \\
%|\vec \mu(t) - \vec F(\vec a(t))| &\leq .
%\end{aligned}
%\end{equation}
\EQ{ \label{eq:nonlin-param} 
\sup_{1 \le k \le n} \| \lam_k(t) + (-1)^k v_k(t) \|_{N_{\gamma+1}} &\le C_\gamma  \\
\sup_{1 \le k \le n}\| M  \mu_k(t)+ (-1)^k F_k( \vec a(t), \vec v(t))\|_{N_{\gamma +1}}   &\le C_\gamma 
}
for any $\gamma < 2$, where $F_k(\vec a, \vec v)$ is defined in~\eqref{eq:F_k-def}. 
\end{lemma}

%\Red{It will look cleaner to get rid of $\gamma_{v_k}^3$ everywhere here by replacing by $1$ and putting error on right-hand-side.}

Before we give a proof (which is based on the Contraction Principle),
we need one more auxiliary result, which provides estimates on the first step
of the Picard iteration. It relies on an explicit ansatz and results in a better bound
than the one obtained from a direct application of the energy estimates.
\begin{lemma}
\label{lem:1st-Picard}
Let $(\bs h_{I}, \vec\lambda_I,\vec\mu_I)$ be the solution of
\begin{equation}
\label{eq:1st-Picard}
\begin{aligned}
\partial_t \bs h_I(t) &= \bs J\vD^2 E(\bs H(\vec a(t), \vec v(t)))\bs h_I(t) + \bs J\vD E(\bs H(\vec a(t), \vec v(t)))\\ %- \p_t \bs H(\vec a(t), \vec v(t)) \\
&\quad + \sum_{k=1}^n\big( \lambda_{I, k}(t) \partial_{a_k}\bs H_k(t)
+ \mu_{I, k}(t)\partial_{v_k}\bs H_k(t)\big)
 \end{aligned}
\end{equation}
given by Lemma~\ref{lem:kg-lin-eq-inf}. For each $0<\gamma<2$ there exists a constant $C_{\gamma} >0$ so that $(\bs h_I, \vec \lambda_I, \vec \mu_I)$ satisfy $\|\bs h_I\|_{N_\gamma(\E)} \le C_\gamma$, and 
\EQ{ \label{eq:1st-Picard-param} 
 \sup_{1 \le k \le n}\Big(  \| \lambda_{I, k}(t) + (-1)^k v_k(t) \|_{N_{\gamma+1}} +  \|M  \mu_{I, k}(t) + (-1)^k F_k( \vec a(t), \vec v(t)) \|_{N_{\gamma+1}}  \Big)&\le C_\gamma. %\\ 
%&\le C_\gamma 
}
%for each $k \in \{1, \dots, n\}$. 
\end{lemma}
\begin{proof}
Let $(\bs h_S(t), \vec\lambda_S(t), \vec \mu_S(t))$ be the solution  given by Lemma~\ref{lem:static-ls} part \ref{it:static-ls-i} of
\EQ{ \label{eq:h_S-eqn} 
\vD^2 E(\bs H(\vec a(t), \vec v(t)))\bs h_S(t) &= -\vD E(\bs H(\vec a(t), \vec v(t)))    + \sum_{k =1}^n  \lambda_{S, k}(t) \bs \alpha_k(t) + \mu_{S, k}(t) \bs \beta_k(t)) \\
 &=\bs f_S +  \sum_{k =1}^n \big( (\lambda_{S, k}(t) +(-1)^k v_k(t)) \bs \alpha_k(t) + \mu_{S, k}(t) \bs \beta_k(t)\big) %\\
 %&\quad -  \pmat{ U'( H( \vec a, \vec v)) - \sum_{k=1}^n (-1)^kU'( H_k) \\ 0}
  }
  with 
  \EQ{
  \la \bs h_S , \, \bs \al_k\ra = \la \bs h_S, \, \bs \beta_k\ra = 0, 
  }
  for all $k \in \{1, \dots, n\}$ and 
  \EQ{ \label{eq:f_S} 
  \bs f_S = -  \pmat{ U'( H( \vec a, \vec v)) - \sum_{k=1}^n (-1)^kU'( H_k) \\ 0}. 
  }
%  \EQ{ \label{eq:h_S-eqn} 
%\vD^2 &E(\bs H(\vec a(t), \vec v(t)))\bs h_S(t) = -\vD E(\bs H(\vec a(t), \vec v(t))) - \bs J\p_t \bs H(\vec a(t), \vec v(t)) \\
%&\qquad \qquad \qquad \qquad\qquad\, \,\, \, \,  + \sum_{k =1}^n ( \lambda_{S, k}(t) \bs \alpha_k(t) + \mu_{S, k}(t) \bs \beta_k(t)) \\
% &= \sum_{k =1}^n \big( (\lambda_{S, k}(t) - (-1)^k ( a_k'(t) - v_k(t)) \bs \alpha_k(t) + (\mu_{S, k}(t) - (-1)^kv_k'(t))\bs \beta_k(t)\big)\\
%&\quad -  \pmat{ U'( H( \vec a, \vec v)) - \sum_{k=1}^n (-1)^kU'( H_k) \\ 0}
%  }
 By~\eqref{eq:h-est-lm}, and~ Lemma~\ref{lem:U-bounds} we have
\EQ{ \label{eq:f_S-est} 
\| \bs h_S(t) \|_{\E} \lesssim \| \bs f_S \|_{\E} \lesssim  \| U'( H( \vec a, \vec v)) - \sum_{k=1}^n (-1)^kU'( H_k) \|_{H^1}  \lesssim \sqrt{y_{\min}(t)} e^{-y_{\min}(t)} 
}
 which, together with~\eqref{eq:traj-cond} gives 
\EQ{ \label{eq:h_S-est} 
\|  \bs h_S \|_{N_{\gamma}} \le C_\gamma
}
for any $\gamma<2$. 
By~\eqref{eq:dHalbe-form} along with Lemma~\ref{lem:U-bounds} we have, for each $k \in \{1, \dots, n\}$
\EQ{
|\la \partial_{v_k}  H_k(t) , \, U'( H( \vec a, \vec v)) - \sum_{k=1}^n (-1)^kU'( H_k) \ra| \lesssim |v_k(t)| \sqrt{y_{\min}(t)}e^{-y_{\min}(t)}
}
It then follows from~\eqref{eq:lam-mu-est-lm} and the above that
\EQ{
| \lambda_{S, k}(t) + (-1)^k v_k(t)| &\lesssim |v_k(t)| \sqrt{y_{\min}(t)}e^{-y_{\min}(t)} \\
&\quad + ( \sqrt{y_{\min}(t)} e^{-y_{\min}(t)} + | \vec v|) \| \bs h_S \|_{\E} \\
&\quad + y_{\min}(t) e^{- y_{\min}(t)} \| U'( H( \vec a, \vec v)) - \sum_{k=1}^n (-1)^kU'( H_k)  \|_{\E} 
}
Using~\eqref{eq:traj-cond}, Lemma~\ref{lem:U-bounds} and~\eqref{eq:traj-cond} we see that for any $\gamma<2$
\EQ{ \label{eq:lam_S-est} 
\| \lambda_{S, k}(t) +(-1)^k v_k(t)\|_{N_{\gamma+1}} \le C_\gamma
}
Using the same logic, we arrive at the bound
\EQ{
| M\gamma_{v_k(t)}^3 \mu_{S, k}(t) - \la \p_{a_k}  H_k(t), \, U'( H( \vec a, \vec v)) - \sum_{k=1}^n (-1)^kU'( H_k)\ra| \le C_\gamma t^{-\gamma -1} 
}
or equivalently, 
\EQ{ \label{eq:mu_S-est} 
\| M  \mu_{S, k}(t) + (-1)^k  F_k( \vec a(t), \vec v(t)) \|_{N_{\gamma+1}} \le C_\gamma
}
where we have used above that $|\gamma_{v_k}^3 - 1| \lesssim (t_0 + t)^{-2}$. 

We will also require bounds for 
\EQ{ 
\partial_t \bs h_{S}(t) = \partial_t \bs h_S( \vec a(t), \vec v(t)) = \sum_{j =1}^n a_j'(t) (\p_{a_j} \bs h_S)(\vec a(t), \vec v(t)) + v_j'(t) (\p_{v_j} \bs h_S)(\vec a(t), \vec v(t))
}
Using~\eqref{eq:h-first} from Lemma~\ref{lem:diff-param} along with~\eqref{eq:dU-bounds} and~\eqref{d2U-bounds} from Lemma~\ref{lem:U-bounds} we note the estimates,
\EQ{ \label{eq:p1h-est} 
\| \p_{a_j} \bs h_S(\vec a(t), \vec v(t)) \|_{\E} &\lesssim \sqrt{y_{\min}(t)} e^{-y_{\min}(t)}  \\
\| \p_{v_j} \bs h_S(\vec a(t), \vec v(t)) \|_{\E} &\lesssim \sqrt{y_{\min}(t)} e^{-y_{\min}(t)}
}
and by the same token, using now~\eqref{eq:h-second} from Lemma~\ref{lem:diff-param} we have 
\EQ{ \label{eq:p2h-est} 
\| \p_{a_j}\p_{a_k} \bs h_S(\vec a(t), \vec v(t)) \|_{\E}  + \| \p_{v_j}\p_{a_k} \bs h_S(\vec a(t), \vec v(t)) \|_{\E}&\lesssim \sqrt{y_{\min}(t)} e^{-y_{\min}(t)}
}
which we record for later. 
Together with~\eqref{eq:traj-cond} it follows that
\EQ{\label{eq:dthS-est} 
\| \p_t \bs h_S \|_{N_{\gamma+1}(\E)} \le C_{\gamma} 
}
for any $\gamma<2$. 

We write $\bs h_I = \bs h_S + \bs  h_R$, $\vec \lambda_I = \vec \lambda_S + \vec \lambda_R$, and $\vec \mu_I= \vec \mu_S+ \vec\mu_R$. Then~\eqref{eq:1st-Picard} is equivalent to $(\bs h_R,\vec \lambda_R, \vec\mu_R)$ solving
\EQ{ \label{eq:h-R-eq} 
\p_t \bs h_R(t) 
%&= \p_t \bs h_I(t) - \p_t \bs h_S(t) \\
%&= \bs J\vD^2 E(\bs H(\vec a(t), \vec v(t)))\bs h^{(I)}(t) + \bs J\vD E(\bs H(\vec a(t), \vec v(t))) \\
%&\quad + \sum_{k=1}^n\big( \lambda^{(I)}_k(t) \partial_{a}\bs H_k(t)
%+ \mu_k^{(I)}(t)\partial_{v}\bs H_k(t)\big) \\
%&\quad  - \p_t \bs h_S(t) \\
%&= \bs J\vD^2 E(\bs H(\vec a(t), \vec v(t)))\ti {\bs h}(t)  - \p_t \bs H(\vec a(t), \vec v(t)) - \p_t \bs h_S(t) \\
%&\quad + \sum_{k=1}^n\big( \ti \lambda_k(t) \partial_{a_k}\bs H_k(t)
%+ \ti \mu_k(t)\partial_{v_k}\bs H_k(t)\big)  \\
& = \bs J\vD^2 E(\bs H(\vec a(t), \vec v(t)))\bs h_R(t)- \p_t \bs h_S(t)  \\
&\quad +  \sum_{k=1}^n  \lambda_{R, k}(t) \partial_{a_k}\bs H_k(t)
+ \mu_{R, k}(t)\partial_{v_k}\bs H_k(t)  
}
and
\EQ{
\la \bs h_R(t), \, \bs \alpha_k(t)\ra = \la \bs h_R(t), \, \bs \beta_k(t)\ra  = 0, 
}
for all $k \in \{1, \dots, n\}$. 
From Lemma~\ref{lem:kg-lin-eq-inf} and~\eqref{eq:dthS-est} we have
\EQ{ \label{eq:ti-h-prelim} 
\|\bs h_R \|_{N_{\gamma}(\E)} \le C_{\gamma}  \| \p_t \bs h_S \|_{N_{\gamma+1}(\E)} \le C_\gamma
}
and
\EQ{ \label{eq:lam-mu-R-est} 
\sup_{1 \le k \le n}\big( \|  \lambda_{R, k} \|_{N_{\gamma+1}} + \| \mu_{R, k}  \|_{N_{\gamma+1}(\E)} \big)\le C_\gamma 
}
for any $\gamma<2$ as claimed. Combining the above with~\eqref{eq:h_S-est},~\eqref{eq:lam_S-est},  and~\eqref{eq:mu_S-est} gives the lemma. 

We also require an improvement over the estimate~\eqref{eq:ti-h-prelim}. We claim that in fact
\EQ{ \label{eq:h-R-est-refined} 
\| \bs h_R \|_{N_{\gamma+1}(\E)} \le C_\gamma
}
for all $\gamma <2$. To see this we rewrite the forcing term $\p_t \bs h_S$ in the equation for ${\bs h}_R$ as 
\EQ{
-\partial_t \bs h_{S} (t)&=  \bs f_{R, 1}(t) + \bs f_{R, 2}(t),
}
where 
\EQ{
\bs f_{R, 1}(t)\coloneqq  \sum_{k =1}^n \Big((a_k'(t) - v_k(t)\Big)(\p_{a_k} \bs h_S)(\vec a(t), \vec v(t)) + \sum_{k =1}^nv_k'(t) (\p_{v_j} \bs h_S)(\vec a(t), \vec v(t)),  
}
and
\EQ{
\bs f_{R, 2}(t)\coloneqq   \sum_{k =1}^n v_k(t)(\p_{a_k} \bs h_S)(\vec a(t), \vec v(t)).
}
We then let $(\bs h_{R, j}, \vec \lambda_{R, j}, \vec \mu_{R, j})$ for $j = 1, 2$ be the solution to 
\EQ{
\p_t \bs h_{R, j} = \bs J\vD^2 E(\bs H(\vec a, \vec v))\bs h_{R, j} + \bs f_{R, j} +  \sum_{k=1}^n  \lambda_{R, j, k} \partial_{a_k}\bs H_k 
+ \mu_{R, j, k}\partial_{v_k}\bs H_k  
}
and
\EQ{
\la \bs h_{R, j}, \, \bs \alpha_k\ra = \la \bs h_{R, j}, \, \bs \beta_k\ra  = 0, 
}
for all $k \in \{1, \dots, n\}$. An application of Lemma~\ref{lem:kg-lin-eq-inf} for the above equation with $j=1$ yields the bound 
\EQ{
\| \bs h_{R, 1} \|_{N_{\gamma +1}(\E)} \le C_{\gamma} \| \bs f_{R, 1} \|_{N_{\gamma + 2}} \le C_{\gamma} 
}
which, using~\eqref{eq:traj-cond}, holds for all $\gamma <2$. To prepare for an application of Lemma~\ref{lem:kg-lin-eq-trick} to treat the case $j =2$ we first compute 
\EQ{
\p_t \bs f_{R, 2}(t) &= \sum_{k =1}^n v_k'(t)(\p_{a_k} \bs h_S)(\vec a(t), \vec v(t)) + \sum_{k =1}^n \sum_{j =1}^n v_k(t) a_{j}'(t)(\p_{a_j} \p_{a_k} \bs h_S)(\vec a(t), \vec v(t)) \\
&\quad + \sum_{k =1}^n v_k(t) v_j'(t) (\p_{v_j}\p_{a_k} \bs h_S)(\vec a(t), \vec v(t)).
}
Using~\eqref{eq:p1h-est},~\eqref{eq:p2h-est} and~\eqref{eq:traj-cond} we see that in fact 
\EQ{
\| \p_t \bs f_{R, 2} \|_{N_{\gamma + 2}(\E)} \le C_\gamma, 
}
for any $\gamma < 2$ and hence by Lemma~\ref{lem:kg-lin-eq-trick} we have 
\EQ{
\| \bs h_{R, 2} \|_{N_{\gamma +1}} \le C_\gamma 
}
for any $\gamma < 2$. This proves~\eqref{eq:h-R-est-refined}. 
%\Red{then split up the right hand side of the equation for $\ti{\bs h}$, using Yvan's trick only on the term that just has $v_k(t)$. Once differentiated in time, this term decays like $t^{-4+}$. Need estimates for $\| \p_{a_j} \p_{a_k} \bs h_S \|_{\E}$ and $\| \p_{v_j} \p_{a_k} \bs h_S \|_{\E}$ }
\end{proof}

\begin{proof}[Proof of Lemma~\ref{lem:kg-nonlin-eq-inf}] We divide the proof into five steps. 

\noindent

\textbf{Step 1.} (Formulation as a fixed point problem.)
Given $\bs g \in N_1(\E)$ such that $\|\bs g\|_{N_1(\E)} \leq \delta$, let $(\bs h, \vec\lambda, \vec\mu)$ be the unique solution of
 \begin{equation}
 \label{eq:kg-nonlin-eq-inf-iter}
\begin{aligned}
 \partial_t \bs h(t) &= \bs J\vD^2 E(\bs H(\vec a(t), \vec v(t))) \bs h(t) +  J \vD E(\bs H(\vec a(t), \vec v(t)))  \\
&+ \bs J\big(\vD E(\bs H(\vec a(t), \vec v(t))+\bs g(t)) - \vD E(\bs H(\vec a(t), \vec v(t)))- \vD^2 E(\bs H(\vec a(t), \vec v(t)))\bs g(t)\big) \\
& + \sum_{k =1}^n  \lambda_k(t) \partial_{a_k}\bs H_k(t) +  \mu_k(t) \partial_{ v_k}\bs H_k(t) 
 \end{aligned}
\end{equation}
given by Lemma~\ref{lem:kg-lin-eq-inf}, which applies here since %\red{CORRECT}
\begin{equation}
\|\vD E(\bs H(\vec a(t), \vec v(t))+\bs g(t)) - \vD E(\bs H(\vec a(t), \vec v(t)))- \vD^2 E(\bs H(\vec a(t), \vec v(t)))\bs g(t)\|_{N_2} \lesssim \|\bs g\|_{N_1}^2.
\end{equation}
We set $\Phi(\bs g) \coloneqq  \bs h$, so that $\bs h$ solves \eqref{eq:kg-nonlin-eq-inf} if and only if $\Phi(\bs h) = \bs h$.

%\EQ{
%\partial_t \bs h^{(I)}(t) &= \bs J\vD^2 E(\bs H(\vec a(t), \vec v(t)))\bs h^{(I)}(t) + \bs J\vD E(\bs H(\vec a(t), \vec v(t))) \\
%&+ \sum_{k=1}^n\big( \lambda^{(I)}_k(t) \partial_{a_k}\bs H_k(t)
%+ \mu_k^{(I)}(t)\partial_{v_k}\bs H_k(t)\big)
%}
%For arbitrary $\bs g \in N_1$ we decompose $\bs h = \Phi(\bs g) =   \Phi(\bs 0) + \bs h_* = \bs h_I + \bs h_*$, $\vec\lambda = \vec \lambda_I + \vec \lambda_*$, and $\vec \mu = \vec \mu_I +\vec  \mu_*$.  Given that the equation for  $\bs h_I$ is~\eqref{eq:1st-Picard}, we see that $\bs h  = \Phi(\bs g)$ solves~\eqref{eq:kg-nonlin-eq-inf-iter} if and only if $\bs h_* = \bs h - \bs h_I$ solves 
%\EQ{
%\p_t \bs h_*(t) &= \bs J \vD^2 E( \bs H(\vec a(t), \vec v(t)))\bs h_*(t) \\
% &\quad + \bs J\big(\vD E(\bs H(\vec a(t), \vec v(t))+\bs g(t)) - \vD E(\bs H(\vec a(t), \vec v(t)))- \vD^2 E(\bs H(\vec a(t), \vec v(t)))\bs g(t)\big) \\
% &\quad + \sum_{k=1}^n  \lambda_{*, k}(t) \partial_{a_k}\bs H_k(t) +  \mu_{*, k}(t) \partial_{ v_k}\bs H_k(t) 
%}

\textbf{Step 2.} (Estimates on $\Phi(\bs 0)$.) The so-called first Picard iterate is given by $ \bs h_I\coloneqq  \Phi( \bs 0)$ where $\bs h_I$ is as in Lemma~\ref{lem:1st-Picard} and the equation for  $\bs h_I$ is~\eqref{eq:1st-Picard}.   %We set
%\begin{equation}
%\chi(x) \coloneqq  \frac{1}{1 + \eee^{x}}, \qquad \wt\chi(x) = \chi(-x) = 1-\chi(x) = \frac{1}{1+\eee^{-x}}
%\end{equation}
%and
%\begin{equation}
%h^{(0)}(\vec a; x) \coloneqq  \kappa\sum_{j=1}^n (-1)^j\Big[ \wt \chi\Big(x - \frac{a_{j-1}+a_j}{2}\Big)\wt P(x - a_j)
%- \chi\Big(x - \frac{a_j + a_{j+1}}{2}\Big)P(x - a_j) \Big],
%\end{equation}
%where $P(x)$ is defined in Lemma~\ref{lem:P} and $\wt P(x) \coloneqq  P(-x)$. We claim that
%\begin{equation}
%\begin{aligned}
%({-}\partial_x^2 + U''(H(\vec a, \vec v)))h^{(0)}(\vec a) &= -\partial_x^2 H(\vec a, \vec v) + U'(H(\vec v, \vec a)) \\
%&+ \sum_{j=1}^n \mu^{(0)}_j \partial_x H(x - a_j) + O(|v|^3 + \ldots).
%\end{aligned}
%\end{equation}
%\red{TODO}
By Lemma~\ref{lem:1st-Picard} we have 
\begin{equation} \label{eq:P0-bound} 
\|\Phi(\bs 0)\|_{N_\gamma(\E)} \leq C_\gamma, \qquad \text{for all }\gamma < 2, 
\end{equation}
and thus  $\|\Phi(\bs 0)\|_{N_1(\E) } \le \frac{\delta}{2}$  provided $\tau_0$ is chosen large enough.

\textbf{Step 3.} (Estimates on $\Phi(\sh{\bs g}) - \Phi(\bs g)$.)
Let $\bs g, \sh{\bs g} \in N_1(\E)$ and $\max(\|\bs g\|_{N_1(\E)}, \|\sh{\bs g}\|_{N_1(\E)}) \leq \delta$.
Set
\begin{equation}
\begin{aligned}
\bs f(t) &\coloneqq  \bs J\big(\vD E(\bs H(\vec a(t), \vec v(t))+{\bs g}(t))-\vD E(\bs H(\vec a(t), \vec v(t))) - \vD^2 E(\bs H(\vec a(t), \vec v(t))){\bs g}(t)\big) \\
&= \big(0, {-}U'(H(\vec a(t), \vec v(t)) + g(t)) + U'(H(\vec a(t), \vec v(t)) + U''(H(\vec a(t), \vec v(t)))g(t)\big),
\end{aligned}
\end{equation}
and define $\sh{\bs f}(t)$ by the same formula, with $\bs g$ replaced by $\sh{\bs g}$.
Then $\Phi(\sh{\bs g}) - \Phi(\bs g) = \bs h$, where $(\bs h, \vec\lambda, \vec\mu)$
is the solution of
 \begin{equation}
 \label{eq:kg-nonlin-eq-inf-diff}
\begin{aligned}
 \partial_t \bs h(t) &= J\vD^2 E(\bs H(\vec a(t), \vec v(t))) \bs h(t) + \sh{\bs f}(t) - \bs f(t) \\
&+ \vec \lambda(t)\cdot \partial_{\vec a}\bs H(\vec a(t), \vec v(t))
+ \vec \mu(t)\cdot \partial_{\vec v}\bs H(\vec a(t), \vec v(t))
 \end{aligned}
\end{equation}
Applying~\eqref{eq:U'-taylor-diff} with $w = \sh w = H(\vec a, \vec v)$, we obtain
\begin{equation}
\|\sh{\bs f}(t) - \bs f(t)\|_{\cE} \lesssim (\|\sh{\bs g}(t)\|_\cE + \|\bs g(t)\|_\cE)\|\sh{\bs g}(t) - \bs g(t)\|_\cE,
\end{equation}
hence
\begin{equation}
\|\sh{\bs f} - \bs f\|_{N_2} \lesssim (\|\sh{\bs g}\|_{N_1} + \|\bs g\|_{N_1})\|\sh{\bs g} - \bs g\|_{N_1}.
\end{equation}
It thus follows from Lemma~\ref{lem:kg-lin-eq-inf} that
\begin{equation}
\|\Phi(\sh{\bs g}) - \Phi(\bs g)\|_{N_1} \leq \frac 12\|\sh{\bs g} - \bs g\|_{N_1}.
\end{equation}
if $\delta$ is small enough. Once $\delta$ is fixed, we choose $t_0$ large enough, so that
Step 2 yields $\|\Phi(\bs 0)\|_{N_1(\E)} \leq \frac{\delta}{2}$.
Then $\Phi$ is a strict contraction on the closed ball of radius $\delta$ in $N_1(\E)$ and the Contraction Mapping Principle yields a unique fixed point. 

\textbf{Step 4.} (Improved time decay of $\bs h(t)$.) 
Fix any $\gamma \in (1, 2)$. Because of~\eqref{eq:P0-bound} the closed unit ball in $N_\gamma(\E)$ is invariant under $\Phi$ for sufficiently large $\tau_0$ and by the same argument as above we see that $\Phi$ is a strict contraction on the closed unit ball in $N_{\gamma}(\E)$, based on the estimate
\EQ{
\| \Phi(\bs g) - \Phi(\sh{\bs g}) \|_{N_{2 \gamma -1}(\E)} \le C_\gamma (\| \sh{\bs g} \|_{N_{\gamma}(\E)} + \| \bs g \|_{N_\gamma(\E)} ) \| \bs g - \sh{\bs g} \|_{N_{\gamma}(\E)}. 
}
 By the Contraction Mapping Principle there is a unique fixed point in the unit ball in $N_{\gamma}(\E)$, but since $N_{\gamma}( \E) \subset N_1(\E)$ we deduce that the unique fixed point in the $\delta$-ball in $N_1(\E)$ belongs to $N_{\gamma}(\E)$ for all $1<\gamma< 2$. 
%We check that $\Phi$ is a strict contraction on the closed unit ball of $N_\gamma$ if $t_0$ is large enough. \red{TODO}

\textbf{Step 5.}(Estimates on the parameters) Let $\bs h$ denote the solution to~\eqref{eq:kg-nonlin-eq-inf} found above. 

Denote $\bs h_I \coloneqq  \Phi(0)$, let $\bs h$ be the fixed point of $\Phi$ found above and let $\bs h_* \coloneqq  \bs h - \bs h_I$.
By the same argument as in Step 4, we have
\begin{equation}
\|\Phi(\bs h_I + \bs h_*) - \Phi(\bs 0)\|_{N_{2\gamma - 1}(\E)} \lesssim \|\bs h_I + \bs h_*\|_{N_\gamma}^2.
\end{equation}
But $\Phi(\bs h_I + \bs h_*) - \Phi(0) = \bs h - \bs h_I = \bs h_*$, hence we obtain
\begin{equation}
\label{eq:hstar-decay}
\|\bs h_*\|_{N_\gamma} \leq C_\gamma, \qquad \text{for all }\gamma < 3.
\end{equation}
Similarly, write $\vec\lambda = \vec \lambda_I + \vec \lambda_*$, and $\vec \mu = \vec \mu_I +\vec  \mu_*$, where $(\vec \lam_I, \vec \mu_I)$ are as in Lemma~\ref{lem:1st-Picard}. We see that $(\bs h_*, \vec \lam_*, \vec \mu_*)$ solve 
\EQ{ \label{eq:hstar} 
\p_t \bs h_*(t) &= \bs J \vD^2 E( \bs H(\vec a(t), \vec v(t)))\bs h_*(t) \\
 & + \bs J\big(\vD E(\bs H(\vec a(t), \vec v(t))+\bs h(t)) - \vD E(\bs H(\vec a(t), \vec v(t)))- \vD^2 E(\bs H(\vec a(t), \vec v(t)))\bs h(t)\big) \\
 & + \sum_{k=1}^n  \lambda_{*, k}(t) \partial_{a_k}\bs H_k(t) +  \mu_{*, k}(t) \partial_{ v_k}\bs H_k(t). 
}
By Lemma~\ref{lem:kg-lin-eq-inf},~\eqref{eq:hstar-decay}, and that $ \bs h \in N_{\gamma}( \E)$ for $\gamma <2$  we see that
\begin{equation}
\label{eq:modstar-decay}
\|\vec \lambda_*\|_{N_\gamma} + \|\vec \mu_*\|_{N_\gamma} \leq C_\gamma,\qquad\text{for all }\gamma < 4.
\end{equation}
Then~\eqref{eq:nonlin-param} follow from the above and Lemma~\ref{lem:1st-Picard}. 
\end{proof}

%Denote $\bs h_I \coloneqq  \Psi(0)$, let $\bs h$ be the fixed point of $\Phi$ found above and let $\bs h_* \coloneqq  \bs h - \bs h_I$.
%By the same argument, as in the proof above, we have
%\begin{equation}
%\|\Phi(\bs h_I + \bs h_*) - \Phi(0)\|_{N_{2\gamma - 1}} \lesssim \|\bs h_I + \bs h_*\|_{N_\gamma}^2.
%\end{equation}
%But $\Phi(\bs h_I + \bs h_*) - \Phi(0) = \bs h - \bs h_I = \bs h_*$, hence we obtain
%\begin{equation}
%\label{eq:hstar-decay}
%\|\bs h_*\|_{N_\gamma} \leq C_\gamma, \qquad \text{for all }\gamma < 3.
%\end{equation}
%
%Taking the difference of the equations satisfied by $\bs h$ and $\bs h_I$, we obtain the following equation for $\bs h_*$:
%\begin{equation}
%\label{eq:hstar}
%\begin{aligned}
% \partial_t \bs h_*(t) &= J\vD^2 E(\bs H(\vec a(t), \vec v(t))) \bs h_*(t) \\
%&+ J\big(\vD E(\bs H(\vec a(t), \vec v(t))+\bs h_I(t) + \bs h_*(t)) - \vD E(\bs H(\vec a(t), \vec v(t)))  \\
%&\qquad- \vD^2 E(\bs H(\vec a(t), \vec v(t)))(\bs h_I(t) + \bs h_*(t))\big) \\
%&+ \vec \lambda_*(t)\cdot \partial_{\vec a}\bs H(\vec a(t), \vec v(t))
%+ \vec \mu_*(t)\cdot \partial_{\vec v}\bs H(\vec a(t), \vec v(t)),
%\end{aligned}
%\end{equation}
%where we have set
%\begin{equation}
%\vec \lambda_* \coloneqq  \vec\lambda - \vec\lambda_I, \qquad \vec \mu_* \coloneqq  \vec \mu - \vec \mu_I.
%\end{equation}
%Observe that, by Lemma~\ref{lem:kg-lin-eq-inf},
%\begin{equation}
%\label{eq:modstar-decay}
%\|\vec \lambda_*\|_{N_\gamma} + \|\vec \mu_*\|_{N_\gamma} \leq C_\gamma,\qquad\text{for all }\gamma < 4.
%\end{equation}

Finally, we study the dependence of the solution $(\bs h, \vec \lambda, \vec \mu)$
obtained in Lemma~\ref{lem:kg-nonlin-eq-inf} on the choice of $(\vec a, \vec v)$.
Again, we begin by studying the first step of the Picard iteration.
\begin{lemma}
\label{lem:1st-Picard-diff}
Let $(\bs h_I, \vec\lambda_I, \vec\mu_I)$ be the solution of \eqref{eq:1st-Picard}
given by Lemma~\ref{lem:kg-lin-eq-inf}
and $(\sh{\bs h}_I, \sh{\vec\lambda}_I, \sh{\vec\mu}_I)$ the solution of the same equation with $(\vec a, \vec v)$ replaced by
$(\sh{\vec a}, \sh{\vec v})$. Then for all $\gamma < 2$ there exists $C_\gamma>0$ such that
%\begin{equation}
%\begin{aligned}
%&\|\sh{\bs h}_I - \bs h_I\|_{N_\gamma}+\|(\sh{\vec\lambda}_I - \vec\lambda_I)
%- (\sh{\vec v} - \vec v)\|_{N_{\gamma}}+\|(\sh{\vec\mu}_I - \vec\mu_I) - (F(\sh{\vec a}) - F(\vec a))\|_{N_{\gamma+1}} \leq \\
%&\qquad\leq C_\gamma(\|\sh{\vec a} - \vec a\|_{L^\infty} + \|\sh{\vec v} - \vec v\|_{L^\infty}).
%\end{aligned}
%\end{equation}
\EQ{\label{eq:fl-h-I-est} 
\|\sh{\bs h}_I - \bs h_I\|_{N_\gamma(\E)} &\le C_\gamma(\|\sh{\vec a} - \vec a\|_{L^\infty} + \|\sh{\vec v} - \vec v\|_{N_1} +  \|(\sh{\vec a})' - \vec a'\|_{N_1} + \|(\sh{\vec v})' - \vec v'\|_{N_1}  ), 
}
\begin{multline}  \label{eq:fl-lam-I-est} 
\sup_{1 \le k \le n} \| (\sh \lambda_{I, k}-  \lambda_{I, k})  +(-1)^k  (\sh v_k-  v_k)\|_{N_{\gamma+1}} \\
\le C_\gamma(\|\sh{\vec a} - \vec a\|_{L^\infty} + \|\sh{\vec v} - \vec v\|_{N_1} +   \|(\sh{\vec a})' - \vec a'\|_{N_1} + \|(\sh{\vec v})' - \vec v'\|_{N_1}  ), 
\end{multline} 
and, 
\begin{multline} \label{eq:fl-mu-I-est} 
\sup_{ 1 \le k \le n} \| (\sh \mu_{I, k}-  \mu_{I, k})  + (-1)^k ( M ^{-1} F_k(\sh{\vec a}, \sh{\vec v}) - M ^{-1}F_k({\vec a}, {\vec v}))  \|_{N_{\gamma +1}}\\
 \le C_\gamma(\|\sh{\vec a} - \vec a\|_{L^\infty} + \|\sh{\vec v} - \vec v\|_{N_1} +   \|(\sh{\vec a})' - \vec a'\|_{N_1} + \|(\sh{\vec v})' - \vec v'\|_{N_1}  ). 
\end{multline} 
\end{lemma}

\begin{proof} 
%We write the equation satisfied by $\sh{\bs h}_I - \bs h_I$ as follows, 
%\EQ{
%\p_t( \sh{ \bs h}_I - \bs h_I ) &= \bs J \vD^2 E( \bs H(\vec a, \vec v))(\sh{ \bs h}_I- \bs h_I  )  + \bs J \big((\vD^2 E( \bs H(\sh{\vec a}, \sh{\vec v})) -  \vD^2 E( \bs H( {\vec a},  {\vec v})))\sh{ \bs h}_I\big)\\
%&\quad +\pmat{ 0\\U'( H( \sh{\vec a}, \sh{\vec v})) - \sum_{k=1}^n (-1)^kU'( \sh H_k) } -  \pmat{ 0\\U'( H( \vec a, \vec v)) - \sum_{k=1}^n (-1)^kU'( H_k) }\\
% & \quad -\sum_{k =1}^n \big(\sh \lambda_{I, k}  + (-1)^k  \sh v_k\big) \big( \sh{ \p_{a_k} \bs H}_k - \p_{a_k} \bs H_k\big)  -\sum_{k=1}^n  \sh \mu_{I, k}\big( \sh{\p_{v_k} \bs H}_k - \p_{v_k} \bs H_k \big)\\
% &\quad + \sum_{k =1}^n \big((\sh\lambda_{I, k}- \lam_{I, k})  + (-1)^k   (\sh v_k-  v_k)\big) \bs \p_{a_k} \bs H_k   +\sum_{k=1}^n  (\sh \mu_{I, k}- \mu_{I, k})     \p_{v_k} \bs H_k 
%}
Following the main idea in the proof of Lemma~\ref{lem:1st-Picard} we will use an ansatz to extract and estimate the main terms. We decompose $\sh{\bs h}_I - { \bs h}_I= \fl{\bs h}_S + \fl{ {\bs h}}_R$,  $\sh{\vec \lam}_I - {\vec \lam}_I = \fl{\vec \lam}_S +  \fl{\vec{ \lam}_R}$, and $\sh{\vec \mu}_I  - {\vec \mu}_I = \fl{\vec \mu}_S + \fl{\vec{ \mu}}_R$, where $ \fl{\bs h}_S = \sh{ \bs h}_S - \bs h_S$, $\fl{ {\bs h}}_R = \sh{ \bs h}_R - \bs h_R$, $\fl{\vec \lam}_S = \sh{\vec \lam}_S - \vec \lam_S$, $\fl{\vec \lam}_R = \sh{\vec \lam}_R - \vec \lam_R$, $\fl{\vec \mu}_S = \sh{\vec \mu}_S - \vec \mu_S$, and $\fl{\vec \mu}_R = \sh{\vec \mu}_R - \vec \mu_R$. 

It follows that $(\fl{\bs h}_S, \fl{\vec \lambda}_S, \fl{\vec \mu}_S)$ solves 
\EQ{
 \vD^2 E( \bs H(\vec a, \vec v)) \fl{\bs h}_S &= \fl{\bs f}_S 
+ \sum_{k =1}^n \big(\fl{\lambda}_{S, k}  + (-1)^k (\sh v_k-  v_k))\big) \bs \alpha_k  +\sum_{k=1}^n \fl{\mu}_{S, k}   \bs \beta_k, 
}
with 
\EQ{
\la \fl{\bs h}_S , \, \bs \al_k\ra & =  \la \sh {\bs h}_S , \, \bs \al_k - \sh{\bs \al}_k\ra  = \ell_k\\
\la \fl{\bs h}_S ,  \, \bs \be_k \ra& = \la \sh {\bs h}_S , \, {\bs \beta}_k - \sh{\bs \beta}_k \ra = m_k, 
}
and
\EQ{
\fl{\bs f}_S &= -\big(\vD^2 E(\bs H(\sh{\vec a}, \sh{\vec v})) - \vD^2 E(\bs H({\vec a}, {\vec v}))\big) \sh{\bs h}_S\\
 &\quad -\pmat{ U'( H( \sh{\vec a}, \sh{\vec v})) - \sum_{k=1}^n (-1)^kU'( \sh H_k) \\ 0} +  \pmat{ U'( H( \vec a, \vec v)) - \sum_{k=1}^n (-1)^kU'( H_k) \\ 0}\\
 &\quad + \sum_{k =1}^n (\sh{\lam}_{S, k} + (-1)^k \sh{v}_k) \big( \sh{\bs \al}_k - \bs \al_k\big) + \sum_{k=1}^n \sh{\mu}_{S, k} \big( \sh{ \bs \beta}_k -  \bs \beta_k \big), 
}
where we have used the notation $ \sh H_k = H(\sh a_k, \sh v_k), \sh{ \bs \al}_k = \bs J \p_{a} \bs H( \sh a_k, \sh \be_k)$ and $\sh{ \bs \be}_k = \bs J \p_{v} \bs H( \sh a_k, \sh \be_k)$.

From Lemma~\ref{lem:diff-param}~\eqref{eq:h-diff-est} we obtain the estimates
\EQ{
\| \fl{\bs h}_S \|_{\E} \lesssim \big( | \sh{\vec a} - \vec a| + | \sh{\vec v} - \vec v| \big)\big( \| \sh{\bs f}_S \|_{\E} + \| \bs f_S \|_{\E}\big) + \| \sh{ \bs f_S } - \bs f_S \|_{\E}, 
}
where $\bs f_S, \sh{\bs f}_S$ are as in~\eqref{eq:f_S}. 
For each $\gamma < 2$ it follows, using~\eqref{eq:f_S-est} that 
%\EQ{
% \| \fl{\bs h}_S \|_{\E}  \le C_\gamma t^{-\gamma}  ( | \vec a - \sh{\vec a}| + | \vec v - \sh{\vec v}| )%( 1+  \| \sh{\bs h}_I \|_{N_{\gamma}(\E)} )  
%}
%and thus, 
\EQ{ \label{eq:fl-h-S-est} 
\| \fl{\bs h}_S\|_{N_{\gamma}(\E)}  \le C_\gamma \big( \| \vec a - \sh{\vec a}\|_{L^\infty} + \| \vec v - \sh{\vec v}\|_{L^\infty}\big) .
}
We next use~\eqref{eq:lam-diff}, which gives 
\EQ{ \label{eq:lam-diff-est-prelim} 
| \fl{\lambda}_{S, k}&   + (-1)^k  (\sh v_k-  v_k) | \lesssim \big((\gamma_{\sh v_k}^{3} M)^{-1}\la \partial_{v_k}\sh{\bs H}_k, \sh{\bs f}_S\ra- (\gamma_{v_k}^3 M)^{-1}\la \partial_{v_k}\bs H_k, \bs f_S\ra\big)\big|\\
&\quad + | \sh{\vec a} - \vec a|( | \sh{\vec v} | + | \vec v| + y_{\min} e^{-y_{\min}})( \| \bs f_S \|_{\E} + \| \sh{ \bs f }_S\|_{\E})  \\
&\quad +   | \sh{\vec v} - \vec v| ( \| \bs f \|_{\E} + \| \sh{ \bs f }_S\|_{\E}) +C \| \sh{\bs f}_S - \bs f_S \|_{\E} (  | \sh{\vec v} | + | \vec v| + y_{\min} e^{-y_{\min}}),
}
%\EQ{
%\sup_{1 \le k \le n} \big| (  \sh \lam_k -  \lam_k ) - \big((\gamma_{\sh v_k}^{3} M)^{-1}\la \partial_{v_k}\sh{\bs H}_k, \sh{\bs f}\ra- (\gamma_{v_k}^3 M)^{-1}\la \partial_{v_k}\bs H_k, \bs f\ra\big)\big|  \\
%\le C | \sh{\vec a} - \vec a|( | \sh{\vec v} | + | \vec v| + y_{\min} e^{-y_{\min}})( \| \bs f \|_{\E} + \| \sh{ \bs f }_{\E})  \\+ 
%C  | \sh{\vec v} - \vec v| ( \| \bs f \|_{\E} + \| \sh{ \bs f }_{\E}) +C \| \sh{\bs f} - \bs f \|_{\E} (  | \sh{\vec v} | + | \vec v| + y_{\min} e^{-y_{\min}}),
%}
from which we deduce that 
\EQ{ \label{eq:fl-lam-S-est} 
\| \fl{\lambda}_{S, k}&   + (-1)^k  (\sh v_k-  v_k) \|_{N_{\gamma+1}} \le C_\gamma \big( \| \vec a - \sh{\vec a}\|_{L^\infty} + \| \vec v - \sh{\vec v}\|_{N_1}\big).
}
for each $\gamma <2$. 
Similarly, from~\eqref{eq:mu-diff} we obtain 
\begin{multline} 
|\fl{\mu}_{S, k} + (-1)^k ((\gamma_{\sh v_k}^{3} M)^{-1}F_k(\sh{\vec a}, \sh{\vec v}) - (\gamma_{ v_k}^{3} M)^{-1}F_k({\vec a}, {\vec v}))| \\ \lesssim  | \sh{\vec a} - \vec a|( | \sh{\vec v} | + | \vec v| + y_{\min} e^{-y_{\min}})( \| \bs f_S \|_{\E} + \| \sh{ \bs f }_S\|_{\E})  \\
 +   | \sh{\vec v} - \vec v| ( \| \bs f \|_{\E} + \| \sh{ \bs f }_S\|_{\E}) +C \| \sh{\bs f}_S - \bs f_S \|_{\E} (  | \sh{\vec v} | + | \vec v| + y_{\min} e^{-y_{\min}}),
\end{multline} 
from which we see that 
\EQ{\label{eq:fl-mu-S-est} 
\|\fl{\mu}_{S, k} + (-1)^k (M^{-1}F_k(\sh{\vec a}, \sh{\vec v}) -  M^{-1}F_k({\vec a}, {\vec v}))\|_{N_{\gamma+1}} \le C_\gamma\big( \| \vec a - \sh{\vec a}\|_{L^\infty} + \| \vec v - \sh{\vec v}\|_{N_1}\big),
}
for each $\gamma<2$. 
%for all $\gamma<2$.  -- note that we do not obtain an estimate in the space $N_{\gamma +1}$ for the term above due to the presence of the first term on the right in~\eqref{eq:lam-diff-est-prelim}.  \Red{This next part is subtle...need to proceed as in JKL} using~\eqref{eq:lam-mu-est-lm-ii} we obtain the estimate, 
%\EQ{  \label{eq:fl-mu-S-est} 
%\| \fl{\mu}_{S, k}  + (-1)^k (F_k(\sh{\vec a}, \sh{\vec v}) - F_k({\vec a}, {\vec v}))  \|_{N_{\gamma +1}} \le C_\gamma \big( \| \vec a - \sh{\vec a}\|_{L^\infty} + \| \vec v - \sh{\vec v}\|_{L^\infty}\big)
%} 

Next, we consider the equation for $\fl{\bs h}_R$, which is
\EQ{ \label{eq:flat-h-R} 
\p_t \fl{{\bs h}}_R 
& = \bs J\vD^2 E( \bs H(\vec a, \vec v)) \fl{{\bs h}}_R + \fl{{\bs{f}}}_R
 + \sum_{k=1}^n \fl{\lam}_{R, k}     \p_{a_k} \bs H_k + \fl{ \mu}_{R, k}   \p_{v_k} \bs H_k,
}
where 
\EQ{
\fl{{\bs f}}_R &\coloneqq  - \p_t \fl{\bs h}_S  + \bs J\big((\vD^2 E( \bs H(\sh{\vec a}, \sh{\vec v})) - \vD^2 E( \bs H({\vec a}, {\vec v}))) \sh{{\bs h}}_R\big)  \\
 & \quad -\sum_{k =1}^n\sh \lambda_{R, k} \big( \sh{ \p_{a_k} \bs H}_k - \p_{a_k} \bs H_k\big)  -\sum_{k=1}^n  \sh \mu_{R, k} \big( \sh{\p_{v_k} \bs H}_k - \p_{v_k} \bs H_k \big),
}
and
\EQ{ \label{eq:fl-ti-ell-m} 
\la \fl{{\bs h}}_R, \, \bs \al_k \ra &=  \la \sh{ {\bs h}}_R, \, { \bs {\al}}_k - \sh{\bs \al}_k\ra \\
 \la \fl{{\bs h}}_R, \, \bs \be_k \ra &=  \la \sh{ {\bs h}}_R, \, { \bs {\be}}_k - \sh{\bs \be}_k\ra.
 }
 Define 
 \EQ{
 \fl{ \de}(t)\coloneqq \Big( \sum_{k =1}^n (\la \sh{ {\bs h}}_R, \, { \bs {\al}}_k - \sh{\bs \al}_k\ra)^2 + ( \la \sh{ {\bs h}}_R, \, { \bs {\be}}_k - \sh{\bs \be}_k\ra)^2 \Big)^{\frac{1}{2}}.
 }
 Using~\eqref{eq:h-R-est-refined} we have
 \EQ{
 \| \fl{\de} \|_{N_{\gamma +1}} \le C_\gamma(\|\sh{\vec a} - \vec a\|_{L^\infty} + \|\sh{\vec v} - \vec v\|_{L^\infty}) ,
 }
 for all $\gamma<2$.  We claim that
 \EQ{ \label{eq:fl-f-est} 
 \| \fl{ {\bs f}}_R \|_{N_{\gamma +1}(\E)}  \le C_\gamma(\|\sh{\vec a} - \vec a\|_{L^\infty} + \|\sh{\vec v} - \vec v\|_{N_1} +   \|(\sh{\vec a})' - \vec a'\|_{N_1} + \|(\sh{\vec v})' - \vec v'\|_{N_1}  )
 }
 for all $\gamma < 2$.  Using Lemma~\ref{lem:en-est-lin-inf} and the previous two lines we obtain the estimates
 \EQ{ \label{eq:fl-h-R-est} 
 \| \fl{\bs h}_R \|_{N_{\gamma}} \le C_\gamma(\|\sh{\vec a} - \vec a\|_{L^\infty} + \|\sh{\vec v} - \vec v\|_{N_1} +   \|(\sh{\vec a})' - \vec a'\|_{N_1} + \|(\sh{\vec v})' - \vec v'\|_{N_1}  ),
 }
 for all $\gamma<2$. 
 
We prove~\eqref{eq:fl-f-est}.  We have 
 \EQ{
\| \fl{{\bs f}}_R \|_{\E} &\lesssim  \| \p_t \fl{\bs h}_S\|_{\E}  + \| \bs J\big((\vD^2 E( \bs H(\sh{\vec a}, \sh{\vec v})) - \vD^2 E( \bs H({\vec a}, {\vec v}))) \sh{{\bs h}}_R\big)\|_{\E}  \\
 & \quad +\sum_{k =1}^n |\sh \lambda_{R, k} | \| \sh{ \p_{a_k} \bs H}_k - \p_{a_k} \bs H_k\|_{\E} 
  +\sum_{k=1}^n  |\sh \mu_{R, k}|\| \sh{\p_{v_k} \bs H}_k - \p_{v_k} \bs H_k \|_{\E}.
}
We use the refined estimate~\eqref{eq:h-R-est-refined} to estimate the second term by 
\EQ{
\| \bs J\big((\vD^2 E( \bs H(\sh{\vec a}, \sh{\vec v})) - \vD^2 E( \bs H({\vec a}, {\vec v}))) \sh{\ti{\bs h}}\big)\|_{N_{\gamma +1}(\E)} \le  C_\gamma (\|\sh{\vec a} - \vec a\|_{L^\infty} + \|\sh{\vec v} - \vec v\|_{L^\infty}   ),
}
for any $\gamma <2$. The same bound for the third and fourth terms follows from~\eqref{eq:1st-Picard-param}. We claim that $\p_t \fl{\bs h}_S(t)$ satisfies
\EQ{ \label{eq:dt-fl-h-S-est} 
\| \p_t \fl{\bs h}_S \|_{N_{\gamma +1}(\E)} \le C_\gamma(\|\sh{\vec a} - \vec a\|_{L^\infty} + \|\sh{\vec v} - \vec v\|_{N_1} +   \|(\sh{\vec a})' - \vec a'\|_{N_1} + \|(\sh{\vec v})' - \vec v'\|_{N_1}  ),
}
for all $\gamma<2$.  To see this, we write
\EQ{
\p_t \fl{\bs h}_S &= \p_t \sh{\bs h}_S - \p_t \bs h_S = \sum_{ k =1}^n \big((\sh a_k)' \p_{a_k} \sh{\bs h}_S - a_k' \p_{a_k} \bs h_S \big) + \big((\sh v_k)' \p_{v_k} \sh{\bs h}_S - v_k' \p_{v_k} \bs h_S \big),
}
and the bound~\eqref{eq:dt-fl-h-S-est} follows using~\eqref{eq:dh-diff} from Lemma~\ref{lem:diff-param}. This completes the proof of~\eqref{eq:fl-f-est} and thus of~\eqref{eq:fl-h-R-est}.

 We turn to the estimates for the parameters.  We claim that 
 \begin{multline}  \label{eq:flat-param-R}
\| \fl{\vec \lam}_R\|_{N_{\gamma+1}} + \| \fl{\vec \mu}_R\|_{N_{\gamma+1}} \\
 \le C_\gamma(\|\sh{\vec a} - \vec a\|_{L^\infty} + \|\sh{\vec v} - \vec v\|_{N_1} +  \|(\sh{\vec a})' - \vec a'\|_{N_1} + \|(\sh{\vec v})' - \vec v'\|_{N_1}  )
\end{multline} 
for all $\gamma < 2$. Arguing as usual, i.e., multiplying the equation~\eqref{eq:flat-h-R} by $\bs \al_k$, $\bs \beta_k$, we arrive at the following expression for the parameters %and using the above along with the invertibility of the matrix $\calM(\vec a, \vec v)$ we arrive at the estimates, 
\EQ{
\begin{pmatrix}
\fl{\vec\mu}_R \\ \fl{\vec\lam}_R
\end{pmatrix} = 
- \big[\calM(\vec a, \vec v)^{-1} \big]^{\intercal} \pmat{ -\la \p_t \fl{\bs h}_R , \,  \bs \al_1\ra  - \la\fl{\bs h}_R, \,  \vD^2 E( \bs H(\vec a, \vec v))\bs  J \bs \alpha_1 \ra + \la \fl{\bs f}_R, \, \bs \al_1\ra \\ -\la \p_t \fl{\bs h}_R , \,  \bs \beta_1\ra  - \la\fl{\bs h}_R, \,  \vD^2 E( \bs H(\vec a, \vec v))\bs  J \bs \beta_1 \ra + \la \fl{\bs f}_R, \, \bs \beta_1\ra\\ \dots \\ \dots \\  -\la \p_t \fl{\bs h}_R , \,  \bs \al_n\ra  - \la\fl{\bs h}_R, \,  \vD^2 E( \bs H(\vec a, \vec v)) \bs J \bs \alpha_n \ra + \la \fl{\bs f}_R, \, \bs \al_n\ra \\ -\la \p_t \fl{\bs h}_R , \,  \bs \beta_n\ra  - \la\fl{\bs h}_R, \,  \vD^2 E( \bs H(\vec a, \vec v)) \bs J \bs \beta_n \ra + \la \fl{\bs f}_R, \, \bs \beta_n  \ra }.
}
First, we show that for each $k$
  \begin{multline}  \label{eq:pthR-terms} 
  \|\la \p_t \fl{{\bs h}}_R, \, \bs \al_k \ra \|_{N_{\gamma +1}} + \|  \la \p_t \fl{{\bs h}}_R, \, \bs \be_k \ra \|_{N_{\gamma +1}} \\\le C_\gamma(\|\sh{\vec a} - \vec a\|_{L^\infty} + \|\sh{\vec v} - \vec v\|_{N_1} +   \|(\sh{\vec a})' - \vec a'\|_{N_1} + \|(\sh{\vec v})' - \vec v'\|_{N_1}  ) , 
  \end{multline} 
  for each $\gamma <2$. 
 To see the above we differentiate~\eqref{eq:fl-ti-ell-m} to obtain the identities
 \EQ{ \label{eq:dt-fl-ti-ell-m} 
\la \p_t \fl{{\bs h}}_R, \, \bs \al_k \ra &= - \la \fl{{\bs h}}_R, \, \p_t \bs \al_k \ra+  \la \p_t \sh{ {\bs h}}_R, \,{ \bs {\al}}_k - \sh{\bs \al}_k\ra  +  \la \sh{ {\bs h}}_R, \, \p_t { \bs {\al}}_k - \p_t \sh{\bs \al}_k\ra,\\
 \la \p_t \fl{{\bs h}}_R, \, \bs \be_k \ra &= - \la \fl{{\bs h}}_R, \, \p_t \bs \be_k \ra+  \la \p_t \sh{ {\bs h}}_R, \, { \bs {\beta}}_k -\sh{ \bs \beta}_k\ra  +  \la \sh{ {\bs h}}_R, \, \p_t { \bs {\beta}}_k - \p_t \sh{\bs \beta}_k\ra.
  }
  We consider only the second line above as the treatment of the first line is analogous. First, 
  \EQ{
 | \la \fl{{\bs h}}_R, \, \p_t \bs \be_k \ra| \lesssim \| \fl{{\bs h}}_R \|_{\E} \big( | \vec a'| + | \vec v'|\big) ,
  }
  which is bounded by the right-hand side of~\eqref{eq:pthR-terms}  by~\eqref{eq:fl-h-R-est} along with~\eqref{eq:traj-cond}. 
For the second term on the  right of the second line we use~\eqref{eq:h-R-eq}, giving 
  \EQ{ 
 \la \p_t \sh{ \bs h}_R, \,{ \bs {\beta}}_k - \sh{\bs \beta}_k\ra  &= 
 - \la \sh{\bs h}_R, \vD^2 E(\bs H(\sh{\vec a }, \sh{\vec v}))\bs J( \,{ \bs {\beta}}_k - \sh{\bs \beta}_k)\ra - \la \p_t \sh{\bs h}_S, \,{ \bs {\beta}}_k - \sh{\bs \beta}_k\ra  \\
&\quad +  \sum_{k=1}^n \sh{ \lambda}_{R, k} \la \partial_{a_k} \sh{\bs H}_k, \,{ \bs {\beta}}_k - \sh{\bs \beta}_k\ra
+\sh  \mu_{R, k}\la \partial_{v_k}\sh{\bs H}_k , \,{ \bs {\beta}}_k - \sh{\bs \beta}_k\ra.
 }
The last three terms on the right are bounded by the right-hand side of~\eqref{eq:pthR-terms},  using~\eqref{eq:dthS-est} and~\eqref{eq:lam-mu-R-est}. The first term on the right can be written, using~\eqref{eq:beta-id} as 
\EQ{
- &\la \sh{\bs h}_R, \vD^2 E(\bs H(\sh{\vec a }, \sh{\vec v}))\bs J( \,{ \bs {\beta}}_k - \sh{\bs \beta}_k)\ra  = \la \sh{\bs h}_R, \, \bs \al_k - \sh{\bs \al}_k \ra + \la \sh{\bs h}_R, \, \sh v_k \p_x \sh{ \bs \beta}_k - v_k \p_x \bs \beta_k \ra  \\
%&\quad + \la \sh{\bs h}_R, \,\big( (\vD^2 E (\bs H(\vec a, \vec v)) - \vD^2 E( \bs H_k))\p_v \bs H_k\big) - \big( (\vD^2 E (\bs H(\sh {\vec a}, \sh{\vec v})) - \vD^2 E(\sh{ \bs H_k}))\p_v \sh{\bs H}_k\big)\\
&\quad + \la \sh{\bs h}_R, \, \big( \vD^2 E (\bs H(\sh {\vec a}, \sh{\vec v})) - \vD^2 E(\sh{ \bs H_k})\big)\big(\p_v \bs H_k- \p_v \sh{\bs H}_k\big)\ra\\
&\quad + \la \sh{\bs h}_R, \big( \vD^2 E(\sh{ \bs H_k})-\vD^2 E( \bs H_k)\big) \p_v \bs H_k \ra,
}
and the terms on the right are all bounded by the right-hand side of~\eqref{eq:pthR-terms}  using~\eqref{eq:h-R-est-refined}. Similarly, using~\eqref{eq:alpha-id} and~\eqref{eq:beta-id} and~\eqref{eq:fl-h-R-est} we deduce that 
\EQ{
\sup_{1 \le k \le n} \big(\| \la\fl{\bs h}_R, \,  \vD^2 E( \bs H(\vec a, \vec v))\bs  J \bs \alpha_k \ra  \|_{N_{\gamma+1}} + \| \la\fl{\bs h}_R, \,  \vD^2 E( \bs H(\vec a, \vec v))\bs  J \bs \beta_k \ra \|_{N_{\gamma+1}} \big)
}
is bounded by the right-hand side of~\eqref{eq:flat-param-R}. Lastly, 
  \EQ{
\sup_{1 \le k \le n} \big(\| \la \fl{\bs f}_R, \, \bs \alpha_k\ra  \|_{N_{\gamma+1}} + \| \la \fl{\bs f}_R, \, \bs \beta_k\ra\|_{N_{\gamma+1}} \big)
}
is bounded by the right-hand side of~\eqref{eq:flat-param-R} by~\eqref{eq:fl-f-est}. The bounds~\eqref{eq:flat-param-R} follow from the preceding estimates and the uniform boundedness of the operator norm of the matrix $(\calM(\vec a, \vec v)^{-1})^{\intercal}$ for $(\vec a, \vec v)$ with $\rho(\vec a, \vec v) \le \eta_0$. 
%\EQ{
%\sum_{k =1}^n\big|\fl{\lambda}_{S, k}  - (-1)^k ( (a_k- \sh a_k)' - (v_k- \sh v_k)) \big|  \| a_k' \p_{a_k}\bs \alpha_k + v_k' \p_{v_k} \bs \al_k \|_{\E} \lesssim t^{-1-\gamma}\big( \| \vec a - \sh{\vec a} \|_{L^\infty} + \| \vec v - \sh{\vec v} \|_{L^\infty} \big) 
%}
%$ = \p_t \fl{\bs h}_S(\vec a(t), \sh{\vec a}(t), \vec v(t), \sh{\vec v}(t))$, which admits the expansion, 
%\EQ{
%\p_t \fl{\bs h}_S(t)  = \sum_{k =1} a_k'(t) \p_{a_k} \fl{\bs h}_S(t) + (\sh{a_k})'(t) \p_{\sh{a}_k} \fl{\bs h}_S(t) +v_k'(t) \p_{v_k} \fl{\bs h}_S(t) + (\sh{v_k})'(t) \p_{\sh{v}_k} \fl{\bs h}_S(t) 
%}
%By Lemma~\ref{lem:kg-} we have the estimate, 
%\EQ{
%\|  \fl{\ti{\bs h}}\|_{N_\gamma(\E)} \le C_\gamma \| \p_t \fl{\bs h}_S\|_{N_{\gamma +1}(\E)}  
%}
%\Red{But here we can only get $\gamma <1$ it seems due to bad terms in $\p_t \fl{\bs h}_S$.} 

The bound~\eqref{eq:fl-h-I-est} for $\sh{\bs h}_I - \bs h_I$ follow from~\eqref{eq:fl-h-S-est} and~\eqref{eq:fl-h-R-est}. The bounds~\eqref{eq:fl-lam-I-est} and~\eqref{eq:fl-mu-I-est} for the parameters follow from~\eqref{eq:fl-lam-S-est},~\eqref{eq:fl-mu-S-est}, and~\eqref{eq:flat-param-R}. 
\end{proof}

\begin{lemma} \label{lem:diff-est} 
Let $(\vec a, \vec v)$ and $(\sh{\vec a}, \sh{\vec v})$ satisfy \eqref{eq:traj-cond}. 
%and $\sup_{t \geq 0}|\sh{\vec a}(t) - \vec a(t)| \ll 1$ \red{!!!},
Let $(\bs h, \vec \lambda, \vec\mu)$ be the solution of \eqref{eq:kg-nonlin-eq-inf}
obtained in Lemma~\ref{lem:kg-nonlin-eq-inf} associated to $(\vec a, \vec v)$, and let $(\sh{\bs h}, \sh{\vec \lambda}, \sh{\vec\mu})$ be the solution associated to $(\sh{\vec a}, \sh{\vec v})$.
For all $\gamma < 2$ there exists $C_\gamma>0$ such that
\EQ{ \label{eq:h-diff-est-final} 
\|\sh{\bs h} - \bs h\|_{N_\gamma(\E)} &\le C_\gamma(\|\sh{\vec a} - \vec a\|_{L^\infty} + \|\sh{\vec v} - \vec v\|_{N_1} +   \|(\sh{\vec a})' - \vec a'\|_{N_1} + \|(\sh{\vec v})' - \vec v'\|_{N_1}  ), 
}
and the parameters satisfy
\begin{multline} \label{eq:lam-v-diff} 
\sup_{1 \le k \le n} \| (\sh \lambda_{ k}-  \lambda_{ k})  +(-1)^k  (\sh v_k-  v_k)\|_{N_{\gamma+1}} \\
\le C_\gamma(\|\sh{\vec a} - \vec a\|_{L^\infty} + \|\sh{\vec v} - \vec v\|_{N_1} +   \|(\sh{\vec a})' - \vec a'\|_{N_1} + \|(\sh{\vec v})' - \vec v'\|_{N_1}  ), 
\end{multline} 
and, 
\begin{multline} \label{eq:mu-F-diff} 
\sup_{ 1 \le k \le n} \| (\sh \mu_{ k}-  \mu_{ k})  + (-1)^k ( M^{-1}F_k(\sh{\vec a}, \sh{\vec v}) - M^{-1}F_k({\vec a}, {\vec v}))  \|_{N_{\gamma +1}}\\
 \le C_\gamma(\|\sh{\vec a} - \vec a\|_{L^\infty} + \|\sh{\vec v} - \vec v\|_{N_1} +   \|(\sh{\vec a})' - \vec a'\|_{N_1} + \|(\sh{\vec v})' - \vec v'\|_{N_1}  ). 
\end{multline} 
%\begin{align}
%|\sh\lambda_j(t) etc.
%\end{align}
\end{lemma}
%\Red{in argument below, move $\gamma_{v_k}$ terms to become errors on the right}
\begin{proof}
We decompose ${\bs h} = {\bs h}_I + {\bs h}_*$ and $\sh{\bs h} = \sh{\bs h}_I + \sh{\bs h}_*$, as explained above.
Let $\fl{\bs h} \coloneqq  \sh {\bs h}_* - \bs h_*$, $\fl \lambda \coloneqq  \sh{\vec \lambda}_* - \vec\lambda_*$ and
$\fl \mu \coloneqq  \sh{\vec \mu}_* - \vec\mu_*$.
Taking the difference of \eqref{eq:hstar} and the analogous equation for $\sh{\bs h}_*$,
we obtain that $\fl {\bs h}$ satisfies
\begin{equation}
\label{eq:nth-eq}
\begin{aligned}
\partial_t \fl{\bs h}(t) &= \bs J\vD^2 E(\bs H(\vec a(t), \vec v(t)))\fl{\bs h}(t) + \fl{\bs f}(t) \\
&+ \sum_{k=1}^n \fl{\lambda}_k(t) \partial_{ a_k}\bs H_k(t)
+ \fl{\mu}_k(t) \partial_{ v_k}\bs H_k(t),
\end{aligned}
\end{equation}
with 
\EQ{
\la \fl{\bs h}, \, \bs \al_k \ra = \la \sh{\bs h}_*, \, \bs{\al}_k - \sh{ \bs \al}_k \ra , \quad \la \fl{\bs h}, \, \bs \beta_k \ra = \la \sh{\bs h}_*, \, \bs{\beta}_k - \sh{ \bs \beta}_k ,\ra 
}
and 
where $\fl{\bs f} = \bs f_l + \bs f_q + \bs f_m$,
\begin{equation}
\label{eq:fl}
\begin{aligned}
\bs f_l(t) &\coloneqq \bs J(\vD^2 E(\bs H(\sh{\vec a}(t), \sh{\vec v}(t))) - \vD^2 E(\bs H(\vec a(t), \vec v(t))))\sh{\bs h}_*(t),
\end{aligned}
\end{equation}
\begin{equation}
\label{eq:fq}
\begin{aligned}
\bs f_q(t) &\coloneqq J\big(\vD E(\bs H(\sh{\vec a}(t), \sh{\vec v}(t))+\sh{\bs h}_I(t) + \sh{\bs h}_*(t)) - \vD E(\bs H(\sh{\vec a}(t), \sh{\vec v}(t)))  \\
&\qquad- \vD^2 E(\bs H(\sh{\vec a}(t), \sh{\vec v}(t)))(\sh{\bs h}_I(t) + \sh{\bs h}_*(t))\big) \\
&- J\big(\vD E(\bs H(\vec a(t), \vec v(t))+\bs h_I(t) + \bs h_*(t)) - \vD E(\bs H(\vec a(t), \vec v(t)))  \\
&\qquad- \vD^2 E(\bs H(\vec a(t), \vec v(t)))(\bs h_I(t) + \bs h_*(t))\big),
\end{aligned}
\end{equation}
and
\begin{equation}
\label{eq:fm}
\begin{aligned}
\bs f_m(t) &\coloneqq \sum_{k =1}^n \sh{ \lambda}_{*, k} (\partial_{ a_k}\sh{\bs H}_k(t) - \partial_{a_k}\bs H_k(t) ) + \sh{ \mu}_{*, k}(\partial_{ v_k}\sh{\bs H}_k(t) - \partial_{ v_k}\bs H_k(t)).
\end{aligned}
\end{equation}
Using \eqref{eq:hstar-decay}, we obtain
\begin{equation}
\|\bs f_l\|_{N_{\gamma+1}} \leq C_\gamma(\|\sh{\vec a} -\vec  a\|_{L^\infty} + \|\sh {\vec v} -\vec  v\|_{L^\infty}),
\end{equation}
for all $\gamma <2$. 
This is the worst term: for $\bs f_q$ and $\bs f_m$ an analogous estimate holds for $\gamma < 3$. Setting
\EQ{
\fl{\de} = \Big(\sum_{k =1}^n (\la \sh{\bs h}_*, \, \bs{\al}_k - \sh{ \bs \al}_k \ra)^2 + (\la \fl{\bs h}, \, \bs \beta_k \ra = \la \sh{\bs h}_*, \, \bs{\beta}_k - \sh{ \bs \beta}_k \ra)^2 \Big)^{\frac{1}{2}},
}
it follows from Lemma~\ref{lem:en-est-lin-inf}, the above, and~\eqref{eq:hstar-decay} that
\EQ{
\| \fl{ \bs h} \|_{N_{\gamma}(\E)} \lesssim_{\gamma}  \| \fl{\bs f}\|_{N_{\gamma +1}(\E)} + \| \fl\de \|_{N_{\gamma +1}} \lesssim_\gamma \|\sh{\vec a} -\vec  a\|_{L^\infty} + \|\sh {\vec v} -\vec  v\|_{L^\infty},
}
for all $\gamma <2$. Arguing exactly as in the proof of~\eqref{eq:flat-param-R} it follows as well that 
\EQ{
\| \fl{\vec \lam}  \|_{N_{\gamma +1}} + \| \fl{\vec \mu} \|_{N_{\gamma +1}} \lesssim_\gamma \|\sh{\vec a} -\vec  a\|_{L^\infty} + \|\sh {\vec v} -\vec  v\|_{L^\infty}  +   \|(\sh{\vec a})' - \vec a'\|_{L^\infty} + \|(\sh{\vec v})' - \vec v'\|_{L^\infty} .
}
%In fact, if we only want $\gamma < 3$, then in Lemma~\ref{lem:1st-Picard-diff} $\gamma < 1$ should be enough.
%We have
%\begin{equation}
%|\la \bs \alpha(\vec a(t), \vec v(t)), \fl{\bs h}(t)\ra| = |\la \bs \alpha(\sh{\vec a}(t), \sh{\vec v}(t))- \bs \alpha(\vec a(t), \vec v(t)), \sh{\bs h}_*(t)\ra| \lesssim |\sh{\vec a}(t) - \vec a(t)| \|\sh{\bs h}_*(t)\|_\cE,
%\end{equation}
%and similarly for $\bs\beta$ instead of $\bs\alpha$.
%Consider the first and second line of the right hand side of \eqref{eq:nth-eq-f}.
Using the conclusions of Lemma~\ref{lem:1st-Picard-diff} and the previous two estimates completes the proof.  
\end{proof}

%%%%%%%%%%%%%%%%%%%%%%%%%%%%%%%%%%%%%%%%%%%%%%%%%%%%%%%%%%%%%%%%%%%%%%%%%%%%

\section{Analysis of the bifurcation equation} \label{sec:bifurcation} 

Recall that our goal is to study solutions to~\eqref{eq:csf} of the form 
\EQ{
\bs \phi(t, x) = \bs H( \vec a(t), \vec v(t); x) + \bs h(t, x) 
}
with trajectories $(\vec a(t), \vec v(t))$ as in~\eqref{eq:traj-cond} and such that the error satisfies $ \lim_{t \to \infty} \| \bs h(t, x) \|_{\E}  = 0$ and the trajectories satisfy $\lim_{t \to \infty} \rho( \vec a(t) ,\vec v(t)) = 0$. The relationship between the trajectories and the error is fixed by requiring the orthogonality conditions~\eqref{eq:h-orth}. 
%\EQ{
%\la \bs\alpha_k(t), \bs h(t)\ra = 
%\la \bs\beta_k(t), \bs h(t)\ra = 0, \quad \forall \,t .
%}
We can thus rewrite~\eqref{eq:csf} as the following coupled equations for $( \bs h(t), \vec a(t), \vec v(t))$:
\EQ{ \label{eq:h-nl} 
\p_t \bs h(t) &= \bs J \vD E( \bs H(\vec a(t), \vec v(t)) + \bs h(t)) -  \sum_{k=1}^n (-1)^k \big( a_k'(t) \p_{a} \bs H_k(t)  +  v_k'(t) \p_v \bs H_k(t)) \\
0 & = \la \bs\alpha_k(t), \bs h(t)\ra = \la \bs\beta_k(t), \bs h(t)\ra, \quad \forall \, k . 
}
In view of Lemma~\ref{lem:kg-nonlin-eq-inf}, we see that solving~\eqref{eq:h-nl} is equivalent to finding trajectories $(\vec a(t), \vec v(t))$ so that the triplet $(\bs h(t), \vec \lam(t), \vec \mu(t))$ given by Lemma~\ref{lem:kg-nonlin-eq-inf} satisfies
\EQ{
a_k'(t) = -(-1)^k \lam_k(t), \quad v_k'(t) =- (-1)^k \mu_k(t) 
} 
for each $k \in \{1, \dots, n\}$. This is achieved in the following proposition.

\begin{proposition}[Solution of the bifurcation equations]
\label{prop:sol-bif}
There exist $\tau_0, \eta_0 > 0$ such that for all $t_0 \geq \tau_0$ and $a_{1, 0}, \ldots a_{n, 0}$ real numbers satisfying
\begin{equation} \label{eq:a-data} 
\max_{1\leq k < n}\bigg|\big(a_{k+1, 0} - a_{k, 0}\big) - \Big(2\log(\kappa t_0) - \log\frac{Mk(n-k)}{2}\Big)\bigg| < \eta_0
\end{equation}
there exists a unique pair of trajectories $(\vec a(t), \vec v(t))$ satisfying \eqref{eq:traj-cond} with $\vec a(t_0)= \vec a_0$ and
\begin{equation} \label{eq:a=l-v=m}
a_k'(t) + (-1)^k\lambda_k(t) = v_k'(t) + (-1)^k\mu_k(t) = 0,\qquad\text{for all }k, 
\end{equation}
where $(\bs h(t), \vec \lam(t), \vec \mu(t))$ is the solution associated to $(\vec a(t), \vec v(t))$ given by Lemma~\ref{lem:kg-nonlin-eq-inf}. 
\end{proposition}

%\Red{Is there an $M$ in $\mu, v$ equation above or not?} 
The remainder of this section is devoted to proving Propostion~\ref{prop:sol-bif}. 

\subsection{Linearized problem}

Consider the equation
\EQ{ \label{eq:toda} 
a_k'(t) &= v_k(t) \\
M v_k'(t) & = 2 \kappa^2 \big( e^{-(a_{k+1}(t) - a_k(t))} - e^{-(a_{k}(t) - a_{k-1}(t))} \big) . 
}
This has an explicit solution (defined up to translation) by the formulas
\EQ{ \label{eq:expl} 
a_{k+1}(t) - a_{k}(t) = 2 \log( \kappa t) - \log\Big( \frac{M k (n-k)}{2} \Big) , \quad v_k(t) =  - \frac{ n+1 - 2k}{t}. 
}
In what follows we fix the solution $(\vec a_\tx{expl}(t), \vec v_\tx{expl}(t) )$ of~\eqref{eq:toda} satisfying~\eqref{eq:expl} with barycenter at the origin (i.e., we require that $ \sum_{k =1}^n a_{\tx{expl}, k}(t) = 0$ for each $t$, but this is an arbitrary choice). 

The proof of Proposition~\ref{prop:sol-bif} begins with an analysis of the linearization of the bifurcation equations about the solution $(\vec a_\tx{expl}(t), \vec v_\tx{expl}(t) )$. Define 
\EQ{
F_{\tx{expl}, k}(\vec a) = 2 \kappa^2 \big( e^{-(a_{k+1}(t) - a_k(t))} - e^{-(a_{k}(t) - a_{k-1}(t))} \big). 
}
The Jacobian matrix of $\vec F_{\tx{expl}} (\vec a)$ evaluated at the explicit solution $\vec a_\tx{expl}(t)$ is given by 
\EQ{
\vD \vec F_{\tx{expl}} (\vec a_\tx{expl}(t) ) =  M t^{-2} \scrU, 
}
where the matrix $\scrU = \scrU_n = (u_{jk}) \in \bR^{n\times n}$ is defined by
\begin{equation}
\label{eq:U-def}
u_{k,k+1} \coloneqq  -k(n-k), \quad u_{k+1,k} \coloneqq  -k(n-k),\quad u_{j,k} = 0\ \text{ if }\ |j - k| \geq 2,
\end{equation}
and the elements on the diagonal are determined by the requirement that the sum of the entries in each row equals 0, that is
\begin{equation}
u_{k,k} \coloneqq  -u_{k, k-1}-u_{k,k+1} = (k-1)(n-k+1) + k(n-k),
\end{equation}
where here and later we adopt the convention
\begin{equation}
u_{0,1} = u_{1,0} = u_{n+1,n} = u_{n,n+1} = 0.
\end{equation}
Note that \eqref{eq:U-def} still holds for $k = 0$ and $k = n$.
\begin{definition}
We define the \emph{Legendre vectors} $\vec P_0, \vec P_1, \ldots, \vec P_{n-1} \in \bR^n$ as follows:
\begin{itemize}
\item $\vec P_0 \coloneqq  (1, 1, \ldots, 1)$,
\item for $ 1 \leq \ell \leq n-1$, $\vec P_\ell$ is the orthogonal projection of $(0^\ell, 1^\ell, \ldots, (n-1)^\ell)$
on the space orthogonal to $\vec P_0, \ldots, \vec P_{\ell-1}$.
\end{itemize}
\end{definition}
\begin{proposition}
For $\ell \in \{0, 1, \ldots, n-1\}$ it holds
\begin{equation}
\scrU\vec P_\ell = \ell(\ell+1)\vec P_\ell.
\end{equation}
\end{proposition}
\begin{proof}
An explicit expansion of all the terms as polynomials with respect to $j-1$ reveals that
\begin{equation}
\sum_{k=1}^n u_{j, k} (k-1)^\ell = \ell(\ell + 1)(j-1)^\ell + \sum_{\ell' = 0}^{\ell-1} c_{n, \ell, \ell'}(j-1)^{\ell'},
\end{equation}
where $c_{n, \ell, \ell'}$ are some coefficients whose precise form is not needed here
(it is only crucial that they do not depend on $j$).
The conclusion thus follows from the theory of symmetric matrices.
\end{proof}
In the next lemma we study the bifurcation equations linearized around the explicit trajectory
(if we translate in space the explicit trajectory, the linearized equations are unchanged).
%\begin{lemma}
%The linear system
%\begin{equation}
%\vec b'(t) = \vec w(t) + \vec f(t), \qquad \vec w'(t) = (t_0+t)^{-2}\scrU \vec b(t) + \vec g(t)
%\end{equation}
%has a unique solution converging to $0$ as $t \to \infty$ and such that $\vec b(0) = \vec b_0$.
%It satisfies \red{estimates}
%\begin{equation}
%|\vec b(t) - \vec b_0| \lesssim \qquad \text{for all }t \geq 0
%\end{equation}
%\end{lemma}
\begin{lemma} \label{lem:Euler} 
Let $t_0\ge 1$ and let $\vec F= (F_1, \dots, F_n)$ and $\vec G = (G_1, \dots, G_n)$ be such that $F_j, G_j \in L^1([t_0, \infty); \R)$ for each $j \in \{1, \dots, n\}$. Let $\gamma_1, \gamma_2 \in (1, 2)$ and assume that $F_j \in N_{\gamma_1+1}$ for each $j  \in {1, \dots, n}$ and assume $G_j \in N_{2}$ for each $j \in \{1, \dots, n\}$. Let $\vec B_0 \in \R^n$.

Define the projection coefficients of $\vec F$, $\vec G$ onto the Legendre vectors by 
\EQ{
f_\ell(t) \coloneqq  | \vec P_{\ell}|^{-2} \la \vec F(t), \, \vec P_\ell \ra_{\R^n}, \quad g_\ell(t) \coloneqq  | \vec P_{\ell}|^{-2}  \la \vec G(t), \, \vec P_\ell \ra_{\R^n}
}
for $\ell = \{0, \dots, n-1\}$. 
Assume further that  
\EQ{
g_{ 0}(t)  = \frac{1}{n}(G_1(t) + \dots + G_n(t))
}
satisfies $g_0 \in N_{\gamma_2+1}$. 
 The linear system
\begin{equation}
\vec B'(t) = \vec W(t) + \vec F(t), \qquad \vec W'(t) = (t_0+t)^{-2}\scrU \vec B(t) + \vec G(t)
\end{equation}
has a unique solution that remains bounded as $t \to \infty$ and such that $\vec B(0) = \vec B_0$. Moreover, there exists $C_0 = C_0(n) \ge 0$ such that it satisfies the estimates 
\EQ{\label{eq:B-bound} 
|\vec B(t)| \le C  \Big( |\vec B_0| + t_0^{-\gamma_1} \| \vec F \|_{N_{\gamma_1 +1}} +  \| \vec G \|_{N_{2}} + t_0^{-\gamma_2 +1}\| g_0 \|_{N_{\gamma_2 + 1}}  \Big)\qquad \text{for all }t \geq 0, 
}
%\EQ{
% |B_{k+1} (t) -B_k(t)|  \lesssim \max_{\ell = 1, \dots, n-1} \| g_\ell \|_{N_2}  + \dots 
%}
for each $k = 1, \dots, n-1$, as well as 
\EQ{ \label{eq:W-bound} 
\| \vec W(t) \|_{N_1} \le C_0 \Big(   |\vec B_0|+ t_0^{-\gamma_1} \| \vec F \|_{N_{\gamma_1 +1}} +  \| \vec G \|_{N_{2}} + t_0^{-\gamma_2 +1}\| g_0 \|_{N_{\gamma_2 + 1}}\Big) 
}

\end{lemma}
\begin{proof}
Decomposing everything in the eigenbasis of $\scrU$, we obtain the system
\begin{equation}
b_{ \ell}'(t) = w_\ell(t) + f_\ell(t), \qquad w_\ell'(t) = \ell(\ell + 1)(t_0+t)^{-2}b_\ell(t) + g_\ell(t).
\end{equation}
This is a standard Euler differential equation. A particular solution is given by
\begin{equation}
\begin{aligned}
b_\ell^{\tx p}(t) &\coloneqq \frac{(t_0+t)^{-\ell}}{2\ell+1} \int_0^t \big((\ell+1)(t_0+s)^\ell f_\ell(s)-  (t_0+s)^{\ell+1} g_\ell(s)\big)\ud s\\
&+ \frac{(t_0+t)^{\ell + 1}}{2\ell + 1}\int_t^\infty \big( \ell(t_0 + s)^{-\ell - 1}f_\ell(s)
- (t_0 +s)^{-\ell}g_\ell(s)\big)\ud s
\end{aligned}
\end{equation}
and
\begin{equation}
\begin{aligned}
w_\ell^{\tx p}(t) &\coloneqq \frac{-\ell(t_0+t)^{-\ell-1}}{2\ell+1} \int_0^t \big((\ell+1)(t_0+s)^\ell f_\ell(s)-  (t_0+s)^{\ell+1} g_\ell(s)\big)\ud s\\
&+ \frac{(\ell + 1)(t_0+t)^{\ell}}{2\ell + 1}\int_t^\infty \big( \ell(t_0 + s)^{-\ell - 1}f_\ell(s)
- (t_0 +s)^{-\ell}g_\ell(s)\big)\ud s.
\end{aligned}
\end{equation}

The solutions of the homogenous equation which are bounded are spanned by $(t_0 + t)^{-\ell}$,
and we add an appropriate multiple so that the initial condition is satisfied. Indeed, for each initial data $\vec B(0)  = \vec B_0\in \R^n$ the unique bounded solution $(\vec B(t), W(t)) = (\sum_{\ell =0}^{n-1} b_{\ell}(t) \vec P_{\ell}, \sum_{\ell =0}^{n-1} w_{\ell}(t) \vec P_{\ell} )$ with $\vec B(0) = \vec B_0 = \sum_{\ell =0}^{N-1} b_{0, \ell} \vec P_\ell$ is given  by 
\EQ{
b_{\ell}(t) &= t_0^{\ell}(b_{0, \ell}  - b^{\tx p}_{\ell}(0))(t_0 + t)^{-\ell} + b_\ell^{\tx p}(t) \\
w_{\ell}(t) &=    -\ell t_0^{\ell}(b_{0, \ell}  - b^{\tx p}_{\ell}(0))(t_0 + t)^{-\ell-1}   + w_\ell^{\tx p}(t)
}
For $\ell = 0$ we have 
\EQ{
b_{0}(t) - b_{0, 0} & = b_0^{\tx p}(t) - b_0^{\tx p}(0) \\
& =  \int_0^t \big( f_0(s)-  (t_0+s)g_0(s)\big)\ud s - (t_0 + t)\int_t^\infty g_0(s)\ud s   + t_0 \int_0^\infty g_0(s)\ud s, 
}
and hence 
\EQ{
|b_{0}(t) - b_{0, 0}| \lesssim t_0^{-\gamma_1} \| \vec F \|_{N_{\gamma_1 + 1}} + t_0^{-\gamma_2 +1} \| g_0 \|_{N_{\gamma_2 +1}}  . 
}
For $\ell \ge 1$ we have
\EQ{
b_{\ell}(t) - \frac{t_0^\ell}{(t_0 + t)^{\ell}} b_{0, \ell} %&= t_0^{\ell}(b_{0, \ell}  - b^{\tx p}_{\ell}(0))(t_0 + t)^{-\ell} + b_\ell^{\tx p}(t)  - b_{\ell, 0} \\
& =   -  \frac{t_0^\ell} {(t_0 + t)^{\ell}}  b^{\tx p}_{\ell}(0) +  b_\ell^{\tx p}(t), 
}
and thus
\EQ{
|b_{\ell}(t) -  \frac{t_0^\ell}{(t_0 + t)^{\ell}} b_{ 0, \ell}| \lesssim t_0^{-\gamma_1} \| \vec F \|_{N_{\gamma_1+1}} + \| \vec G \|_{N_{2}}
}
which gives~\eqref{eq:B-bound}. Next, for $\ell =0$, we have
\EQ{
(t_0 + t) | w_0(t) | \le (t_0 + t) | w_0^p(t)|  \le (t_0 + t) \int_t^\infty | g_0(s)| \, \ud s  \lesssim (t_0 + t)^{-\gamma_2+ 1}  \| g_0 \|_{N_{\gamma_2 + 1}}
}
If $\ell \ge 1$, 
\EQ{
(t_0 + t) | w_\ell(t) |  &\le \frac{\ell t_0^\ell}{ (t_0 + t)^\ell}| b_{0,\ell}| + \frac{ \ell t_0^\ell}{ (t_0 + t)^\ell} |b_{\ell}^p(0)|  + (t_0 +t) | w_\ell^p(t)| \\
& \lesssim | b_{0, \ell}| +  t_0^{-\gamma_1}  \| f_\ell \|_{N_{\gamma_1 +1}} + \| g_\ell \|_{N_2}
}
completing the proof of~\eqref{eq:W-bound}. 
%The function $(t_0 + t)^{-\ell}$ is non-increasing, so we get a bound in $L^\infty$ by $C|\vec b_0|$.
\end{proof}

\begin{proof}[Proof of Proposition~\ref{prop:sol-bif}]
%We  will set up a contraction for the norm appearing on the right hand side in Lemma~\ref{lem:diff-est}.
We write $\vec a(t) = \vec a_\tx{expl}(t) + \vec b(t)$ and $\vec v(t) = \vec v_\tx{expl}(t) + \vec w(t)$, where $(\vec a_\tx{expl}(t), \vec v_\tx{expl}(t))$ is the explicit solution
of the $n$-body problem defined by~\eqref{eq:expl} with the barycenter at $0$ (this last condition is an arbitrary choice and could be any number). The idea is to set up a contraction mapping amongst  perturbations of this explicit trajectory that are small in the norm appearing on the right hand side in Lemma~\ref{lem:diff-est}.

Let $\eta_0>0$ be a small number and $\tau_0>0$ be a large number, both to be fixed below. Let $\vec a_0$ be as in~\eqref{eq:a-data}, and assume (after translating) that $\vec a_0$ also has barycenter $=0$ like the explicit trajectory $ \vec a_{\tx{expl}}(t)$. Let $t_0 \ge \tau_0$ and we consider trajectories $(\vec a(t), \vec v(t))$, satisfying~\eqref{eq:traj-cond} with $\vec a(t_0) = \vec a_0$. For $t_0$ as above and trajectories $(\vec b, \vec w) \in C^1([t_0, \infty); \R^n \times \R^n)$ with $\vec b(t_0) = \vec b_0 \coloneqq  \vec a(t_0) - \vec a_\tx{expl}(t_0)$, define the norm $\calX = \calX([t_0, \infty)$ by 
\EQ{
\| (\vec b, \vec w) \|_{\calX} \coloneqq  \| \vec b \|_{L^\infty([t_0, \infty))} + \| \vec b'\|_{N_1([t_0, \infty))} + \| \vec w \|_{N_1([t_0, \infty))} + \| \vec w' \|_{N_1([t_0, \infty))}.
}
Considering only trajectories that have $\vec b(t_0) = \vec b_0$, for $\eta>0$ we denote the $\eta$-ball in $\calX$ by 
\EQ{
\calB_\calX( \eta) \coloneqq  \{ (\vec b, \vec w) \in C^1([t_0, \infty); \R^n \times \R^n) \mid  \|( \vec b, \vec w) \|_{\calX}  < \eta \mand \vec b(t_0) = \vec b_0\}.
}
To trajectories $(\vec a, \vec v) = (\vec a_\tx{expl} + \vec b, \vec v_\tx{expl} + \vec w)$ with $(\vec b, \vec w ) \in \calB_{\calX}(\eta)$ we associate new trajectories  $(\vec \fa, \vec \fv) = (\vec a_\tx{expl} + \vec \fb, \vec v_\tx{expl} + \vec \fw)$
defined as the solution of the system
\begin{equation}
\begin{aligned}
\fb_k'(t) &= \fw_k(t) - \big((-1)^k \lambda_k(t) + v_k(t)\big), \\
\fw_k'(t) &= (t_0+t )^{-2}(\scrU \vec \fb(t))_k - \big((-1)^k\mu_k(t) + M^{-1}F_{\textrm{expl}, k}(\vec a_{\textrm{expl}}) + (t_0+t )^{-2}(\scrU \vec b(t))_k\big).
\end{aligned}
\end{equation}
where on the right-hand side, $(\vec \lam(t), \vec \mu(t))$ is determined by $(\vec a, \vec v)$ via Lemma~\ref{lem:kg-nonlin-eq-inf}.  
We denote this association as a mapping $\Phi( \vec b(t), \vec w(t)) = (\vec \fb(t), \vec \fw(t))$. We will show that there exist $\eta_0>0$ sufficiently small and $\tau_0>0$ sufficiently large  the mapping $\Phi$  defines a contraction on $\calB_\calX( 8 C_0\eta_0)$ where $C_0$ is the constant from Lemma~\ref{lem:Euler}. 

First we show that $\Phi: \calB_\calX( 8C_0\eta_0) \to \calB_\calX(8 C_0\eta_0)$. Setting up an application of Lemma~\ref{lem:Euler} we expand the forcing on right-hand side of the equation for $\fw_k$ as follows, defining 
\EQ{
-G_k(t)\coloneqq   (-1)^k\mu_k(t) &+ M^{-1}F_{\textrm{expl}, k}(\vec a_{\textrm{expl}}) + (t_0+t )^{-2}(\scrU \vec b(t))_k \\
& = (-1)^k\mu_k(t) + M^{-1} F_k((\vec a_\tx{expl} + \vec b, \vec v_\tx{expl} + \vec w)   \\
&\quad -  M^{-1}\Big(F_k(\vec a_\tx{expl} + \vec b, \vec v_\tx{expl} + \vec w)- F_{\textrm{expl}, k}(\vec a_{\textrm{expl}} + \vec b)  \Big) \\
&\quad - M^{-1} \Big( F_{\textrm{expl}, k}(\vec a_{\textrm{expl}} + \vec b) - F_{\textrm{expl}, k}(\vec a_{\textrm{expl}}) - (t_0+t )^{-2}(\scrU \vec b(t))_k\Big)
}
By~\eqref{eq:nonlin-param} we have 
\EQ{ \label{eq:param1} 
\max_{1 \le k \le n} \Big( \| (-1)^k \lambda_k + v_k\|_{N_{\gamma+1}} + \| (-1)^k\mu_k(t) + M^{-1} F_k(\vec a_{\textrm{expl}} + \vec b) \|_{N_{\gamma+1}} \Big) \le  1
}
for each $\gamma <2$, provided $t_0$ is sufficiently large. By~\eqref{eq:Fz}, we have 
\EQ{ \label{eq:F-Fexpl} 
\| F_k(\vec a_\tx{expl} + \vec b, \vec v_\tx{expl} + \vec w)- F_{\textrm{expl}, k}(\vec a_{\textrm{expl}} + \vec b) \|_{N_3} \le 1
}
provided $t_0$ is large enough. Using Taylor's formula with remainder we find a constant $C = C(n)$ so that 
\EQ{
\| F_{\textrm{expl}}(\vec a_{\textrm{expl}} + \vec b) - F_{\textrm{expl}}(\vec a_{\textrm{expl}}) - (t_0+t )^{-2}(\scrU \vec b(t)) \|_{N_2} \le C \| \vec b \|_{L^\infty}^2.
}
And finally observe that 
\EQ{
\la \vec F_{\textrm{expl}}(\vec a_{\textrm{expl}} + \vec b) - \vec F_{\textrm{expl}}(\vec a_{\textrm{expl}}) - (t_0+t )^{-2}(\scrU \vec b(t))  , \, \vec P_0 \ra = 0
}
which, combined with~\eqref{eq:param1} and~\eqref{eq:F-Fexpl} gives 
\EQ{
\| \la \vec G, \, \vec P_0 \ra \|_{N_{\gamma+1}} \le 1 
}
for each $\gamma < 2$, provided  $t_0$ is sufficiently large. 

An application of Lemma~\ref{lem:Euler} together with the above estimates yields  
\EQ{
\|(\vec \fb, \vec \fw)\|_{\calX} < 8 C_0 \eta_0  
}
after taking $t_0$ sufficiently large and $\eta_0>0$ sufficiently small, and making use of the equation for the estimates for $\vec \fb', \vec \fw$. 

We address the contraction property. Consider two trajectories $( \vec b, \vec w), ( \sh {\vec b} \sh{\vec w}) \in \calB_{\calX}( 8 C_0 \eta_0)$ and their images $( \vec {\fb} , \vec {\fw}), ( \sh{\vec{\fb}}, \sh{\vec \fw})$ under the mapping $\Phi$. Then $(\fl{\vec \fb}, \fl{\vec \fw}) \coloneqq  ( \sh{\vec{\fb}}, \sh{\vec \fw}) - ( \vec \fb , \vec {\fw})$ satisfies the equation, 
\EQ{ \label{eq:shb-b-eq} 
(\fl {\fb}_k)'(t) &= \fl \fw_k(t) - \Big( \big((-1)^k\sh  \lambda_k(t) +  \sh v_k(t)\big)- \big((-1)^k \lambda_k(t) + v_k(t)\big), \\
(\fl \fw_k)'(t) &= (t_0+t )^{-2}(\scrU \fl{\vec  \fb}(t))_k \\
&\quad  - \Big( \big((-1)^k\sh \mu_k(t) + M^{-1} F_k(\sh{\vec a}(t), \sh{\vec v}(t) )\big) - \big((-1)^k\mu_k(t) + M^{-1} F_k(\vec a(t), \vec v(t)) \big) \Big)  \\
&\quad -  M^{-1}\Big(\big(F_k(\sh{\vec a}(t), \sh{\vec v}(t))- F_{\textrm{expl}, k}(\sh{\vec a}(t))\big) -  \big(F_k(\vec a(t), \vec (t))- F_{\textrm{expl}, k}(\vec a(t))\big) \Big) \\
& \quad - M^{-1} \Big( F_{\textrm{expl}, k}(\sh{\vec a}(t)) - F_{\textrm{expl}, k}(\vec a_{\textrm{expl}}(t)) - (t_0+t )^{-2}(\scrU \sh{\vec b(t)})_k\Big)\\
&\quad + M^{-1} \Big( F_{\textrm{expl}, k}({\vec a}(t)) - F_{\textrm{expl}, k}(\vec a_{\textrm{expl}}(t)) - (t_0+t )^{-2}(\scrU {\vec b(t)})_k\Big)
}
with zero initial data. The estimates~\eqref{eq:lam-v-diff} and~\eqref{eq:mu-F-diff} give that 
\EQ{
\| \big((-1)^k\sh  \lambda_k(t) +  \sh v_k(t)\big)- \big((-1)^k \lambda_k(t) + v_k(t)\big)\|_{N_{\gamma +1}} \le C_{\gamma}  \| ( \sh{\vec{b}}, \sh{\vec w}) - ( \vec {b} , \vec {w}) \|_{\calX} \\
\|  \big((-1)^k(\sh \mu_k(t) - \mu_k(t) \big) + M^{-1} \big( F_k(\sh{\vec a}(t), \sh{\vec v}(t) )\big) -F_k(\vec a(t), \vec v(t)) \big)  \|_{N_{\gamma + 1}} \le C_{\gamma}  \| ( \sh{\vec{b}}, \sh{\vec w}) - ( \vec {b} , \vec {w}) \|_{\calX} 
}
for each $\gamma \in [1, 2)$.  We also note the estimates, %\Red{include proof of these?}
\EQ{
\| \big(F_k(\sh{\vec a}(t), \sh{\vec v}(t))- F_{\textrm{expl}, k}(\sh{\vec a}(t))\big) -  \big(F_k(\vec a(t), \vec v(t))- F_{\textrm{expl}, k}(\vec a(t))\big) \|_{N_3} \lesssim \| \sh{\vec b} - \vec b \|_{L^\infty}
}
and 
\begin{multline} 
\Big\| \big( F_{\textrm{expl}, k}(\sh{\vec a}(t)) - F_{\textrm{expl}, k}(\vec a_{\textrm{expl}}(t)) - (t_0+t )^{-2}(\scrU \sh{\vec b(t)})_k\big) \\
- \big( F_{\textrm{expl}, k}({\vec a}(t)) - F_{\textrm{expl}, k}(\vec a_{\textrm{expl}}(t)) - (t_0+t )^{-2}(\scrU {\vec b(t)})_k\Big)\Big\|_{N_2}  \lesssim ( \| \sh{\vec b} \|_{L^\infty} + \| \vec b \|_{L^\infty}) \| \sh{\vec b} - \vec b \|_{L^\infty}. 
\end{multline} 
As before, we observe that the orthogonal projections of the last two lines of~\eqref{eq:shb-b-eq} onto $\vec P_0$ vanishes identically. Applying Lemma~\ref{lem:Euler} to~\eqref{eq:shb-b-eq}, noting $\sh{\vec b}(t_0) - \vec b(t_0) = \vec 0$ and using the estimates above we find that 
\EQ{ \label{eq:bw-diff} 
\| ( \sh{\vec{\fb}}, \sh{\vec \fw}) - ( \vec \fb , \vec {\fw}) \|_{\calX}  \le \frac{1}{2} \| ( \sh{\vec{b}}, \sh{\vec w}) - ( \vec {b} , \vec {w}) \|_{\calX}
}
provided $\eta_0$ is taken sufficiently small and $t_0$ is sufficiently large (and again making use of the equation for the estimates for the derivatves). The contraction mapping principle yields a  unique trajectory $(\vec b, \vec w) \in \calB_{\calX}(8 C_0 \eta_0)$ such that $\Phi( \vec b, \vec w)= (\vec b, \vec w)$. By Lemma~\ref{lem:Euler} this is the unique solution to~\eqref{eq:shb-b-eq} amongst those trajectories $(\vec a, \vec v)$ satisfying~\eqref{eq:traj-cond} with initial data $\vec a(t_0) = \vec a_0$ proving Proposition~\ref{prop:sol-bif}. 
\end{proof}

We conclude this section with a quick corollary of Proposition~\ref{prop:sol-bif} and its proof.
%Let $T_0, \eta_0$ be as in Proposition~\ref{prop:sol-bif} and fix $t_0  \ge T_0$. We define the open set 
%\EQ{
%\calT(t_0, \eta_0) \coloneqq  \{ \vec a_0 \in \R^n \mid \max_{1\leq k < n}\bigg|\big(a_{k+1, 0} - a_{k, 0}\big) - \Big(2\log(\kappa t_0) - \log\frac{Mk(n-k)}{2}\Big)\bigg| < \eta_0
%\}
%}
Recall the definition of the open set $\calT(\eta_0)$ from~\eqref{eq:T0-def}. 
Proposition~\ref{prop:sol-bif} then defines a mapping 
\begin{multline} \label{eq:pos-to-sol} 
 \calT( \eta_0) \ni \vec a_0 \\ \mapsto ( \bs h(\vec a_0, t_0), \vec a(\vec a_0, t_0) , \vec v(\vec a_0, t_0)) \in N_1(\E) \times C^1([t_0, \infty); \R^n) \times C^1([t_0, \infty); \R^n) 
\end{multline} 
where $t_0\ge \tau_0$ is any choice of initial time for which~\eqref{eq:a-data} holds, and  $(\vec a(\vec a_0, t_0), \vec v(\vec a_0, t_0))$ are the trajectories provided by Proposition~\ref{prop:sol-bif} and  $(\bs h(\vec a_0, t_0), \vec \lam(\vec a_0, t_0),  \vec \mu(\vec a_0, t_0))$ is the solution to~\eqref{eq:kg-nonlin-eq-inf} provided by Lemma~\ref{lem:kg-nonlin-eq-inf} with $\lam_k(\vec a_0, t_0) = - (-1)^k a_k(\vec a_0, t_0)$ and $\mu_k(\vec a_0, t_0) = - (-1)^k v_k(\vec a_0, t_0)$. It is evident from the proof that given $\vec a_0 \in \calT(\eta_0)$ and $\eta_0$ as in Proposition~\ref{prop:sol-bif}, two choices of times $t_0, \ti t_0 \ge \tau_0$ for which~\eqref{eq:a-data} holds lead to trajectories that are time translations of each other, that is if $\ti t_0 \ge t_0$, 
\begin{multline}  \label{eq:tt-traj} 
(\bs h(\vec a_0, t_0; t), \vec a(\vec a_0, t_0; t), \vec v(\vec a_0, t_0; t))\\ = (\bs h(\vec a_0, \ti t_0; \ti t_0 + t - t_0), \vec a(\vec a_0, t_0; \ti t_0 + t - t_0), \vec v(\vec a_0, t_0; \ti t_0 + t - t_0))
\end{multline} 
\begin{corollary} \label{cor:hav-cont} 
The mapping defined by~\eqref{eq:pos-to-sol} is continuous. 
\end{corollary} 

\begin{proof} 
Consider two initial data $\vec a_0, \sh{\vec a}_0 \in \calT(\eta_0)$. By taking $| \vec a- \sh{\vec a}|$ sufficiently small we can ensure that~\eqref{eq:a-data} is satisfied with the same choice of $t_0$ for both $\vec a$ and $\sh{\vec a}$. Consider their images $( \bs h(\vec a_0), \vec a(\vec a_0) , \vec v(\vec a_0))$ and $( \bs{ h}( \sh {\vec a}_0), \vec a(\sh{\vec a}_0) , \vec v(\sh{\vec a}_0))$ under the mapping defined by~\eqref{eq:pos-to-sol}, noting we have removed $t_0$ from the notation to simplify it. We can assume (after a spatial translation) that $\sum_{k} a_{0, k} = 0$ and $\sum_k \sh a_{0, k} =0$ so that we can use the explicit trajectory $(\vec a_{\textrm{expl}}, \vec v_{\textrm{expl}})$ as a reference for both, writing 
\EQ{
\vec a(\vec a_0; t) &= \vec a_{\textrm{expl}}(t) + \vec b(\vec a_0; t) , \quad\vec a(\sh{\vec a}_0; t) = \vec a_{\textrm{expl}}(t) + \vec b(\sh{\vec a}_0; t)\\
\vec v(\vec a_0; t) &= \vec v_{\textrm{expl}}(t) + \vec w(\vec a_0; t) , \quad\vec v(\sh{\vec a}_0; t) = \vec v_{\textrm{expl}}(t) + \vec w(\sh{\vec a}_0; t)
}
Using the same argument used to establish the estimate~\eqref{eq:bw-diff} we obtain the estimate 
\begin{multline} 
\|( \vec b(\vec a_0), \vec w(\vec a_0)) - (\vec b(\sh{\vec a}_0), \vec w(\sh{\vec a}_0))  \|_{\calX}\\ \le  C |\vec b_0(\vec a_0; t_0) - \vec b( \sh{\vec a}_0; t_0) | + \frac{1}{2} \|( \vec b(\vec a_0), \vec w(\vec a_0)) - (\vec b(\sh{\vec a}_0), \vec w(\sh{\vec a}_0))  \|_{\calX} 
\end{multline} 
After absorbing the second term on the right above into the left, we may use the above together with~\eqref{eq:h-diff-est-final} to obtain
\EQ{
\| \bs h(\vec a_0) - \bs h(\sh{\vec a}_0) \|_{N_1} \le C \|( \vec b(\vec a_0), \vec w(\vec a_0)) - (\vec b(\sh{\vec a}_0), \vec w(\sh{\vec a}_0))  \|_{\calX}  \le  C |\vec b_0(\vec a_0; t_0) - \vec b( \sh{\vec a}_0; t_0) |.
}
This completes the proof. 
\end{proof}

\section{Proof of the uniqueness of kink clusters and corollaries} \label{sec:proofs}

We conclude with the proofs of Theorem~\ref{thm:main}, Corollary~\ref{cor:estimates} and Corollary~\ref{cor:stable}. We begin with a quick lemma, adding to the conclusions of Theorem~\ref{thm:asymptotics}.

\begin{lemma}[Dynamical Modulation Lemma]  \label{lem:dyn-modulation} Let $\bs \phi(t)$ be a kink $n$-cluster. Then there exists $t_0,  C_0>0$ such that the following are true.  The $C^1$ modulation parameters $(\vec a, \vec v) = ( \vec a( \bs \phi), \vec v( \bs \phi))$ given by Lemma~\ref{lem:static-mod} are defined on $[t_0, \infty):\R^n \times \R^n)$ and satisfy 
%$\bs h \in C([t_0, \infty); \E)$ such that $\bs \phi(t) =  \bs H( \vec a(t), \vec v(t)) + \bs h(t)$ and 
%\EQ{ \label{eq:h-sol-ortho} 
%\la \bs h(t),  \, \bs \alpha(a_k(t), v_k(t)) \ra = \la \bs h(t)  ,\, \bs \beta(a_k(t), v_k(t)) \ra = 0
%}
%for all $k  = 1, \dots, n$ and  $t \in [t_0, \infty)$. Moreover, 
%\EQ{ \label{eq:h-N1-small} 
%\lim_{ t \to \infty}  t \| \bs h(t) \|_{\E}  = 0,
%}
%and 
\EQ{ \label{eq:av-4.1-2} 
| \vec a\, '(t) - \vec v(t) | + | \vec v\, '(t)|  \le C_0 t^{-2} 
}
for all $t \ge t_0$. 
\end{lemma} 

\begin{proof} Choose $t_0$ sufficiently large so that $ \bfd( \bs \phi(t)) < \eta_0$  for all $t \in [t_0, \infty)$, where $\eta_0>0$ is as in Lemma~\ref{lem:mod-av-intro}. Define 
\EQ{
\bs h(t) = \bs \phi(t) - \bs H(\vec a( \bs \phi(t)), \vec v( \bs \phi(t)).
}
By Theorem~\ref{thm:asymptotics} we have 
\EQ{
\| \bs h \|_{N_1(\E)}< \delta
}
provided $t_0$ is sufficiently large, and where $\delta>0$ is as in Lemma~\ref{lem:kg-nonlin-eq-inf}. Because $\bs \phi(t)$ solves~\eqref{eq:csf} we see that the triplet $( \bs h(t), \vec \lam(t), \vec \mu(t))$ with 
\EQ{ \label{eq:av-lam-mu} 
\lam_k(t) = - (-1)^k a_k(t) \mand \mu_k(t) = -(-1)^k v_k(t)
} 
solves~\eqref{eq:kg-nonlin-eq-inf}. However, while the trajectories $(\vec a(t), \vec v(t))$ satisfy~\eqref{eq:JL9} we do not yet know that they satisfy~\eqref{eq:traj-cond}, particularly we have not yet established~\eqref{eq:av-4.1-2}. To see this, we use the exact same argument used to establish the estimates~\eqref{eq:nonlin-param} with $\gamma = 1$ (i.e., via~\eqref{eq:lam-mu-lin-kg}). Using~\eqref{eq:av-lam-mu} we obtain
\EQ{
 \| a_k' - v_k \|_{N_2} \lesssim 1 \mand \| v_k'  - F_k( \vec a, \vec v) \|_{N_2} \lesssim 1
}
for each $k \in \{1, \dots, n\}$. 
The first bound is the same as the required estimate for the first term in~\eqref{eq:av-4.1-2}. Using~\eqref{eq:JL9} we see that 
\EQ{
\max_{k} \| F_k( \vec a, \vec v) \|_{N_2} \lesssim 1
} 
and combining this with the above we obtain $\max_k  \| v_k' \|_{N_2} \lesssim 1$, which completes the proof. 
\end{proof}

\begin{proof}[Proof of the Theorem~\ref{thm:main}] 

Let $\tau_0, \eta_0>0$ to be fixed below, $t_0 \ge  \tau_0$ and let $\vec a_0 \in \R^n$ be as in~\eqref{eq:a0-initial}. By taking $\eta_0, \tau_0$ as in Proposition~\ref{prop:sol-bif}  we obtain a unique pair of $C^1$ trajectories  $(\vec a(t), \vec v(t))$ on $[t_0, \infty)$ satisfying the conclusions of Proposition~\ref{prop:sol-bif} with $a(t_0) = \vec a_0$. In particular,  using~\eqref{eq:a=l-v=m} and Lemma~\ref{lem:kg-nonlin-eq-inf} we can thus associate to $\vec a_0$ a unique solution $(\bs h, \vec \lam, \vec \mu)  \in C([t_0, \infty), \E) \times C^1( [t_0, \infty), \R^n \times \R^n)$ to~\eqref{eq:h-nl} with $ \| \bs h \|_{N_{1}(\E)} \le \delta$ and with $\lam_k(t) = -(-1)^k a_k(t)$ and $\mu_k(t) = -(-1)^k v_k(t)$. Writing the triplet $(\bs h(t), \vec \lam(t), \vec \mu(t)) = ( \bs h(\vec a_0, t_0; t), \vec \lam(\vec a_0, t_0; t), \vec \mu( \vec a_0, t_0; t))$, the trajectories $(\vec a(t), \vec v(t)) = ( \vec a(a_0, t_0; t), \vec v(\vec a_0, t_0; t))$ and setting 
\EQ{
\bs \phi(\vec a_0, t_0; t) = \bs H( \vec a(\vec a_0, t_0; t), \vec v(\vec a_0, t_0; t)) + \bs h(\vec a_0, t_0; t), 
}
we see that $\bs \phi(\vec a_0, t_0; t)$ solves~\eqref{eq:csf} by construction and is a kink $n$-cluster because $\bs h \in N_1(\E)$ and because the trajectories satisfy~\eqref{eq:traj-cond}. 

Next we prove uniqueness. Fix $\tau_0, \eta_0>0$ to be determined below. Let $t_0 \ge \tau_0$ and fix any  $\vec a_0$ satisfying~\eqref{eq:a0-initial}. Suppose that $\bs \phi(t) \in \calE_{(-1)^n, 1}$ is a kink $n$-cluster with 
$\bfd(\bs \phi(t)) < \eta_0$ for all $t \ge t_0$. By taking $\eta_0$ sufficiently small, by Lemma~\ref{lem:mod-av-intro} we can associate to $\bs \phi(t)$, for each $t \ge t_0$, unique modulation parameters $(\vec a( \bs \phi(t)), \vec v( \bs\phi(t)))$ such that~\eqref{eq:static-orth} holds. To prove uniqueness, we suppose in addition that $\vec a(\bs \phi(t))$ satisfies $\vec a( \bs \phi(t_0)) = \vec a_0$.  

%By ensuring $\tau_0>0$ is sufficiently large, we see that from~\eqref{eq:delta_t-rate} that Lemma~\ref{lem:mod-av}  and Lemma~\ref{lem:dyn-modulation} can also be applied for all $t \ge t_0$ and that the unique modulation parameters found above agree with those given by Lemma~\ref{lem:mod-av}. 
Defining 
\EQ{
\bs h(t) \coloneqq  \bs \phi(t) - \bs H( \vec a(\bs \phi(t)), \vec v( \bs \phi(t)), 
}
we have $ \| \bs h \|_{N_{1}(\E)} < \delta$ (for $\delta$ as in Lemma~\ref{lem:kg-nonlin-eq-inf}) by Theorem~\ref{thm:asymptotics} and $(\bs h(t), \vec \lambda(t), \vec \mu(t))$ solve~\eqref{eq:kg-nonlin-eq-inf} with $\lambda_k(t) \coloneqq  - (-1)^ka_k'( \bs \phi(t))$ and $\mu_k(t) \coloneqq  -(-1)^k v_k(\bs \phi(t))$ for all $k$. Observe that~\eqref{eq:traj-cond} is satisfied for the trajectories $(\vec a( \bs \phi(t)), \vec v( \bs \phi(t)))$ because of~\eqref{eq:JL9} and~\eqref{eq:av-4.1-2}. 
 %By~\eqref{eq:av-4.1} we see that~\eqref{eq:a0-initial} is satisfied with $\eta_0$ as in Proposition~\ref{prop:sol-bif} for $\vec a(t_0) = \vec a_0$ for all sufficiently large $t_0$. 
 Lemma~\ref{lem:kg-nonlin-eq-inf} and Proposition~\ref{prop:sol-bif} imply that  $(\bs h(t), \vec \lam(t), \vec \mu(t)) =( \bs h(\vec a_0,t_0; t), \vec \lam(\vec a_0, t_0; t), \vec \mu( \vec a_0, t_0; t))$ and thus $\bs \phi(t) = \bs \phi( \vec a_0, t_0; t)$.  
\end{proof} 

\begin{proof}[Proof of Corollary~\ref{cor:estimates}]
Let $\bs \phi(\vec a_0, t_0;  t)$ be a kink $n$-cluster as in the statement of Corollary~\ref{cor:estimates} and as above we write 
\EQ{
\bs \phi(\vec a_0,t_0; t) = \bs H( \vec a(\vec a_0,t_0; t), \vec v(\vec a_0, t_0; t)) + \bs h(\vec a_0,t_0; t). 
} 
with the modulation parameters chosen as in the statement of Theorem~\ref{thm:main}, i.e., so that~\eqref{eq:static-orth} holds. 
The estimate~\eqref{eq:h-decay-thm} follows from Lemma~\ref{lem:kg-nonlin-eq-inf}. The estimate%~\eqref{eq:traj-thm} follows from~\eqref{eq:JL9} and
~\eqref{eq:traj'-thm} follows from~\eqref{eq:av-4.1-2}.  
\end{proof}

 We conclude with the proof of Corollary~\ref{cor:stable}.

\begin{proof}[Proof of Corollary~\ref{cor:stable}]

We start by showing that the mapping~\eqref{eq:homeo} is continuous. Fix $\eta_0$ and $\calT(\eta_0)$ as in the statement, and consider $\vec a_0 \in \calT( \eta_0)$ with~\eqref{eq:a0-initial} satisfied for some $t_0 \ge \tau_0$. As usual let $\bs \phi(\vec a_0, t_0; t_0)$ denote the image of $\vec a_0$ under the map~\eqref{eq:homeo}. As discussed in  Remark~\ref{rem:tt-homeo} the image of $\vec a_0$ under this map is  independent of $t_0$ and we simply write $\bs  \phi(\vec a_0)$.

Let $\vec a_{0, n} \in \calT( \eta_0)$ be a sequence with $| \vec a_0 - \vec a_{0, n}| \to 0 $ as $n \to \infty$. For $| \vec a_0 - \vec a_{0, n}|$ sufficiently small, we see that ~\eqref{eq:a0-initial} is satisfied for $\vec a_{0, n}$ with the same $t_0$ for all sufficiently large $n$. Letting $\bs \phi( \vec a_{0, n})$ denote the images under the map~\eqref{eq:homeo} we see that $\| \bs \phi( \vec a_0) - \bs \phi( \vec a_{0, n}) \|_{\E}  \to 0$ as $n \to \infty$ by Corollary~\ref{cor:hav-cont}. 

Next, consider any element of $\calM_n$ in the image of the mapping~\eqref{eq:homeo}. By definition it can be written as $\bs \phi(\vec a_0) = \bs \phi( \vec a_0; t_0, t_0)$ for some $\vec a_0 \in \calT(\eta_0)$. By Theorem~\ref{thm:main},  the uniquely defined modulation parameters $(\vec a( \bs \phi(\vec a_0)), \vec v( \bs \phi( \vec a_0)))$ from Lemma~\ref{lem:mod-av-intro} satisfy $\vec a( \bs \phi(\vec a_0)) = \vec a_0$,  giving the inverse map, which is continuous by Lemma~\ref{lem:mod-av-intro}.

Next, let $\bs \phi \in \calM_n$ be any choice of initial data and let $\bs \phi(t)$ denote its forward trajectory. By~\eqref{eq:JL9} we see that $\bs \phi(t) \in \calM_{n, \loc} = \bs \Phi( \calT(\eta_0))$ for all sufficiently large $t$. By choosing $t_0$ large enough, we may write $\bs \phi(t) = \bs\phi( \vec a( \bs \phi(t_0)), t_0; t)$.  To show that  $\calM_n$ is a topological manifold, we still need to show that an open neighborhood of our arbitrary choice of the initial data $\bs \phi \in \calM_n$ is homeomorphic to an open subset of $\R^n$. But this is a consequence of the local well-posedness theory for~\eqref{eq:csf} by pulling back an open neighborhood of $\bs \phi(\vec a( \bs \phi(t_0)), t_0; t_0)$  in $\calM_{n, \loc}$ via the continuous flow by time $t_0$. 
\end{proof}

\bibliographystyle{plain}
\bibliography{kink-clusters}

\bigskip
\centerline{\scshape Jacek Jendrej}
\smallskip
{\footnotesize
 \centerline{Institut de Math\'ematiques de Jussieu,
Sorbonne Universit\'e}
\centerline{4 place Jussieu,
75005 Paris, France}
\centerline{\email{jendrej@imj-prg.fr}}
} 
\medskip 
\centerline{\scshape Andrew Lawrie}
\smallskip
{\footnotesize
 \centerline{Department of Mathematics, University of Maryland}
\centerline{4176 Campus Drive -- William E. Kirwan Hall,
College Park, MD 20742-4015, U.S.A.}
\centerline{\email{alawrie@umd.edu}}
} 

\end{document}